\documentclass[11pt,reqno]{amsart}
\def\MR#1{}
\usepackage[margin=37mm]{geometry}
\usepackage{amsfonts,amsmath,amssymb,enumerate}
\usepackage{amsthm, bm}
\usepackage{mathrsfs}
\usepackage{setspace}
\usepackage{comment}

\usepackage{pgfplots}
\usepackage{tikz-3dplot}
\usetikzlibrary{arrows,matrix}
\usetikzlibrary{decorations.markings,shapes.geometric,shapes.misc,cd}
\usetikzlibrary{3d}
\usetikzlibrary {graphs.standard}
\pgfplotsset{compat=1.17}
\usepgfplotslibrary{fillbetween}

\usepgfplotslibrary{external}
\def\noexternalize{yes}
\ifdefined\noexternalize\else
\tikzset{external/up to date check = md5}
\fi

\usepackage{mathtools}
\usepackage{hyperref}
\usepackage{todonotes}
\usepackage{aliascnt}

\usepackage[capitalise]{cleveref}
\crefname{subsection}{\S\!}{\S\!}
\crefname{subappendix}{\S\!}{\S\!}
\Crefname{subsection}{\S\!}{\S\!}
\Crefname{subappendix}{\S\!}{\S\!}
\usepackage{upgreek}
\usepackage{simpler-wick}

\usepackage[hyperref,doi,url=false,
backref,
giveninits=true,
backend=biber,
style=alphabetic,
maxbibnames=99]{biblatex}
\usepackage{xpatch}

\makeatletter
\patchcmd\blx@bblinput{\blx@blxinit}{%
  \blx@blxinit
}{}{\fail}
\makeatother

\AtEveryBibitem{% Clean up the bibtex rather than editing it
 \clearlist{address}
 \clearfield{date}
 \clearfield{isbn}
 \clearfield{issn}
 \clearlist{location}
 \clearfield{month}
 \clearfield{series}

 \ifentrytype{book}{}{% Remove publisher and editor except for books
  \clearlist{publisher}
  \clearname{editor}
 }
 \ifentrytype{misc}{}{% Remove arxiv number except for preprints
  \clearfield{eprint}
 }
}

\crefname{equation}{}{}
\Crefname{equation}{}{}

\DeclareFontEncoding{LS1}{}{}
\DeclareFontSubstitution{LS1}{stix}{m}{n}
\DeclareSymbolFont{symbols2}{LS1}{stixfrak}{m}{n}
\DeclareMathSymbol{\typecolon}{\mathbin}{symbols2}{"25}
\theoremstyle{plain}% default

\newtheorem{thm}{Theorem}
\newtheorem*{prop*}{Proposition}
\newtheorem*{thm*}{Theorem}
\let\cref\Cref
\newcommand{\NewTheoremWithAlias}[2]{%
  \newaliascnt{#1}{thm}%
  \newtheorem{#1}[#1]{#2}%
  \aliascntresetthe{#1}%
  \Crefname{#1}{#2}{#2s}%
}

\NewTheoremWithAlias{lem}{Lemma}
\NewTheoremWithAlias{prop}{Proposition}
\NewTheoremWithAlias{cor}{Corollary}
\NewTheoremWithAlias{defn}{Definition}
\NewTheoremWithAlias{defprop}{Definition-Proposition}

\NewTheoremWithAlias{examplex}{Example}
\newenvironment{exmp}
  {\pushQED{\qed}\examplex}
  {\popQED\endexamplex}

\theoremstyle{remark}
\NewTheoremWithAlias{remarkx}{Remark}

\newenvironment{rem}
  {\pushQED{\qed}\remarkx}
  {\popQED\endremarkx}

\newtheorem*{rem*}{Remark}

\newcommand{\be}{\begin{equation}}
\newcommand{\ee}{\end{equation}}
\newcommand{\bmx}{\begin{pmatrix}}
\newcommand{\emx}{\end{pmatrix}}

\newcommand{\bmb}{\begin{matrix}}
\newcommand{\bme}{\end{matrix}}

\newcommand{\bbb}{\!\begin{bmatrix}}
\newcommand{\bbe}{\end{bmatrix}}

\newcommand{\ttb}{\!\left[\!\!\left[\begin{matrix}}
\newcommand{\tte}{\end{matrix}\right]\!\!\right]}

\newcommand{\llb}{\hspace{-.15em}\left(\!\!\!\left(\begin{matrix}}
\newcommand{\lle}{\end{matrix}\right)\!\!\!\right)}

\newcommand{\ul}{\underline}
\newcommand{\del}{\partial}
\newcommand{\g}{{\mathfrak g}}

\newcommand{\p}{{\mathfrak p}}

\newcommand{\B}{{\mathsf B}}

\newcommand{\mc}{\mathcal}

\newcommand{\D}{\mathcal{D}}

\newcommand{\nn}{\nonumber}
\newcommand{\sign}{{\rm sign}}
\newcommand{\X}{{\tilde X}}
\newcommand{\8}{{\infty}}

\newcommand{\ZZ}{{\mathbb Z}}

\newcommand{\CC}{{\mathbb C}}
\newcommand{\kk}{{\mathbf k}}

\renewcommand{\AA}{{\mathbb A}}

\renewcommand{\P}{{\mathcal P}}

\newcommand{\RR}{{\mathbb R}}

\newcommand{\id}{{\textup{id}}}

\newcommand{\wh}{\widehat}

\newcommand{\Qc}[2]{\pi_i \circ Q \circ \pi_j}

\newcommand{\A}{\mathcal A}

\newcommand{\goi}[2]{=}
\DeclareMathOperator{\Hom}{Hom}

\newcommand{\on}{.}

\renewcommand{\binom}[2]{{#1 \brack #2}}

\newcommand{\btp}{\begin{tikzpicture}[baseline=0pt,scale=0.9,line width=0.25pt]}
\newcommand{\etp}{\end{tikzpicture}}

\DeclareMathOperator{\res}{res}
\DeclareMathOperator{\Res}{Res}

\DeclareMathOperator{\Span}{span}

\DeclareMathOperator{\Spec}{Spec}

\DeclareMathOperator{\gr}{gr}

\renewcommand{\O}{\mc O}

\newcommand{\M}{\mathcal M}

\DeclareMathOperator{\Ob}{Ob}

\newcommand{\ox}{\mathbin\otimes}

\newcommand{\bul}{\bullet}

\newcommand{\directlim}{\varinjlim}

\newcommand{\into}{\hookrightarrow}

\newcommand{\onto}{\twoheadrightarrow}

\renewcommand{\i}{\mathsf i}

\def\B{\mc B}

\newcommand{\op}{\mathrm{op}}

\newcommand{\G}{\mc G}
\newcommand{\F}{\mc F}

\newcommand{\isom}{\xrightarrow\sim}

\newcommand{\x}{{\bm x}}

\newcommand{\I}{\mathcal I}

\newcommand{\C}{\mathcal C}

\makeatletter
\newcommand{\extp}{\@ifnextchar^\@extp{\@extp^{\,}}}
\def\@extp^#1{\mathop{\bigwedge\nolimits^{\!#1}}}
\makeatother
\makeatletter
\newcommand{\hextp}{\@ifnextchar^\@hextp{\@hextp^{\,}}}

\def\@hextp^#1{\mathop{\wh{\bigwedge\nolimits^{\!#1}}}}
\makeatother

\newcommand{\catname}[1]{\mathbf{#1}}

\newcommand{\Set}{\catname{Set}}

\newcommand{\Operad}[1]{\mathcal #1}

\newcommand{\LieOperad}{\Operad{Lie}}
\newcommand{\ComOperad}{\Operad{Com}}

\newcommand{\LieInfinityOperad}{\Operad{Lie}_\8}

\newcommand{\CAlg}{\catname{Alg}^{\ComOperad}}

\newcommand{\dgCAlg}{\CAlg(\dgVect_\kk)}

\newcommand{\dgVect}{\catname{dgVect}}

\newcommand{\Conf}{\mathrm{Conf}}

\newcommand{\dd}{\mathrm{d}}

\DeclareMathOperator{\Tot}{Tot}

\DeclareMathOperator{\holim}{holim}

\newcommand{\Cech}{\check C}

\newcommand{\zz}{\mathbf z}
\DeclareMathOperator{\Th}{Th}

\DeclareMathOperator{\Jac}{Jac}

\newcommand{\dfn}[1]{%{\color{blue}
\emph{#1}}
\newcommand{\Linfinity}{{\texorpdfstring{$L_\8\,$}{L-infinity}}}

\newcommand{\ijk}{_{1\leq i<j\leq k}}

\newcommand{\TS}{{\mathbf{P}}}

\newcommand{\ChirComp}[0]{\mathcal{C}}

\newcommand{\ChirCompTot}[0]{\mathcal{F}}
\newcommand{\Coh}[0]{H}

\newcommand{\tl}[0]{t^\lambda}

\newcommand{\shifted}{^\blacklozenge}
\newcommand{\two}{^{[2]}}
\newcommand{\bp}{\circ}

\newcommand{\DAAnk}{\mathcal D_{(\AA^n)^k}}
\newcommand{\ddts}{\dd_{\TS}}
\newcommand{\ddmu}[1]{\dd_{\mu^{#1}}}
\newcommand{\Vol}{\mathrm{Vol}}
\newcommand{\mTree}{\mathbf{mTree}}
\newcommand{\Prop}{P}
\newcommand{\oms}{\omega\shifted_{\AA^n}}
\newcommand{\om}{\omega_{\AA^n}}
\newcommand\xonto[2][]{%
  \mathrel{\ooalign{$\xrightarrow[#1\mkern4mu]{#2\mkern4mu}$\cr%
  \hidewidth$\rightarrow\mkern4mu$}}
}

\newcommand{\basepoint}{\bp}

\newcommand{\polysimplicial}{polysimplicial }
\newcommand{\topsimplex}{[N]}

\crefname{lem}{lemma}{lemmas}
\Crefname{lem}{Lemma}{Lemmas}
\crefname{prop}{proposition}{propositions}
\Crefname{prop}{Proposition}{Propositions}

\begin{document}

\title[Higher chiral algebras in a polysimplicial model]{Higher chiral algebras\\ in a polysimplicial model}

\author{Laura O. Felder, Zhengping Gui, Charles A. S. Young}
\address{}  \email{l.o.felder@herts.ac.uk, zgui@simis.cn, c.young8@herts.ac.uk}
%\date{\today}

\ifdefined\noexternalize\else\tikzexternaldisable\fi
%\begin{comment}

\address{ZG: Shanghai Institute for Mathematics and Interdisciplinary Sciences, Shanghai, China; LOF,CASY: Department of Physics, Astronomy and Mathematics, University of Hertfordshire, College Lane, Hatfield AL10 9AB, UK. }%

\thanks{}%
\subjclass{}%
\keywords{}%

%\date{}%
%\dedicatory{}%
%\commby{}%
% ----------------------------------------------------------------

\begin{abstract}
%We give an algebraic construction of a chiral algebra on $\AA^n$. Since vertex algebras can be regarded as chiral algebras on $\AA^1$, in the sense of Beilinson and Drinfeld \cite{BDChiralAlgebras}, this can be seen as a construction of a higher dimensional vertex algebra.

Vertex algebras are equivalent to
translation-equivariant
chiral algebras on $\AA^1$, in the sense of Beilinson and Drinfeld \cite{BDChiralAlgebras}.
In this paper we give an algebraic construction of a chiral algebra on $\AA^n$; this can be seen as an algebraic construction of a
higher-dimensional vertex algebra.

We introduce a model, in dg commutative algebras, of the derived algebra of functions on the configuration space of $k$ distinct labelled marked points in $\AA^n$. Working in this model -- which we call the \emph{polysimplicial model} -- we obtain a dg operad of chiral operations on a degree-shifted copy of the canonical sheaf. We prove that there is a quasi-isomorphism, to this dg operad, from the Lie-infinity operad.
This result makes the shifted canonical sheaf into a first example of a \emph{homotopy polysimplicial chiral algebra} on $\AA^n$, in a sense which generalizes to higher dimensions Malikov and Schechtman's \cite{MalikovSchechtman} notion of a homotopy chiral algebra.

%It can also be expected to furnish an example of a chiral algebra in the sense of Francis and Gaitsgory \cite{FrancisGaitsgory}.

%As a biproduct, we compute the cohomology of the structure sheaf on configuration space.

\end{abstract}

% 17B67  	Kac-Moody (super)algebras; extended affine Lie algebras; toroidal Lie algebras
%17B69  	Vertex operators; vertex operator algebras and related structures
%81R10  	Infinite-dimensional groups and algebras motivated by physics, including Virasoro, Kac-Moody, $W$-algebras and other current algebras and their representations
%81R12  	Groups and algebras in quantum theory and relations with integrable systems

\maketitle
\setcounter{tocdepth}{1}
\tableofcontents

\section{Introduction}
\subsection{Motivation}
%Our goal in this paper is to introduce, in the very simplest example, a concrete model of chiral algebras in higher dimensions.

In this paper, we give an algebraic construction of a chiral algebra in higher dimensions. To motivate this, and to set the scene, let us recall the relationship between vertex algebras and chiral algebras.

Vertex algebras formalize the properties of local observables in chiral conformal field theories in complex dimension one \cite{Borcherds,KacVertexAlgBook,LLbook,FrenkelBenZvi}. They are a powerful and ubiquitous tool in mathematical physics and representation theory. %{\color{gray}  -- appearing in, for example in the 4D/2D correspondence \cite{xxx} and 3D/2D correspondence \cite{xxx} \dots}

Now, a running theme in recent years has been that important algebraic objects in mathematical physics which appear to be peculiar to complex dimension one sometimes \emph{can} be generalized to higher complex dimensions, provided one is prepared to go to the setting of higher and derived algebras. % and to work up to homotopy.
For example, current algebras and their Kac-Moody central extensions, which are associated to the punctured formal disc in complex dimension one, turn out to have higher dimensional analogs \cite{FHK}, \cite{GWHigherKM}, once one recognizes that one should pass from the algebra of functions on the punctured disc to its derived analog.

It is an important open problem to construct %(algebraic models of) analogs of
vertex algebras in higher dimensions. At first sight it is perhaps far from obvious that this ought to be possible. For example, vertex algebras enjoy a form of associativity -- Borcherds identity -- which follows naturally from physical arguments about applying the residue theorem to chiral correlators/conformal blocks in complex dimension one, but which is rather sui generis from an algebraic perspective.

To make progress, it is valuable to pass to Beilinson and Drinfeld's elegant language of chiral algebras %(and their Koszul duals, factorization algebras)
 \cite{BDChiralAlgebras}. Every vertex algebra has an equivalent description as a chiral algebra\footnote{Specifically, a translation-equivariant chiral algebra on $\AA^1$. See \cite[\S0.15]{BDChiralAlgebras}.}  on $\AA^1$ and crucially a chiral algebra is just a certain sort of Lie algebra. %(More precisely, it is a Lie object in certain pseudo-tensor category of right $\D$-modules and chiral operations.)
In particular, the associativity-type relation is just the usual Jacobi identity.
This opens a path to generalization. And indeed, it has been known since the seminal work of Francis and Gaitsgory \cite{FrancisGaitsgory} that the notion of chiral algebras extends to higher dimensions. The paper \cite{FrancisGaitsgory} works throughout at the level of the derived/infinity categories, independently of any choice of model, and it relaxes the strict notion of chiral algebras of \cite{BDChiralAlgebras} so that, roughly speaking, \emph{all} structures are up to homotopy.

For the purposes of applications -- in particular, to allow explicit computations -- one would like to have at hand %tractable models of higher chiral algebras. In other words, one would like
an algebraic \emph{construction} of higher chiral algebras of interest. In the present work we give such a construction.  We confine our attention to the very simplest example, that of the unit chiral algebra in affine $n$-space $\AA^n$. We construct it explicitly, working in a model we introduce which we shall refer to as the \emph{polysimplicial} model.\footnote{One certainly expects that other examples of higher vertex/chiral algebras interest can be similarly constructed using this model, for example the free boson, free fermion, $\beta\gamma$/$bc$ ghost system, Kac-Moody, and so on. We hope to return to this in future work.}

It will turn out that our notion of chiral algebras in the polysimplicial model is an up-to-homotopy generalization of that of \cite{BDChiralAlgebras}, but in a much milder sense than that of \cite{FrancisGaitsgory}. What we obtain is, rather, an example in higher dimensions of a \emph{homotopy chiral algebra} in a sense Malikov and Schechtman have recently introduced in dimension one \cite{MalikovSchechtman}.\footnote{Although we will not make the connection explicit here, polysimplical chiral algebras in our sense should also furnish examples of chiral algebras in the sense of \cite{FrancisGaitsgory}.}
To explain that statement, let us now give an overview of the content of the paper.

\subsection{Overview}
\subsubsection*{Chiral operations on $\AA^n$}
Let $\Conf_k(\AA^n)$ be the configuration space of $k$ distinguishable marked points in $\AA^n$. We denote by $j:\Conf_k(\AA^n)\hookrightarrow (\AA^n)^k$ the open inclusion and by $\Delta: \AA^n\hookrightarrow (\AA^n)^k$ the diagonal embedding.

In $n=1$ dimensions, a $k$-ary \dfn{chiral operation}
\( (L_1,\dots,L_k) \to M \)
on $\AA^1$ in the sense of \cite{BDChiralAlgebras} is a map
\begin{equation}
j_*j^*( L_1\boxtimes \dots \boxtimes L_k %\underbrace{M \boxtimes \dots \boxtimes M}_k
) \to \Delta_* M \nn\end{equation}
(of right $\D$-modules on $(\AA^1)^k$; here $L_1,\dots,L_k$ and $M$ are quasi-coherent right $\D$-modules on $\AA^1$; cf.  \cref{sec: main}).

Such a map takes as input a section %of $L_1\boxtimes \dots \boxtimes L_k$
over the configuration space of $k$ distinguishable marked points in $\AA^1$ -- in other words, a section which is allowed to be singular along the diagonals in $(\AA^1)^k$, where collisions of marked points occur. That configuration space happens to be affine as a scheme over the ground field $\kk$, namely $\Conf_k(\AA^1) = \Spec\left( \kk[z_i,(z_j-z_\ell)^{-1}]_{1\leq i\leq k;1\leq j < \ell \leq  k}\right)$.

In higher dimensions, $n\geq 2$, the configuration space %$\Conf_k(\AA^n)$
of $k$ distinguishable marked points in $\AA^n$ is no longer affine. % and its structure sheaf has higher cohomology.
We shall introduce a model,
\[ \TS_{\AA^n}^k \simeq R\Gamma(\Conf_k(\AA^n),\O), \]
in dg commutative $\kk$-algebras, of the derived sections of the structure sheaf $\O$. It carries a $\D_{(\AA^n)^k}$-module structure. We also get a model of the derived sections of any quasi-coherent sheaf of right $\D$-modules on $(\AA^n)^k$, the relevant examples being those of the form %$L_1\boxtimes\dots\boxtimes L_k$,
$\TS_{\AA^n}^k(L_1\boxtimes \dots \boxtimes L_k)% &:= \TS_{\AA^n}^k \ox_{\O_{\AA^n}^{\ox k}} \left( (L_1)_{\AA^n} \ox_\kk \dots \ox_\kk (L_k)_{\AA^n}\right) \nn\\&
\simeq R\Gamma\left(\Conf_k(\AA^n) , L_1\boxtimes \dots \boxtimes L_k \right)\simeq Rj_*j^*( L_1\boxtimes \dots \boxtimes L_k).
\nn$ Here we abuse the notation of a sheaf and its sections, see \cref{thm: F model} for the precise statement.

It is then natural to define a $k$-ary \dfn{chiral operation} \( (L_1,\dots,L_k) \to M \) on $\AA^n$ in this polysimplicial model to be a map
\begin{align}  \TS_{\AA^n}^k(L_1\boxtimes \dots \boxtimes L_k) \to \Delta_*M
\nn\end{align}
(again of right $\D$-modules, now on $(\AA^n)^k$; see \cref{sec: main} for details, and cf. \cref{rem: reln to DS}).

\subsubsection*{The unit chiral algebra $\oms$ on $\AA^n$} In particular we get the $k$-ary operations
\[\TS_{\AA^n}^k(\underbrace{\oms\boxtimes\dots\oms}_{k}) \to \Delta_*\oms\]
from a degree-shifted copy
\[\oms:=\omega_{\AA^n}[n-1]\]
of the canonical sheaf $\omega$ on $\AA^n$ to itself. We shall show
(in \cref{sec: operad})
that these operations form a dg operad, $\P^{ch}_{\AA^n,\bul}$,
%\begin{equation} \P^{ch}_{\AA^n}(k) := \Hom_{\DAAnk}\left(\TS_{\AA^n}^{k}(\omega_{\AA^n}[n-1]^{\boxtimes k}), \Delta_*\omega_{\AA^n}[n-1] \right),\nn\end{equation}
which we call the \dfn{polysimplical unit chiral operad}.
The main result of the paper is then the following.
\begin{thm*}[\cref{thm: quasi-isomorphism from LieInfinity}]
    There is a quasi-isomorphism of dg operads
\[\LieInfinityOperad\xrightarrow\sim \P_{\AA^n}^{ch}\]
from the Lie-infinity operad $\LieInfinityOperad$ to the polysimplicial unit chiral operad $\P_{\AA^n}^{ch}$.
\qed\end{thm*}
This result is a natural generalization (in some sense as mild as one could hope for)\footnote{What we mean by this is that the choice of model appears to be somewhat fortunate here.
We expect that one could define, analogously, a unit chiral operad $\mc Q_{\AA^n}^{ch}$ in various other models, but in general (to the authors' knowledge) one cannot always expect to get a quasi-isomorphism $\LieInfinityOperad\xrightarrow\sim\mc Q_{\AA^n}^{ch}$.}
of the statement that there is an \emph{iso}morphism of operads $\LieOperad\cong \P_{\AA^1}^{ch}$ in the one-dimensional case (for which see \cite[Theorem 3.1.5]{BDChiralAlgebras} with $X=\AA^1$).

It makes the shifted canonical sheaf $\oms$ into a first example of a \dfn{homotopy chiral algebra on $\AA^n$}, cf. \cite{MalikovSchechtman} and \cref{sec: homotopy polysimplicial chiral algebras}. We call it the \dfn{unit chiral algebra on $\AA^n$, in the polysimplicial model}.
If one accepts that chiral algebras on $\AA^n$ are the natural algebraic generalization of vertex algebras to higher dimensions, then this theorem can be seen as providing the first example, $\oms$, of a higher-dimensional vertex algebra.

\subsubsection*{Propagators and Arnold relations} To explain more explicitly the structures with which the unit chiral algebra $\oms$ comes equipped, we first give a rough description of the polysimplicial model $\TS_{\AA^n}^k$ itself (see \cref{sec: higher configuration space and movable punctures in two complex dimensions} for the precise statements).

It is helpful to start with the case of $k=2$ marked points. An element of \[\TS_{\AA^n}^2\simeq R\Gamma(\Conf_2(\AA^n),\O)\] is by definition a polynomial differential form on an $n-1$ simplex $\triangle_{n-1}$, taking coefficients in the algebra $\kk[z^r_1,z^r_2,(z^r_1-z^r_2)^{-1}]^{1\leq r\leq n}$ % (which is the algebra of sections of the structure sheaf over $(\AA^n)^2$ with all of the lines $z^r_1=z^r_2$ removed)
, and obeying boundary conditions: let $u^r$, $1\leq r\leq n$, be the coordinates on $\triangle_{n-1}$ (with $\sum_{r=1}^n u^r =1$); then the pullback to the face at $u^r=0$ must be regular in $z_1^r-z_2^r$, for each $r$.

%Here $z_1^r$ and $z_2^r$ are the coordinate functions on $\AA^n\times \AA^n$, and one should think of the simplex $\triangle_{n-1}$ as somehow auxiliary.

The cohomology of $\TS_{\AA^n}^2$ is concentrated in cohomological degrees $0$ and $n-1$. It is generated (as a $\D$-module, and not freely) by the cohomology classes $[1]$ and $[\Prop_{12}]$ of respectively the unit element $1$ and what we shall call the \dfn{propagator}\footnote{Since $\Conf_2(\AA^n) \cong (\AA^n\setminus \{0\}) \times \AA^n$, we are implicitly also describing here a model of the derived sections on punctured affine $n$-space. (More precisely, we have $\TS_{\AA^n}^k \simeq R\Gamma(\AA^n\setminus\{0\},\O)\ox_\kk \Gamma(\AA^n,\O)$.) The reader may find it instructive to compare this model to the Jouanolou model \[A_n^\bul\simeq R\Gamma(\AA^n\setminus\{0\},\O)\] of \cite{FHK}, and in particular to compare the propagator $\Prop_{12}$ to the Martinelli-Bochner form.}
\be \Prop_{12} := \frac{\dd u^1 \dots \dd u^{n-1} %\Vol(\triangle_{n-1})
}{(z^1_1-z^1_2) \dots (z^n_1-z^n_2)} \in \TS_{\AA^n}^{2,n-1}. \nn\ee

%The passage to $k\geq 2$ marked points is just a matter of more of the same, roughly speaking.
Elements of $\TS_{\AA^n}^k\simeq R\Gamma(\Conf_k(\AA^n),\O)$ are then polynomial differential forms on a $\binom k 2$-fold \emph{product} of simplices $\prod_{1\leq i<j\leq k} \triangle_{n-1}^{(i,j)}$ where the copy $\triangle_{n-1}^{(i,j)}\cong \triangle_{n-1}$ is associated to the pair of marked points labelled $i$, $j$.

We compute the cohomology $H^\bul(\TS_{\AA^n}^k) = H^\bul(\Conf_k(\AA^n),\O)$ explicitly in \cref{sec: cohomology}. It is again generated by the classes $[\Prop_{ij}]$ of propagators, $\Prop_{ij}$, $1\leq i<j\leq k$. Recall that $\TS_{\AA^n}^k$ is a model in dg commutative algebras. We shall show that
\be  \Prop_{ij} \Prop_{j\ell}+\Prop_{j\ell} \Prop_{\ell i}+\Prop_{\ell i} \Prop_{ij} \in \dd \left(\TS_{\AA^n}^k\right) \nn\ee
is exact, for all $1\leq i<j<\ell\leq k$. Thus, an analog of the usual Arnold relations holds at the level of cohomology.
See \cref{sec: arnold relations}.

\begin{figure}
  \[\begin{tikzpicture}[every node/.style={font=\scriptsize}]
    \begin{axis}[scale=1.5,axis equal,view={110}{25},xlabel={$u_{12}$},
      ylabel={$u_{13}$}, zlabel={$u_{23}$},xmin=0,xmax=1,ymin=0,ymax=1,zmin=0,zmax=1,xtick={0,1},ytick={0,1},ztick={0,1},clip]

    % Draw the edges of the cube
\draw[thick,dotted] (1,0,0) -- (0,0,0) -- (0,1,0) (0,0,0) -- (0,0,1);
\draw[thick] (1,0,0) -- (1,1,0) -- (1,1,1) -- (1,0,1) -- cycle (1,0,1) -- (0,0,1) -- (0,1,1) -- (0,1,0) -- (1,1,0) (1,1,1) -- (0,1,1);

\draw[thick,opacity=0.5,fill=green,green]   (0,0,0) -- (1,1,0) -- (1,1,1) -- (0,0,1) -- cycle;
\draw[thick,opacity=0.5,fill=purple,purple] (0,0,0) -- (1,0,1) -- (1,1,1) -- (0,1,0) -- cycle;
\draw[thick,opacity=0.5,fill=yellow,yellow] (0,0,0) -- (0,1,1) -- (1,1,1) -- (1,0,0) -- cycle;

%  \draw[thick,dotted] (G) -- (2p1p3)-- (G)--(3p2p1)-- (G) -- (1p3p2);

\end{axis}
  \end{tikzpicture}\]
\caption{Integration domains for the Jacobiator $\Jac_3$ in $n=2$ dimensions; see \cref{sec: examples} \label{fig}}
\end{figure}
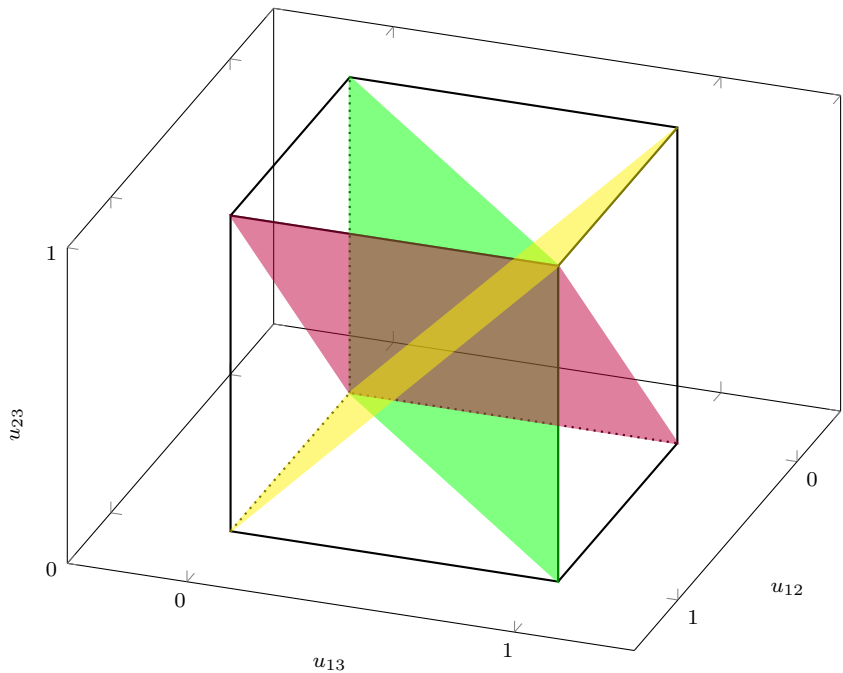

\subsubsection*{The chiral bracket $\mu_2$ and higher chiral operations $(\mu_k)_{k \geq 3}$}\label{sec: higher operations}
Now we can explain the content of our main result, \cref{thm: quasi-isomorphism from LieInfinity}, more concretely. We first define the \dfn{residue map} or \dfn{chiral bracket}.
It is a binary operation of the polysimplicial unit chiral operad, $\mu_2 \in \P^{ch}_{\AA^n,0}(2)$, given explicitly by (the following $\lambda$ variables come from $\Delta_*\oms$)
\begin{align} \mu_2 &:= %\int_{(\triangle_{n-1})_{12}}
\int_{\triangle_{n-1}}
\res_{z^1_2 \to z^1_1} e^{\lambda^1_2(z^1_2- z^1_1) } \dots \res_{z^n_2 \to z^n_1} e^{\lambda^n_2(z^n_2- z^n_1) } .\nn
\end{align}
% where
% \be \int_{(\triangle_{n-1})_{12}} = \int_{u^1_{12}=0}^1 \dots \int_{u^{n-1}_{12}=0}^1 \nn\ee
% is the integral over the $(n-1)$-simplex $\triangle_{n-1}$.
Observe that the degree shift in the definition of $\oms:= \omega_{\AA^n}[n-1]$ is chosen to ensure that the map $\mu_2$ itself is in degree zero. It is a closed element and on the cohomology it gives the Sato construction (for which see e.g. \cite[\S2.7.19]{ginzburg1998lectures}).

The maps \((\oms)^{\boxtimes 2} \into \TS_{\AA^n}^2((\oms)^{\boxtimes 2}) \xrightarrow{\mu_2} \Delta_* \oms\) of right $\D$-modules on $(\AA^n)^k$ induce a short exact sequence of the cohomologies,
\be 0 \to (\oms)^{\boxtimes 2} \to H_{\ddts}( \TS_{\AA^n}^{2,\bul}((\oms)^{\boxtimes 2})) \xrightarrow{\mu_2} \Delta_* \oms \to 0.  \nn\ee
The structure can be pictured as follows. (Here $\dd\zz = \dd z^1\dots \dd z^n$, $p$ and $q$ are polynomials, and $1\leq r\leq n$.)
\be\begin{tikzpicture}
\matrix (m) [matrix of math nodes, row sep=.2em,
column sep=2em, text height=2ex, text depth=1ex]
{
{} & (\oms)^{\boxtimes 2} & \TS_{\AA^n}^{2,\bul}((\oms)^{\boxtimes 2}) & \Delta_* \oms \\
1-n & 0  & p(z^r_1,\del_{z^r_2}) P_{12} \dd\zz_1\dd\zz_2   &  \pm\dd\zz_1p(z^r_1,\lambda^r_2)        \\
\vdots  & 0 & * & 0     \\
\vdots  & \vdots & \vdots & \vdots     \\
\vdots  & 0 & * & 0     \\
2-2n  & q(z_1^r,z_2^r) \dd\zz_1 \dd\zz_2 & q(z_1^r,z_2^r) \dd\zz_1\dd\zz_2 &  0    \\
};
\path[->,font=\scriptsize,shorten <= 2mm,shorten >= 2mm]
(m-1-2) edge[right hook->]
node [above] {} (m-1-3)
(m-1-3) edge[->] node [above] {} (m-1-4);
\path[->,font=\scriptsize,shorten <= 2mm,shorten >= 2mm]
(m-6-2) edge[->] node [above] {} (m-6-3)
(m-2-3) edge[->] node [above] {$\mu_2$} (m-2-4);
\path ($(m-4-1)- (1,0ex)$) node[rotate=90,font=\scriptsize,align=center]  {Cohomological degree};
\draw[black,line width=0.4pt] ($(m-1-1)-(1,2ex)$) -- ($(m-1-4)-(-1,2ex)$);
\draw[black,line width=0.4pt] ($(m-1-1)-(-1,2ex)$) -- ($(m-6-1)-(-1,2ex)$);
\end{tikzpicture}\nn\ee

In the special case of dimension $n=1$, $\mu_2$ reduces to the usual $\D$-module residue map $\res_{z_2 \to z_1} e^{\lambda_2(z_2-z_1)}$ which is the chiral bracket of the unit chiral algebra $\omega_{\AA^1}$. That map is a Lie bracket, i.e., it obeys the Jacobi identity.
In all higher dimensions $n>1$, $\mu_2$ obeys the Jacobi identity only up to homotopy. %a correction term $\mu_3 \circ \ddts$.  %(Compare \cref{exmp: cousin} in \cref{sec: Chevalley-Cousin complex}.)

Indeed, recall that $\LieInfinityOperad$ is the free operad in graded vector spaces generated by operations $\ell_p\in \LieInfinityOperad(p), p\geq 2$, where $\ell_p$ is skewsymmetric and has homological degree $p-2$ (see \cite[\S10.1.6]{LodayVallette}). It becomes a dg operad with the differential given by
\[
\dd \ell_p=\sum_{\substack{ p_1+p_2=p+1\\p_1,p_2>1}} \sum_{\sigma\in \mathrm{Sh}^{-1}_{p_1-1,p_2}}\mathrm{sgn}(\sigma)(-1)^{(p_1-1)p_2}\left(\ell_{p_1}\circ_1\ell_{p_2}\right)^{\sigma}.
\]
Here $\mathrm{Sh}^{-1}_{p_1-1,p_2}$ is the set of $(p_1-1,p_2)$-unshuffles -- see \cite[\S1.3.2]{LodayVallette} -- and $\mathrm{sgn}$ is the signature.

Thus, our main result asserts the existence of higher chiral operations $\mu_p\in \P^{ch}_{\AA^n,p-2}(p)$, $p=3,4,\dots$, of the polysimplicial unit chiral operad, obeying the coherence relations
\begin{align}
  (\mu_p\circ \ddts)(-) & =\sum_{\substack{ p_1+p_2=p+1\\p_1,p_2>1}} \sum_{\sigma\in \mathrm{Sh}^{-1}_{p_1-1,p_2}}\mathrm{sgn}(\sigma)(-1)^{(p_1-1)p_2} (\mu_{p_1}\circ_1\mu_{p_2})^{\sigma}(-)
\label{mucoherence}\end{align}
and the skewsymmetry property $\mu_p^\sigma = \mathrm{sgn}(\sigma)\mu_p$.

(Here we are glossing over how partial composition $\circ_1$ and the action of the symmetric group are actually defined in the polysimplicial unit chiral operad. See \cref{sec: operad}.)

In \cref{sec: examples} we write down the first of these higher chiral products explicitly in the simplest non-trivial example, that of $n=2$ dimensions.

\subsubsection{The Chevalley-Cousin complex} In \cref{sec: Chevalley-Cousin complex} we construct the \dfn{Chevalley-Cousin complex} for the polysimplicial unit chiral algebra $\oms$ on $\AA^n$. It is a complex whose differential encodes the \Linfinity coherence relations \cref{mucoherence} obeyed by the chiral products.
From the algebraic perspective, it is the chiral analog of the Chevalley-Eilenberg chain complex $C(\g)$ one associates to an \Linfinity-algebra $\g$ in dg vector spaces.

In the case of chiral algebras on $\AA^1$, the Chevalley-Cousin complex also has a dual algebro-geometric
interpretation as the Cousin complex for the stratification given by the diagonals \cite{BDChiralAlgebras}.

While we won't mention factorization algebras (the Koszul duals of chiral algebras) at all in the main text of the present paper, we shall establish, in \cref{thm: Chevalley-Cousin}, that the algebro-geometric interpretation of the Chevalley-Cousin complex persists in higher dimensions:
it provides a resolution in $\D$-modules of the global sections of the canonical sheaf on $(\AA^n)^k$, for all $k$.

The reader is encouraged to examine \cref{exmp: cousin}, which illustrates how the usual Chevalley-Cousin complex for chiral algebras on $\AA^1$ generalizes to $\AA^n$.

\subsection{Relation to the smooth setting and holomorphic field theories}
In this paper, we work in the algebraic setting. Let us close this introduction by commenting briefly on the relationship to the smooth setting, and to holomorphic field theories in the physics literature more broadly.

Working in the setting of their smooth notion of factorization algebras, Costello and Gwilliam define a class of holomorphically translation invariant prefactorization algebras whose cohomologies, in the case of complex dimension one, encode vertex algebras (under some technical conditions) \cite[Chapter 5]{CG1}. In higher complex dimensions this class can be seen as providing higher analogs of vertex algebras, for the purposes of those (many and rich) applications in which one
%can safely pass once and for all to
loses nothing by passing once and for all to
the smooth setting. The general relationship between smooth factorization algebras and factorization/chiral algebras in the algebraic setting of \cite{BDChiralAlgebras} appears to be a very subtle topic.
In particular, what isn't known is how to recover, from such smooth prefactorization algebras, higher vertex algebras in the sense we are seeking here: namely algebraic objects suitable for, for example, doing representation theory.

More broadly, %and without making any precise claims,
it is perhaps helpful to remark on our general expectation about the relationship between higher chiral algebras and holomorphic field theories in the physics literature:
\begin{enumerate}[-]
\item Recall that vertex algebras capture in axiomatic form the local behaviour %(the operator product expansions and the state-field correspondence)
of chiral conformal field theories in complex dimension one.   %, in an algebraic form suitable for doing e.g., representation theory.
\item In the same way, the \dfn{raviolo vertex algebras} recently introduced by Garner and Williams \cite{GarnerWilliamsRaviolo}, cf. \cite{arXiv:2310.08524,MR4861815}, are defined to capture the local behaviour of topological-holomorphic field theories on spacetimes locally modelled on $\RR\times \CC$.
\item
Our general expectation is that higher chiral algebras on $\AA^n$  in the polysimplicial model we introduce here (specialized to the ground field $\kk = \CC$) will stand in much the same relation to holomorphic field theories in higher complex dimensions.
\end{enumerate}
Notions of higher products have been introduced in the holomorphic field theory literature \cite{BGKWY2023} (see also \cite{wang2025feynman}). The higher products vanish in complex dimension one by a theorem of Li \cite{li2023vertex} (within the perturbative BV-formalism). In higher dimensions, the higher products are not trivial, as shown in \cite{BGKWY2023}. In \cref{sec: triangle diagram} we find a match between our first higher product, $\mu_3$, and the one computed by very different methods (summing Feynman diagrams) in the work of Budzik, Gaiotto, Kulp, Wu and Yu \cite{BGKWY2023}. One benefit of having at hand a concrete algebraic model in dg commutative algebras is that it clarifies the conceptual status of these products: for us they are very precisely the products of a homotopy/\Linfinity algebra in the chiral setting, as we discussed above. While our construction is purely algebraic, an \Linfinity chiral algebra in the Jouanolou model can be constructed using renormalized Feynman integrals \cite{GWW25}.

% \begin{tabular}{c|c}
% Algebraic object  & Physical theory \\
% \hline
% Vertex algebras  & Chiral CFTs/holomorphic field theories in one complex dimension\\
% Raviolo vertex algebras   & Topological-holomorphic theories on spacetimes locally $\sim\RR\times \CC$\\
% Chiral algebras on $\AA_\CC^n$ & Holomorphic theories in higher complex dimensions
% \end{tabular}

\subsection{Notations}\label{sec: notations}
\begin{enumerate}[-]%$\triangleright\,\,$]
\item $\kk$ --- ground field of characteristic zero
\item $n\geq 1$ --- number of dimensions
\item $k\geq 1$ --- number of marked points
\item $z^r_i$, $r\in \{1,\dots,n\}$, $i\in \{1,\dots,k\}$ --- coordinates on $(\AA^n)^k$
\item $z^r_{ij} := z^r_i - z^r_j$
\item $u^r_{ij}$ --- coordinates on the polysimplex $\triangle_{n-1}^{\times \binom k2}$, cf. \cref{sec: polysimplices}
\item $I,J,K,\dots \subseteq \{1,\dots,k\}$ --- finite sets of indices, labelling marked points
\item $(\AA^n)^I$ --- a copy of $(\AA^n)^{|I|}$ associated to the index set $I$
\item $I\two := \{(i,j) \in I\times I: i\neq j\}$ for any subset $I\subset \{1,\dots,k\}$
\item $\O_{\AA^1}^I := \kk[z_i]_{i\in I}$, $\B_{\AA^1}^I := \kk[z_i,\frac 1{z_{j\ell}}]_{i\in I; (j,\ell)\in I\two}$
\item $\O_{(\AA^n)^I} := \left( \O_{\AA^1}^I \right)^{\ox n} = \kk[z^r_i]^{1\leq r\leq n}_{i\in I}$, $\B_{\AA^n}^I := \left( \B_{\AA^1}^I \right)^{\ox n} = \kk[z^r_i,\frac 1{z^r_{j\ell}}]^{1\leq r\leq n}_{i\in I; (j,\ell) \in I\two}$
\item $[p]$ --- the $\kk$-vector space of dimension one sitting in homological degree $p$
\item $V[p]:= [p]\ox V$ --- degree-shifted copy of a (possibly graded) $\kk$-vector space $V$
\item $\omega_{\AA^n}$ --- global sections of the canonical sheaf on $\AA^n$
\item $\oms:= \omega_{\AA^n}[n-1]$ --- degree-shifted copy of $\omega_{\AA^n}$
\item $\dd \mathbf z\shifted := \dd z^1 \dots \dd z^n[n-1] \in \oms$ --- shifted volume form
\item $\TS^I_{\AA^n}(\F)$ -- \polysimplicial model of derived sections of  $\F$  on $(\AA^n)^I$; see \cref{sec: TS def}
\end{enumerate}

\subsection{Acknowledgements}
The authors would like to thank Severin Bunk, Hyungrok Kim, Kai Wang, Gerard Watts and Keyou Zeng, for helpful communications and discussions.
Some of the results in this paper were announced at the 4th International Conference on Operad Theory and Related Topics (Tianjing, November 2024). Z.G. would like to thank the organizers and participants of this conference.
C.A.S.Y. gratefully acknowledges the financial support of the Leverhulme Trust in the form of Research Project Grant RPG-2021-092, during which part of this work was initiated.
Some of the results of this paper were presented at a Hamburg U./DESY ZMP colloquium in May 2025 and C.A.S.Y. would like to thank the organisers, especially Sven M\"oller, and audience members for many insightful suggestions.

\section{A \polysimplicial model of derived sections}
%Higher configuration space and movable punctures}
\label{sec: higher configuration space and movable punctures in two complex dimensions}
We work over a field $\kk$ of characteristic zero (for example $\CC$, the complex numbers).
Let us consider the configuration space $\Conf_k(\AA^n)$ of $k$-tuples of distinguishable, pairwise distinct closed points in affine $n$-space
\be \AA^n:= \AA^n_\kk:= \Spec \kk[z^1,\dots,z^n]. \nn\ee
%(Here and in what follows, we use super- and subscripts for indices of formal variables.)
By definition, %it
$\Conf_k(\AA^n)$
is obtained by starting with the affine scheme
\begin{equation} (\AA^{n})^k = \underbrace{\left(\AA^n \times \dots \times \AA^n\right)}_{\text{$k$ copies}} = \Spec(\kk[z^1_1,\dots,z^n_1;\dots;z^1_k,\dots,z^n_k])
\nn\end{equation}
and removing the diagonals:
\begin{equation} \Conf_k(\AA^n)  := (\AA^{n})^k \setminus \bigcup\ijk \{ z^1_i=z^1_j,\dots,z^n_i=z^n_j\} \label{def: CNN}. \end{equation}
%Here $\{ z^1_i=z^1_j,\dots,z^n_i=z^n_j\}$ is the Zariski-closure of the generalised point corresponding to the prime ideal generated by $(z^1_i=z^1_j)\dots(z^n_i=z^n_j)$.
In the special case of dimension $n=1$, these diagonals are of codimension one and the configuration space is itself an affine scheme,
\be \Conf_k(\AA^1) = \Spec \B_k, \nn\ee
where
\be  \B_k := \Gamma(\Conf_k(\AA^1),\O) = \kk[z_i,(z_j-z_\ell)^{-1}]_{1\leq i\leq k;1\leq j<\ell\leq k}\label{Bkdef}\ee
is the commutative algebra of global sections over $\Conf_k(\AA^1)$ of the structure sheaf.
For all $n>1$, however, the configuration space $\Conf_k(\AA^n)$ is not affine and its structure sheaf has higher cohomology.

The space $R\Gamma(\Conf_k(\AA^n),\O)$ of derived sections %of the structure sheaf $\O$
 comes with a dg commutative algebra structure, unique up to zigzags of quasi-isomorphisms of dg commutative algebras.
Let us describe the model of $R\Gamma(\Conf_k(\AA^n),\O)$ we shall use.

\subsection{Forms on polysimplices}
\label{sec: polysimplices}
Let $\Omega^\bul(\triangle_{n-1})$ denote the dg commutative algebra
\begin{equation} %\Omega_n^\bul:=
\Omega^\bul(\triangle_{n-1})  :=  \kk[u^1,\dots,u^{n}; \dd u^1, \dots \dd u^{n}]\big/\langle \sum_{r=1}^n u^r -1, \sum_{r=1}^n \dd u^r \rangle \nn\end{equation}
with $u^r$ in degree $0$ and $\dd u^r$ in degree $1$, for each $r$, and equipped with the usual de Rham differential.
One should think of $\Omega^\bul(\triangle_{n-1})$ as the algebra of polynomial differential forms on the algebro-geometric $(n-1)$-simplex
\be \triangle_{n-1} := \Spec\left( \kk[u^1,\dots,u^{n}]\big/\langle \sum_{r=1}^{n} u^r -1\rangle\right) .\nn\ee
A product $\triangle_{n-1}\times \dots \times \triangle_{n-1}$ of several copies of this simplex is called a  \dfn{polysimplex}. %(prod simplex)
We shall need the algebra of polynomial differential forms on the $\binom k 2$-fold polysimplex:
\begin{align} \Omega^\bul(\triangle_{n-1}^{\times \binom k 2}) &:= \Omega^\bul(\underbrace{\triangle_{n-1}\times \dots \times \triangle_{n-1}}_{\text{$\binom k 2$ times}})\nn\\
&= \kk[u^1_{ij},\dots,u^{n}_{ij}; \dd u^1_{ij}, \dots \dd u^{n}_{ij}]\ijk\big/\langle \sum_{r=1}^{n} u^r_{ij} -1, \sum_{r=1}^{n} \dd u^r_{ij} \rangle\ijk \nn.
\end{align}

\begin{rem}\label{rem: cubes}
In the case of $n=2$ dimensions, the polysimplex is a hypercube
\be \square_{\binom k 2} := \underbrace{\triangle_1\times \dots \times \triangle_1}_{\text{$\binom k 2$ copies}}\nn\ee
and the dg commutative algebra of forms on it admits a simpler description
\begin{equation}
\Omega^\bul(\square_{\binom k 2}) \cong \kk[u_{ij},\dd u_{ij}]\ijk. \nn\end{equation}
%(where $u_{ij} := u_{ij}^1$).
\end{rem}

% More generally, for any totally ordered finite set of indices $(J,<)$, it will be useful to introduce
% \begin{align} \Omega_{n,J}^\bul &:=  \kk[u^1_{ij},\dots,u^{n}_{ij}; \dd u^1_{ij}, \dots \dd u^{n}_{ij}]\ijII \big/\langle \sum_{r=1}^{n} u^r_{ij} -1, \sum_{r=1}^{n} \dd u^r_{ij} \rangle\ijII. \nn
%\end{align}

\subsection{The model $\TS$}\label{sec: TS def}
Recall that $\B_k := \Gamma(\Conf_k(\AA^1),\O)=\kk[z_i,(z_j-z_\ell)^{-1}]_{1\leq i\leq k;1\leq j<\ell\leq k}$ as in \cref{Bkdef}. We think of the tensor product
\begin{align} %\underbrace{\B_k \ox \dots \ox \B_k}_{\text{$n$ fold}}
\B_k^{\ox n} \ox \Omega(\triangle_{n-1}^{\times \binom k 2})
\cong \kk[z^r_i,(z^r_j-z^r_\ell)^{-1}]^{1\leq r\leq n}_{1\leq i\leq k; 1\leq j<\ell<k} \ox \Omega(\triangle_{n-1}^{\times \binom k 2})  \nn
\end{align}
as the algebra of $\B_k^{\ox n}$-valued polynomial differential forms on the polysimplex $\triangle_{n-1}^{\times \binom k 2}$.
Let $\TS_{\AA^n}^{k,\bul}$ denote the following dg commutative subalgebra, in which we place certain boundary conditions on such forms:
\begin{align} \TS_{\AA^n}^{k,\bul}  := &\bigl\{ \tau \in
 \kk[z^r_i,(z^r_j-z^r_\ell)^{-1}]^{1\leq r\leq n}_{1\leq i\leq k; 1\leq j<\ell<k} \ox \Omega^\bul(\triangle_{n-1}^{\times \binom k 2}) \label{def: TS}\\
& \quad :          \text{$\tau|_{u^r_{ij}=0}$ is regular in $(z^r_i-z^r_j)$,  for $1\leq r\leq n$ and $1\leq i<j\leq k$} \bigr\}. \nn
\end{align}
Here $\tau|_{u^r_{ij}=0}$ denotes the pull-back of the form $\tau$ to the face of the polysimplex on which $u^r_{ij} =0$. The differential $\dd_{\TS}$ on $\TS_{\AA^n}^k$ is given by the de Rham differential on the polysimplex.

\begin{thm}\label{thm: polysimplex model NN}
This dg commutative algebra $(\TS_{\AA^n}^k,\dd_{\TS})$ is a model of $R\Gamma(\Conf_k(\AA^n),\O)$:
\begin{equation} \TS_{\AA^n}^k \simeq R\Gamma(\Conf_k(\AA^n),\O). \nn\end{equation}
\end{thm}
\begin{proof}
The proof is given in \cref{sec: proof of polysimplex model NN}.
\end{proof}

%Recall that if $\F$ is a quasicoherent sheaf of (possibly dg) $\O$-modules on $(\AA^{n})^k$ then for every affine open $U\subset (\AA^n)^k$ we have
%\be \Gamma(U,\F) \equiv \F_U \quad\cong_{\O_U}\quad \O_U  \ox_{\O_{(\AA^{n})^k}} \F_{(\AA^{n})^k}.  \nn\ee
Here and subsequently, $\underline{E}$ stands for the quasi-coherent sheaf on ${(\AA^{n})^k}$ corresponding to a module $E$ over $\kk[z^r_i]^{1\leq r\leq n}_{1\leq i\leq k}$. For a sheaf $\F$ of dg $\O_{(\AA^n)^k}$-modules, we define the sheaf $\underline{\TS_{\AA^n}^k}(\F)$ to be the totalization
\be  \underline{\TS_{\AA^n}^{k}}(\F) := \Tot\left(\underline{\TS_{\AA^n}^{k}} \ox_{\O_{(\AA^{n})^k}} \F\right). \nn\ee
Note that $\TS_{\AA^n}^k$ is not only a dg algebra over the regular functions $\mathcal{O}_{(\AA^{n})^k}$, but also a $\mathcal{D}_{(\AA^{n})^k}$-module: we take the usual derivatives $\frac{\partial}{\partial z^r_j}$ and view elements in $\Omega^\bul(\triangle_{n-1}^{\times \binom k 2})$ as constants. Thus, if $\F$ is furthermore a sheaf of right (resp. left) $\D_{(\AA^n)^k}$-modules, % on $(\AA^{n})^k$,
 then $\underline{\TS_{\AA^n}^{k}}(\F)$ is a sheaf of right (resp. left) $\D_{(\AA^n)^k}$-modules via the left $\mathcal{D}_{(\AA^{n})^k}$-module structure on $\underline{\TS_{\AA^n}^{k}}$.
We give our conventions on $\D$-modules in \cref{sec: D module background}.
\begin{thm}\label{thm: F model}
Let $j:\Conf_k(\AA^n)\hookrightarrow (\AA^{n})^k$ denote the open embedding.
\begin{enumerate}[(1)]
\item Let $\F$ be a bounded complex of $\O_{(\AA^{n})^k}$-modules with $\mathcal{O}_{(\AA^{n})^k}$-quasi-coherent cohomology. Then
\be \underline{\TS_{\AA^n}^{k,\bul}}(\F)  \simeq \underline{R\Gamma(\Conf_k(\AA^n),j^{\centerdot}\F)}\simeq Rj_{\centerdot}j^{\centerdot}\F \nn\ee
in the bounded derived category of $\O_{(\AA^{n})^k}$-modules.

\item Let $\F$ be a bounded complex of $\D_{(\AA^{n})^k}$-modules with $\mathcal{O}_{(\AA^{n})^k}$-quasi-coherent cohomology. Then
\be \underline{\TS_{\AA^n}^{k,\bul}}(\F)  \simeq \int_jj^{*}\F \nn\ee
in the bounded derived category of $\D_{(\AA^{n})^k}$-modules.
\end{enumerate}

Here $j_{\centerdot},j^{\centerdot}$ are direct and inverse image for $\O_{(\AA^{n})^k}$-modules. And $\int_j,j^*$ are the (derived) direct and inverse image for $\D_{(\AA^{n})^k}$-modules (in our case $j^*=j^{\centerdot}$).
\end{thm}
\begin{proof}
The proof is given in \cref{sec: proof of F model}.
\end{proof}

In this paper, we will mainly deal with quasi-coherent $\D$-modules on $(\AA^n)^k$ and we will not distinguish the global section $\Gamma((\AA^n)^k,\F)$ and $\F$ itself. For example, we will just write ${\TS_{\AA^n}^{k,\bul}}(\F) $ instead of $\underline{\TS_{\AA^n}^{k,\bul}}(\F) $.

\subsection{Generators and relations for $\TS$}\label{sec: gens and rels for TS}
In \cref{sec: TS def} we gave the algebro-geometric definition of the model $\TS$, as the dg commutative algebra of polynomial differential forms on a polysimplex, valued in a particular ring, and obeying certain boundary conditions.

The model also has the following purely algebraic description in terms of generators and relations.
\begin{thm}\label{thm: TS gens and rels} There is an isomorphism of dg commutative algebras
\be \TS^k_{\AA^n} \cong \directlim_{m\geq 0}\kk\left[ z^r_i, \frac{u^r_{j\ell}}{(z^r_j-z^r_\ell)^m}, \frac{\dd u^r_{j\ell}}{(z^r_j-z^r_\ell)^m}, \right]^{1\leq r\leq n}_{\substack{1\leq i\leq k\\ 1\leq j<\ell\leq k}}\big/\langle \sum_{r=1}^{n} u^r_{ij} -1, \sum_{r=1}^{n} \dd u^r_{ij} \rangle\ijk \nn. \ee
\end{thm}
\begin{proof} \cref{thm: TS gens and rels} is a special case of \cref{thm: TSIE gens and rels}, below.
\end{proof}
Let $I\two := \left\{ (i,j) \in I \times I \,\middle|\, i\neq j\right\}$ denote the set of ordered pairs of distinct elements of an index set $I$.
Let $I\two/S_2$ denote the set of unordered pairs.
Given an ordered pair $e=(i,j) \in I\two$ we write
\be z_{e} := z_{ij} := z_i - z_j = - z_{ji}. \nn\ee
Given an unordered pair $e\in I\two/S_2$, for definiteness we set $z_e := z_i-z_j$ for the choice of representative $(i,j)$ such that $i<j$ relative to some fixed but arbitrary ordering $<$ of the index set $I$.
%
%(Mostly we are concerned with unordered pairs; the exceptions occur where it is important to keep track of the relative sign in $z_{ij} = -z_{ji}$, as for example in the Arnold relations \cref{eq: arnold relations} below.)
%
We think of elements of $I$ as vertices and elements of $I\two/S_2$ as undirected edges.
For a set of undirected edges $E\subset I\two/S_2$, let us define% such that $(i,j)\in E\Leftrightarrow (j,i) \in E$:
\begin{enumerate}[-]
%\item  $\B_{\AA^1}^{I,E} := \kk[z_i,\frac 1{z_e}]_{i\in I; e\in E}$,
%\item $\B_{\AA^n}^{I,E} := \left( \B_{\AA^1}^{I,E} \right)^{\ox n} = \kk[z^r_i,\frac 1{z^r_e}]^{1\leq r\leq n}_{i\in I; e \in E}$
\item $\Omega_{n}^{I,E} := \kk[u^r_e; \dd u^r_e]^{1\leq r\leq n}_{e \in E} \big/\langle \sum\limits_{r=1}^{n} u^r_e -1, \sum\limits_{r=1}^{n} \dd u^r_e \rangle_{e\in E}$, we sometimes write $u^s_{ij}=u^s_{ji}=u^s_e$ if $e=(i,j)$.
\item $\mathbf \Omega_{\AA^n}^{I,E} := \kk[z^r_i,\frac 1{z^r_e}]^{1\leq r\leq n}_{i\in I; e \in E} \ox_\kk \Omega_{n}^{I,E}$
\end{enumerate}
Given also an undirected edge $e\in E$ and a direction $s\in \{1,\dots,n\}$, let
\begin{enumerate}[-]
\item $\Omega^{I,E}_{n,(e,\del_s)} := \Omega_{n}^{I,E}\big/ \langle u_e^s = 0, \dd u_e^s=0\rangle$
\item $\mathbf \Omega^{I,E}_{\AA^n,(e,\del_s)} %= \B^{I,E}_{\AA^n} \ox \Omega^{I,E}_{\AA^n,(e,\del_s)}
:= \mathbf \Omega^{I,E}_{\AA^n}\big/  \langle u_e^s = 0, \dd u_e^s=0\rangle=\kk[z^r_i,\frac 1{z^r_e}]^{1\leq r\leq n}_{i\in I; e \in E} \ox_\kk \Omega_{n,(e,\partial_s)}^{I,E}$
\item $\mathbf \Omega^{I,E}_{\AA^n,(e,\del_s)^{\mathbf{reg}}} := \kk[z^r_i,\frac 1{z^t_{f}}]^{(r,i)\in \{1,\dots,n\}\times I}_{(t,f)\in (\{1,\dots,n\}\times E) \setminus \{(s,e)\}} \ox \Omega^{I,E}_{n,(e,\del_s)}$
\item $\mathbf \Upsilon^{I,E}_{\AA^n,u^s_e=0} := \mathbf \Omega^{I,E}_{\AA^n,(e,\del_s)}\big/  \mathbf \Omega^{I,E}_{\AA^n,(e,\del_s)^{\mathbf{reg}}}$
\end{enumerate}
We get the canonical quotient maps
$\Res_e^{I,E}|_{u^s_e=0} : \mathbf \Omega^{I,E}_{\AA^n} \to \mathbf\Upsilon^{I,E}_{\AA^n,u^s_e=0}$
(these quotient maps behave like residue maps) and define the dg commutative algebra
\be \TS^{I,E}_{\AA^n} := \bigcap_{\substack{e\in E\\ s\in\{1,\dots,n\}}} \ker\left( \Res_e^{I,E}|_{u^s_e=0}\right) \subset \mathbf \Omega^{I,E}_{\AA^n} .\label{def: TSIE}\ee
Observe that then the model $\TS^k_{\AA^n}$ as we defined it in \cref{def: TS} is equivalently given by
\be \TS_{\AA^n}^k \cong \TS^{I}_{\AA^n} := \TS^{I,I\two/S_2}_{\AA^n} \nn\ee
with $I= \{1,\dots,k\}$. Define the dg commutative algebra
\be \mathbf B^{I,E} := \directlim_{m\geq0}\kk\left[ z^r_i, \frac{u^r_{e}}{(z^r_{e})^m}, \frac{\dd u^r_{e}}{(z^r_{e})^m}\right]^{1\leq r\leq n}_{i\in I, e\in E}\big/\langle \sum_{r=1}^{n} u^r_{e} -1, \sum_{r=1}^{n} \dd u^r_e\rangle_{e\in E}\nn. \ee
%Note that $\bigotimes_{\O^I_{\AA^n}}^{e\in E}\mathbf{B}_{\AA^n}^{I,\{e\}}=\mathbf{B}_{\AA^n}^{I,E}$.

\begin{thm}\label{thm: TSIE gens and rels}
There is an isomorphism of commutative dg algebras
\be \TS^{I,E}_{\AA^n} \cong \mathbf B^{I,E}_{\AA^n} ,\nn\ee
for any finite sets of indices $I$ and of undirected edges $E\subset I\two/S_2$.% such that $(i,j) \in E \Leftrightarrow (j,i)\in E$.
\end{thm}
\begin{proof} The proof is given in \cref{sec: proof of thm TSIE gens and rels}.
\end{proof}

\section{Main result: The homology of the \polysimplicial unit chiral operad}\label{sec: main}

In this section we define the unit chiral operad on $\AA^n$, in the \polysimplicial model. It is a dg operad -- that is, an operad in chain complexes of $\kk$-vector spaces.
Our main result is that its homology is isomorphic to the Lie operad.

\subsection{The shifted canonical sheaf $\omega\shifted$}\label{sec: shifted canonical bundle}

We will first introduce some notations. The canonical sheaf $\omega_{\AA^n}$ on $\AA^n$ is a right module over $\D_{\AA^n} = \kk[ z^r, \del_{z^r}]^{1\leq r\leq n}$. The action is given by
\begin{align} (p(z) \dd z^1\dots \dd z^n) \cdot \del_{z^s} &:= - \left(\del_{z^s} p(z)\right) \dd z^1\dots \dd z^n,\nn\\
 (p(z) \dd z^1\dots \dd z^n) \cdot z^s &:= \phantom{+} \left(z^s p(z)\right) \dd z^1\dots \dd z^n .\nn
\end{align}
Similarly, $\Gamma((\AA^{n})^k,\omega_{\AA^n}^{\boxtimes k}) = \omega_{(\AA^{n})^k}$ is a right module over $\DAAnk= \kk[ z^r_i, \del_{z^r_i}]^{1\leq r\leq n}_{1\leq i\leq k}$.

Whenever $\F$ is a left $\D$-module and $\G$ a right $\D$-module, the tensor product $\F\ox_\O \G$ is canonically a right $\D$-module
(with $(f \ox g)\cdot \del := (-\del f)\ox \tau + f\ox(\tau\cdot \del)$ for any derivation $\del$).
Thus, as $\TS_{\AA^n}^{k}$ (defined in \cref{def: TS}) is a left $\D$-module,
\be \TS_{\AA^n}^{k}((\omega_{\AA^n})^{\boxtimes k}) := \TS_{\AA^n}^{k} \ox_{\O_{(\AA^{n})^k}} \omega_{(\AA^{n})^k} \nn\ee
is a right $\DAAnk$-module.

The $\D$-module pushforward $\Delta_* \omega_{\AA^n}$ of $\omega_{\AA^n}$ along the closed diagonal embedding $\AA^n\into (\AA^n)^k$
%(defined, dually, by $z^r_i \mapsto z^r$)
is given as a $\kk$-vector space by
\be \Delta_* \omega_{\AA^n} := \omega_{\AA^n} \ox_{\kk[\lambda^r]^{1\leq r\leq n}} \kk[\lambda^r_i]^{1\leq r\leq n}_{1\leq i\leq k}. \nn\ee
The commutative variables $\lambda^s, \lambda^s_i$ should be viewed as $\del_{z^s}, \del_{z^s_i}$ respectively. The right action of $\lambda^s$ on $\omega_{\AA^n}$ is given by the right action of $\del_{z^s}$ as above; the left action of $\lambda^s$ on $\kk[\lambda^r_i]^{1\leq r\leq n}_{1\leq i\leq k}$ is via the map $\lambda^s\mapsto \sum\limits^k_{i=1}\lambda^s_i$. The right $\DAAnk$-module structure of $\Delta_* \omega_{\AA^n}$ is given by
\[
\left(\tau\otimes P(\lambda)\right)\del_{z^s_i}:=\tau\otimes P(\lambda)\cdot\lambda^s_i,\quad \left(\tau\otimes P(\lambda)\right){z^s_i}:=\tau\otimes (\del_{\lambda^s_i}P(\lambda))+(z^s\tau)\otimes P(\lambda).
\]

Our main interest is in the shifted canonical sheaf, the copy of $\omega_{\AA^n}$ degree-shifted into cohomological degree $1-n$:
\be \oms := \omega_{\AA^n}[n-1] .\nn\ee
(The reason for this shift is that it will put the binary chiral bracket $\mu_2$ correctly in degree zero, as we shall see in \cref{sec: higher operations}.)

As a shorthand, we shall write
$\dd \mathbf z\shifted := \dd z^1 \dots \dd z^n[n-1]$.
Then sections of $(\oms)^{\boxtimes k}= \omega_{(\AA^{n})^k}[k(n-1)]$ over $(\AA^{n})^k$ take the form
\be f\left((z^r_i)^{1\leq r\leq n}_{1\leq i\leq k} \right) \dd\zz_1\shifted \dots  \dd\zz_k\shifted \nn\ee
and sit in cohomological degree $k(1-n)$. We will mainly work with the $\D_{(\AA^n)^k}$-modules $\TS_{\AA^n}^{k}((\oms)^{\boxtimes k})$and $\Delta_*\oms$. Notice that in our convention, the degree $\bullet$ in $\TS_{\AA^n}^{k,\bullet}((\oms)^{\boxtimes k})$ is the total degree
$$
\TS_{\AA^n}^{k,\bullet}((\oms)^{\boxtimes k})=\Tot^{\bullet}\left(\underline{\TS_{\AA^n}^{k}} \ox_{\O_{(\AA^{n})^k}} \oms\right).
$$

\subsection{The \polysimplicial unit chiral operad}\label{sec: TS unit chiral operad}
We shall define a dg operad $\P^{ch}_{\AA^n}$.
%To that end we must first specify a chain complex $\P^{ch}_{\AA^n,\bul}(k)$ of $k$-ary operations, for each $k\geq 0$.
%We set
%$ \P^{ch}_{\AA^n,\bul}(0) := 0$
%and f
For every strictly positive $k$ we let
\be \P^{ch}_{\AA^n}(k) := \Hom_{\DAAnk}\left(\TS_{\AA^n}^{k}((\oms)^{\boxtimes k}), \Delta_*\oms \right).\nn\ee
Recall that $\Delta_*\oms$ is by our definition concentrated in cohomological degree $-(n-1)$. Therefore, more explicitly, one has
\be \P^{ch}_{\AA^n,\bul}(k) := \mathop\oplus\limits_{l\in\mathbb{Z}}\Hom_{\DAAnk}\left(\TS_{\AA^n}^{k,\bul+l}((\oms)^{\boxtimes k}), (\Delta_*\oms)^{l} \right)\nn\ee
\be=\Hom_{\DAAnk}\left(\TS_{\AA^n}^{k,\bul-(n-1)}((\oms)^{\boxtimes k}), \Delta_*\omega_{\AA^n}\right),\nn\ee
%cohom degree of map = cohom degree of target - cohom degree of source = (n-1) - \bul
%hom degree of map = - cohom degree of map = \bul - (n-1)
with the differential given by
\be \dd_\P(-) := (-)\circ \ddts .\nn\ee
% Recall here from \cref{sec: shifted canonical bundle} that
% \begin{enumerate}[--]
% \item $\oms:= \omega_{\AA^n}[n-1]$ denotes the global sections of the canonical bundle put into cohomological degree $-(n-1)$, and $\Delta_*\oms$ its $\D$-module pushforward by the diagonal embedding $\Delta: \AA^n \into  (\AA^n)^k$; and
% \item $\TS_{\AA^n}^{k,\bul}((\oms)^{\boxtimes k})$ is defined in \cref{def: TSomega}.
% \end{enumerate}
Thus, for each strictly positive $k$ we have a chain complex of $k$-ary operations\footnote{Indeed, since elements of $\TS_{\AA^n}^{k,p}((\oms)^{\boxtimes k})$ are given by  $p+ k(n-1)$-forms on the polysimplex $\triangle_{n-1}^{\times \binom k2}$, the cochain complex $\TS_{\AA^n}^{k,\bul}((\oms)^{\boxtimes k})$ is concentrated in cohomological degrees $[-k(n-1), -k(n-1) + \binom k 2 (n-1)]$. Therefore the chain complex $\P^{ch}_{\AA^n,\bul}(k)$ is concentrated in homological degrees $[-(k-1)(n-1),-(k-1)(n-1) + \binom k 2 (n-1)]$. }
\begin{align}
0 \leftarrow \P^{ch}_{\AA^n,-(k-1)(n-1)}(k) \xleftarrow{\dd_\P} \dots \xleftarrow{\dd_\P} \P^{ch}_{\AA^n,0}(k) \xleftarrow{\dd_\P} \dots \xleftarrow{\dd_\P} \P^{ch}_{\AA^n, \binom{k-1} 2 (n-1)}(k) \leftarrow 0.
\nn\end{align}
\begin{prop}\label{prop: operad}
$\P^{ch}_{\AA^n}$ is a dg operad.
\end{prop}
\begin{proof}
The proof is given in \cref{sec: operad}.
\end{proof}

\begin{comment}
In particular, for the $1$-ary and $2$-ary operations we have
\begin{align}
\P^{ch}_{\AA^n,\bul}(1)&: &\bigl( 0\leftarrow  \P^{ch}_{\AA^n,0}(1)=\Hom_{\DAAnk}\left(\oms,\oms\right)=
&\kk\leftarrow 0 \bigr)                       \nn\\
\P^{ch}_{\AA^n,\bul}(2)&: &\bigl( 0\leftarrow \P^{ch}_{\AA^n,-(n-1)}(2) \xleftarrow{\dd_\P} \dots \xleftarrow{\dd_\P} &\P^{ch}_{\AA^n,0}(2)\leftarrow 0\bigr).\nn
\end{align}
\end{comment}

\begin{rem}\label{rem: reln to DS} From \cref{thm: F model}, $\TS_{\AA^n}^{k}((\omega_{\AA^n})^{\boxtimes k})$ represents $\int_jj^*(\omega_{\AA^n})^{\boxtimes k}$.
In the case of dimension $n=1$, the higher direct image vanishes and we have
  \[
  \Gamma\bigl((\AA^{1})^k,\smallint_j(j^*M)\bigr)\simeq     \Gamma\bigl((\AA^{1})^k,j_{\centerdot}j^{\centerdot}M\bigr)=\TS^{k}_{\AA^1}(M).
  \]
Beilinson and Drinfeld \cite{BDChiralAlgebras} define the unit chiral operad, i.e., the operad of chiral operations on $\omega_X$, for curves $X$; in the case of $X=\AA^1$ their definition is
\begin{align}\P_{\AA^1}^{ch}(k)&:=\Hom_{\D_{(\AA^1)^k}
}\left(j_*j^{*}\omega_{\AA^1}^{\boxtimes k},\Delta_*\omega_{\AA^1}\right)\nn\\
% \end{align}
% When $X=\AA^1$, it is enough to consider global sections (since $\AA^k$ are $\D$-affine) and one has
% \begin{align}
%\P_{\AA^1}^{ch}(k)
&=\Hom_{\D_{(\AA^1)^k}}\left(\TS^{k}_{\AA^1}(\omega_{\AA^1}^{\boxtimes k}),\Delta_*\omega_{\AA^1}\right)\nn.
\end{align}
It is thus natural to use $\TS^{k}_{\AA^n}(\omega_{\AA^n}^{\boxtimes k})$ to construct higher analogs of the unit chiral operad for affine spaces.
\end{rem}
\begin{prop}\label{prop: mu2Residue}
    The residue operation $\mu$ defined in \cref{sec: higher operations} is an antisymmetric element in $\P^{ch}_{\AA^n,0}(2)$ and it is closed, $\dd_\P(\mu_2)=0$.
\end{prop}
\begin{proof}
    Recall that $\mu_2\in \Hom_{\mathbf{k}}\left(\TS_{\AA^n}^{2,-(n-1)}((\oms)^{\boxtimes 2}), \Delta_*\omega_{\AA^n}\right)$ is defined by
    $$
    \mu_2=
\mu_2^{2\rightarrow 1}=\int_{\triangle_{n-1}}
\res_{z^1_2 \to z^1_1} e^{\lambda^1_2(z^1_2- z^1_1) } \dots \res_{z^n_2 \to z^n_1} e^{\lambda^n_2(z^n_2- z^n_1) }
    ,$$
    where we identify $\Delta_*\omega_{\AA^n}=\omega_{\AA^n} \ox_{\kk[\lambda^r]^{1\leq r\leq n}} \kk[\lambda^r_1,\lambda^r_2]^{1\leq r\leq n}$ with $\omega_{\AA^n} \ox_{\kk} \kk[\lambda^r_2]^{1\leq r\leq n}$. Here we use the notation $\res_{z^s_2 \to z^s_1} e^{\lambda^s_2(z^s_2- z^s_1) }$ following \cite{vanEkerenHeluaniChiralHomologyEllipticCurvesZhu}. The fact that $\mu_2$ is a $\D_{(\AA^n)^2}$-module map follows from the construction, and the antisymmetry can be checked directly after we introduce the symmetric group action in \cref{sec: operad}. To see that $\mu_2$ is closed, we use Stokes' theorem and the defining boundary conditions \cref{def: TS} of the model.
\end{proof}
\subsection{The main theorem}
Our main result is the following.

\begin{thm}\label{thm: ChiralIsoLie}
    The homology of the \polysimplicial unit chiral operad $\P^{ch}_{\AA^n,\bul}$ is isomorphic to the usual Lie operad:
\be \LieOperad \cong H(\P^{ch}_{\AA^n,\bul}). \nn\ee
\end{thm}

The proof is given in \cref{sec: main proof} and consists of two steps. First for given $k\geq 2$, we fixed a base point to reduce the complex $\P^{ch}_{\AA^n,\bul}(k)$ to a complex of Hom spaces of translation invariant $\D$-modules. Then we show that these in turn map injectively to the dual of the top degree part of the de Rham cohomology of the configuration space, which is isomorphic to the Lie operad (see \cite{totaro1996configuration} and references therein). Then we use the chain level residue map to show that this map is also surjective.

Beilinson and Drinfeld prove (\cite[Theorem 3.1.5]{BDChiralAlgebras} with $X=\AA^1$)  that the operad of chiral operations on $\omega_{\AA^1}$ is isomorphic to the Lie operad:
\be \LieOperad \cong \P^{ch}_{\AA^1} .\nn\ee
Our theorem generalizes this to affine spaces $\AA^n$ of any dimension $n$.

Theorem \ref{thm: ChiralIsoLie}
%implies
is equivalent to the following result (whose meaning we described more explicitly in \cref{sec: higher operations} above).
Let $\LieInfinityOperad$ denote the dg operad governing  \Linfinity-algebras.
\begin{thm}\label{thm: quasi-isomorphism from LieInfinity}
    There is a quasi-isomorphism, of operads in chain complexes of $\kk$-vector spaces,
\[\LieInfinityOperad\xrightarrow\sim \P_{\AA^n}^{ch}.\]
\end{thm}
\begin{proof}
The argument is standard in the theory of dg operads; see for instance \cite{berger2003axiomatic,HinichHomotopyAlgebras,LodayVallette} and also the recent survey \cite{campos2023operadic}.
The category of dg operads has a model category structure in which the weak equivalences are quasi-isomorphisms and the fibrations are degree-wise surjections; see \cite[\S1.4.3]{campos2023operadic} and references given there.
The dg-operad $\LieInfinityOperad$ is the minimal cofibrant resolution of the Lie operad $\LieOperad$; see \cite[\S1.5]{campos2023operadic}.
We have the following diagram
\be \begin{tikzcd}                                                                        &                                       & \P_{\AA^n}^{ch}                                                                                              \\
\bullet \arrow[d, "\text{cofibration}"'] \arrow[rr]                      &                                       & \P^{ch}_{\AA^n,\geq 0} \arrow[d, "\text{fibration}\cap\text{W}", two heads] \arrow[u, "\text{quasi-iso}"'] \\
\LieInfinityOperad \arrow[r, "\text{quasi-iso}"'] \arrow[rru, dashed] & \LieOperad \arrow[r, "\text{iso}"] & {H\P_{\AA^n}^{ch}}
\end{tikzcd}\nn\ee
Here $\P_{\AA^n,\geq 0}^{ch}$ is the truncation
\be
\begin{tikzcd}
\P_{\AA^n}^{ch}(k):                                        & \dots \arrow[r] & \P_{\AA^n,1}^{ch}(k) \arrow[r,"d_\mathcal{P}"]                                & \P_{\AA^n,0}^{ch}(k)\arrow[r, "d_\mathcal{P}"]             & \P_{\AA^n,-1}^{ch}(k) \arrow[r] & \dots \\
\P_{\AA^n,\geq 0}^{ch}(k): \arrow[u, "\text{quasi-iso}"] & {\cdots} \arrow[r]    & \P_{\AA^n,1}^{ch}(k) \arrow[r,"d_\mathcal{P}"] \arrow[u] & \mathrm{Ker}(\dd_\P) \arrow[r] \arrow[u] & 0 \arrow[u]           &
\end{tikzcd}\nn\ee
and the vertical map is a quasi-isomorphism by virtue of \cref{thm: ChiralIsoLie}.
By the left lifting property of the cofibration $\bullet\rightarrow \LieInfinityOperad$ we have the quasi-isomorphism $\LieInfinityOperad\xrightarrow\sim \P_{\AA^n,\geq 0}^{ch}$ and hence the required quasi-isomorphism
$\LieInfinityOperad\xrightarrow\sim \P_{\AA^n}^{ch}$.
\end{proof}

\subsection{Homotopy polysimplicial chiral algebras on $\AA^n$}\label{sec: homotopy polysimplicial chiral algebras}
Malikov and Schechtman \cite{MalikovSchechtman} define the notion of a \dfn{homotopy chiral algebra} on a curve $X$ as a map
from $\LieInfinityOperad$ to the usual chiral operad $\P_{X}^{ch}[\A](k):=\Hom_{\D_{X^k}}\left(j_*j^{*}\mathcal{A}^{\boxtimes k},\Delta^{}_*\A\right)$ associated to a dg $\D_X$-module $\mathcal{A}$ (here $j:\Conf_k(X)\hookrightarrow X^k\hookleftarrow X:\Delta$). It is natural to generalize their definition to higher dimensions, at least for affine spaces $\AA^n$, using the \polysimplicial model, as follows.

Suppose that $(\A,\dd_\A)$ is a quasi-coherent dg $\D_{\AA^n}$-module.
We define
\[\P^{ch}_{\AA^n}[\A](k):=\Hom_{\DAAnk}\left(\TS^{k}_{\AA^n}(\mathcal{A}^{\boxtimes k}),\Delta_*\A\right)\]
with the differential $\dd_\P[\A]$ given by
\[\left(\dd_\P[\A]\mu\right)(-)=\dd_\A\left(\mu(-)\right)\pm\mu\left((\dd_{\A^{\boxtimes k}}+\ddts)(-)\right)\]
(where $\pm$ is the required Koszul sign).

This defines a dg operad (by a generalization of \cref{prop: operad}).
\begin{defn}
%  Let $(\A,\dd_\A)$ be a dg $\D_{\AA^n}$-module.
A \dfn{homotopy polysimplicial chiral algebra} structure on $\mathcal{A}$ is a map of dg operads
\[  \LieInfinityOperad\rightarrow \P_{\AA^n}^{ch}[\A].\]
\end{defn}

In this language, our result \cref{thm: quasi-isomorphism from LieInfinity} implies the following:
\begin{cor}
The degree-shifted canonical sheaf $\oms := \omega_{\AA^n}[n-1]$ on $\AA^n$ has the structure of a homotopy polysimplicial chiral algebra.  \qed
\end{cor}
\begin{rem}
Our definition is not equivalent to that in \cite{FrancisGaitsgory} since we are working with a specific model (the \polysimplicial model) and the dg operad depends on the choice of model. In our definition, the hom spaces are not derived; one needs to work with $\infty-$operads to handle the fully derived chiral operations.  However, using Drinfeld's dg quotient construction, one can embed our polysimplicial hom space into the derived one. Thus, we expect homotopy chiral algebras in the \polysimplicial model can be turned into special examples of higher dimensional chiral algebras in \cite{FrancisGaitsgory}. Details are left to future work.
\end{rem}

\section{Operad structure of $\P^{ch}_{\AA^n,\bul}$}\label{sec: operad}

In \cite[\S1.3.1]{BDChiralAlgebras}, an operad consisting of the chiral operations of a $\mathcal{D}_X$-module $M$ on a smooth curve $X$ is constructed. For a more detailed study of this operad and its relation to vertex algebra cohomology, we refer the reader to \cite{bakalov2020chiral} and the references given there. In this section, we extend the construction in \cite[\S1.3.1]{BDChiralAlgebras} to $X=\AA^n$ for all $n\geq 1$ using our polysimplicial model.

We will first define the symmetric group action on each component of $\P^{ch}_{\AA^n}[M]$ explicitly. For a nonempty finite index set $I$ with $k>0$-many elements, the \polysimplicial chiral $I$-operations are the chain complex of $\kk$-vector spaces
\be \P^{ch}_{\AA^n}[M](k) := \P^{ch}_{\AA^n}[M](I\to \{\star\}) := \Hom_{\D_{(\AA^n)^I}}\left(\TS_{\AA^n}^{I}(M^{\boxtimes I}), \Delta_*^{I/\{\star\}}M \right).\nn\ee
Here, $\Delta^{I/\{\star\}}$ is defined as the diagonal embedding $\Delta^{I/\{\star\}}:(\AA^n)^{\{\star\}}\cong\AA^n\into (\AA^n)^I$.
\begin{defn}
We define the action of the symmetric group $S_{|I|}$ on  $\P^{ch}_{\AA^n}[M](I\to \{\star\})$ by
  \[
 (\sigma\mu)(-):=\sigma\mu(\sigma^{-1}-)
  \]
where
\begin{enumerate}[-]
\item
the action on $\TS_{\AA^n}^I(M^{\boxtimes I})$ is given by
\begin{align}
&\sigma\left(\alpha(z^r_i, u_{ij},\dd u_{ij})\cdot m_1\boxtimes \cdots\boxtimes m_k\right)\nn\\
&
\quad :=(-1)^{\chi(m_1,\dots,m_k;\sigma)}\alpha(z^r_{\sigma(i)},u^r_{\sigma(i)\sigma(j)},\dd u^r_{\sigma(i)\sigma(j)}) \cdot m_{\sigma^{-1}(1)}\boxtimes \cdots\boxtimes m_{\sigma^{-1}(k)}\nn
\end{align}
where $m_i\in \Gamma(\AA^n,M)$ and $\chi(m_1,\dots,m_k;\sigma)$ is the Koszul sign.\footnote{
Note here that $z^r_1 = z^r \ox 1 \ox \dots \ox 1 \in \O_{\AA^n}^{\ox k} \cong \O_{(\AA^{n})^k}$ and so on, and hence
\[
\cdots\boxtimes \underbrace{f(z)m}_{\text{i-th}}\boxtimes\cdots =f(z_i)\cdot \left(\cdots\boxtimes \underbrace{m}_{\text{i-th}}\boxtimes\cdots \right).
\]
}
\item the action on $\Delta_*M$ is given by
\[
\sigma\left(m\otimes f(\lambda^r_1,\dots,\lambda^r_k)\right):=m\otimes f(\lambda^r_{\sigma(1)},\dots,\lambda^r_{\sigma(k)}).
\]
\end{enumerate}
\end{defn}

\begin{comment}
\begin{rem}
It maybe more clear if we write $M$ as $N\ox_{\mathcal{O}}\omega_{\AA^n}$ where $N$ is a left $\D$-module. For a section in $\Gamma\left((\AA^n)^k,M^{\boxtimes k}\right)=\Gamma\left((\AA^n)^k,(N\ox_{\mathcal{O}}\omega_{\AA^n})^{\boxtimes k}\right)$
\[
\alpha(z_1,\dots,z_k,u_{ij},\dd u_{ij})\cdot n_1d\mathbf{z}_1\boxtimes \cdots\boxtimes n_kd\mathbf{z}_k,\quad n_i\in \Gamma(\AA^n,M)=\Gamma(\AA^n,N\ox_{\mathcal{O}}\omega_{\AA^n}),
\]
the symmetric group action can be written as
\[
\sigma\left(\alpha(z_1,\dots,z_k,u_{ij},\dd u_{ij})\cdot n_1d\mathbf{z}_1\boxtimes \cdots\boxtimes n_kd\mathbf{z}_k\right)
\]
\[
=(-1)^{\chi(n_1,\dots,n_k;\sigma^{-1})}\alpha(z_{\sigma(1)},\dots,z_{\sigma(k)},u_{\sigma(i)\sigma(j)},\dd u_{\sigma(i)\sigma(j)})\cdot n_{\sigma^{-1}(1)}d\mathbf{z}_1\boxtimes \cdots\boxtimes n_{\sigma^{-1}(k)}d\mathbf{z}_k.
\]
\end{rem}
\end{comment}

Our main result of this section is the following.

\begin{prop}\label{prop: operad}
Let $M$ be a quasicoherent right $\D$-module on $\AA^n$.
  The dg $S$-module defined by
  \[
 \P^{ch}_{\AA^n}[M](k)= \begin{cases}
    0, & \mbox{if } k=0 \\
    \Hom_{\D_{(\AA^n)^k}}\left(\TS_{\AA^n}^{k}(M^{\boxtimes k}), \Delta_*^{\{1,\dots,k\}/\{\star\}}M\right) , & \mbox{}k\geq 1.
  \end{cases}
  \]
  has a dg operad structure.
\end{prop}
\begin{proof}
In the remainder of this section, we shall define the composition map
\[
\gamma:
\P^{ch}_{\AA^n}[M](J\onto \{\star\})\ox_{\mathbf{k}}\left(\ox_{\mathbf{k}}^{j\in J}
\P^{ch}_{\AA^n}[M](K_j\onto \{j\})\right)\rightarrow \P^{ch}_{\AA^n}[M](K\onto \{\star\})
\]
for a surjective map $K\onto J$ (here $K_j\subset K$ denotes the preimage of $j\in J$) and check in detail that composition is associative. The other axioms of the classical definition of an operad due to \cite{May} are straightforward to verify.
\end{proof}
%For this proof, to alleviate the notational burden we shall use a couple of shorthands:
%\begin{enumerate}[-]
%\item
Recall the notation $\TS^{I,E}_{\AA^n}$ from \cref{sec: gens and rels for TS}, where $E$ are edges. Given an $I$-family of disjoint sets $\{J_i\}_{i\in I}$, let us write
\begin{equation} \TS^{\{J_i\}_{i\in I}} := \TS^{(\sqcup_{i\in I} J_{i}), (\sqcup_{i\neq i'} J_i \times J_{i'})}. \nn\end{equation}
(In this section, we shall sometimes omit the subscript $\AA^n$.)

By \cref{thm: TSIE gens and rels}, we have
\[
\TS^K(M^{\boxtimes K})=\TS^{\{K_j\}_{j\in J}}
\otimes_{\O_{(\AA^n)^{K}}}\left(\mathop{\boxtimes}_{j\in J}\TS^{K_j}(M^{\boxtimes K_j})\right)=\TS^{\{K_j\}_{j\in J}}
\left(\mathop{\boxtimes}_{j\in J}\TS^{K_j}(M^{\boxtimes K_j})\right).
\]
Given chiral operations $\mu_{K_j} \in \P^{ch}_{\AA^n}[M](K_j\onto \{j\})$ for each $j\in J$, we get the map
\[
\mathop{\boxtimes}_{j\in J}\mu_{K_j}: \TS^K\left(M^{\boxtimes K}\right)
\rightarrow  \TS^{\{K_j\}_{j\in J}}
\left(\mathop{\boxtimes}_{j\in J}\Delta^{ K_{j}/\{j\}}_*M\right).
\]
In order to apply a chiral $J$-operation $\mu_J$ to the result, we need a map
\begin{equation}\label{DiagonalPullBack}
\TS^{\{K_j\}_{j\in J}}\left(\mathop{\boxtimes}_{j\in J}\Delta^{ K_{j}/\{j\}}_*M\right)\rightarrow \Delta^{ K/J}_*\TS^{J}\left(M^{\boxtimes J}\right).
\end{equation}
Here, $\Delta^{K/J}$ is denotes the diagonal embedding $\Delta^{k/J}:(\AA^n)^J\into (\AA^n)^K$ induced by the surjection $K\onto J$.

With the next step in mind, namely checking associativity of composition, we will construct such maps in a slightly more general setting.

Let $K\onto J\onto I$ be a sequence of surjective maps, $J_i\subset J$ the preimage of $i\in I$ and $K_{J_i}\subset K$ the preimage of $J_i$. Such a situation is depicted in \cref{composition}. Suppose that $M_{J_i}$ are $\D$-modules on $(\AA^n)^{J_i}$. We will define, in \cref{def: Imap} below, a map
\[
\mathbf{I}^{[K_{J_i}]_{i\in I}}_{[J_i]_{i\in I}}: \TS^{\{K_{J_i}\}_{i\in I}}\left(\mathop{\boxtimes}_{i\in I}\Delta^{ K_{J_i}/J_i}_*M_{J_i}\right)\rightarrow \Delta^{K/J}_*\TS^{\{J_i\}_{i\in I}}\left(\mathop{\boxtimes}_{i\in I}M_{J_i}\right).
\]
(Taking $J=I$ and $M_{J_i}=M$, we recover (\ref{DiagonalPullBack}).)
Here, the source of this map is
\[\TS^{\{K_{J_i}\}_{i\in I}}\left(\mathop{\boxtimes}_{i\in I}\Delta^{ K_{J_i}/J_i}_*M_{J_i}\right):=\left(  (\mathop{\boxtimes}_{i\in I}M_{J_i})\ox_{\kk[\lambda^r_{j}]^{1\leq r\leq n}_{j\in J}}\kk[\lambda^r_{k}]^{1\leq r\leq n}_{k\in K}\right)\ox_{\O_{(\AA^n)^K}}\TS^{\{K_{J_i}\}_{i\in I}},\]
where $\mathop{\boxtimes}\limits_{i\in I}M_{J_i}$ is a $\O_{(\AA^n)^K}$-module via the map $\O_{(\AA^n)^K}\rightarrow \O_{(\AA^n)^J}$. It is generated by $(\mathop{\boxtimes}\limits_{i\in I}M_{J_i})\ox_{\O_{(\AA^n)^K}}\TS^{\{K_{J_i}\}_{i\in I}}$ as a $\D$-module. The identity map
  \[
  \id:(\mathop\boxtimes_{i\in I}M_{J_i})\ox_{\O_{(\AA^n)^K}}\TS^{\{K_{J_i}\}_{i\in I}}\rightarrow (\mathop\boxtimes_{i\in I}M_{J_i})\ox_{\O_{(\AA^n)^K}}\TS^{\{K_{J_i}\}_{i\in I}}
  \]
induces a map of $\D$-modules
  \begin{gather}
v^{[K_{J_i}]_{i\in I}}_{[J_i]_{i\in I}}: \left(  (\mathop\boxtimes_{i\in I}M_{J_i})\ox_{\kk[\lambda^r_{j}]^{1\leq r\leq n}_{j\in J}}\kk[\lambda^r_{k}]^{1\leq r\leq n}_{k\in K}\right)\ox_{\O_{(\AA^n)^K}}\TS^{\{K_{J_i}\}_{i\in I}}\nn\\
\rightarrow\left(  (\mathop\boxtimes_{i\in I}M_{J_i})\ox_{\O_{(\AA^n)^K}}\TS^{\{K_{J_i}\}_{i\in I}}\right)\ox_{\kk[\lambda^r_{j}]^{1\leq r\leq n}_{j\in J}}\kk[\lambda^r_{k}]^{1\leq r\leq n}_{k\in K}.\nn
\end{gather}
given by the Leibniz rule (by regarding $\lambda^r_k$ as derivatives $\del_{z^r_k}$ and moving them out to the right).

\begin{figure}[htp]

\tikzset{every picture/.style={line width=0.75pt}} %set default line width to 0.75pt

\begin{tikzpicture}[]
\begin{scope}[canvas is xz plane at y=6]
\draw[thick,dashed] (-2,0,0) rectangle (8,7,0) node[label={[black]right:$ K$}] {};
   \draw[dashed] (6,3,0) circle (1.5cm);
        \draw[thick] (6.8,3.5,0) circle (0.5cm);
            \draw[fill=black] (6.8,3.7,0) circle (0.1cm);
            \draw[fill=black] (6.6,3.5,0) circle (0.1cm);
            \draw[fill=black] (7.0,3.5,0) circle (0.1cm);
        \draw[thick] (5.3,3.5,0) circle (0.5cm);
            \draw[fill=black] (5.3,3.4,0) circle (0.1cm);
        \draw[thick] (6,2.5,0) circle (0.5cm);
            \draw[fill=black] (6.3,2.5,0) circle (0.1cm);
            \draw[fill=black] (6.0,2.3,0) circle (0.1cm);
    \draw[dashed] (2,2,0) circle (1.5cm) node[label={[black]right:$\;\;\quad\qquad K_{J_i}$}] {};
        \draw[thick] (2,2.8,0) circle (0.5cm);
            \draw[fill=black] (1.8,3.0,0) circle (0.1cm);
            \draw[fill=black] (2.3,2.6,0) circle (0.1cm);
            \draw[fill=black] (2.2,3.0,0) circle (0.1cm);
            \draw[fill=black] (1.8,2.7,0) circle (0.1cm);
        \draw[thick] (2,1.3,0) circle (0.5cm);
            \draw[fill=black] (2,1.5,0) circle (0.1cm);
            \draw[fill=black] (2.1,1.1,0) circle (0.1cm);
    \draw[dashed] (0,5,0) circle (1.5cm);
        \draw[thick] (0,5,0) circle (0.5cm) node[label={[black]right:$\quad K_j$}] {};
            \draw[fill=black] (0,5.2,0) circle (0.1cm);
            \draw[fill=black] (0.2,4.8,0) circle (0.1cm);
            \draw[fill=black] (-0.1,4.8,0) circle (0.1cm);
\end{scope}
\begin{scope}[canvas is xz plane at y=3]
\draw[thick,dashed] (-2,0,0) rectangle (8,7,0) node[label={[black]right:$ J$}] {};
    \draw[thick] (6,3,0) circle (1.5cm);
        \draw[fill=black] (6.8,3.5,0) circle (0.1cm);
        \draw[fill=black] (5.3,3.5,0) circle (0.1cm);
        \draw[fill=black] (6,2.5,0) circle (0.1cm);
    \draw[thick] (2,2,0) circle (1.5cm) node[label={[black]right:$\;\;\quad\qquad J_i$}] {};
        \draw[fill=black] (2,2.8,0) circle (0.1cm);
        \draw[fill=black] (2,1.3,0) circle (0.1cm);
    \draw[thick] (0,5,0) circle (1.5cm);
        \draw[fill=black] (0,5,0) circle (0.1cm) node[label={[black]right:$ j$}] {};
\end{scope}
\begin{scope}[canvas is xz plane at y=0]
\draw[thick,dashed] (-2,0,0) rectangle (8,7,0) node[label={[black]right:$ I$}] {};
    \draw[fill=black] (6,3,0) circle (0.1cm);
    \draw[fill=black] (2,2,0) circle (0.1cm) node[label={[black]right:$ i$}] {};
    \draw[fill=black] (0,5,0) circle (0.1cm);
\end{scope}
\end{tikzpicture}
  \centering
  \caption{Depiction of $K\onto J\onto I$.}\label{composition}
\end{figure}

\begin{subequations}\label{def: c maps}
Next we observe that there is a well-defined map
\begin{equation}
 \TS^{\{K_{J_i}\}_{i\in I}}\rightarrow \TS^{\{J_i\}_{i\in I}}
\nn\end{equation}
defined by sending
\be z^r_k \mapsto z^r_{\pi(k)},\qquad
 u^r_{kk'} \mapsto u^r_{\pi(k)\pi(k')},\ee
where $\pi : K \onto J$ is the surjection sending $K_j \onto \{j\}$ for each $j\in J$. This is a $\kk[\lambda^r_j]^{1\leq r\leq n}_{j\in J}-$module map if we use $\lambda^r_j=\sum\limits_{k\in K_j} \lambda^r_k$ to define the $\kk[\lambda^r_j]^{1\leq r\leq n}_{j\in J}-$module structure on $ \TS^{\{K_{J_i}\}_{i\in I}}$.
Thus we get a $\D$-module map
\begin{gather}
c^{[K_{J_i}]_{i\in I}}_{[J_i]_{i\in I}}:  \left(  (\mathop\boxtimes_{i\in I}M_{J_i})\ox_{\O_{(\AA^n)^K}}\TS^{\{K_{J_i}\}_{i\in I}}\right)\ox_{\kk[\lambda^r_{j}]^{1\leq r\leq n}_{j\in J}}\kk[\lambda^r_{k}]^{1\leq r\leq n}_{k\in K}\nn\\
\rightarrow \left(  (\mathop\boxtimes_{i\in I}M_{J_i})\ox_{\O_{(\AA^n)^J}}\TS^{\{J_i\}_{i\in I}}\right)\ox_{\kk[\lambda^r_{j}]^{1\leq r\leq n}_{j\in J}}\kk[\lambda^r_{k}]^{1\leq r\leq n}_{k\in K}.
\end{gather}
\end{subequations}

\begin{defn}\label{def: Imap}
  We define a $\D$-module map
  \[
\mathbf{I}^{[K_{J_i}]_{i\in I}}_{[J_i]_{i\in I}}:  \TS^{\{K_{J_i}\}_{i\in I}}\left(\mathop\boxtimes_{i\in I}\Delta^{ K_{J_i}/J_i}_*M_{J_i}\right)\rightarrow \Delta^{K/J}_*\TS^{\{J_i\}_{i\in I}}\left(\mathop\boxtimes_{i\in I}M_{J_i}\right)
  \]
  to be the composition $c^{[K_{J_i}]_{i\in I}}_{[J_i]_{i\in I}}\circ v^{[K_{J_i}]_{i\in I}}_{[J_i]_{i\in I}}$.
\end{defn}
\iffalse
Using the above map, an element $\mu\in  \P^{ch}_{\AA^n}[M](I\xonto \pi T)$ gives rise to
\[
\mathbf{I}^{[I_t,\{k\}]_{t\in T,k\in K}}_{[\{t\},\{k\}]_{t\in T,k\in K}}\circ\left(\mu\boxtimes\mathrm{id}_{M^{\boxtimes K}}\right)\in \P^{ch}_{\AA^n}[M](I\sqcup K\xonto \pi T\sqcup K) .
\]
This will be used frequently in [Chevalley-Cousin].
\fi
\begin{rem}
  In the case of $\AA^1$, the above map is an isomorphism of $\D$-modules. One can show that it is a quasi-isomorphism in the general case. In fact, the map $\mathbf{I}^{[K_{J_i}]_{i\in I}}_{[J_i]_{i\in I}}$ defined above is an explicit construction of the quasi-isomorphism in the projection formula \cite[Corollary 1.7.5]{HTT_Dmodules2008}. %, but we will not use this fact in the construction.
\end{rem}

We are now ready to define \dfn{operadic composition}. Given a surjection $K \onto J$, define
\[
\gamma:\P^{ch}_{\AA^n}[M](J\onto \{\star\})\ox_{\mathbf{k}}\left(\ox_{\mathbf{k}}^{j\in J}\P^{ch}_{\AA^n}[M](K_j\onto \{j\})\right)\rightarrow \P^{ch}_{\AA^n}[M](K\onto \{\star\})
\]
by
\[
\gamma\left(\mu_{J}\otimes (\otimes^{j\in J}\mu_{K_j})\right):=\left(\Delta^{K/J}_*\mu_J\right)\circ\left(\mathbf{I}^{[K_{j}]_{j\in J}}_{[\{j\}]_{j\in J}}\circ(\mathop\boxtimes_{j\in J}\mu_{K_j})\right).
\]
% Here
% \[
% \mathbf{I}^{[K_{j}]_{j\in J}}_{[\{j\}]_{j\in J}}:\TS^{\{K_j\}_{j\in J}}\left(\boxtimes_{j\in J}\Delta^{ K_{j}/j}_*M\right)\rightarrow \Delta^{ K/J}_*\TS^{J}\left(M^{\boxtimes J}\right).
% \]

Now we check the associativity of composition. We need to show that
\[
\gamma\left(\gamma\left(\mu_{I}\otimes(\otimes^{i\in I}\mu_{J_i})\right)\otimes (\otimes^{j\in J_i}\mu_{K_j})\right)=\gamma\left(\mu_{I}\otimes\left(\otimes^{i\in I}\gamma\left(\mu_{J_i}\otimes (\otimes^{j\in J_i}\mu_{K_j})\right)\right)\right)
\]
We first prove the following lemma.

\begin{lem}\label{OperadLem}
\begin{enumerate}[(1)]
\item  We have  $\mathbf{I}^{[K_j]_{j\in J}}_{[\{j\}]_{j\in J}}=\mathbf{I}^{[K_{J_i}]_{i\in I}}_{[{J_i}]_{i\in I}}\circ\mathop{\boxtimes}\limits_{i\in I}\mathbf{I}^{[K_j]_{j\in J_i}}_{[\{j\}]_{j\in J_i}}$. That is, the following diagram commutes
  \[
% https://tikzcd.yichuanshen.de/#N4Igdg9gJgpgziAXAbVABwnAlgFyxMJZABgBpiBdUkANwEMAbAVxiRABUBlAPWAGlSAHUEBbOjgAWENMGFwAjgGMmaAL58A+gCtheEfAAEm4FoDkq1cIYwAZjgAUwsZOmzBAIwgAPPfEuCGLBFcOA0TYSwwAwApfwARGAYcOl4jbQB6LVUNACoAWWEAJywAcwkcAEoDEFVSdExcfEIUMgAmKlpGFjYuXgEncSkZOSUVdTDojSx-Xzg04EngLHMLK1sHAZdhj29Z-0DgnFCliKiASXjE5NTjSen0u+yc3oWp1Xs83mFPHyDDR4qRVK5SqNTqIAw2DwBCIZAALB16MxWIgQMIEkkUvwHk9hL1oms7I5RINXN9dn84PsgiEwlhTgYLp83D9ZjE3kCypVqqoOjAoCV4ERQDZChAREgyCAcBAkABGWoisUSxBy6gypCtRUgUXiyXq2WIVrUdwwMBQJAAdgAHNQkd1Ue4ahRVEA
\begin{tikzcd}
{\TS^{\{K_j\}_{j\in J}}\left(\mathop{\boxtimes}\limits_{j\in J}\Delta^{ K_j/\{j\}}_*M\right) } \arrow[dd,"\mathop{\boxtimes}\limits_{i\in I}\mathbf{I}^{[K_j]_{j\in J_i}}_{[\{j\}]_{j\in J_i}}"] \arrow[dddd, "\mathbf{I}^{[K_j]_{j\in J}}_{[\{j\}]_{j\in J}}", bend left=78]         \\
                                                                                                                                                              \\
{\TS^{\{K_{J_i}\}_{i\in I}}\left(\mathop{\boxtimes}\limits_{i\in I}\Delta^{ K_{J_i}/J_i}_*\TS^{J_i}(M^{\boxtimes J_i})\right) } \arrow[dd,"\mathbf{I}^{[K_{J_i}]_{i\in I}}_{[{J_i}]_{i\in I}}"] \\
                                                                                                                                                              \\
\Delta^{K/J}_*\TS^J\left(\mathop{\boxtimes}\limits_{i\in I}M^{\boxtimes J_i}\right)
\end{tikzcd}
\]
\item We have $ \Delta^{K/J}_*\left(\mathop{\boxtimes}\limits_{i\in I}\mu_{J_i}\right)\circ \mathbf{I}^{[K_{J_i}]_{i\in I}}_{[J_i]_{i\in I}}=\mathbf{I}^{[K_{J_i}]_{i\in I}}_{[J_i]_{i\in I}}\circ\left(\mathop{\boxtimes}\limits_{i\in I}\Delta^{K_{J_i}/J_i}_*\mu_{J_i}\right)$. That is, the following diagram commutes
\[
% https://tikzcd.yichuanshen.de/#N4Igdg9gJgpgziAXAbVABwnAlgFyxMJZABgBpiBdUkANwEMAbAVxiRAAUBlAHW4ZgBmOABS8ARhAAeeALbwA+sCy8sYAAQBJAL68AIjAY46APWC84ARwDGTNIoBWK9QCl5WLWoDS8+wHpX7rwATlgA5gAWOACUIFqk6Ji4+IQoZACMVLSMLGx6Bkam5ta2Dk5qzh7efhVcvPxCotwS0lhycIrK3KqaOk1SsgrAjl0ubr0hEdGx8SAY2HgERABMpBnU9MysiCB5hiZm3JY2dkNlFV4+-lq1fIIi4v2tg53d2rsFwAG+gdwTkTFxBLzZLLciZDY5bY3er3PotNodMpvbj6PaFQ7FE7DboBSqXXHvfZfH5-KZaTIwKCheBEUACIIQGRIMggHAQJBpQEgemMjnUNlIJZcnlMxAsgWIADMwoZosl-PZiCFFC0QA
\begin{tikzcd}
\TS^{\{K_{J_i}\}_{i\in I}}\left(\mathop{\boxtimes}\limits_{i\in I}\Delta^{ K_{J_i}/J_i}_*\TS^{J_i}(M^{\boxtimes J_i})\right) \arrow[d,"\mathbf{I}^{[K_{J_i}]_{i\in I}}_{[{J_i}]_{i\in I}}"'] \arrow[rr,"\mathop{\boxtimes}\limits_{i\in I}\Delta^{ K_{J_i}/J_i}_*\mu_{J_i}"]       &  & \TS^{\{K_{J_i}\}_{i\in I}}\left(\mathop{\boxtimes}\limits_{i\in I}\Delta^{ K_{J_i}/J_i}_*\Delta^{J_i/\{i\}}_*M\right) \arrow[d,"\mathbf{I}^{[K_{J_i}]_{i\in I}}_{[{J_i}]_{i\in I}}"] \\
\Delta^{K/J}_*\TS^J\left(\mathop{\boxtimes}\limits_{i\in I}M^{\boxtimes J_i}\right) \arrow[rr,"\Delta^{K/J}_*\left(\mathop{\boxtimes}\limits_{i\in I} \mu_{J_i}\right)"] &  & \Delta^{K/J}_*\TS^{\{J_i\}_{i\in I}}\left(\mathop{\boxtimes}\limits_{i\in I}\Delta^{J_i/\{i\}}_*M\right)
\end{tikzcd}
\]
\item We have $\mathbf{I}^{[K_{J_i}]_{i\in I}}_{[\{i\}]_{i\in I}}=\Delta^{K/J}_*\mathbf{I}^{[{J_i}]_{i\in I}}_{[\{i\}]_{i\in I}}\circ \mathbf{I}^{[K_{J_i}]_{i\in I}}_{[J_i]_{i\in I}}$. That is, the following diagram commutes
\[
% https://tikzcd.yichuanshen.de/#N4Igdg9gJgpgziAXAbVABwnAlgFyxMJZABgBpiBdUkANwEMAbAVxiRAB12AVAZQD1gAaVKcAtnRwALCGmCc4ARwDGTNAF9BAfWAApTVjWc8o+AAItu7VgDkau5wYwAZjgAUYidNmcARhAAexvCG7AxYorhwVpxYYKYAkiEAIjAMOHQC5tp6BgD0OWqaAFScKWkZlli5BsUAspwATlgA5pI4AJQgaqTomLj4hChkAExUtIwsbKWp6QKC+YUl3Py6IY4u7uziUjJy7H6B4cEO4ZHR7LEJyTMVOdWL9exNrR1dPSAY2HgERGQALGN6MxWIgOOwyrMhAtipxeAIdGtnG4PDtvPsAkE4GtTjgosAsDE4olpuV4fp7nVGi02p01GMYFBmvAiKAnA0IKIkGQQDgIEgAIzdVnszmIfnUXlIYZCkBsjlciV8xDDag+GBgKBIADsf2oQMmoJ8XQoaiAA
\begin{tikzcd}
{\TS^{\{K_{J_i}\}_{i\in I}}\left(\mathop{\boxtimes}\limits_{i\in I}\Delta^{ K_{J_i}/J_i}_*\Delta^{J_i/\{i\}}_*M\right)} \arrow[dd,"\mathbf{I}^{[K_{J_i}]_{i\in I}}_{[{J_i}]_{i\in I}}"] \arrow[dddd, "\mathbf{I}^{[K_{J_i}]_{i\in I}}_{[\{i\}]_{i\in I}}", bend left=78] \\
                                                                                                                                                                                    \\
\Delta^{K/J}_*\TS^{\{J_i\}_{i\in I}}\left(\mathop{\boxtimes}\limits_{i\in I}\Delta^{J_i/\{i\}}_*M\right) \arrow[dd,"\Delta^{K/J}_*\mathbf{I}^{[J_i]_{i\in I}}_{[\{i\}]_{i\in I}}"]                                                                                    \\
                                                                                                                                                                                    \\
\Delta^{K/I}_*\TS^{I}\left(M^{\boxtimes I}\right)
\end{tikzcd}
\]
\end{enumerate}
\end{lem}

\begin{proof}
  For (1), we only need to note that
  \[
  \TS^{\{K_j\}_{j\in J}}\left(\mathop\boxtimes_{j\in J}\Delta^{ K_j/j}_*M\right)=  \TS^{\{K_{J_i}\}_{i\in I}}\left(\mathop\boxtimes_{i\in I}\TS^{\{K_j\}_{j\in J_i}}\left(\mathop\boxtimes_{j\in J_i}\Delta^{ K_j/\{j\}}_*M\right)\right).
  \]
Part  (2) follows from the fact that $\mu_{J_i}$ are $\D$-module maps and (3) follows from the definition.
\end{proof}

By the above lemma, we have
\begin{align*}
   &\gamma\left(\gamma\left(\mu_{I}\otimes(\otimes^{i\in I}\mu_{J_i})\right)\otimes (\otimes^{j\in J_i}\mu_{K_j})\right)  \\
   & =\Delta^{K/J}_*\left(\Delta^{J/I}_*\mu_I\circ \mathbf{I}^{[J_i]_{i\in I}}_{[\{i\}]_{i\in I}}\circ \left(\mathop\boxtimes_{i\in I}\mu_{J_i}\right)\right)\circ\underline{\mathbf{I}^{[K_j]_{j\in J}}_{[\{j\}]_{j\in J}}}\circ \left(\mathop\boxtimes_{j\in J}\mu_{K_j}\right)\\
   &\underset{\cref{OperadLem}-(1)}{=}\Delta^{K/I}_*\mu_I\circ \Delta^{K/J}_*\mathbf{I}^{[J_i]_{i\in I}}_{[\{i\}]_{i\in I}}\circ \underline{\Delta^{K/J}_*\left(\mathop\boxtimes_{i\in I}\mu_{J_i}\right)\circ \mathbf{I}^{[K_{J_i}]_{i\in I}}_{[J_i]_{i\in I}}}\circ\mathop\boxtimes_{i\in I}\mathbf{I}^{[K_j]_{j\in J_i}}_{[\{j\}]_{j\in J_i}}\circ \left(\mathop\boxtimes_{j\in J}\mu_{K_j}\right)\\
   &\underset{ \cref{OperadLem}- (2)}{=}\Delta^{K/I}_*\mu_I\circ \underline{\Delta^{K/J}_*\mathbf{I}^{[J_i]_{i\in I}}_{[\{i\}]_{i\in I}}\circ \mathbf{I}^{[K_{J_i}]_{i\in I}}_{[J_i]_{i\in I}}}\circ\left(\mathop\boxtimes_{i\in I}\Delta^{K_{J_i}/J_i}_*\mu_{J_i}\right)\circ\mathop\boxtimes_{i\in I}\mathbf{I}^{[K_j]_{j\in J_i}}_{[\{j\}]_{j\in J_i}}\circ \left(\mathop\boxtimes_{j\in J}\mu_{K_j}\right)\\
    &\underset{\cref{OperadLem}-(3)}{=}\Delta^{K/I}_*\mu_I\circ \mathbf{I}^{[K_{J_i}]_{i\in I}}_{[\{i\}]_{i\in I}}\circ\left(\mathop\boxtimes_{i\in I}\Delta^{K_{J_i}/J_i}_*\mu_{J_i}\right)\circ\mathop\boxtimes_{i\in I}\mathbf{I}^{[K_j]_{j\in J_i}}_{[\{j\}]_{j\in J_i}}\circ \left(\mathop\boxtimes_{j\in J}\mu_{K_j}\right)\\
    &=\Delta^{K/I}_*\mu_I\circ \mathbf{I}^{[K_{J_i}]_{i\in I}}_{[\{i\}]_{i\in I}}\circ\left(\mathop\boxtimes_{i\in I}\Delta^{K_{J_i}/J_i}_*\mu_{J_i}\circ\mathbf{I}^{[K_j]_{j\in J_i}}_{[\{j\}]_{j\in J_i}}\circ \left(\mathop\boxtimes_{j\in J_i}\mu_{K_j}\right)\right)\\
    &=\gamma\left(\mu_{I}\otimes\left(\otimes^{i\in I}\gamma\left(\mu_{J_i}\otimes (\otimes^{j\in J_i}\mu_{K_j})\right)\right)\right),
\end{align*}
which is the required associativity property of the composition map $\gamma$.

\section{Propagators and Arnold relations}\label{sec: arnold relations}
In this section we introduce elements
\[ \Prop_{ij}\in \TS_{\AA^n}^k \simeq R\Gamma(\Conf_k(\AA^n),\O) \]
in cohomological degree $n-1$, called \dfn{propagators}. They generate nontrivial cohomology classes
\be [\Prop_{ij}] \in H(\TS_{\AA^n}^k) \cong H(\Conf_k(\AA^n),\O).\nn\ee
We shall show that these cohomology classes obey (a natural analog of) the \dfn{Arnold relations}; see \cref{thm: arnold} below.
Recall that the usual Arnold relations are the identities
\be \frac1{z_{ij} z_{j\ell}} + \frac1{z_{j\ell} z_{\ell i}} + \frac1{z_{\ell i} z_{ij}} = 0, \qquad 1\leq i<j<\ell\leq k, \label{eq: arnold relations}\ee
which hold in the commutative algebra of global sections of the structure sheaf over the configuration space of $k$ points in one dimension,
\be \kk\left[z_1,\dots,z_k\right]\left[\prod\ijk \frac1 {z_{ij}}\right]= \Gamma(\Conf_k(\AA^1),\O) .\nn\ee
Here, as throughout, we continue to %write $z_1,\dots,z_k$ for coordinates on $\AA^k$, and to
use the abbreviation $z_{ij} := z_i - z_j$.

%\subsection{The volume form}
Recall that for us $u^1,u^2,\dots,u^n$ denote coordinates on the algebro-geometric $(n-1)$-simplex
\(\triangle_{n-1} := \Spec\left(\kk[u^1,\dots,u^n]/\left< \sum\limits_{r=1}^n u^r - 1 \right> \right).\)
We choose volume form given by, for every $1\leq r\leq n$,
\begin{align} \Vol(\triangle_{n-1}) := (-1)^n\dd u^1 \dd u^2 \dots  \dd u^{n-1} &=
(-1)^r\dd u^1 \dots \widehat{\dd u^r} \dots \dd u^{n} \nn\\
& =\sum_{s=1}^n (-1)^s u^s \dd u^1 \dots \widehat{\dd u^s} \dots \dd u^{n}.\nn\end{align}

 % For every $r\in \{1,\dots,n\}$ one has
 % \begin{align} \Vol(\triangle_{n-1}) &= (-1)^r \dd u^1 \dots \dd u^{r-1} \dd u^{r+1} \dots \dd u^{n}\nn\\
 % %& = (-1)^{r + (r-1)(n-r)} \dd u^{r+1} \dots \dd u^n \dd u^1 \dots \dd u^{r-1} \nn\\
 % & = (-1)^{r + (r-1)(n-r)} \dd u^{r+1} \dots \dd u^n \dd u^1 \dots \dd u^{r-1} \label{spun}\\
 % & = (-1)^{r + (r-1)(n-r)} \dd\left( u^{r+1} \dd u^{r+2} \dots \dd u^n \dd u^1 \dots \dd u^{r-1}\right). \nn
 % \end{align}
 % For each $r$, this expresses $\Vol(\triangle_{n-1})$ as the differential of a form which vanishes on pullback to every face of the simplex except for the face given by $u^r=0$.

 % One also observes -- making use of \cref{spun} in the second step -- that
 % \begin{align} \Vol(\triangle_{n-1}) &= \left(\sum_{r=1}^n u^r\right) (-1)^n \dd u^1 \dots \dd u^{n-1} \nn\\
 % &= \sum_{r=1}^n (-1)^{r(n-r)} u^r \dd u^{r+1} \dots \dd u^n \dd u^1 \dots \dd u^{r-1}
 % %\nn\\&
 % = n(-1)^{n-1} u^{[1} \dd u^2 \dots \dd u^{n]}.\nn
 % \end{align}
 % (We define skew-symmetrization with weight 1.)

\subsection{Propagators}
Recall that $\{u^r_{ij}\}^{1\leq r\leq n}\ijk$ are coordinates on the polysimplex $\triangle_{n-1}^{\times \binom k2}$.
As a matter of notation, we shall write
\be u^r_{ji} := u^r_{ij}, \qquad 1\leq i<j\leq k .\nn\ee
We may then define, for all $1\leq i\neq j\leq k$,
\be \Vol(\triangle_{n-1})_{ij} := (-1)^n\dd u^1_{ij} \dd u^2_{ij} \dots \dd u^{n-1}_{ij} \label{def: vol}\ee
and the \dfn{propagator}
\be \Prop_{ij} := \frac{\Vol(\triangle_{n-1})_{ij}}{z^1_{ij} \dots z^n_{ij}} \in \TS_{\AA^n}^{k,n-1}. \nn\ee
Each $\Prop_{ij}$ defines a nontrivial cohomology class $[\Prop_{ij}] \in H^{n-1}(\TS_{\AA^n}^{k,\bul})$.
For each direction $r$, we have $(z^r_i-z^r_j)[P_{ij}]=0$ by noting that
\[
(z_i^r-z^r_j)P_{ij}=\frac{1}{n-1} \dd_{\TS}\left(\frac{\sum\limits_{s\neq r}\pm u^s_{ij}\dd u^1_{ij}\cdots \widehat{\dd u^{r}_{ij}}\cdots\widehat{\dd u^{s}_{ij}}\cdots \dd u^{n}_{ij}}{z^1_{ij}\cdots \widehat{z^r_{ij}}\cdots z^n_{ij}}\right).
\]
For example, when $n=2$, we have
\[
(z_i^1-z^1_j)P_{ij}=\frac{\dd u^1_{ij}}{z^2_{ij}}=-\ddts\left(\frac{u^2_{ij}}{z^2_{ij}}\right) \quad \text{(recall that}\ \dd u^1_{ij}+ \dd u^2_{ij}=0).
\]
Notice that we should not write $(z_i^1-z^1_j)P_{ij}$ as $\ddts(\frac{u^1_{ij}}{z^2_{ij}})$ because $\frac{u^1_{ij}}{z^2_{ij}}\notin \TS_{\AA^2}^{k,\bul}$ (it does not satisfy the boundary conditions).

The non-exactness of $P_{ij}$ can be seen using the residue operation $\mu_2^{i\rightarrow j}$ of \cref{prop: mu2Residue}: if $P_{ij}$ were $\ddts$-exact, then $\mu_2^{i\rightarrow j}(P_{ij}\dd\zz_i\shifted\boxtimes \dd\zz_j\shifted)$ would vanish since  $\mu_2^{i\rightarrow j}$ is closed, whereas in fact
\[\mu_2^{i\rightarrow j}(P_{ij}\dd\zz_i\shifted\boxtimes \dd\zz_j\shifted  )= \pm\dd\zz_j\shifted. \]

\subsection{Partial application, and the pairing of residues and propagators}\label{sec: pairing of residues and propagators}
The notation $2 \to 1$ in $\mu_2^{2\to 1}$ is a convenient shorthand: the chiral operation $\mu_2^{2\to 1}$ is associated to a surjection of index sets $\{1,2\} \onto \{1\}$. It is instructive to consider an example of \emph{partial} application of this chiral operation:
\[ \mu_2^{2\to 1} \left( P_{12} P_{23} \dd\zz_1\shifted\boxtimes \dd\zz_2\shifted \boxtimes\dd\zz_3\shifted \right) = \pm P_{13} \dd\zz_1\shifted\boxtimes\dd\zz_3\shifted\quad \in \TS^{\{1,3\}}_{\AA^n}\left((\oms)^{\boxtimes \{1,3\}}\right). \]
In practice, here the steps are: first, apply $\mu_2^{2\to 1}$ as defined in \cref{prop: mu2Residue}; second, pull back to a diagonally embedded polysimplex, via the map $u^r_{32} \mapsto u^r_{31}$, $1\leq r\leq n$, cf. \cref{def: c maps}.  More precisely, in operadic language what we are doing here is applying the \emph{parallel composition} of $\mu_2^{\{1,2\}\onto \{1\}}$ with $\id^{\{3\} \to \{3\}}$, which is given by
\[  \mathbf I_{[\{1\},\{3\}]}^{[\{1,2\},\{3\}]} \circ (\mu_2^{2\to 1}\boxtimes \id^{3\to 3}) \]
where $\mathbf I$ is the map of \cref{def: Imap}.
We may then go on to compute
\[ \mu_2^{3\to 1} \circ \mu_2^{2\to 1} \left( P_{12} P_{23} \dd\zz_1\shifted\boxtimes \dd\zz_2\shifted \boxtimes\dd\zz_3\shifted \right) = \pm\dd\zz_1\shifted. \]
where $\mu_2^{3\to 1} \circ \mu_2^{2\to 1}$ is really shorthand for the operadic composition of \cref{sec: operad}:
\[\mu_2^{3\to 1} \circ \mu_2^{2\to 1} := \gamma( \mu_2^{3\to 1} \ox \left( \mu_2^{2\to 1} \ox \id^{3\to 3} \right) := \Delta_*^{2\to 1}(\mu_2^{3\to 1}) \circ \left( \mathbf I_{[\{1\},\{3\}]}^{[\{1,2\},\{3\}]} \circ (\mu_2^{2\to 1}\boxtimes \id^{3\to 3}) \right) .\]

Similarly, one has
\[ \mu_2^{2\to 1} \circ \mu_2^{3\to 2} \left( P_{12} P_{23} \dd\zz_1\shifted\boxtimes \dd\zz_2\shifted \boxtimes\dd\zz_3\shifted \right) = \pm\dd\zz_1\shifted, \]
but for example $\mu_2^{3\to 1} \left( P_{12} P_{23} \dd\zz_1\shifted\boxtimes \dd\zz_2\shifted \boxtimes\dd\zz_3\shifted \right)=0$ and hence
\[ \mu_2^{2\to 1} \circ \mu_2^{3\to 1} \left( P_{12} P_{23} \dd\zz_1\shifted\boxtimes \dd\zz_2\shifted \boxtimes\dd\zz_3\shifted \right) = 0 . \]
These examples illustrate the way in which the residue maps pair with products of propagators,\footnote{In a fashion which is very familiar from the one-dimensional case, where it appears e.g. as the pairing between of flags and forms in \cite{SV}.} a pattern which will be important in our proofs in \cref{sec: main proof}.

%We will show that they generate the cohomology as a $\D$-module.
\subsection{The Arnold relations}
\begin{thm}\label{thm: arnold} The Arnold relations hold in cohomology:
\be [\Prop_{ij}][\Prop_{j\ell}] + [\Prop_{j\ell}][\Prop_{\ell i}] + [\Prop_{\ell i}][ \Prop_{ij}]  = 0 , \qquad 1\leq i<j<\ell\leq k, \nn\ee
in $H(\TS_{\AA^n}^k)$.
\end{thm}
\begin{proof}
We must show that $\Prop_{ij} \Prop_{j\ell}+\Prop_{j\ell} \Prop_{\ell i}+\Prop_{\ell i} \Prop_{ij}$ is exact, for all $1\leq i<j<\ell\leq k$.
And indeed, by a somewhat lengthy direct calculation, one checks that
\begin{align}\nn%\label{arnoldExplicit}
    &\Prop_{ij} \Prop_{j\ell}+\Prop_{j\ell} \Prop_{\ell i}+\Prop_{\ell i} \Prop_{ij}\\\nonumber
    &\quad=\frac{1}{n}\sum_{(I,J,K)\in \{1,\dots,n\}^{\{1,2,3\}}}\sign(I,J,K)\ddts(\Prop_{ij}^{\overline{I}} \Prop_{j\ell}^{\overline{J}} \Prop_{\ell i}^{\overline{K}}),
\end{align}
where:
\begin{enumerate}[-]
\item
For a subset $I\subset\{1,\dots,n\}$, we write $\overline I$ for the ordered tuple of the elements of the \emph{complement}, $\{1,\dots,n\} \setminus I$, in increasing order.
\item
For a set partition \( \{1,\dots,n\} = I \sqcup J \sqcup K \), we define
 $\sign(I,J,K)$ to be the sign of the unique permutation in $\sigma\in S_{2n}$ such that $\sigma(n+r)>\sigma(r)$ for each $1\leq r\leq n$ and $\sigma\on (1,\dots,n,1,\dots,n) = \overline{I}\#\overline{J}\#\overline{K}$, where $\#$ denotes concatination of tuples.
\item
Given any ordered $p$-tuple $R=(r_1,\dots,r_p)$ of elements of $\{1,\dots,n\}$, we write
\begin{align}
    P^R_{ij} :=\sum_{q=1}^p (-1)^q \frac{u_{ij}^{r_q} \dd u_{ij}^{r_1}\dots \wh{\dd u_{ij}^{r_q}}\dots \dd u_{ij}^{r_p}}{z^{r_1}_{ij}\dots z^{r_p}_{ij}}.
\nn\end{align}
\end{enumerate}
\end{proof}

The following example illustrates the main computational trick.

\begin{exmp} Consider the special case  of $n=2$ dimensions and $k=3$ marked points. We write $u^1_{ij}=u_{ij}$, and $u^2_{ij} = 1-u_{ij}$. The key is to notice that by rewriting the propagator $P_{12}$ in the form
\[ P_{12} = u_{13} \frac{\dd u_{12}}{z^1_{12} z^2_{12}} + (1-u_{13}) \frac{ \dd u_{12}}{z^1_{12} z^2_{12}}\]
one can rewrite $P_{12}P_{23}$ using the usual Arnold relations \cref{eq: arnold relations} while respecting the defining boundary conditions of $\TS_{\AA^2}^3$. One finds
\begin{align*}
    P_{12}P_{23}=\left(\frac{\dd u_{12}}{z^1_{12}}\frac{u_{31}-1}{z^2_{31}}-\frac{u_{31}}{z^1_{31}}\frac{\dd(u_{12}-1)}{z^2_{12}}\right)P_{23}+\left(\frac{u_{31}}{z^1_{31}}\frac{\dd(u_{23}-1)}{z^2_{23}}-\frac{\dd u_{23}}{z^1_{23}}\frac{u_{31}-1}{z^2_{31}}\right)P_{12}.
\end{align*}
Alternatively, one can derive the above identity using $z^1_{ij}P_{ij}=-\ddts(\frac{1-u_{ij}}{z^2_{ij}}),z^2_{ij}P_{ij}=\ddts(\frac{u_{ij}}{z^1_{ij}})$ and
\begin{align*}
     P_{12}P_{23}&=\left(u_{13}\cdot\frac{z^1_{13}}{z^1_{13}}+(1-u_{13})\frac{z^2_{13}}{z^2_{13}}\right)\cdot  P_{12}P_{23}\\
     &=\left(\frac{u_{13}}{z^1_{13}}\cdot (z^1_{12}+z^1_{23})+\frac{1-u_{13}}{z^2_{13}}\cdot (z^2_{12}+z^2_{23})\right)\cdot  P_{12}P_{23}\\
     &=\left(\frac{u_{13}}{z^1_{13}}\cdot z^1_{12}P_{12}+\frac{1-u_{13}}{z^2_{13}}\cdot z^2_{12}P_{12}\right)P_{23}+\left(\frac{u_{13}}{z^1_{13}}\cdot z^1_{23}P_{23}+\frac{1-u_{13}}{z^2_{13}}\cdot z^2_{23}P_{23}\right)P_{12}.
\end{align*}
In this way one sees that
\begin{align*}
&P_{12}P_{23}+P_{23}P_{31}+P_{31}P_{12}=\ddts\left(\left(\frac{u_{12}}{z^1_{12}}\frac{u_{31}-1}{z^2_{31}}-\frac{u_{31}}{z^1_{31}}\frac{u_{12}-1}{z^2_{12}}\right)P_{23}\right.\\
&\left.\qquad+\left(\frac{u_{23}}{z^1_{23}}\frac{u_{12}-1}{z^2_{12}}-\frac{u_{12}}{z^1_{12}}\frac{u_{23}-1}{z^2_{23}}\right)P_{31}
+\left(\frac{u_{31}}{z^1_{31}}\frac{u_{23}-1}{z^2_{23}}-\frac{u_{23}}{z^1_{23}}\frac{u_{31}-1}{z^2_{31}}\right)P_{12}\right).
\end{align*}
\end{exmp}
\begin{rem}
    One can also derive the Arnold relations in the Jouanolou model \cite{GWW25}.
\end{rem}

\section{Cohomology of the structure sheaf on configuration space}\label{sec: cohomology}
In this section we compute the cohomology of the \polysimplicial model
\be \TS_{\AA^n}^{k} \simeq R\Gamma(\Conf_k(\AA^n),\O) \nn\ee
of the derived global sections of the structure sheaf over configuration space for $k$ points in $\AA^n$ we defined in \cref{sec: TS def}. The main result is \cref{thm: cohFin}. As a corollary, we also get an explicit basis of $\Coh_{\dd_{\TS}}(\TS_{\AA^n}^{k})$ as a $\kk$-vector space, \cref{cor: basis}.

As mentioned in the previous section, the \polysimplicial model $\TS_{\AA^n}^{k}$ is not only a dg $\mathcal{O}_{(\AA^{n})^k}$-module but also a dg $\mathcal{D}_{(\AA^{n})^k}$-module. Remarkably, this $\mathcal{D}_{(\AA^{n})^k}$-module structure simplifies the spectral sequence computation.

Suppose we have a finite set of indices $I$, and an additional index $\ast \notin I$.

Let $\mathbf{H}_{\AA^n}^{I\smile\{\ast\}}$ denote the graded module over
\be\mathcal{D}_{(\mathbb{A}^{n})^{I\sqcup \{\ast\}}}= \kk[z_i^r,z_\ast^r,\partial_{z^r_i},\partial_{z^r_{\ast}}]_{i\in I}^{1\leq r\leq n}\nn\ee
with generators $\mathbf{1}$ and $P_{i\ast},i\in I$ such that
\begin{itemize}
  \item $\deg(\mathbf{1})=0$ and $\deg(P_{i\ast})=n-1,i\in I$,
  \item $\partial_{z^r_i}P_{i\ast}=-\partial_{z^r_\ast}P_{i\ast},\partial_{z^r_j}P_{i\ast}=0,i,j\in I, i\neq j,$
      \item$\del_{z^r_\ast} \mathbf 1 =\del_{z^r_i} \mathbf 1=0$ and $(z_i^r - z_\ast^r) \Prop_{i\ast} = 0.$
\end{itemize}

Denote $(\del_{z^1_{\ast}})^{p_1}\cdots (\del_{z^n_{\ast}})^{p_n}\Prop_{i\ast}$ by $\Prop_{i\ast}^{p_1,\dots,p_r,\dots,p_n}$. As an $\mathcal{O}_{(\mathbb{A}^n)^I}$-module, $\mathbf{H}_{\AA^n}^{I\smile\{\ast\}}$ is free:
\be\mathbf{H}_{\AA^n}^{I\smile\{\ast\}}=\mathcal{O}_{(\mathbb{A}^n)^{I\sqcup\{\ast\}}}\oplus\bigoplus^{}_{i\in I,p_1,\dots,p_n\geq 0}\mathcal{O}_{(\mathbb{A}^n)^{I}}\cdot\Prop_{i\ast}^{p_1,\dots,p_n}.\nn\ee
The $\mathcal{O}_{(\mathbb{A}^n)^{I\sqcup\{\ast\}}}$-module structure is given by rewriting the previous relations:
\be z^r_\ast  \cdot \Prop_{i\ast}^{p_1,\dots,p_n} =z^r_i \cdot \Prop_{i\ast}^{p_1,\dots,p_n} - p_r \Prop_{i\ast}^{p_1,\dots,p_r-1,\dots,p_n}.\nn\ee

\begin{thm}[Cohomology of $\TS$]
\label{thm: cohFin}
Let $k>1$, $I=\{1,\dots,k\},I_1=\{1\},\dots,I_{k-1}=\{1,\dots,k-1\}$. We have the isomorphism
\begin{align} \Coh_{\dd_{\TS}}(\TS_{\AA^n}^{I}) &\cong \mathbf{H}_{\AA^n}^{I_1\smile\{2\}} \ox_{\O_{(\AA^n)^{I_2}}} \mathbf{H}_{\AA^n}^{I_2\smile\{3\}} \ox_{\O_{(\AA^n)^{I_3}}} \dots%\nn\\&\qquad\qquad  \dots
\ox_{\O_{(\AA^n)^{I_{k-1}}}} \mathbf{H}_{\AA^n}^{I_{k-1}\smile\{k\}}\nn\end{align}
of graded modules over $\mathcal{D}_{(\mathbb{A}^{n})^I}=\kk[z_i^r,\partial_{z^r_i}]_{i\in I}^{1\leq r\leq n}$.
\end{thm}

The proof of \cref{thm: cohFin} is given in section \cref{sec: proof of thm cohFin} below.

Intuitively speaking, \cref{thm: cohFin} says that the cohomology is generated by products of propagators such that no two propagators ever ``end'' at the same vertex (where $\Prop_{ij}$, $i<j$ ``starts'' at $i$ and ``ends'' at $j$). For example the following is allowed,
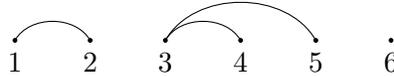
\be    \begin{tikzpicture}[]
        \foreach \x in {1,2,3,4,5,6}{
            \coordinate (\x) at (\x,0);
           \node[draw,circle,fill=black,minimum size=0.3pt,inner sep=.5pt,label=below:{$\x$}]() at (\x) {};
        }
        \draw (1) to[out=60,in=120] (2);
        \draw (3) to[out=60,in=120] (4);
        \draw (3) to[out=60,in=120] (5);
 %       \draw[dashed] (2) to[out=-45,in=-135] (6);
 %       \draw[dashed] (5) to[out=-45,in=-135] (6);
    \end{tikzpicture}\nn\ee
as also is
\be
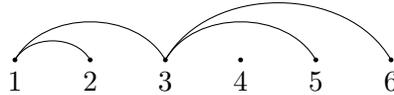
\begin{tikzpicture}[]
        \foreach \x in {1,2,3,4,5,6}{
            \coordinate (\x) at (\x,0);
           \node[draw,circle,fill=black,minimum size=0.3pt,inner sep=.5pt,label=below:{$\x$}]() at (\x) {};
        }
        \draw (1) to[out=60,in=120] (2);
        \draw (1) to[out=60,in=120] (3);
        \draw (3) to[out=60,in=120] (6);
        \draw (3) to[out=60,in=120] (5);
 %       \draw[dashed] (2) to[out=-45,in=-135] (6);
 %       \draw[dashed] (5) to[out=-45,in=-135] (6);
    \end{tikzpicture}\nn\ee
but the following is \emph{not}
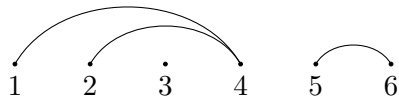
\be\begin{tikzpicture}[]
        \foreach \x in {1,2,3,4,5,6}{
            \coordinate (\x) at (\x,0);
           \node[draw,circle,fill=black,minimum size=0.3pt,inner sep=.5pt,label=below:{$\x$}]() at (\x) {};
        }
        \draw (1) to[out=60,in=120] (4);
        \draw (2) to[out=60,in=120] (4);
        \draw (5) to[out=60,in=120] (6);
 %       \draw[dashed] (2) to[out=-45,in=-135] (6);
 %       \draw[dashed] (5) to[out=-45,in=-135] (6);
    \end{tikzpicture}\nn\ee
Thus, one can think that what \cref{thm: cohFin} is doing is choosing certain representatives from among the equivalence classes of the product of propagators defined by the Arnold relations in cohomology, cf. \cref{sec: arnold relations}.

From \cref{thm: cohFin} we can write down a $\kk$-linear basis of the cohomology $\Coh_{\dd}(\TS_{\AA^n}^{k})$ regarded as a $\kk$-vector space. Indeed, consider the following collection of vectors,
   \begin{align}
        \label{basisFin}
        \left(\prod_{i\in  I_{J,(\ell_j)}}\prod_{r=1}^n(z_{i}^r-z_{k}^r)^{n_{r,i}}\right)\left(\prod_{r=1}^n(z_{k}^r)^{n_r}\right)
        \left(\prod_{j\in J}[\Prop_{\ell_jj}^{m_{1,j},\dots,m_{n,j}}]\right),
    \end{align}
each of which is labelled by:
\begin{enumerate}[--]
\item a choice of subset $J\subset\{2,\dots,k\}$\\ (for each $j\in J$ there will be a propagator that ``ends'' at $j$)
\item for each $j\in J$ an integer $\ell_j$ such that $1\leq \ell_j< j$ \\(specifying where the corresponding propagator ``starts'')
\item additional labels $n_{r,i},n_r,m_{r,j}\geq 0$, where $1\leq r\leq n$ and $j \in I$, and where $i\in I_{J,(\ell_j)}$ runs over the following set
\begin{align*}
        I_{J,(\ell_j)}:=\{1,\dots,k-1\}\setminus\{\ell_j:j\in J\}\setminus \{j\in J : \exists j'\in J\text{ with } j'>j\text{ and }l_{j'}=l_j\}
    \end{align*}
(controlling which factors $(z_i-z_{k})$ may appear for a given product of propagators).
\end{enumerate}

For example when $k=6$ we may choose $J=\{2,4,5\}$ and $l_1=1,l_3=3,l_4=4$. Then $I_{J,(\ell_j)}=\{2,5\}$. This can be visualised as follows, where the solid lines again represent propagators $\Prop_{\ell_jj}^{m_{1,j},\dots,m_{n,j}}$ while the dashed lines represent allowed factors of $(z_{i}^r-z_{k}^r)^{n_{r,i}}$.
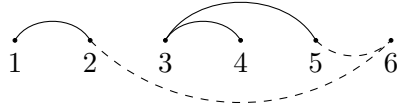
\begin{figure}[h!]
    \begin{tikzpicture}[]
        \foreach \x in {1,2,3,4,5,6}{
            \coordinate (\x) at (\x,0);
           \node[draw,circle,fill=black,minimum size=0.3pt,inner sep=.5pt,label=below:{$\x$}]() at (\x) {};
        }
        \draw (1) to[out=60,in=120] (2);
        \draw (3) to[out=60,in=120] (4);
        \draw (3) to[out=60,in=120] (5);
        \draw[dashed] (2) to[out=-45,in=-135] (6);
        \draw[dashed] (5) to[out=-45,in=-135] (6);
    \end{tikzpicture}
\end{figure}

\begin{cor}
    \label{cor: basis}
The elements \cref{basisFin} form a $\kk$-linear basis of $\Coh_{\dd_{\TS}}(\TS_{\AA^n}^{k})$.
\qed\end{cor}

The proof is straightforward and we omit the details, but the idea is the following.
First, a $\kk$-linear basis of $\mathbf{H}_{\AA^n}^{I\smile\{\ast\}}$ is evidently given by
\begin{align}
    \label{basisInterm}
    &\prod_{j\in I}\prod_{r=1}^n(z_j^r-z_\ast^r)^{n_{r,j}}\prod_{r=1}^n(z_\ast^r)^{n_r},\quad n_{r,j},n_r\geq0\\
    &\prod_{j\in I\setminus\{i\}}\prod_{r=1}^n(z_j^r-z_\ast^r)^{n_{r,j}}\prod_{r=1}^n(z_\ast^r)^{n_r}[\Prop_{i\ast}^{m_1,\dots,m_n}],\quad i\in I, \qquad n_{r,j},n_r,m_r\geq 0.\nn
\end{align}
In view of \cref{thm: cohFin}, by taking tensor products of such basis vectors we arrive at an (overcomplete) spanning set of $\Coh_{\ddts}(\TS_{\AA^n}^{k})$.
% Using \cref{thm: cohFin} we construct a spanning set from the basis \cref{basisInterm}. Here, elements look like:
% \begin{figure}[!h]
%     \centering
%     \begin{tikzpicture}[]
%         \foreach \x in {1,2,3,4,5,6}{
%             \coordinate (\x) at (\x,0);
%            \node[draw,circle,fill=black,minimum size=0.3pt,inner sep=.5pt,label=below:{$\x$}]() at (\x) {};
%         }
%         \draw (1) to[out=60,in=120] (2);
%         \draw (3) to[out=60,in=120] (4);
%         \draw (3) to[out=60,in=120] (5);
%         \draw[dashed] (5) to[out=-45,in=-135] (6);
%         \draw[dashed] (4) to[out=-45,in=-135] (6);
%         \draw[dashed] (3) to[out=-45,in=-135] (6);
%         \draw[dashed] (2) to[out=-45,in=-135] (6);
%         \draw[dashed] (1) to[out=-45,in=-135] (6);
%         \draw[dashed] (4) to[out=-45,in=-135] (5);
%         \draw[dashed] (2)
%         to[out=-45,in=-135] (5);
%         \draw[dashed] (1) to[out=-45,in=-135] (5);
%         \draw[dashed] (2) to[out=-45,in=-135] (4);
%         \draw[dashed] (1) to[out=-45,in=-135] (4);
%         \draw[dashed] (2) to[out=-45,in=-135] (3);
%         \draw[dashed] (1) to[out=-45,in=-135] (3);
%     \end{tikzpicture}
% \end{figure}
Then we note that, using $(z_i-z_j)=(z_i-z_{\ell}) - (z_j-z_\ell)$, we can perform replacements of dotted lines in favour of dashed lines as in the following examples:

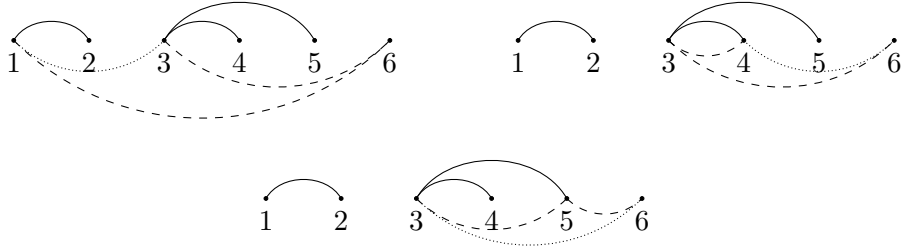
\begin{figure}[!h]
    \centering
    \begin{tikzpicture}[]
        \foreach \x in {1,2,3,4,5,6}{
            \coordinate (\x) at (\x,0);
           \node[draw,circle,fill=black,minimum size=0.3pt,inner sep=.5pt,label=below:{$\x$}]() at (\x) {};
        }
        \draw (1) to[out=60,in=120] (2);
        \draw (3) to[out=60,in=120] (4);
        \draw (3) to[out=60,in=120] (5);
        \draw[densely dotted] (1) to[out=-45,in=-135] (3);
        \draw[dashed] (1) to[out=-45,in=-135] (6);
        \draw[dashed] (3) to[out=-45,in=-135] (6);
    \end{tikzpicture}
    \hspace{1cm}
    \begin{tikzpicture}[]
        \foreach \x in {1,2,3,4,5,6}{
            \coordinate (\x) at (\x,0);
           \node[draw,circle,fill=black,minimum size=0.3pt,inner sep=.5pt,label=below:{$\x$}]() at (\x) {};
        }
        \draw (1) to[out=60,in=120] (2);
        \draw (3) to[out=60,in=120] (4);
        \draw (3) to[out=60,in=120] (5);
        \draw[densely dotted] (4) to[out=-45,in=-135] (6);
        \draw[dashed] (3) to[out=-45,in=-135] (6);
        \draw[dashed] (3) to[out=-45,in=-135] (4);
        \draw[dashed,opacity=0] (1) to[out=-45,in=-135] (6); %This is heere only for spacing
    \end{tikzpicture}
    \begin{tikzpicture}[]
        \foreach \x in {1,2,3,4,5,6}{
            \coordinate (\x) at (\x,0);
           \node[draw,circle,fill=black,minimum size=0.3pt,inner sep=.5pt,label=below:{$\x$}]() at (\x) {};
        }
        \draw (1) to[out=60,in=120] (2);
        \draw (3) to[out=60,in=120] (4);
        \draw (3) to[out=60,in=120] (5);
        \draw[densely dotted] (3) to[out=-45,in=-135] (6);
        \draw[dashed] (5) to[out=-45,in=-135] (6);
        \draw[dashed] (3) to[out=-45,in=-135] (5);
\end{tikzpicture}
\end{figure}

Dashed lines parallel to solid lines can immediately be absorbed into propagators (which then might become exact and vanish at the level of cohomology). Iterating, we arrive at the vectors of the form \cref{basisFin} above.

\subsection{Proof of \cref{thm: cohFin}}\label{sec: proof of thm cohFin}
Our general strategy is inspired by \cite{CLM}: we shall consider a ``relative'' version of the complex in which all but one of the points are ``fixed'', leaving only one ``moving'' point, labelled $*$. This relative complex is sufficiently simple that its cohomology can be computed directly. Then we compute the cohomology of the original complex using induction and a spectral sequence.

Recall from \cref{sec: gens and rels for TS} the definition of $\TS_{\AA^n}^{I,E}$ for a finite set $I$ of vertices and $E\subset I^{[2]}/S_2$ of undirected edges. Given a set of directed edges $E\subset I\two$, we shall write
$\TS_{\AA^n}^{I,E} := \TS_{\AA^n}^{I,(E\cup \sigma E)/S_2}$,
where $S_2 = \{\id, \sigma\}$.

%We shall compute the cohomology of $\TS_{\AA^n}^{I} = \TS_{\AA^n}^{I,I\two}$ inductively by adding vertices.

We begin with the simple case of two points $\{i,*\}$ and one edge $e = (i,*)$. Notice that the differential $\dd$ on $\TS_{\AA^n}^{\{i,\ast\},\{(i,\ast)\}}$ is
$$
\dd_{(i,\ast)}=\sum_{r=1}^n \dd u_{i\ast}^r \frac{\del}{\del u^r_{i\ast}} .
$$
\begin{lem} There is an isomorphism of graded $\D_{(\AA^n)^{\{i,*\}}}$-modules
    \label{lem: cohMin}
    \begin{align*}
        \Coh_{\dd_{(i,\ast)}}(\TS_{\AA^n}^{\{i,\ast\},\{(i,\ast)\}})\cong \mathbf{H}_{\AA^n}^{\{i\}\smile\{\ast\}}.
    \end{align*}
\end{lem}
\begin{proof}
For $S\subset \{1,\dots,n\}$ consider the algebra $\Omega_S^\bul(\triangle_{n-1})=\{\tau \in   \Omega^\bul(\triangle_{n-1}) : \tau\vert_{\cap_{s\in S}(u^s=0)}=0\}\nn$ of polynomial differential forms on the simplex $\triangle_{n-1}$ vanishing on boundaries along directions $s\in S$. It is standard that
    \begin{align*}
        \Coh^p(\Omega^\bul_S(\triangle_{n-1}))=\left.
        \begin{cases}
            \kk, & \text{for $S=\emptyset$ and $p=0$} \\
            \kk [\Vol(\triangle_{n-1})], & \text{for $S=\{1,\dots,n\}$ and $p=n-1$} \\
             0, & \text{else},
        \end{cases}
        \right.
    \end{align*}
where $\Vol(\triangle_{n-1})_{ij} := (-1)^n\dd u^1_{ij} \dd u^2_{ij} \dots \dd u^{n-1}_{ij}$.
See for example \cite[Lemma 9.5]{GriffithsMorgan}

Now note that (since $z_i^r=z^r_\ast+(z_i^r-z^r_\ast)$) we have
    \begin{align*}
        \TS_{\AA^n}^{\{i,\ast\},\{(i,\ast)\}}
        &= \{\tau \in  \kk[z_\ast^r,(z_i^r-z_\ast^r)^{\pm 1}]^{1\leq r\leq n}\ox\Omega(\triangle_{n-1})\nn\\
        &\qquad: \tau|_{u^r_{i\ast} =0} \text{ regular in $(z_i^r-z_\ast^r)$, for $1\leq r\leq n$}\}.
    \end{align*}
As a dg vector space,
\be \TS_{\AA^n}^{\{i,\ast\},\{(i,\ast)\}} \cong \oplus_{f}\Omega_{S_f}^\bul(\triangle_{n-1})\nn\ee where $f$ runs over monic  monomials in $\kk[z_\ast^r,(z_i^r-z_\ast^r)^{\pm1}]^{1\leq r\leq n}$ and $S_f:=\{s\in\{1,\dots,n\} : f \text{ singular in $(z_i^s-z_\ast^s)$}\}$.
    %The result then follows due to the following observation:
    %Every primitive monomial $f$ with $I_f=\{1,\dots,n\}$ is of the form $\prod_{r=1}^n\frac{1}{(z_l^r-z_{k+1}^r)^{n'_r}}\prod_{i=1}^{k+1}(z_i^r)^{m'_{r,i}}$ for $n'_r>0$ and $m'_{r,i}\geq 0$. Using $z_i^r=z^r_{k+1}+(z_i^r-z^r_{k+1})$ it is clear that $f$ can be written as a linear combination of elements of the form $\prod_{r=1}^n\frac{(z_{k+1}^r)^{m_r}}{(z_l^r-z_{k+1}^r)^{n_r}}$ for $n_r>0$ and $m_r\geq 0$ and elements $\wt f$ with $I_{\wt f}\neq \{1,\dots,n\}$.
   The monomials $f$ such that $S_f=\{1,\dots,n\}$ are of the form
\be (z_*^1)^{m_1}\dots (z_*^n)^{m_n} \left(\frac{\del}{\del z^1_{*}}\right)^{p_1} \dots \left(\frac{\del}{\del z^n_{*}}\right)^{p_n} \prod_{r=1}^n \frac1 {z^r_i-z^r_*} ,\qquad m_1,\dots,m_n,p_1,\dots,p_n\geq 0.\nn\ee
Thus, with the identification
\be\label{propIdentify} \Prop_{i*} \simeq  \left(\prod_{r=1}^n \frac1 {z^r_i-z^r_*}\right)[\Vol(\triangle_{n-1})], \ee
we have the result.
    %This establishes the proof.
\end{proof}
Next we note that
\be\label{decompTS}\TS_{\AA^n}^{I\sqcup\{\ast\},I\times\{\ast\}} \cong \bigotimes_{i\in I}\TS_{\AA^n}^{\{i,\ast\},\{(i,\ast)\}}.\ee
Here the tensor product on the right-hand side is of modules over $\O_{\AA^n}=\kk[z_*^r]^{1\leq r\leq n}$. Note that both $\TS_{\AA^n}^{\{i,*\},\{(i,*)\}}$ and its cohomology $\mathbf{H}_{\AA^n}^{\{i\}\smile\{\ast\}}$ are projective (in fact, free) over $\O_{\AA^n}=\kk[z_*^r]^{1\leq r\leq n}$. Therefore by the K\"unneth theorem \cite[Theorem 3.6.3]{Weibel} we obtain that
\be\label{decompTS2} \Coh_{\dd_{I\times\{\ast\}}}(\TS_{\AA^n}^{I\sqcup\{\ast\},I\times\{\ast\}}) \cong \bigotimes_{i\in I}\Coh(\TS_{\AA^n}^{\{i,\ast\},\{(i,\ast)\}}) = \bigotimes_{i\in I}  \mathbf{H}_{\AA^n}^{\{i\}\smile\{\ast\}}\nn\ee
as an isomorphism of graded vector spaces. This isomorphism is compatible with $\mathcal{D}_{(\mathbb{A}^{n})^{I\sqcup \{\ast\}}}$-structure. Here we have introduced an explicit notation for the differential
\be \dd_{I\times\{\ast\}} = \sum_{1\leq i\leq k} \sum_{r=1}^n \dd u_{i\ast}^r \frac{\del}{\del u^r_{i\ast}} \nn\ee
of the complex $\TS_{\AA^n}^{I\sqcup\{\ast\},I\times\{\ast\}}$.

At this point, roughly speaking, what we have shown is that the cohomology $\Coh_{\dd_{I\times\{*\}}}(\TS_{\AA^n}^{I\sqcup\{\ast\},I\times\{\ast\}})$ is generated by \emph{products} of propagators $\Prop_{i*}$ (and their derivatives). The next result shows that when we change base (dg) ring and work over $\TS_{\AA^n}^{I}$, then the cohomology of $\dd_{I\times \{*\}}$ is \emph{linearly} generated by such propagators (and their derivatives).
\begin{lem}\label{lem: isom}
There is an isomorphism of graded $\mathcal{D}_{(\mathbb{A}^{n})^{I\sqcup \{\ast\}}}$-modules
\begin{align}           \Coh_{\dd_{I\times\{\ast\}}}\left(\TS_{\AA^n}^{I,I^{[2]}}\ox_{\O_{(\AA^n)^{I}}} \TS_{\AA^n}^{I\sqcup\{\ast\},I\times\{\ast\}}\right)\cong\TS_{\AA^n}^{I,I^{[2]}} \ox_{\O_{(\AA^n)^{I}}} \mathbf{H}_{\AA^n}^{I\smile\{\ast\}}.
\end{align}
\end{lem}
\begin{proof}
%By \cref{cohInterm} $\Coh_{\dd_{I\times\{\ast\}}}\left( \TS_{\AA^n}^{I\sqcup\{\ast\},I\times\{\ast\}}\right)$ consists of products of elements in $\kk[z^r_\ast,(z_i^r-z^r_\ast)]^{1\leq r\leq n}$ respectively $\\Prop_{\AA^n}^{i,\ast}$ for $i\in I$.
We have that $\TS_{\AA^n}^{I,I^{[2]}}$ is flat over $\ox_{\O_{(\AA^n)^{I}}} $ (by \cref{thm: TSIE gens and rels} together with \cref{lem: FlatE}). Hence, we find using \cref{decompTS} that
\begin{align*}
    \Coh_{\dd_{I\times\{\ast\}}}\left(\TS_{\AA^n}^{I,I^{[2]}} \ox_{\O_{(\AA^n)^{I}}} \TS_{\AA^n}^{I\sqcup\{\ast\},I\times\{\ast\}}\right)
&\cong \TS_{\AA^n}^{I,I^{[2]}} \ox_{\O_{(\AA^n)^{I}}} \Coh_{\dd_{I\times\{\ast\}}}\left( \TS_{\AA^n}^{I\sqcup\{\ast\},I\times\{\ast\}}\right)\nn\\
&\cong \TS_{\AA^n}^{I,I^{[2]}} \ox_{\O_{(\AA^n)^{I}}}  \bigotimes_{i\in I} \mathbf{H}_{\AA^n}^{\{i\}\smile\{\ast\}} \nn
\end{align*}
The result follows if we can show that, in this graded ring, all products of the form $[\Prop^{p_1,\dots,p_n}_{i\ast}][\Prop^{p'_1,\dots,p'_n}_{j\ast}]$ are zero for $i\neq j\in I$.
Indeed, we have that
\begin{align*}
    &1\ox[\Prop_{i\ast}][\Prop_{j\ast}]=1\ox[\Prop_{i\ast}\Prop_{j\ast}]\\
    &\qquad\qquad=\sum_{r=1}^nu_{ij}^r\ox[\Prop_{i\ast}\Prop_{j\ast}]=\sum_{r=1}^nu_{ij}^r\frac{z^r_{ij}}{z^r_{ij}}\ox[\Prop_{i\ast}\Prop_{j\ast}]\\
    &\qquad\qquad=\sum_{r=1}^n\frac{u_{ij}^r}{z^r_{ij}}\ox[z^r_{ij}\Prop_{i\ast}\Prop_{j\ast}]\\
    &\qquad\qquad=\sum_{r=1}^n\frac{u_{ij}^r}{z^r_{ij}}\ox[z^r_{i\ast}\Prop_{i\ast}\cdot\Prop_{j\ast}]-\sum_{r=1}^n\frac{u_{ij}^r}{z^r_{ij}}\ox[\Prop_{i\ast}\cdot z^r_{j\ast}\Prop_{j\ast}]=0.
\end{align*}

The same argument generalizes to arbitrary elements in $\mathbf{H}_{\AA^n}^{\{i\}\smile\{\ast\}}$ and $\mathbf{H}_{\AA^n}^{\{j\}\smile\{\ast\}}$ of degree $n-1$ if we replace $\frac{z^r_{ij}}{z^r_{ij}}\to (\frac{z^r_{ij}}{z^r_{ij}})^N$ for $N$ sufficiently large.
\end{proof}

Next we observe that (by \cref{thm: TS gens and rels})
\be  \TS_{\AA^n}^{I\sqcup\{\ast\}} \cong \TS_{\AA^n}^{I} \ox_{\O_{(\AA^n)^I}} \mathbf{Q}^I_{\AA^n},\quad \mathbf{Q}^I_{\AA^n}:=\TS_{\AA^n}^{I\sqcup \{\ast\}, I \times \{\ast\}}.\nn\ee
In this way $\TS_{\AA^n}^{I\sqcup\{\ast\}}$ is a bigraded $\kk$-vector space. There is a corresponding decomposition of the differential $\dd$ on $\TS_{\AA^n}^{I\sqcup\{\ast\}}$ as a sum of two anticommuting differentials,
\be \dd_{\TS} = \dd'_{\mathbf{P}} + \dd_{\mathbf{Q}},\nn\ee
\be \dd'_{\mathbf{P}} = \sum_{1\leq i<j\leq k} \sum_{r=1}^n \dd u_{ij}^r \frac{\del}{\del u^r_{ij}}, \qquad
    \dd_{\mathbf{Q}} =\dd_{I\times\{\ast\}}= \sum_{1\leq i\leq k} \sum_{r=1}^n \dd u_{i\ast}^r \frac{\del}{\del u^r_{i\ast}} .\nn\ee
In this way, the complex $\TS_{\AA^n}^{I\sqcup\{\ast\}}$ becomes the totalization a bicomplex:
\be \left(\TS_{\AA^n}^{I\sqcup\{\ast\}},\dd_{\TS}\right)  \cong \Tot\left( \TS_{\AA^n}^{I} \ox_{\O_{(\AA^n)^I}} \mathbf{Q}^I_{\AA^n} , \dd'_{\mathbf{P}}, \dd_{\mathbf{Q}}\right).\nn\ee

Now we can complete the proof of \cref{thm: cohFin}.
\begin{proof}[Proof of \cref{thm: cohFin}]
The proof is by induction on $k$. The base case is the case $I=\{1\}$, $\ast=2$ of \cref{lem: cohMin}.
Now suppose inductively that we know that
\begin{align} \Coh_{\dd_{\TS}}(\TS_{\AA^n}^{I}) &\cong \mathbf{H}_{\AA^n}^{I_1\smile\{2\}} \ox_{\O_{(\AA^n)^{I_2}}} \mathbf{H}_{\AA^n}^{I_2\smile\{3\}} \ox_{\O_{(\AA^n)^{I_3}}} \dots%\nn\\&\qquad\qquad  \dots
\ox_{\O_{(\AA^n)^{I_{k-1}}}} \mathbf{H}_{\AA^n}^{I_{k-1}\smile\{k\}}\nn\end{align}
as $\mathcal{D}_{(\mathbb{A}^{n})^I}$-modules. We compute
$$
H\left(\TS_{\AA^n}^{I\sqcup\{\ast\}},\dd_{\TS}\right)  =H\left(\mathbf{P}^I_{\AA^n}\ox_{\O_{(\AA^n)^I}}\mathbf{Q}^I_{\AA^n},\dd'_{\mathbf{P}}+\dd_\mathbf{Q}\right)
$$
by using the following spectral sequence
\begin{itemize}
  \item $E^{p,q}_0=\mathbf{P}^{I,p}_{\AA^n}\ox_{\O_{(\AA^n)^I}}\mathbf{Q}^{I,q}_{\AA^n}$ with the differential
      \be d_0=\dd_{\mathbf{Q}}:\mathbf{P}^{I,p}_{\AA^n}\ox_{\O_{(\AA^n)^I}}\mathbf{Q}^{I,q}_{\AA^n}\rightarrow \mathbf{P}^{I,p}_{\AA^n}\ox_{\O_{(\AA^n)^I}}\mathbf{Q}^{I,q+1}_{\AA^n}.\nn\ee
  \item $E^{p,q}_1=\mathbf{P}^{I,p}_{\AA^n}\ox_{\O_{(\AA^n)^I}}H^q_{d_{\mathbf{Q}}}(\mathbf{Q}^{I}_{\AA^n})\cong \mathbf{P}^{I,p}_{\AA^n}\ox_{\O_{(\AA^n)^I}}   \mathbf{H}_{\AA^n}^{I\smile\{\ast\},q}$ with the differential (here we use \cref{lem: isom}) \be d_1=\dd'_{\mathbf{P}}:\mathbf{P}^{I,p}_{\AA^n}\ox_{\O_{(\AA^n)^I}}   \mathbf{H}_{\AA^n}^{I\smile\{\ast\},q}\rightarrow \mathbf{P}^{I,p+1}_{\AA^n}\ox_{\O_{(\AA^n)^I}}   \mathbf{H}_{\AA^n}^{I\smile \{\ast\},q}.\nn\ee
\end{itemize}
Since $\mathbf{H}_{\AA^n}^{I\smile \{\ast\}}$ is a free $\O_{(\AA^n)^I}$-module, on the $E_2$-page we have
\be E^{\bullet,q}_2=  \Coh_{\dd_{\TS}}^\bullet(\TS_{\AA^n}^{I})\ox_{\O_{(\AA^n)^{I}}} \mathbf{H}_{\AA^n}^{I\smile\{\ast\},q}\nn\ee
\be=\mathbf{H}_{\AA^n}^{I_1\smile\{2\}} \ox_{\O_{(\AA^n)^{I_2}}} \mathbf{H}_{\AA^n}^{I_2\smile\{3\}} \ox_{\O_{(\AA^n)^{I_3}}} \dots%\nn\\&\qquad\qquad  \dots
\ox_{\O_{(\AA^n)^{I_{k-1}}}} \mathbf{H}_{\AA^n}^{I_{k-1}\smile\{k\}}\ox_{\O_{(\AA^n)^{I}}} \mathbf{H}_{\AA^n}^{I\smile\{\ast\},q}.\nn\ee
As $\mathbf{H}^{I\smile\{\ast\}}_{\AA^n}$ is concentrated at degree 0 and $n-1$, we only have two non-zero rows on the $E_2$-page. By the induction hypothesis and the fact that $\mathbf{H}_{\AA^n}^{I_l\smile\{l+1\}}$ are all concentrated at degree 0 and $n-1$, we have
$$
E^{p,q}_0\neq 0\Longleftrightarrow p=r(n-1),q=s(n-1),\quad r\geq 0, s=0,1.
$$

When $n\geq 3$, from Fig. \ref{SpectralAn3} we can see that the edge maps, $d_r=0,r\geq 2$, are zero.

\begin{figure}[htp]
  \centering

\tikzset{every picture/.style={line width=0.75pt}} %set default line width to 0.75pt

\begin{tikzpicture}[x=0.75pt,y=0.75pt,yscale=-1,xscale=1]

%Shape: Axis 2D [id:dp6795686252960798]
\draw [dash pattern={on 15pt off 6pt}] (169.08,202.82) -- (417.08,202.82)(193.88,40) -- (193.88,220.91) (410.08,197.82) -- (417.08,202.82) -- (410.08,207.82) (188.88,47) -- (193.88,40) -- (198.88,47)  ;
%Straight Lines [id:da7657144957335233]
\draw    (194,148) -- (240.26,168.92) ;
\draw [shift={(242.08,169.74)}, rotate = 204.33] [color={rgb, 255:red, 0; green, 0; blue, 0 }  ][line width=0.75]    (10.93,-3.29) .. controls (6.95,-1.4) and (3.31,-0.3) .. (0,0) .. controls (3.31,0.3) and (6.95,1.4) .. (10.93,3.29)   ;
\draw [shift={(194,148)}, rotate = 24.33] [color={rgb, 255:red, 0; green, 0; blue, 0 }  ][fill={rgb, 255:red, 0; green, 0; blue, 0 }  ][line width=0.75]      (0, 0) circle [x radius= 3.35, y radius= 3.35]   ;
%Straight Lines [id:da4226917017052789]
\draw    (194,203) ;
\draw [shift={(194,203)}, rotate = 0] [color={rgb, 255:red, 0; green, 0; blue, 0 }  ][fill={rgb, 255:red, 0; green, 0; blue, 0 }  ][line width=0.75]      (0, 0) circle [x radius= 3.35, y radius= 3.35]   ;
%Straight Lines [id:da15341605298992023]
\draw    (249,203) ;
\draw [shift={(249,203)}, rotate = 0] [color={rgb, 255:red, 0; green, 0; blue, 0 }  ][fill={rgb, 255:red, 0; green, 0; blue, 0 }  ][line width=0.75]      (0, 0) circle [x radius= 3.35, y radius= 3.35]   ;
%Straight Lines [id:da7678976551618808]
\draw    (305,203) ;
\draw [shift={(305,203)}, rotate = 0] [color={rgb, 255:red, 0; green, 0; blue, 0 }  ][fill={rgb, 255:red, 0; green, 0; blue, 0 }  ][line width=0.75]      (0, 0) circle [x radius= 3.35, y radius= 3.35]   ;
%Straight Lines [id:da5991913810498737]
\draw    (250,148) ;
\draw [shift={(250,148)}, rotate = 0] [color={rgb, 255:red, 0; green, 0; blue, 0 }  ][fill={rgb, 255:red, 0; green, 0; blue, 0 }  ][line width=0.75]      (0, 0) circle [x radius= 3.35, y radius= 3.35]   ;
%Straight Lines [id:da4506260334677803]
\draw    (305,148) ;
\draw [shift={(305,148)}, rotate = 0] [color={rgb, 255:red, 0; green, 0; blue, 0 }  ][fill={rgb, 255:red, 0; green, 0; blue, 0 }  ][line width=0.75]      (0, 0) circle [x radius= 3.35, y radius= 3.35]   ;
%Straight Lines [id:da8739369279438574]
\draw    (194,148) -- (267.42,196.63) ;
\draw [shift={(269.08,197.74)}, rotate = 213.52] [color={rgb, 255:red, 0; green, 0; blue, 0 }  ][line width=0.75]    (10.93,-3.29) .. controls (6.95,-1.4) and (3.31,-0.3) .. (0,0) .. controls (3.31,0.3) and (6.95,1.4) .. (10.93,3.29)   ;
%Straight Lines [id:da10659649287649309]
\draw    (194,148) -- (302.5,231.52) ;
\draw [shift={(304.08,232.74)}, rotate = 217.59] [color={rgb, 255:red, 0; green, 0; blue, 0 }  ][line width=0.75]    (10.93,-3.29) .. controls (6.95,-1.4) and (3.31,-0.3) .. (0,0) .. controls (3.31,0.3) and (6.95,1.4) .. (10.93,3.29)   ;
%Straight Lines [id:da8340693194620743]
\draw    (362,203) ;
\draw [shift={(362,203)}, rotate = 0] [color={rgb, 255:red, 0; green, 0; blue, 0 }  ][fill={rgb, 255:red, 0; green, 0; blue, 0 }  ][line width=0.75]      (0, 0) circle [x radius= 3.35, y radius= 3.35]   ;
%Straight Lines [id:da24671703606709405]
\draw    (362,148) ;
\draw [shift={(362,148)}, rotate = 0] [color={rgb, 255:red, 0; green, 0; blue, 0 }  ][fill={rgb, 255:red, 0; green, 0; blue, 0 }  ][line width=0.75]      (0, 0) circle [x radius= 3.35, y radius= 3.35]   ;

% Text Node
\draw (222.37,197) node [anchor=north] [inner sep=0.75pt]  [font=\footnotesize,color={rgb, 255:red, 0; green, 0; blue, 0 }  ,opacity=1 ] [align=left] {0};
% Text Node
\draw (195.37,170) node [anchor=north] [inner sep=0.75pt]  [font=\footnotesize,color={rgb, 255:red, 0; green, 0; blue, 0 }  ,opacity=1 ] [align=left] {0};
% Text Node
\draw (222.37,170) node [anchor=north] [inner sep=0.75pt]  [font=\footnotesize,color={rgb, 255:red, 0; green, 0; blue, 0 }  ,opacity=1 ] [align=left] {0};
% Text Node
\draw (195.37,114) node [anchor=north] [inner sep=0.75pt]  [font=\footnotesize,color={rgb, 255:red, 0; green, 0; blue, 0 }  ,opacity=1 ] [align=left] {0};
% Text Node
\draw (195.37,86) node [anchor=north] [inner sep=0.75pt]  [font=\footnotesize,color={rgb, 255:red, 0; green, 0; blue, 0 }  ,opacity=1 ] [align=left] {0};
% Text Node
\draw (195.37,58) node [anchor=north] [inner sep=0.75pt]  [font=\footnotesize,color={rgb, 255:red, 0; green, 0; blue, 0 }  ,opacity=1 ] [align=left] {0};
% Text Node
\draw (278.37,197) node [anchor=north] [inner sep=0.75pt]  [font=\footnotesize,color={rgb, 255:red, 0; green, 0; blue, 0 }  ,opacity=1 ] [align=left] {0};
% Text Node
\draw (250.37,170) node [anchor=north] [inner sep=0.75pt]  [font=\footnotesize,color={rgb, 255:red, 0; green, 0; blue, 0 }  ,opacity=1 ] [align=left] {0};
% Text Node
\draw (278.37,170) node [anchor=north] [inner sep=0.75pt]  [font=\footnotesize,color={rgb, 255:red, 0; green, 0; blue, 0 }  ,opacity=1 ] [align=left] {0};
% Text Node
\draw (305.37,170) node [anchor=north] [inner sep=0.75pt]  [font=\footnotesize,color={rgb, 255:red, 0; green, 0; blue, 0 }  ,opacity=1 ] [align=left] {0};
% Text Node
\draw (222.37,114) node [anchor=north] [inner sep=0.75pt]  [font=\footnotesize,color={rgb, 255:red, 0; green, 0; blue, 0 }  ,opacity=1 ] [align=left] {0};
% Text Node
\draw (250.37,114) node [anchor=north] [inner sep=0.75pt]  [font=\footnotesize,color={rgb, 255:red, 0; green, 0; blue, 0 }  ,opacity=1 ] [align=left] {0};
% Text Node
\draw (278.37,114) node [anchor=north] [inner sep=0.75pt]  [font=\footnotesize,color={rgb, 255:red, 0; green, 0; blue, 0 }  ,opacity=1 ] [align=left] {0};
% Text Node
\draw (305.37,114) node [anchor=north] [inner sep=0.75pt]  [font=\footnotesize,color={rgb, 255:red, 0; green, 0; blue, 0 }  ,opacity=1 ] [align=left] {0};
% Text Node
\draw (222.37,142) node [anchor=north] [inner sep=0.75pt]  [font=\footnotesize,color={rgb, 255:red, 0; green, 0; blue, 0 }  ,opacity=1 ] [align=left] {0};
% Text Node
\draw (278.37,142) node [anchor=north] [inner sep=0.75pt]  [font=\footnotesize,color={rgb, 255:red, 0; green, 0; blue, 0 }  ,opacity=1 ] [align=left] {0};
% Text Node
\draw (334.37,197) node [anchor=north] [inner sep=0.75pt]  [font=\footnotesize,color={rgb, 255:red, 0; green, 0; blue, 0 }  ,opacity=1 ] [align=left] {0};
% Text Node
\draw (390.37,197) node [anchor=north] [inner sep=0.75pt]  [font=\footnotesize,color={rgb, 255:red, 0; green, 0; blue, 0 }  ,opacity=1 ] [align=left] {0};
% Text Node
\draw (334.37,170) node [anchor=north] [inner sep=0.75pt]  [font=\footnotesize,color={rgb, 255:red, 0; green, 0; blue, 0 }  ,opacity=1 ] [align=left] {0};
% Text Node
\draw (362.37,170) node [anchor=north] [inner sep=0.75pt]  [font=\footnotesize,color={rgb, 255:red, 0; green, 0; blue, 0 }  ,opacity=1 ] [align=left] {0};
% Text Node
\draw (334.37,142) node [anchor=north] [inner sep=0.75pt]  [font=\footnotesize,color={rgb, 255:red, 0; green, 0; blue, 0 }  ,opacity=1 ] [align=left] {0};
% Text Node
\draw (231,150.4) node [anchor=north west][inner sep=0.75pt]  [font=\scriptsize]  {$d_{2}$};
% Text Node
\draw (261,178.4) node [anchor=north west][inner sep=0.75pt]  [font=\scriptsize]  {$d_{3}$};
% Text Node
\draw (291,209.4) node [anchor=north west][inner sep=0.75pt]  [font=\scriptsize]  {$d_{4}$};
% Text Node
\draw (170,33.4) node [anchor=north west][inner sep=0.75pt]    {$q$};
% Text Node
\draw (413,210.4) node [anchor=north west][inner sep=0.75pt]    {$p$};

\end{tikzpicture}
  \caption{The $\mathbb{A}^3$ case}\label{SpectralAn3}
\end{figure}

Now we assume that $n=2$. From Fig. \ref{SpectralAn2}, we only need to show that $d_2=0$. Notice that $d_2$ is a $\mathcal{D}_{(\AA^n)^{I\sqcup\{\ast\}}}$-map
$$
d_2:E^{p,1}_2\rightarrow E^{p+2,0}_2.
$$
Suppose that $\tau=\rho\otimes P_{i\ast}\in E^{p,1}_2=\Coh^p_{\dd_{\TS}}(\TS_{\AA^n}^{I})\ox_{\O_{(\AA^n)^{I}}} \mathbf{H}_{\AA^n}^{I\smile\{\ast\},1}$, then we can find $N\gg 0$ such that
$$
(\partial_{z^1_\ast})^Nd_2(\rho\otimes P_{i\ast})=0
$$
since $ E^{p+2,0}_2=\Coh^{p+2}_{\dd_{\TS}}(\TS_{\AA^n}^{I})\ox_{\O_{(\AA^n)^{I}}} \mathbf{H}_{\AA^n}^{I\smile\{\ast\},0}=\Coh^{p+2}_{\dd_{\TS}}(\TS_{\AA^n}^{I})\ox_{\O_{(\AA^n)^{I}}}\O_{(\AA^n)^{I\sqcup\{\ast\}}}$ contains only polynomials in $z^1_\ast$. Then
\begin{align*}
   d_2(\rho\otimes P_{i\ast})&=\frac{(-1)^N}{N!}d_2\left((z^1_i-z^1_{\ast})^N\cdot\rho\otimes (\partial_{z^1_\ast})^NP_{i\ast}\right) \\
   &=\frac{(-1)^N}{N!}(z^1_i-z^1_{\ast})^N\cdot(\partial_{z^1_\ast})^Nd_2\left(\rho\otimes P_{i\ast}\right)\\
   &=0.
\end{align*}
Since $E^{p,1}_2$ is generated by elements like $\rho\otimes P_{i\ast}$ as a $\mathcal{D}$-module, we conclude that $d_2=0$. The proof is complete.
\begin{figure}[htp]
  \centering

\tikzset{every picture/.style={line width=0.75pt}} %set default line width to 0.75pt

\begin{tikzpicture}[x=0.75pt,y=0.75pt,yscale=-1,xscale=1]

%Shape: Axis 2D [id:dp7920357376985787]
\draw [dash pattern={on 15pt off 6pt}] (169.08,202.82) -- (417.08,202.82)(193.88,40) -- (193.88,220.91) (410.08,197.82) -- (417.08,202.82) -- (410.08,207.82) (188.88,47) -- (193.88,40) -- (198.88,47)  ;
%Straight Lines [id:da6066897171894272]
\draw    (194,203) ;
\draw [shift={(194,203)}, rotate = 0] [color={rgb, 255:red, 0; green, 0; blue, 0 }  ][fill={rgb, 255:red, 0; green, 0; blue, 0 }  ][line width=0.75]      (0, 0) circle [x radius= 3.35, y radius= 3.35]   ;
%Straight Lines [id:da01838087727922666]
\draw    (249,203) ;
\draw [shift={(249,203)}, rotate = 0] [color={rgb, 255:red, 0; green, 0; blue, 0 }  ][fill={rgb, 255:red, 0; green, 0; blue, 0 }  ][line width=0.75]      (0, 0) circle [x radius= 3.35, y radius= 3.35]   ;
%Straight Lines [id:da1994810909780056]
\draw    (305,203) ;
\draw [shift={(305,203)}, rotate = 0] [color={rgb, 255:red, 0; green, 0; blue, 0 }  ][fill={rgb, 255:red, 0; green, 0; blue, 0 }  ][line width=0.75]      (0, 0) circle [x radius= 3.35, y radius= 3.35]   ;
%Straight Lines [id:da6645918787363285]
\draw    (277,203) ;
\draw [shift={(277,203)}, rotate = 0] [color={rgb, 255:red, 0; green, 0; blue, 0 }  ][fill={rgb, 255:red, 0; green, 0; blue, 0 }  ][line width=0.75]      (0, 0) circle [x radius= 3.35, y radius= 3.35]   ;
%Straight Lines [id:da8901469507684103]
\draw    (221,203) ;
\draw [shift={(221,203)}, rotate = 0] [color={rgb, 255:red, 0; green, 0; blue, 0 }  ][fill={rgb, 255:red, 0; green, 0; blue, 0 }  ][line width=0.75]      (0, 0) circle [x radius= 3.35, y radius= 3.35]   ;
%Straight Lines [id:da880075439381385]
\draw    (194,176) -- (240.25,196.31) ;
\draw [shift={(242.08,197.11)}, rotate = 203.71] [color={rgb, 255:red, 0; green, 0; blue, 0 }  ][line width=0.75]    (10.93,-3.29) .. controls (6.95,-1.4) and (3.31,-0.3) .. (0,0) .. controls (3.31,0.3) and (6.95,1.4) .. (10.93,3.29)   ;
\draw [shift={(194,176)}, rotate = 23.71] [color={rgb, 255:red, 0; green, 0; blue, 0 }  ][fill={rgb, 255:red, 0; green, 0; blue, 0 }  ][line width=0.75]      (0, 0) circle [x radius= 3.35, y radius= 3.35]   ;
%Straight Lines [id:da04240930658083175]
\draw    (222,176) ;
\draw [shift={(222,176)}, rotate = 0] [color={rgb, 255:red, 0; green, 0; blue, 0 }  ][fill={rgb, 255:red, 0; green, 0; blue, 0 }  ][line width=0.75]      (0, 0) circle [x radius= 3.35, y radius= 3.35]   ;
%Straight Lines [id:da4669346277976818]
\draw    (251,176) ;
\draw [shift={(251,176)}, rotate = 0] [color={rgb, 255:red, 0; green, 0; blue, 0 }  ][fill={rgb, 255:red, 0; green, 0; blue, 0 }  ][line width=0.75]      (0, 0) circle [x radius= 3.35, y radius= 3.35]   ;
%Straight Lines [id:da4011592841313769]
\draw    (278,176) ;
\draw [shift={(278,176)}, rotate = 0] [color={rgb, 255:red, 0; green, 0; blue, 0 }  ][fill={rgb, 255:red, 0; green, 0; blue, 0 }  ][line width=0.75]      (0, 0) circle [x radius= 3.35, y radius= 3.35]   ;
%Straight Lines [id:da7437735960871015]
\draw    (306,176) ;
\draw [shift={(306,176)}, rotate = 0] [color={rgb, 255:red, 0; green, 0; blue, 0 }  ][fill={rgb, 255:red, 0; green, 0; blue, 0 }  ][line width=0.75]      (0, 0) circle [x radius= 3.35, y radius= 3.35]   ;
%Straight Lines [id:da8939227378442636]
\draw    (194,176) -- (275.39,227.05) ;
\draw [shift={(277.08,228.11)}, rotate = 212.1] [color={rgb, 255:red, 0; green, 0; blue, 0 }  ][line width=0.75]    (10.93,-3.29) .. controls (6.95,-1.4) and (3.31,-0.3) .. (0,0) .. controls (3.31,0.3) and (6.95,1.4) .. (10.93,3.29)   ;
%Straight Lines [id:da5926701760199524]
\draw    (194,176) -- (297.44,247.97) ;
\draw [shift={(299.08,249.11)}, rotate = 214.83] [color={rgb, 255:red, 0; green, 0; blue, 0 }  ][line width=0.75]    (10.93,-3.29) .. controls (6.95,-1.4) and (3.31,-0.3) .. (0,0) .. controls (3.31,0.3) and (6.95,1.4) .. (10.93,3.29)   ;

% Text Node
\draw (195.37,114) node [anchor=north] [inner sep=0.75pt]  [font=\footnotesize,color={rgb, 255:red, 0; green, 0; blue, 0 }  ,opacity=1 ] [align=left] {0};
% Text Node
\draw (195.37,86) node [anchor=north] [inner sep=0.75pt]  [font=\footnotesize,color={rgb, 255:red, 0; green, 0; blue, 0 }  ,opacity=1 ] [align=left] {0};
% Text Node
\draw (195.37,58) node [anchor=north] [inner sep=0.75pt]  [font=\footnotesize,color={rgb, 255:red, 0; green, 0; blue, 0 }  ,opacity=1 ] [align=left] {0};
% Text Node
\draw (224.37,142) node [anchor=north] [inner sep=0.75pt]  [font=\footnotesize,color={rgb, 255:red, 0; green, 0; blue, 0 }  ,opacity=1 ] [align=left] {0};
% Text Node
\draw (252.37,142) node [anchor=north] [inner sep=0.75pt]  [font=\footnotesize,color={rgb, 255:red, 0; green, 0; blue, 0 }  ,opacity=1 ] [align=left] {0};
% Text Node
\draw (279.37,142) node [anchor=north] [inner sep=0.75pt]  [font=\footnotesize,color={rgb, 255:red, 0; green, 0; blue, 0 }  ,opacity=1 ] [align=left] {0};
% Text Node
\draw (195.37,142) node [anchor=north] [inner sep=0.75pt]  [font=\footnotesize,color={rgb, 255:red, 0; green, 0; blue, 0 }  ,opacity=1 ] [align=left] {0};
% Text Node
\draw (334.37,197) node [anchor=north] [inner sep=0.75pt]  [font=\footnotesize,color={rgb, 255:red, 0; green, 0; blue, 0 }  ,opacity=1 ] [align=left] {0};
% Text Node
\draw (390.37,197) node [anchor=north] [inner sep=0.75pt]  [font=\footnotesize,color={rgb, 255:red, 0; green, 0; blue, 0 }  ,opacity=1 ] [align=left] {0};
% Text Node
\draw (308.37,142) node [anchor=north] [inner sep=0.75pt]  [font=\footnotesize,color={rgb, 255:red, 0; green, 0; blue, 0 }  ,opacity=1 ] [align=left] {0};
% Text Node
\draw (362.37,197) node [anchor=north] [inner sep=0.75pt]  [font=\footnotesize,color={rgb, 255:red, 0; green, 0; blue, 0 }  ,opacity=1 ] [align=left] {0};
% Text Node
\draw (232,177.4) node [anchor=north west][inner sep=0.75pt]  [font=\scriptsize]  {$d_{2}$};
% Text Node
\draw (265,205.4) node [anchor=north west][inner sep=0.75pt]  [font=\scriptsize]  {$d_{3}$};
% Text Node
\draw (170,33.4) node [anchor=north west][inner sep=0.75pt]    {$q$};
% Text Node
\draw (413,210.4) node [anchor=north west][inner sep=0.75pt]    {$p$};
% Text Node
\draw (334.37,171) node [anchor=north] [inner sep=0.75pt]  [font=\footnotesize,color={rgb, 255:red, 0; green, 0; blue, 0 }  ,opacity=1 ] [align=left] {0};
% Text Node
\draw (363.37,171) node [anchor=north] [inner sep=0.75pt]  [font=\footnotesize,color={rgb, 255:red, 0; green, 0; blue, 0 }  ,opacity=1 ] [align=left] {0};
% Text Node
\draw (284,226.4) node [anchor=north west][inner sep=0.75pt]  [font=\scriptsize]  {$d_{4}$};

\end{tikzpicture}
  \caption{The $\mathbb{A}^2$ case}\label{SpectralAn2}
\end{figure}
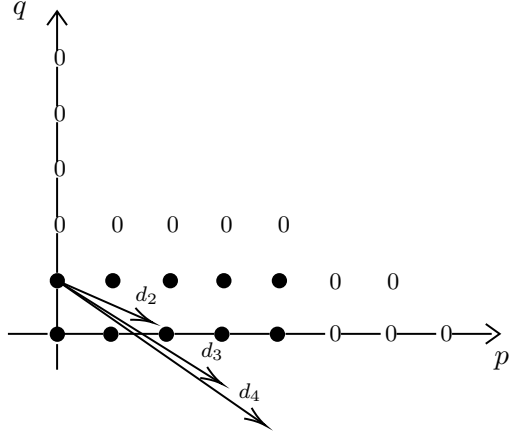
\end{proof}

\begin{rem}
 One may expect to use the K\"unneth formula \cite[Theorem 3.6.3]{Weibel}. However, although $\mathbf{P}^I_{\AA^n},\mathbf{Q}^I_{\AA^n}$ are complexes of flat $\mathcal{O}_{(\mathbb{A}^n)^I}$-module, the modules of coboundary $d_{\mathbf{P}}(\mathbf{P}^{I,p}_{\AA^n}),d_{\mathbf{Q}}(\mathbf{Q}^{I,q}_{\AA^n})$ are not flat. Thus, the K\"unneth formula does not apply directly.
\end{rem}

\subsection{Translation invariant version}\label{sec:TransInv-Coh}
Later in this paper, we will use a translation invariant version of \cref{thm: cohFin}.  If we choose a base point $\bp=1\in I=\{1,\dots,k\}$, then we can define the translation invariant subcomplex of $\TS_{\AA^n}^{I}$
\begin{align} \TS_{\AA^n,\circ}^{I}&=\TS_{\AA^n,\circ}^{I,\bul}  := \bigl\{ \tau \in
 \kk[(z^r_i-z^r_{\circ}),(z^r_j-z^r_\ell)^{-1}]^{1\leq r\leq n}_{{i\in I- \{\circ\}, j\neq \ell\in I}} \ox \Omega^\bul(\triangle_{n-1}^{\times \binom k 2}) \\
& \quad :          \text{$\tau|_{u^r_{ij}=0}$ is regular in $(z^r_i-z^r_j)$,  for $1\leq r\leq n$ and $i\neq j\in I$} \bigr\}. \nn
\end{align}

By an argument that parallels to the proof of \cref{thm: TS gens and rels}, there is an isomorphism of commutative dg algebras
\be \TS^I_{\AA^n,\circ} \cong \mathbf{B}^I_{\mathbb{A}^n,\circ} :=\directlim_{m\geq 0}\frac{\kk\left[ (z^r_i-z^r_{\circ}), \frac{u^r_{j\ell}}{(z^r_j-z^r_\ell)^m}, \frac{\dd u^r_{j\ell}}{(z^r_j-z^r_\ell)^m}, \right]^{1\leq r\leq n}_{i\in I-\{\circ\},j\neq l\in I}}{\langle \sum\limits_{r=1}^{n} u^r_{ij} -1, \sum\limits_{r=1}^{n} \dd u^r_{ij} \rangle\ijk} \nn. \ee

Then
$$
\TS_{\AA^n}^{I}\cong \mathbf{B}_{\AA^n}^{I}=\mathbf{k}[z^r_{\circ}]^{1\leq r\leq n}\otimes_{\mathbf{k}}\mathbf{B}_{\AA^n,\circ}^{I}\cong \mathbf{k}[z^r_{\circ}]^{1\leq r\leq n}\otimes_{\mathbf{k}}\TS_{\AA^n,\circ}^{I}.
$$
Thus the cohomology of $\TS_{\AA^n}^{I}$ can be identified as the cohomology of the translation invariant part with a tensor factor $\mathbf{k}[z^r_{\circ}]^{1\leq r\leq n}$
\be
 \Coh_{\dd_{\TS}}(\TS_{\AA^n}^{I}) \cong  \mathbf{k}[z^r_{\circ}]^{1\leq r\leq n}\otimes_{\mathbf{k}} \Coh_{\dd_{\TS}}(\TS_{\AA^n,\circ}^{I}) .\nn\ee
For $I_l=\{1=\circ,\dots,l\},l<k$, we have
 \be
 \mathbf{H}_{\AA^n}^{I_l\smile\{l+1\}}=\mathbf{k}[z^r_{\circ}]^{1\leq r\leq n}\otimes_{\mathbf{k}} \mathbf{H}_{\AA^n,\circ}^{I_l\smile\{l+1\}}
 \nn\ee
 where\be\mathbf{H}_{\AA^n,\circ}^{I_l\smile\{l+1\}}=\mathcal{O}_{(\mathbb{A}^n)^{I_l\setminus\{\circ\}\sqcup\{l+1\}}}\oplus\bigoplus^{}_{i\in I_l,p_1,\dots,p_n\geq 0}\mathcal{O}_{(\mathbb{A}^n)^{I_l\setminus\{\circ\}}}\cdot\Prop_{i\ast}^{p_1,\dots,p_n},\ast=l+1,\nn\ee
 \be\mathcal{O}_{(\mathbb{A}^n)^{I_l\setminus\{\circ\}}}\cong \mathbf{k}[(z^r_i-z^r_{\circ})]_{i\in I_l\setminus\{\circ\}}^{1\leq r\leq n}.\nn\ee

With the above notation, we have the following theorem.
\begin{thm}
\label{thm: cohFinTransInv}
 We have the isomorphism
\begin{align} \Coh_{\dd_{\TS}}(\TS_{\AA^n,\circ}^{I}) &\cong \mathbf{H}_{\AA^n,\circ}^{I_1\smile\{2\}} \ox_{\O_{(\AA^n)^{I_2\setminus\{\circ\}}}} \mathbf{H}_{\AA^n,\circ}^{I_2\smile\{3\}} \ox_{\O_{(\AA^n)^{I_3\setminus\{\circ\}}}} \dots%\nn\\&\qquad\qquad  \dots
\ox_{\O_{(\AA^n)^{I_{k-1}\setminus\{\circ\}}}} \mathbf{H}_{\AA^n,\circ}^{I_{k-1}\smile\{k\}}\nn\end{align}
of graded modules over $\kk[(z_i^r-z_\circ^r),\partial_{z^r_i},\partial_{z^r_\circ}]_{i\in I\setminus\{\circ\}}^{1\leq r\leq n}$.
\end{thm}
\begin{proof}
    The same argument as in the proof of \cref{thm: cohFin}.
\end{proof}
\begin{cor}\label{cor:TransInvBasis}
    Within the same notation in \cref{cor: basis}, we have the following basis of  $\Coh_{\dd_{\TS}}(\TS_{\AA^n,\circ}^{I})$
    $$
        \left(\prod_{i\in  I_{J,(\ell_j)}}\prod_{r=1}^n(z_{i}^r-z_{k}^r)^{n_{r,i}}\right)
        \left(\prod_{j\in J}[\Prop_{\ell_jj}^{m_{1,j},\dots,m_{n,j}}]\right).
    $$
\end{cor}
\section{Proof of the main theorem}\label{sec: main proof}
In this section we prove the main result of the paper, \cref{thm: ChiralIsoLie}. The following result, \cref{prop: Injectivity}, is a key step.
\subsection{Injectivity property%: Injectivity of the induced map on homology
}\label{sec: InjectivityHomology}
Recall that the chain complex of $k$-ary operations of the chiral operad is
\be \P^{ch}_{\AA^n,\bul}(k) := \Hom_{\DAAnk}\left(\TS_{\AA^n}^{k,\bul-(n-1)}((\oms)^{\boxtimes k}), \Delta_*\om \right)\nn\ee
and it is equipped with the differential given by
\begin{equation} \dd_\P\mu := \mu \circ \ddts .\nn\end{equation}
%Our goal in this section is to prove the following extension result.

\begin{prop}[Injectivity property]
\label{prop: Injectivity}
If a chiral operation $\mu\in \P^{ch}_{\AA^n,p}(k)$  is such that
\[ \mu|_{\ker(\ddts)}=0,\]
(which implies that $\dd_{\TS}\mu=0$) then there exists $\tilde{\mu}\in \P^{ch}_{\AA^n,p+1}(k)$
such that
\[ %\dd_{\P}  \tilde{\mu} %\equiv
\mu=\tilde\mu\circ\ddts
=\dd_\P\mu.\]
Equivalently, the following induced map out of the homology of the chiral operad is injective:
 \begin{align}
H^p_{\dd_P}(\P^{ch}_\bul(k)) %=&H_{\dd_{\P}}\left(\Hom_{\DAAnk}\left(\TS_{\AA^n}^{k,\bul}((\oms)^{\boxtimes k}), \Delta_*\oms \right)\right)\nn\\
&\rightarrow\Hom_{\DAAnk}\left(H^p_{\ddts}\left(\TS_{\AA^n}^{k,\bul-(n-1)}((\oms)^{\boxtimes k})\right),\Delta_*\om\right),\quad p\in \ZZ.
 \nn\end{align}
\end{prop}
\begin{proof} In \cref{prop: translation invariant Pch} below, we shall show that the chiral operad has an equivalent translation-invariant description. It is then enough to establish the statement in that translation-invariant form, and we shall do so in \cref{prop: translation invariant injectivity} below.
\end{proof}
\iffalse
First observe that the two formulations in the statement of the proposition are indeed equivalent:

Let the induced map out of $H_{\dd_P}(\P^{ch}_{\AA^n,\bul}(k))$ be called $\phi$. It is given explicitly by
$\phi([\mu])([\tau]) := \mu(\tau)$ for any closed $\mu$ and (co)closed $\tau$. (One checks easily that this is well-defined). So by definition $\phi([\mu])$ is the zero map precisely if $\mu(\tau)=0$ whenever $\ddts\tau =0$. That is, $\phi([\mu])=0$ precisely if $\mu|_{\ker(\ddts)} = 0$. By definition $[\mu]$ itself is zero precisely if $\mu = \dd_\P \tilde\mu$ for some $\tilde\mu$. So the existence of such a $\tilde\mu$ whenever  $\mu|_{\ker(\ddts)} = 0$  is indeed the condition for the map $\phi$ to be injective.
\fi
\begin{rem}
  It seems that a better way than below to prove the above proposition is to find a subcategory of $\D$-modules such that $\Delta_*\om$ is an injective object \cite{BDChiralAlgebras}.
\end{rem}

\subsection{Translation invariant description}\label{sec: translation invariant}
Our marked points in $\AA^n$ are labelled by a finite index set $I$. It is convenient to pick one of them as the base point. Let $\bp\in I$ denote the label of this preferred point.
We have
\begin{equation} \O_{(\AA^n)^I}\equiv \kk\left[z^r_i\right]^{1\leq r\leq n}_{i\in I}
   = \kk[z^r_\bp]^{1\leq r\leq n} \ox_\kk \kk[z^r_{\bp j}]^{1\leq r\leq n}_{j\in I\setminus\{\bp\}}  = \O_{\AA^n} \ox_\kk \O_{(\AA^n)^I,\basepoint}, \nn\end{equation}
where we think of
\begin{equation} \O_{(\AA^n)^I,\basepoint}:= \kk[z^r_{\bp j}]^{1\leq r\leq n}_{j\in I\setminus\{\bp\}}\nn\end{equation}
as the translation invariant factor of $\O_{(\AA^n)^I}$.
For $E\subset I\two/S_2$ we define $\TS_{\AA^n,\basepoint}^{I,E}$, $\mathbf B_{\AA^n,\basepoint}^{I,E}$ by obvious analogy with the definitions of $\TS_{\AA^n}^{I,E}$, $\mathbf B_{\AA^n}^{I,E}$ in \cref{sec:TransInv-Coh}. By an argument that parallels the proof of \cref{thm: TSIE gens and rels}, one shows that
%\begin{thm}\label{thm: trans inv TSIE gens and rels}
there is an isomorphism of commutative dg algebras
\begin{equation} \TS^{I,E}_{\AA^n,\basepoint} \cong \mathbf B^{I,E}_{\AA^n,\basepoint} \nn\end{equation}
for any finite sets of indices $I$ and edges $E\subset I\two/S_2$.
% such that $(i,j) \in E \Leftrightarrow (j,i)\in E$.
%\qed\end{thm}

\begin{prop}[Translation invariant description of chiral operad]\label{prop: translation invariant Pch}
We have
\begin{align}
\P^{ch}_{\AA^n,\bul}(k)&:=\Hom_{\DAAnk}\left(\TS_{\AA^n}^{k,\bul-(n-1)}((\oms)^{\boxtimes k}),\Delta_*\om\right)\nn\\
&\cong\Hom_{\kk[z^s_{v\bp},\del_{z^s_v}]^{1\leq s\leq n}_{v\in I\setminus\{\bp\}}}\left(\TS_{{\AA^n},\basepoint}^{k,\bul-(n-1)}((\oms)^{\boxtimes k}),\dd\zz\ox \kk[\lambda^r_v]^{1\leq r\leq n}_{v\in I\setminus\{\bp\}}\right),\nn
\end{align}
where the right $\kk[z^s_{v\bp},\del_{z^s_v}]^{1\leq s\leq n}_{v\in I\setminus\{\bp\}}$-module structure on $\dd\zz\ox \kk[\lambda_v]$ is given by
\[ \dd\zz\cdot z^s_{v\bp}=0,\qquad \dd\zz\cdot \del_{z^s_v}=d\mathbf{z}\cdot\lambda^s_v.\]
\end{prop}
\begin{proof}
Recall that in \cref{sec: shifted canonical bundle} we have
\begin{equation}
\Delta_*\om=\kk[z^r]^{1\leq r\leq n} \dd\zz\ox_{\kk[\lambda^r]^{1\leq r\leq n}}\kk[\lambda^r_i]^{1\leq r\leq n}_{i\in I}
\label{DAn}\end{equation}
where $\kk[\lambda^r]^{1\leq r\leq n}$ acts on $\kk[\lambda^r]^{1\leq r\leq n}_{i\in I}$  as $\lambda^r \cdot p := \sum\limits_{i\in I} \lambda^r_i p$ and acts on $\kk[z^r]^{1\leq r\leq n} \dd\zz$ as $p(z) \dd\zz \cdot \lambda^r := -\del_{z^r} p(z)\cdot\dd\zz$.
Here the $z^1,\dots,z^n$ are the coordinates on the main-diagonal copy of $\AA^n$.
Relative to our choice of base point $\bp\in I$, we have an isomorphism
\begin{equation} \Delta_*
\om
\cong \kk[z^r]^{1\leq r\leq n}  \dd\zz \ox_\kk \kk[\lambda^r_v]^{1\leq r\leq n}_{v\in I \setminus \{\bp\}}. \nn\end{equation}
Since
\(\TS_{\AA^n}^{k,\bul}=\kk[z^s_\bp]^{1\leq s\leq n}\ox_\kk\TS_{\AA^n,\basepoint}^{k,\bul},\)
we certainly have the map
\begin{align} &\Hom_{\DAAnk}\left(\TS_{\AA^n}^{k,\bul-(n-1)}((\oms)^{\boxtimes k}),\Delta_*\om\right)\nn\\
&\rightarrow\Hom_{\kk[z^s_{v\bp},\del_{z^s_v}]^{1\leq s\leq n}_{v\in I\setminus\{\bp\}}}\left(\TS_{{\AA^n},\basepoint}^{k,\bul-(n-1)}((\oms)^{\boxtimes k}),\kk[z^r]^{1\leq r\leq n}%\cdot
 \dd\zz\ox_\kk\kk[\lambda^r_i]^{1\leq r\leq n}_{i\in I\setminus\{\bp\}}\right).
\nn\end{align}
We first want to show that for any chiral operation $\mu$ and any $\alpha \in \TS_{{\AA^n},\basepoint}^{k,\bul-(n-1)}((\oms)^{\boxtimes k})$, we in fact have
\begin{equation} \mu(\alpha)\in \dd\zz\ox_\kk\kk[\lambda^r_i]^{1\leq r\leq n}_{i\in I\setminus\{\bp\}}\nn\end{equation}
with no remaining dependence on the coordinates $z^r$. Suppose that
\begin{equation} \mu(\alpha)=\sum_{\mathbf{m}} f_{\mathbf{m}}(z^1,\dots,z^n) \dd\zz\ox \lambda^{\mathbf{m}}.
\nn\end{equation}
Then for every $s$, since $\alpha \cdot \sum\limits_{i\in I} \del_{z_i^s} = 0$, we have
\begin{align*}
 0  =\mu(\alpha\cdot \sum_{i\in I}\del_{z^s_i} )=\mu(\alpha )\cdot \sum_{i\in I}\del_{z^s_i} &=\left( \sum_{\mathbf{m}} f_{\mathbf{m}}(z^1,\dots,z^n) \dd\zz\ox \lambda^{\mathbf{m}}\right) \cdot \sum_{i\in I}\del_{z^s_i} \nn\\
&= - \sum_{\mathbf{m}} \left(\del_{z^s} f_{\mathbf{m}}(z^1,\dots,z^n)\right) \dd\zz\ox \lambda^{\mathbf{m}}.
\end{align*}
This implies that the polynomial coefficients $f_{\mathbf{m}}(z^1,\dots,z^n) $ are indeed constants.
Thus, we in fact have a map
\begin{align} &\Hom_{\DAAnk}\left(\TS_{\AA^n}^{k,\bul-(n-1)}((\oms)^{\boxtimes k}),\Delta_*\om\right)\nn\\
&\rightarrow\Hom_{\kk[z^s_{v\bp},\del_{z^s_v}]^{1\leq s\leq n}_{v\in I\setminus\{\bp\}}}\left(\TS_{{\AA^n},\basepoint}^{k,\bul-(n-1)}((\oms)^{\boxtimes k}), \dd\zz\ox_\kk\kk[\lambda^r_i]^{1\leq r\leq n}_{i\in I\setminus\{\bp\}}\right).
\nn\end{align}
To complete the proof, we must show it is invertible. We shall construct the inverse.
Suppose that we have \[\mu^{\bp}\in \Hom_{\kk[z^s_{v\bp},\del_{z^s_v}]^{1\leq s\leq n}_{v\in I\setminus\{\bp\}}}\left(\TS_{\AA^n,\basepoint}^{k,\bul-(n-1)}((\oms)^{\boxtimes k}), \dd\zz\ox\kk[\lambda^s_v]^{1\leq s\leq n}_{v\in I\setminus\{\bp\}}\right).\]
Using again the fact that $\TS_{\AA^n}^{k,\bul}=\kk[z^s_\bp]^{1\leq s\leq n}\ox_\kk\TS_{\AA^n,\basepoint}^{k,\bul}$. If $\mu$ is to be a $\D$-module map, we certainly must have
\[\mu(f(z^1_\bp,\dots,z^n_\bp)\cdot \alpha):= f(z^1,\dots,z^n)\mu^{\bp}(\alpha).\]
We check that
  \begin{align*}
    \mu\left((f(z_\bp)\cdot \alpha)\del_{z^s_\bp}\right) &=  \mu\left(-(\del_{z^s_\bp}f(z_\bp))\cdot \alpha\right)+    \mu\left(f(z_\bp)\cdot (\alpha\del_{z^s_\bp})\right)\\
    &= -\left(\del_{z^s}f(z)\right)\cdot \mu^{\bp}(\alpha)- f(z)\cdot \mu^{\bp}(\alpha\cdot\sum_{v\in I\setminus\{\bp\}}\del_{z^s_v}) \\
     & =    \mu\left((f(z_\bp)\cdot \alpha)\right)\del_{z^s_\bp},
  \end{align*}
and this suffices to establish that $\mu$ is indeed a $\D$-module map, i.e. \[\mu \in \Hom_{\DAAnk}\left(\TS_{\AA^n}^{k,\bul-(n-1)}((\oms)^{\boxtimes k}),\Delta_*\om\right).\] It is by construction a preimage of $\mu^\bp$.
\end{proof}

In view of \cref{prop: translation invariant Pch}, to prove \cref{prop: Injectivity} it suffices to prove the following translation invariant analog. (It is actually this version we shall ultimately use in the proof of \cref{thm: ChiralIsoLie}.)
\begin{prop}[Injectivity property, translation invariant version]\label{prop: translation invariant injectivity}
If
\[\mu\in \Hom_{\kk[z^s_{v\bp},\del_{z^s_v}]^{1\leq s\leq n}_{v\in I\setminus\{\bp\}}}\left(H_{\ddts}\left(\TS_{\AA^n,\basepoint}^{k,p-(n-1)}((\oms)^{\boxtimes k})\right),\dd\zz\ox\kk[\lambda^s_v]^{1\leq s\leq n}_{v\in I\setminus\{\bp\}}\right)\]
is such that
    \[
    \mu|_{\ker(\ddts)}=0,
    \]
    then there exists
\[ \tilde{\mu}\in \Hom_{\kk[z^s_{v\bp},\del_{z^s_v}]^{1\leq s\leq n}_{v\in I\setminus\{\bp\}}}\left(H_{\ddts}\left(\TS_{\AA^n,\basepoint}^{k,p+1-(n-1)}((\oms)^{\boxtimes k})\right),\dd\zz\ox\kk[\lambda^s_v]^{1\leq s\leq n}_{v\in I\setminus\{\bp\}}\right)\]
such that
    \[
\tilde{\mu}\circ \ddts =\mu.
    \]
Equivalently, the following induced map from the homology is injective:
  \begin{align}
 &H^p_{\dd_{\P}}\left(  \Hom_{\kk[z^s_{v\bp},\del_{z^s_v}]^{1\leq s\leq n}_{v\in I\setminus\{\bp\}}}\left(\TS_{\AA^n,\basepoint}^{k,\bul-(n-1)}((\oms)^{\boxtimes k}),\dd\zz\ox\kk[\lambda^s_v]^{1\leq s\leq n}_{v\in I\setminus\{\bp\}}\right)\right)\nn\\
 &\quad\rightarrow\Hom_{\kk[z^s_{v\bp},\del_{z^s_v}]^{1\leq s\leq n}_{v\in I\setminus\{\bp\}}}\left(H^p_{\ddts}\left(\TS_{\AA^n,\basepoint}^{k,\bul-(n-1)}((\oms)^{\boxtimes k})\right),\dd\zz\ox\kk[\lambda^s_v]^{1\leq s\leq n}_{v\in I\setminus\{\bp\}}\right).\quad p\in  \ZZ.
  \nn\end{align}
\end{prop}
To prove this we require various lemmas.

We first introduce a $\kk^{\times}$-action on the coordinates $\rho_t:z^s_i\rightarrow t\cdot z^s_i$, $t\in \kk^{\times}$. We extend this action to other variables as follows
 \[
 \rho_t\left(\dd\zz\shifted_i\right)=t^{n}\cdot \dd\zz\shifted_i,\quad \rho_t\left(\lambda^s_v\right)=t^{-1}\cdot \lambda^s_v.
 \]
We say $\alpha \in \TS_{{\AA^n},\basepoint}^{k,\bul}((\oms)^{\boxtimes k})$ is of homogeneous $z$-degree $d$ if $\rho_t(\alpha)=t^d\cdot \alpha$.
\begin{prop}\label{prop: Scaling}
Chiral operations preserve the homogeneous $z$-degree.
That is,
\[ \mu\left(\rho_t(\alpha (\dd\zz\shifted)^I)\right)=  \rho_t\left(\mu(\alpha (\dd\zz\shifted)^I)\right),\]
for all $\mu\in\P^{ch}_{\AA^n,\bul}(I)$.
\end{prop}
\begin{proof}
Consider the Euler vector field
$\mathbf{E} :=\sum\limits_{v\in I}\sum\limits^n_{s=1}z^s_{v}\del_{z^s_v}$.
Since $\mathbf{E} \in \DAAnk$, chiral operations are $\mathbf{E}$-equivariant by definition:
$\mu(\alpha) \cdot \mathbf E = \mu(\alpha\cdot \mathbf E)$.
It is enough to note that eigenvectors of $\mathbf E$ with eigenvalue $d$ are precisely elements of homogeneous $z$-degree $-d$.
\end{proof}

Now we want to argue that chiral operations are determined by their action on certain preferred elements of $\TS_{\AA^n}^{k,\bul-(n-1)}$, whose pole structures correspond to tuples of maximal trees, as follows.

Suppose we are given a directed graph $\Gamma = (V(\Gamma), E(\Gamma))$ with vertex set $V(\Gamma) = I = \{1,\dots,k\}$ and edge set $E(\Gamma) \subset V(\Gamma)\two$.
We shall write
\[ \frac{1}{(z)_{\Gamma}} :=\prod_{e\in E(\Gamma)}\frac{1}{z_e}\equiv \prod_{(i,j)\in E(\Gamma)}\frac{1}{z_i-z_j} .\]
The graph $\Gamma$ is a tree if its underlying undirected graph contains no cycles.
%; a tree $\Gamma$ is maximal if its underlying undirected graph has exactly $k-1$ edges.
Let
\[
\mathbf{Q}_{\mathbf{Tree}}^{k,p}
:=\Span_\kk\left\{\frac{F(u)}{(z^1)_{\Gamma_1}\cdots (z^n)_{\Gamma_n}} (\dd\zz\shifted)^I\in \TS_{\AA^n,\basepoint}^{k,p-(n-1)}\left((\oms)^{\boxtimes k}\right) \,\middle|\, \Gamma_1,\dots,\Gamma_n \text{ are trees}%,\frac{F(u)}{(z^1)_{\Gamma_1}\cdots (z^n)_{\Gamma_n}}\in \TS_{\AA^n,\basepoint}^{k,p}
\right\}.
\]
We have the direct sum decomposition as a $\kk$-vector space
\[\mathbf{Q}_{\mathbf{Tree}}^{k,\bul}=\mathbf{Q}_{\mTree}^{k,\bul}\oplus\mathbf{Q}_{\mathbf{Tree},<}^{k,\bul},\]
where:\begin{enumerate}[--]
\item  $\mathbf{Q}_{\mTree}^{k,p}$ is the subspace spanned by monomials such that the trees $\Gamma_1,\dots,\Gamma_n$ are maximal, i.e. $|E(\Gamma_1)/S_2|=\dots=|E(\Gamma_n)/S_2|=k-1$, and where
\item $\mathbf{Q}_{\mathbf{Tree},<}^{k,\bul}$ is the subspace spanned by monomials such that at least one of the trees is non-maximal, i.e., there exists $s$ such that $|E(\Gamma_s)/S_2|<k-1$.
\end{enumerate}
Elements of $\mathbf{Q}_{\mTree}^{k,\bul}$ are homogeneous of $z$-degree \[nk-n(k-1) = n.\] Homogeneous elements of $\mathbf{Q}_{\mathbf{Tree},<}^{k,\bul}$ have $z$-degrees $>n$.
% The Euler vector field $\mathbf{E}_\bp$ from the proof of \cref{prop: Scaling} acts on $\mathbf{Q}_{\mathbf{Tree}}^{k,\bul}$ with integer eigenvalues.
% \[ \mathbf{Q}_{\mathbf{Tree}}^{k,\bul}=\oplus^{n(k-1)}_{l=0}\mathbf{Q}_{\mathbf{Tree}}^{k,\bul}[-l]=\underbrace{\mathbf{Q}_{\mathbf{Tree}}^{k,\bul}[-n(k-1)]}_{\mathbf{Q}_{\mTree}^{k,\bul}}\oplus\underbrace{\oplus^{n(k-1)-1}_{l=0}\mathbf{Q}_{\mathbf{Tree}}^{k,\bul}[-l]}_{\mathbf{Q}_{\mathbf{Tree},<}^{k,\bul}}.
% \]
\begin{lem}\label{lem: tree less}
  As a $\kk[z^s_{v\bp},\del_{z^s_v}]^{1\leq s\leq n}_{v\in I\setminus\{\bp\}}$-module,  $\TS_{\AA^n,\basepoint}^{k,\bul-(n-1)}\left((\oms)^{\boxtimes k}\right)$ is generated by the $\kk$-vector subspace
  \[
\mathbf{Q}_{\mathbf{Tree}}^{k,\bul}=\mathbf{Q}_{\mTree}^{k,\bul}\oplus\mathbf{Q}_{\mathbf{Tree},<}^{k,\bul}\subset  \TS_{\AA^n,\basepoint}^{k,\bul-(n-1)}\left((\oms)^{\boxtimes k}\right).
  \]
Let $\TS_{\mathbf{Tree},<}^{k,\bul}$ denote the sub $\kk[z^s_{v\bp},\del_{z^s_v}]^{1\leq s\leq n}_{v\in I\setminus\{\bp\}}$-module generated by $\mathbf{Q}_{\mathbf{Tree},<}^{k,\bul}$; then, furthermore, for any $\mu\in \P^{ch}_{\AA^n,\bul}(k)$ we have
  \[
  \mu|_{\TS_{\mathbf{Tree},<}^{k,\bul}}=0,\quad\text{and}\quad \mu\left(\mathbf{Q}_{\mathbf{Tree}}^{k,\bul}\right)\subset\mathbf{k}\cdot \dd\mathbf z.
  \]

\end{lem}
\begin{proof} We shall use the algebraic description of the model from \cref{sec: gens and rels for TS} freely without further comment.
  Suppose that we have an element in $\TS_{\AA^n,\basepoint}^{k,\bul-(n-1)}\left((\oms)^{\boxtimes k}\right)$ which contains a factor
  \[
  \alpha=\frac{\varepsilon u^1_{e_1}}{(z^1_{e_1})^{l_1}}\cdots \frac{\varepsilon u^1_{e_K}}{(z^1_{e_K})^{l_K}},\quad\text{where each}\quad \varepsilon u^1_{e_r}=u^1_{e_r} \text{or}\ du^1_{e_r},
  \]
  such that the edges $(e_1,\dots,e_K)$ form a directed cycle, i.e. such that
  \[
  \sum^K_{r=1} z^1_{e_r}=0.
  \]
  Notice that
  %\[
  %\alpha=\frac{\varepsilon u^1_{e_1}}{z^1_{e_1}}\cdots \frac{\varepsilon u^1_{e_k}}{z^1_{e_k}}=\sum^k_{i=2}\frac{\varepsilon u^1_{e_1}}{(z^1_{e_1})^2}\cdots  {\varepsilon u^1_{e_i}}\cdots \frac{\varepsilon u^1_{e_k}}{z^1_{e_k}}.
  %\]
  \begin{align*}
   -\sum^K_{i=2}\frac{\varepsilon u^1_{e_1}}{(z^1_{e_1})^{l_1+1}}\cdots  \frac{\varepsilon u^1_{e_i}}{(z^1_{e_i})^{l_i-1}}\cdots \frac{\varepsilon u^1_{e_K}}{(z^1_{e_K})^{l_K}}  &=-\sum^K_{i=2}z^1_{e_i}\cdot\frac{\varepsilon u^1_{e_1}}{(z^1_{e_1})^{l_1+1}}\cdots \frac{\varepsilon u^1_{e_i}}{(z^1_{e_i})^{l_i}}\cdots \frac{\varepsilon u^1_{e_K}}{(z^1_{e_K})^{l_K}}  \\
     & =z^1_{e_1}\cdot \frac{\varepsilon u^1_{e_1}}{(z^1_{e_1})^{l_1+1}}\cdots \frac{\varepsilon u^1_{e_K}}{(z^1_{e_K})^{l_K}}=\alpha.
  \end{align*}
Repeating the same procedure to each $\frac{\varepsilon u^1_{e_1}}{(z^1_{e_1})^{l_1+1}}\cdots  \frac{\varepsilon u^1_{e_i}}{(z^1_{e_i})^{l_i-1}}\cdots \frac{\varepsilon u^1_{e_K}}{(z^1_{e_K})^{l_K}} $, we can decrease the power of $z^{1}_{e_i},2\leq i\leq K$ until one of them is zero. We repeat this procedure for other factors again until finally we have reduced the element into tree expressions.

Each such expression is in the $\kk[z^s_{v\bp},\del_{z^s_v}]^{1\leq s\leq n}_{v\in I\setminus\{\bp\}}$-module generated by $\mathbf{Q}_{\mathbf{Tree}}^{k,\bul}$. (Recall that $\del_{z^s_v}$ increases the strength of poles in $z_{jv}$, $j\in I \setminus\{v\}$.)

Now we proceed to the second part of the lemma. Since $\TS_{\mathbf{Tree},<}^{k,\bul-(n-1)}$ is the sub $\kk[z^s_{v\bp},\del_{z^s_v}]^{1\leq s\leq n}_{v\in I\setminus\{\bp\}}$-module generated by $\mathbf{Q}_{\mathbf{Tree},<}^{k,\bul}$ and $\mu$ is a map of $\kk[z^s_{v\bp},\del_{z^s_v}]^{1\leq s\leq n}_{v\in I\setminus\{\bp\}}$-modules, we only need to show that
  \[
  \mu|_{\mathbf{Q}_{\mathbf{Tree},<}^{k,\bul}%\left((\oms)^{\boxtimes k}\right)
}=0,\quad \mu\left(\mathbf{Q}_{\mathbf{Tree}}^{k,\bul}%\left((\oms)^{\boxtimes k}\right)
\right)\subset\mathbf{k}\cdot \dd\mathbf z.
  \]
This follows from Proposition \ref{prop: Scaling}.
\end{proof}

\begin{lem}\label{lem: tree tilde mu}
  Suppose that we are given a $\mu\in \P^{ch,\bul-(n-1)}_{\AA^n}(k)$ satisfying $ \mu|_{\ker(\ddts)}=0$. We can find a $\kk$-linear map
  \[ \tilde{\mu}_{\mathbf{Tree}}:\mathbf{Q}_{\mathbf{Tree}}^{k,\bul+1}%\left((\oms)^{\boxtimes k}\right)
\rightarrow\kk\cdot\dd\mathbf z
  \]
  such that
  \[
\tilde{\mu}_{\mathbf{Tree}}(\ddts\alpha)=\mu(\alpha),\quad\text{for all}\quad \alpha\in\mathbf{Q}_{\mathbf{Tree}}^{k,\bul}%((\oms)^{\boxtimes k})
  \]
and
  \[ \tilde{\mu}_{\mathbf{Tree}}|_{\mathbf{Q}_{\mathbf{Tree},<}^{k,\bul+1}%((\oms)^{\boxtimes k})
}=0.
  \]
\end{lem}
\begin{proof}
The differential $ \ddts$ decomposes into a direct sum
  \[
  \ddts=  \ddts^{\mTree}\oplus   \ddts^{\mathbf{Tree},<}\]
with
\[ \ddts^{\mTree} : \mathbf{Q}_{\mTree}^{k,p} \to \mathbf{Q}_{\mTree}^{k,p+1},\quad\text{and}\quad
 \ddts^{\mathbf{Tree,<}} : \mathbf{Q}_{\mathbf{Tree,<}}^{k,p} \to \mathbf{Q}_{\mathbf{Tree,<}}^{k,p+1}.\]
The differential $\ddts^{\mTree}$ defines a $\kk$-linear isomorphism
\[
 \mathbf{Q}_{\mTree}^{k,p}%\left((\oms)^{\boxtimes k}\right)
/\ker(  \ddts^{\mTree})\xrightarrow{\sim }\mathrm{Im}(\ddts^{\mTree})\subset \mathbf{Q}_{\mTree}^{k,p+1}.%\left((\oms)^{\boxtimes k}\right)
\]
We may pick a complement $\mathbf{U}$ such that $ \mathbf{Q}_{\mTree}^{k,p+1}=\mathrm{Im}(\ddts^{\mTree})\oplus \mathbf{U}$ as a $\kk$-vector space. Let $\pi_U: \mathbf{Q}_{\mTree}^{k,p+1} \to \mathrm{Im}(\ddts^{\mTree})$ denote the $\kk$-linear projection along $\mathbf{U}$.  Relative to the choice of $\mathbf{U}$ we then define $\tilde\mu$ to be the map
  \[
   \tilde{\mu}_{\mathbf{Tree}}:\mathbf{Q}_{\mathbf{Tree}}^{k,\bul+1}=\mathbf{Q}_{\mTree}^{k,\bul+1}\oplus\mathbf{Q}_{\mathbf{Tree},<}^{k,\bul+1} %\left((\oms)^{\boxtimes k}\right)
 %\tilde{\mu}_{\mathbf{Tree}}
\xrightarrow{\pi_U\oplus 0} \mathrm{Im}(\ddts^{\mTree})\xrightarrow{\sim}\mathbf{Q}_{\mTree}^{k,\bul}%\left((\oms)^{\boxtimes k}\right)
/\ker(  \ddts^{\mTree})\xrightarrow{\mu}\mathbf{k}\dd\mathbf z\shifted.
  \]
Then $ \tilde{\mu}_{\mathbf{Tree}}$ satisfies the property in the lemma.
\end{proof}

\begin{rem}
  In this paper, we focus on the algebraic aspect of the space of chiral operations. It is important to study the topology of the space of chiral operations (which we left for future study).
\end{rem}

\begin{lem}\label{lem: IntersectionZero}
  We have
  \[
 z^s_{v\bp}\cdot  \mathbf{Q}_{\mTree}^{k,\bul}\subset \mathbf{Q}_{\mathbf{Tree},<}^{k,\bul},\quad\text{and}\quad \mathbf{Q}_{\mTree}^{k,\bul}\cap\TS_{\mathbf{Tree},<}^{k,\bul}=0.
  \]
\end{lem}
\begin{proof}
Consider the expression
\[
 z^s_{v\bp}\cdot \frac{1}{(z^1)_{\Gamma_1}\cdots (z^n)_{\Gamma_n}}.
\]
Whenver each tree graph $\Gamma_s$ is maximal, we can write \[z^s_{v\bp}=z^{s}_{vv_1}+z^{s}_{v_1v_2}+\cdots+z^s_{v_r\bp}\] in such a way that $(v,v_1),(v_1,v_2),\dots,(v_r,\bp)$ are edges of the underlying undirected graph of $\Gamma_s$. This implies that $ z^s_{v\bp}\cdot  \mathbf{Q}_{\mTree}^{k,\bul}\subset \mathbf{Q}_{\mathbf{Tree},<}^{k,\bul}$.

For the second part, let us suppose that we have an element
\[ \alpha \in \mathbf{Q}_{\mTree}^{k,\bul}\cap\TS_{\mathbf{Tree},<}^{k,\bul}.\]
Such an element is of the form
\[
\alpha=  \sum_{(\Gamma_1,\dots,\Gamma_n)\in \mTree^{\times n}}\frac{F_{\Gamma_1,\dots,\Gamma_n}(u)}{(z^1)_{\Gamma_1}\cdots (z^n)_{\Gamma_n}}(\dd\zz\shifted)^I
  \]
where the sum is over tuples of maximal trees and where, for each such tuple, the polynomial differential form $F_{\Gamma_1,\dots,\Gamma_n}(u)$ on the polysimplex must obey boundary conditions specified by those trees, as in \cref{sec: TS def}. Now let us consider the following way of reexpressing $\alpha$. Suppose we are given any monomial
\be \frac{1}{(z)_\Gamma} = \prod_{(i,j)\in E(\Gamma)} \frac 1 {z_{ij}} \nn\ee
where $\Gamma$ is drawn from the set $\mTree$ of maximal trees.
We may always use the Arnold relations, \cref{eq: arnold relations}, and skewsymmetry, $z_{ij} = - z_{ji}$, to rewrite $\frac{1}{(z)_\Gamma}$ as a linear combination of monomials of the form $\frac{1}{(z)_{\Gamma'}}$ where $\Gamma'$ is a maximal tree of a special type: for each $i\in \{2,\dots,k\}$ there is a unique directed edge $(i,j) \in \Gamma'$ out of $i$ and its target $j$ has $j<i$.
Let $\mTree_\rightarrow$ denote this special set of maximal trees:
\begin{align} &\mTree_\rightarrow\nn\\&\qquad := \left\{ \Gamma \in \mTree\, \middle| \,E(\Gamma) = \{ (i,t(i))\}_{2\leq i\leq k} \text{ for some strictly decreasing map $t$}\right\}. \label{special trees}
\end{align}
In this way we rewrite our element $\alpha$ in the form
\[
\alpha=  \sum_{(\Gamma_1,\dots,\Gamma_n)\in \mTree_\rightarrow^{\times n}}\frac{G_{\Gamma_1,\dots,\Gamma_n}(u)}{(z^1)_{\Gamma_1}\cdots (z^n)_{\Gamma_n}} (\dd\zz\shifted)^I%\in \mathbf{Q}_{\mTree}^{k,\bul}\cap\TS_{\mathbf{Tree},<}^{k,\bul}.
  \]
(Let us stress that in this new form, it will not in general be true that each individual term in the sum obeys the defining boundary conditions of $\TS$, though of course the sum itself still does.)

Given a special maximal tree $\Gamma \in \mTree_\rightarrow$ consider the residue operation
\[\Res^z_{\Gamma}:=\res_{z^2\rightarrow z^{t(2)=1}}\circ\cdots \circ\res_{z^k\rightarrow z^{t(k)}} \]
where $t$ is the corresponding strictly decreasing function as in \cref{special trees}.
Observe that
\be \Res_{\Gamma} \frac{1}{(z)_{\Gamma'}} \dd\zz = \pm \delta_{\Gamma,\Gamma'} \quad\text{for}\quad \Gamma,\Gamma' \in \mTree_\rightarrow \nn\ee
up to a sign which is unimportant for us here.
Therefore for each tuple
$(\Gamma_1,\dots,\Gamma_n)\in \mTree_\rightarrow^{\times n}$
of such special maximal trees,
we have
%the product of such residue operators over the coordinate directions picks out the coefficients $G_{\Gamma_1,\dots,\Gamma_n}(u)$ up to a sign and $z$-degree-shift:
\be \prod_{r=1}^n \Res^{z^r}_{\Gamma_r} \alpha = \pm G_{\Gamma_1,\dots,\Gamma_n}(u)[k(n-1)]. \nn\ee
But on the other hand, notice that $\prod\limits_{r=1}^n \Res^{z^r}_{\Gamma_r}$ annihilates the space $\TS_{\mathbf{Tree},<}^{k,\bul}$. Therefore $\prod\limits_{r=1}^n \Res^{z^r}_{\Gamma_r} \alpha=0$. We conclude that $G_{\Gamma_1,\dots,\Gamma_n}(u) =0$ for each tuple $(\Gamma_1,\dots,\Gamma_n)\in \mTree_\rightarrow^{\times n}$, and hence that $\alpha = 0$. That is, $\mathbf{Q}_{\mTree}^{k,\bul}\cap\TS_{\mathbf{Tree},<}^{k,\bul}=0$, as required.
\end{proof}

Now we are ready to prove \cref{prop: translation invariant injectivity}
\begin{proof}[Proof of \cref{prop: translation invariant injectivity}]
Consider the $\kk$-linear map
\[ \tilde{\mu}_{\mathbf{Tree}}:\mathbf{Q}_{\mathbf{Tree}}^{k,\bul+1}%\left((\oms)^{\boxtimes k}\right)
\rightarrow\kk\cdot\dd\mathbf z
\]
we obtain from \cref{lem: tree tilde mu}. We shall first show that it can be uniquely extended to a well-defined $\kk[z^s_{v\bp},\del_{z^s_v}]^{1\leq s\leq n}_{v\in I\setminus\{\bp\}}$-module map
\[
  \tilde{\mu}:\TS_{\AA^n,\basepoint}^{k,\bul+1-(n-1)}\left((\oms)^{\boxtimes k}\right)\rightarrow\kk[\lambda_v]\cdot \dd\mathbf z.\]

Suppose $\alpha\in \TS_{\AA^n,\basepoint}^{k,\bul+1-(n-1)}\left((\oms)^{\boxtimes k}\right)$. Making use of \cref{lem: tree less} and \cref{lem: IntersectionZero}, we may write
\begin{align}
\alpha
&\equiv \sum_{\substack{\mathbf{\Gamma}=(\Gamma_1,\dots,\Gamma_n)\\ \in\mTree^{\times n}}} \frac{F_{\mathbf\Gamma}(u)}{(z^1)_{\Gamma_1}\cdots (z^n)_{\Gamma_n}} (\dd\zz\shifted)^I \cdot P_{\mathbf{\Gamma}}(\del_z) \mod  \TS_{\mathbf{Tree},<}^{k,\bul+1}.
\nn\end{align}
We set
\be \tilde\mu(\alpha) :=
\sum_{\substack{\mathbf{\Gamma}=(\Gamma_1,\dots,\Gamma_n)\\ \in\mTree^{\times n}}} \tilde\mu_{\mathbf{Tree}}\left(\frac{F_{\mathbf\Gamma}(u)}{(z^1)_{\Gamma_1}\cdots (z^n)_{\Gamma_n}} (\dd\zz\shifted)^I\right) \ox P_{\mathbf{\Gamma}}(\lambda) .\nn\ee
To check that this is well defined, suppose
\begin{align}
\alpha
&\equiv\sum_{\substack{\mathbf{\Gamma}=(\Gamma_1,\dots,\Gamma_n)\\ \in\mTree^{\times n}}} \frac{G_{\mathbf\Gamma}(u)}{(z^1)_{\Gamma_1}\cdots (z^n)_{\Gamma_n}} (\dd\zz\shifted)^I \cdot Q_{\mathbf{\Gamma}}(\del_z) \mod \TS_{\mathbf{Tree},<}^{k,\bul+1}
\nn\end{align}
is any other expression for $\alpha$ in the same form. We need to show that
\begin{align}& \sum_{\substack{\mathbf{\Gamma}=(\Gamma_1,\dots,\Gamma_n)\\ \in\mTree^{\times n}}} \tilde\mu_{\mathbf{Tree}}\left(\frac{F_{\mathbf\Gamma}(u)}{(z^1)_{\Gamma_1}\cdots (z^n)_{\Gamma_n}} (\dd\zz\shifted)^I\right) \ox P_{\mathbf{\Gamma}}(\lambda)\nn\\&\qquad =
\sum_{\substack{\mathbf{\Gamma}=(\Gamma_1,\dots,\Gamma_n)\\ \in\mTree^{\times n}}} \tilde\mu_{\mathbf{Tree}}\left(\frac{G_{\mathbf\Gamma}(u)}{(z^1)_{\Gamma_1}\cdots (z^n)_{\Gamma_n}} (\dd\zz\shifted)^I\right) \ox Q_{\mathbf{\Gamma}}(\lambda)\nn
\end{align}
We may suppose without loss of generality that $\alpha$ is homogeneous, of $z$-degree say $n-d$. In that case all the polynomials $P_{\mathbf{\Gamma}}(\lambda)$ and $Q_{\mathbf{\Gamma}}(\lambda)$ appearing here are of $z$-degree $d$ in their arguments (the variables $\lambda^s_v$, which are homogeneous of $z$-degree $-1$).
Therefore it is enough to check that
\begin{align}&
\sum_{\substack{\mathbf{\Gamma}=(\Gamma_1,\dots,\Gamma_n)\\ \in\mTree^{\times n}}} \tilde\mu_{\mathbf{Tree}}\left(\frac{F_{\mathbf\Gamma}(u)}{(z^1)_{\Gamma_1}\cdots (z^n)_{\Gamma_n}} (\dd\zz\shifted)^I\right) \ox \left(\del_{\lambda^{s_1}_{v_1}}\cdots\del_{\lambda^{s_d}_{v_d}} P_{\mathbf{\Gamma}}(\lambda)\right)\nn\\&\qquad =
\sum_{\substack{\mathbf{\Gamma}=(\Gamma_1,\dots,\Gamma_n)\\ \in\mTree^{\times n}}} \tilde\mu_{\mathbf{Tree}}\left(\frac{G_{\mathbf\Gamma}(u)}{(z^1)_{\Gamma_1}\cdots (z^n)_{\Gamma_n}} (\dd\zz\shifted)^I\right) \ox \left(\del_{\lambda^{s_1}_{v_1}}\cdots\del_{\lambda^{s_d}_{v_d}} Q_{\mathbf{\Gamma}}(\lambda)\right)\nn
\end{align}
for all $z$-degree $d$ differential operators $\del_{\lambda^{s_1}_{v_1}}\cdots\del_{\lambda^{s_d}_{v_d}}$.
% , or equivalently (recalling that the image of $\tilde\mu_{\mathbf{Tree}}$ is a number $\in \kk$) that
% \begin{align}&
% \left(\sum_{\substack{\mathbf{\Gamma}=(\Gamma_1,\dots,\Gamma_n)\\ \in\mTree^{\times n}}} \tilde\mu_{\mathbf{Tree}}\left(\frac{F_{\mathbf\Gamma}(u)}{(z^1)_{\Gamma_1}\cdots (z^n)_{\Gamma_n}} (\dd\zz\shifted)^I\right) \ox P_{\mathbf{\Gamma}}(\lambda)\right)\cdot z^{s_1}_{v_1\bp}\cdots z^{s_d}_{v_d\bp} \nn\\&\qquad =
% \left(\sum_{\substack{\mathbf{\Gamma}=(\Gamma_1,\dots,\Gamma_n)\\ \in\mTree^{\times n}}} \tilde\mu_{\mathbf{Tree}}\left(\frac{G_{\mathbf\Gamma}(u)}{(z^1)_{\Gamma_1}\cdots (z^n)_{\Gamma_n}} (\dd\zz\shifted)^I\right) \ox Q_{\mathbf{\Gamma}}(\lambda)\right)\cdot
% z^{s_1}_{v_1\bp}\cdots z^{s_d}_{v_d\bp}\nn
% \end{align}

To prove this, we consider multiplying our two different expressions for $\alpha$ by $z^{s_1}_{v_1\bp}\cdots z^{s_d}_{v_d\bp}$. We find that
\begin{align}
&\sum_{\substack{\mathbf{\Gamma}=(\Gamma_1,\dots,\Gamma_n)\\ \in\mTree^{\times n}}} \frac{F_{\mathbf\Gamma}(u)}{(z^1)_{\Gamma_1}\cdots (z^n)_{\Gamma_n}} (\dd\zz\shifted)^I \cdot P_{\mathbf{\Gamma}}(\del_z) \cdot z^{s_1}_{v_1\bp}\cdots z^{s_d}_{v_d\bp} \nn\\
&\qquad \equiv
\sum_{\substack{\mathbf{\Gamma}=(\Gamma_1,\dots,\Gamma_n)\\ \in\mTree^{\times n}}} \frac{G_{\mathbf\Gamma}(u)}{(z^1)_{\Gamma_1}\cdots (z^n)_{\Gamma_n}} (\dd\zz\shifted)^I \cdot Q_{\mathbf{\Gamma}}(\del_z) \cdot z^{s_1}_{v_1\bp}\cdots z^{s_d}_{v_d\bp} \mod \TS_{\mathbf{Tree},<}^{k,\bul+1}
\nn\end{align}
In view of Lemma \ref{lem: IntersectionZero}, we obtain a strict equality in $\mathbf{Q}_{\mTree}^{k,\bul+1}\subset \TS_{\AA^n,\basepoint}^{k,\bul+1-(n-1)}\left((\oms)^{\boxtimes k}\right)$:
\begin{align}
&\sum_{\substack{\mathbf{\Gamma}=(\Gamma_1,\dots,\Gamma_n)\\ \in\mTree^{\times n}}} \frac{F_{\mathbf\Gamma}(u)}{(z^1)_{\Gamma_1}\cdots (z^n)_{\Gamma_n}} (\dd\zz\shifted)^I \cdot \left[\cdots\left[\Prop_{\mathbf{\Gamma}}(\del_z), z^{s_1}_{v_1\bp}\right],\cdots ,z^{s_d}_{v_d\bp}\right] \nn\\
&\qquad =
\sum_{\substack{\mathbf{\Gamma}=(\Gamma_1,\dots,\Gamma_n)\\ \in\mTree^{\times n}}} \frac{G_{\mathbf\Gamma}(u)}{(z^1)_{\Gamma_1}\cdots (z^n)_{\Gamma_n}} (\dd\zz\shifted)^I \cdot \left[\cdots\left[Q_{\mathbf{\Gamma}}(\del_z), z^{s_1}_{v_1\bp}\right],\cdots ,z^{s_d}_{v_d\bp}\right]
\nn\end{align}
The argument that $\tilde\mu$ is well defined is complete, once we note that
\[ \left(\del_{\lambda^{s_1}_{v_1}}\cdots\del_{\lambda^{s_d}_{v_d}} P_{\mathbf{\Gamma}}(\lambda)\right) = \left[\cdots\left[\Prop_{\mathbf{\Gamma}}(\del_z), z^{s_1}_{v_1\bp}\right],\cdots ,z^{s_d}_{v_d\bp}\right]. \]
are equal in $\kk$, and likewise for $Q_{\mathbf \Gamma}$.

Our map $\tilde\mu$ is $\kk[z^s_{v\bp},\del_{z^s_v}]^{1\leq s\leq n}_{v\in I\setminus\{\bp\}}$-equivariant by construction, given \cref{lem: IntersectionZero}.

It remains to check that we have $\tilde\mu(\ddts\alpha) = \mu(\alpha)$ for all $\alpha\in\TS_{\AA^n,\basepoint}^{k,\bul-(n-1)}\left((\oms)^{\boxtimes k}\right)$ (not merely all $\alpha\in\mathbf{Q}_{\mathbf{Tree}}^{k,\bul}$). To see this we check that
\begin{align}
\tilde\mu(\ddts\alpha) &=
\sum_{\substack{\mathbf{\Gamma}=(\Gamma_1,\dots,\Gamma_n)\\ \in\mTree^{\times n}}} \tilde\mu_{\mathbf{Tree}}\left(\frac{\ddts F_{\mathbf\Gamma}(u)}{(z^1)_{\Gamma_1}\cdots (z^n)_{\Gamma_n}} (\dd\zz\shifted)^I\right) \ox P_{\mathbf{\Gamma}}(\lambda) \nn\\
&=\sum_{\substack{\mathbf{\Gamma}=(\Gamma_1,\dots,\Gamma_n)\\ \in\mTree^{\times n}}} \mu\left(\frac{F_{\mathbf\Gamma}(u)}{(z^1)_{\Gamma_1}\cdots (z^n)_{\Gamma_n}} (\dd\zz\shifted)^I\right) \ox P_{\mathbf{\Gamma}}(\lambda) \nn\\
&=\sum_{\substack{\mathbf{\Gamma}=(\Gamma_1,\dots,\Gamma_n)\\ \in\mTree^{\times n}}} \mu\left(\frac{F_{\mathbf\Gamma}(u)}{(z^1)_{\Gamma_1}\cdots (z^n)_{\Gamma_n}} (\dd\zz\shifted)^I\cdot P_{\mathbf{\Gamma}}(\del_z)\right) \nn\\
&= \mu(\alpha),\nn
\end{align}
where in the final equality we used the fact that $\mu|_{\TS^{k,\bul}_{\mathbf{Tree,<}}}=0$ from \cref{lem: tree less}.
This completes the proof of \cref{prop: translation invariant injectivity}, and hence, in view of \cref{prop: translation invariant Pch}, the proof of \cref{prop: Injectivity}.
\end{proof}

\subsection{Completion of the proof}\label{sec: Surjectivity}
In this section we will finish the proof of the main theorem, \cref{thm: ChiralIsoLie}.

%Recall the injective map from \cref{prop: Injectivity}.  We will first show that there is an injective map  $\Hom_{\kk[z^s_{v\bp},\del_{z^s_v}]^{1\leq s\leq n}_{v\in I \setminus\{\bp\}}\left(H_{\ddts}\left(\TS_{\AA^n,\basepoint}^{k,-(k-1)(n-1)-\bullet}((\oms)^{\boxtimes k})\right),\dd\zz\shifted\ox\kk[\lambda^r_v]^{1\leq r\leq n}_{v\in I\setminus\{\bp\}}\right)\to \Lie(k)$. Then we shall observe that this map is also surjective.

We may define the de Rham cohomology, i.e. the quotient by all total $z$ derivatives, of $ H_{\ddts}\left(\TS_{\AA^n,\basepoint}^{k,\bul-(n-1)}((\oms)^{\boxtimes k})\right)$ as
\begin{align}
&h_{dR}\left( H_{\ddts}\left(\TS_{\AA^n,\basepoint}^{k,\bul-(n-1)}((\oms)^{\boxtimes k})\right)\right)\nn\\&\qquad:=\frac{H_{\ddts}\left(\TS_{\AA^n,\basepoint}^{k,\bul-(n-1)}((\oms)^{\boxtimes k})\right)}{\left< \left[\alpha\cdot \del_{z^s_v}\right]\,\middle|\,\alpha\in \TS_{\AA^n,\basepoint}^{k,\bul-(n-1)}((\oms)^{\boxtimes k}),v\in I\setminus\{\bp\}, 1\leq s\leq n\right>}.
\end{align}

%(Strictly $h_{dR}$ is the quotient by the submodule obtained by acting with the de Rham differential $\dd_{dR} := \sum_{s=1}^n \sum_{v\in I\setminus\{\bp\}} \dd z^s_v \del_{z^s_v}$ on $\TS^{k,\bul}_{\AA^n,\basepoint}( \Omega^{(kn-1)}_{(\AA^{n})^k}[k(n-1)])$, where $\left(\Omega^{*}_{(\AA^{n})^k},\dd_{dR}\right)$ is the sheaf of sections of the algebraic de Rham complex on $(\AA^{n})^k$. But our definition above is evidently equivalent.)

Our goal is now to prove \cref{lem: de rham} below which identifies this de Rham cohomology with the dual $\LieOperad(k)^\vee$ of the space of $k$-ary operations of the Lie operad.

First, the following lemma and its corollary are useful:
\begin{lem}
Pick any nonempty $J\subset \{2,\dots,k\}$ and for each $j\in J$ pick $\ell_j \in \{1,\dots,j-1\}$. Pick $(m_{r,j}\in \ZZ_{\geq 0})^{1\leq r\leq n}_{j\in J}$ and consider the element
\be \alpha = \prod_{j\in J} \left(\prod_{r=1}^n \del_{z^r_j}^{m_{r,j}}\right) \Prop_{\ell_jj}
\nn\ee
If any one of the $m_{r,j}$ is nonzero, then this element is a total $z$ derivative.
\end{lem}
\begin{proof}[Proof (sketch)]
Suppose $j_1>j_2>\dots$ are the elements of $J$ in decreasing order. By induction on $p$ one checks that if $m_{r,j_p}>0$ for any $r$ then $\alpha$ is a total $z$ derivative. The induction is straightforward, and we shall just give an illustrative an example. The following sketch shows the situation when $J=\{6,5,3,2,1\}$.
\be
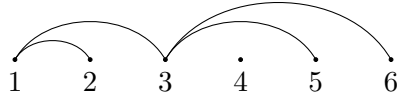
\begin{tikzpicture}[]
        \foreach \x in {1,2,3,4,5,6}{
            \coordinate (\x) at (\x,0);
           \node[draw,circle,fill=black,minimum size=0.3pt,inner sep=.5pt,label=below:{$\x$}]() at (\x) {};
        }
        \draw (1) to[out=60,in=120] (2);
        \draw (1) to[out=60,in=120] (3);
        \draw (3) to[out=60,in=120] (5);
        \draw (3) to[out=60,in=120] (6);
        % \draw[densely dotted] (1) to[out=-45,in=-135] (3);
        % \draw[dashed] (1) to[out=-45,in=-135] (6);
        % \draw[dashed] (3) to[out=-45,in=-135] (6);
    \end{tikzpicture}\nn\ee
Clearly any derivatives $\del_{z^r_6}$ act nontrivially only on $\Prop_{36}$ and can immediately be pulled out. Likewise for any derivatives $\del_{z^r_5}$. For the nontrivial inductive step, we note that
\begin{align}
 (\del_{z^r_3} \Prop_{13}) \Prop_{35} \Prop_{36} &= \del_{z^r_3} \left( \Prop_{13} \Prop_{35} \Prop_{36} \right) - \Prop_{13} \left( \del_{z^r_3} \Prop_{35} \Prop_{36} \right) \nn\\
&= \del_{z^r_3} \left( \Prop_{13} \Prop_{35} \Prop_{36} \right) + \Prop_{13} \left(\del_{z^r_5}  \Prop_{35} \right) \Prop_{36} + \Prop_{13} \Prop_{35}\left( \del_{z^r_6} \Prop_{36}\right)  \nn
\end{align}
where the key observation is that $\del_{z^r_i} \Prop_{ij} = - \del_{z^r_j} \Prop_{ij}$.

(Note that this argument is exactly as in the case of dimension $n=1$, since the index $r$ plays no real role.)
\end{proof}

Now recall the basis of $\Coh_{\dd_{\TS}}(\TS_{\AA^n,\basepoint}^{k,\bul})$ we found in \cref{sec: cohomology}, \cref{cor:TransInvBasis}
\begin{align}\label{b1}
        \left(\prod_{i\in  I_{J,(\ell_j)}}\prod_{r=1}^n(z_{i}^r-z_{k}^r)^{n_{r,i}}\right)
        \left(\prod_{j\in J}[\Prop_{\ell_jj}^{m_{1,j},\dots,m_{n,j}}]\right)
    \end{align}
where the set $I_{J,(\ell_j)}$ is as we defined it in \cref{sec: cohomology}. We have the following immediate corollary.

\begin{cor}
Consider the representative
\begin{align}
      \alpha=  \left(\prod_{i\in  I_{J,(\ell_j)}}\prod_{r=1}^n(z_{i}^r-z_{k}^r)^{n_{r,i}}\right)
        \left(\prod_{j\in J}\Prop_{\ell_jj}^{m_{1,j},\dots,m_{n,j}}\right)
\nn\end{align}
in $\TS_{\AA^n,\basepoint}^{k,\bul}$
of any one of the basis vectors in \cref{b1} above. We can find $\beta\in \TS_{\AA^n,\basepoint}^{k,\bul}$ which is a linear combination of the following terms
$$
\left(\prod_{i\in  I_{J,(\ell_j)}}\prod_{r=1}^n(z_{i}^r-z_{k}^r)^{n'_{r,i}}\right)
        \left(\prod_{j\in J}\Prop_{\ell_jj}\right)
$$
and $\alpha-\beta$ is a total $z$ derivative.
\end{cor}

We are now ready to compute the de Rham cohomology.
\begin{lem}\label{lem: de rham}
The de Rham cohomology is isomorphic to $\LieOperad(k)^{\vee}$. That is,
\[ h_{dR}\left( H^p_{\ddts}\left(\TS_{\AA^n,\basepoint}^{k,\bullet-(n-1)}((\oms)^{\boxtimes k})\right)\right)\cong \begin{cases} \LieOperad(k)^{\vee} & p = 0 ,\nn\\ 0 & \text{otherwise.} \end{cases}\]
Moreover this isomorphism is compatible with the action of the symmetric group $S_k$.
\end{lem}
\begin{proof}
Consider the following Euler vector field, which performs rescaling relative to our chosen base point:% labelled by the index $\bp\in I$:
\[
\mathbf{E}_\bp:=\sum_{v\in I\setminus\{\bp\}}\sum^n_{s=1}z^s_{v\bp}\del_{z^s_v}.
\]
Clearly $\TS_{\AA^n,\basepoint}^{k,\bullet-(n-1)}((\oms)^{\boxtimes k})$ is the direct sum of its eigenspaces for the action of $\mathbf E_\bp$. Observe that if $\alpha \in \TS_{\AA^n,\basepoint}^{k,\bullet-(n-1)}((\oms)^{\boxtimes k})$ has nonzero eigenvalue for the action (from the right, since this is a right $\D$-module) of $\mathbf E_{\bp}$ then it is a total $z$ derivative:
\be\left( \alpha \cdot \mathbf{E}_\bp = \lambda \alpha ,\, \lambda\neq 0\right) \implies \alpha = \frac 1 \lambda \sum_{v\in I\setminus\{\bp\}}\sum^n_{s=1} \left(\alpha z^s_{v\bp}\right)\cdot \del_{z^s_v}.\nn\ee
Therefore we may confine our attention to the kernel of $\mathbf E_\bp$.

Now, from the basis in \cref{cor:TransInvBasis} we get a basis of $H_{\ddts}\left(\TS_{\AA^n,\basepoint}^{k,\bullet-(n-1)}((\oms)^{\boxtimes k})\right)$. And by the preceding corollary, it is enough to consider those basis vectors represented by elements of the form
\begin{align} \label{b2}
       \left(\prod_{i\in  I_{J,(\ell_j)}}\prod_{r=1}^n(z_{i}^r-z_{k}^r)^{n_{r,i}}\right)
        \left(\prod_{j\in J}\Prop_{\ell_jj}\right)(\dd\zz\shifted)^I.
\end{align}
(i.e. with no derivatives acting on the propagators $\Prop_{ij}$).
\iffalse
The kernel of $\mathbf E_{\bp}$ acting (from the right) on $\TS_{\AA^n,\basepoint}^{k}((\oms)^{\boxtimes k})$ is isomorphic to the eigenspace of eigenvalue $-n(k-1)$ of $\mathbf E_{\bp}$ acting (from the left) on $\TS_{\AA^n,\basepoint}^{k}((\O_{\AA^n}))^{\boxtimes k})$.% (cf. the proof of \cref{prop: Scaling}).
\fi
Since propagators $\Prop_{ij}$ have eigenvalue $-n$ for $\mathbf E_\bp$ and factors $(z^r_i-z^r_j)$ have eigenvalue $+1$, we see that the vectors of the form \cref{b2} which lie in the kernel of $\mathbf E_{\bp}$ are precisely those corresponding to maximal trees (i.e. with $k-1$ propagators  $\Prop_{ij}$) and with no nontrivial polynomial prefactor.

Therefore, every cohomology class in $ H_{\ddts}\left(\TS_{\AA^n,\basepoint}^{k,\bullet-(n-1)}((\oms)^{\boxtimes k})\right)$ is represented, modulo some total $z$ derivative, by a $\kk$-linear combination of the cohomology classes represented by the vectors
\begin{align} \label{b3}
        \left(\prod_{j=2}^k\Prop_{\ell_j,j}\right)(\dd\zz\shifted)^I, \qquad 1\leq \ell_j<j.
\end{align}
We want to show that this representative is unique. In other words, we want to show that the subspace, of the span of these vectors, consisting of total $z$ derivatives is trivial. Consider a general linear combination
\be \alpha = \sum_{(\ell_j)} c_{(\ell_j)} \left(\prod_{j=2}^k\Prop_{\ell_j,j}\right)(\dd\zz\shifted)^I , \qquad c_{(\ell_j)} \in \kk.\nn\ee
For each tuple $(\ell_j)$ in the sum, there is a residue operation (here, cf. \cref{sec: pairing of residues and propagators})
\be R_{(\ell_j)} = \mu_2^{2\to 1}\circ\dots\mu_2^{j\to l_j}\circ\dots\circ\mu_2^{k\to \ell_k} \label{residue operator}\ee
which picks out precisely that term
\be R_{(\ell_j)}(\alpha) = c_{(\ell_j)} \dd\zz_1  \in \dd\zz_1 \ox \kk[\lambda^s_v]^{1\leq s\leq n}_{v\geq 2}.\label{pickout}\ee
(Here we pick, for the first time, $1=\bp$ as the label of our base point.)

The operation $R_{(\ell_j)}$ is a map of $\D$-modules.
Therefore if $\alpha$ is a total $z$ derivative then so is $R_{(\ell_j)}(\alpha)$. But to be a total   $z$ derivative in $\dd\zz_1 \ox \kk[\lambda^s_v]^{1\leq s\leq n}_{v\geq 2}$ is to belong to the subspace $\bigoplus\limits_{r=1}^n \bigoplus\limits_{i\geq 2} \dd\zz_1 \ox \lambda^r_i \kk[\lambda^s_v]^{1\leq s\leq n}_{v\geq 2}$, and the only way $R_{(\ell_j)}(\alpha) = c_{(\ell_j)} \dd\zz_1 $ belongs to that subspace is if $c_{(\ell_j)}=0$. This establishes that the only linear combination of the vectors in \cref{b3} which is a total $z$ derivative is the zero vector.

We conclude that the de Rham cohomology
\[ h_{dR}\left( H_{\ddts}\left(\TS_{\AA^n,\basepoint}^{k,\bullet-(n-1)}((\oms)^{\boxtimes k})\right)\right)\]
has a $\kk$-basis given by the representatives of the vectors
\be \left(\prod_{j=2}^k \Prop_{\ell_jj} \right) (\dd\zz\shifted)^{\boxtimes k}\label{b4}\ee
labelled by tuples $(\ell_2,\dots,\ell_k)$ such that $\ell_j<j$ for each $j$.
In particular, it is concentrated in the cohomological degree
\be (k-1)(n-1) - k(n-1) = -(n-1) \Rightarrow \bullet=0.\nn\ee
(Recall that $(\dd\zz\shifted)^I$ is in cohomological degree $-k(n-1)$, and the propagators $\Prop_{ij}$ are in cohomological degree $n-1$.)

To complete the proof, let us recall an explicit description of $\LieOperad(k)^{\vee}$, the dual of the $\kk$-vector space of $k$-ary operations of the Lie operad. The Orlik-Solomon algebra $OS(k) = \bigoplus_{p\geq 0} OS(k)^p$ is the quotient of the free graded-commutative unital $\kk$-algebra in generators $(A_{ij}\equiv A_{ji})_{1\leq i<j\leq k}$ in degree one, by the ideal generated by the Arnold relations $A_{ij}A_{j\ell} + A_{j\ell}A_{\ell i} + A_{\ell i} A_{ij}=0$. (The relations $A_{ij}A_{ij}=0$ are automatic since $A_{ij}$ is in odd degree.)
As a $\kk$-vector space, $\LieOperad(k)^\vee \cong OS(k)^{k-1}$. It has a $\kk$-basis
\be \prod_{j=2}^{k} A_{\ell_jj} \nn\ee
labelled by tuples $(\ell_2,\dots,\ell_k)$ such that $\ell_j<j$ for each $j$. On comparing this with \cref{b4}, we have the required isomorphism of vector spaces.
The compatibility with $S_k$-module structure follows from the Arnold relations, \cref{sec: arnold relations}, \cref{thm: arnold}.

(See, for example, \cite{totaro1996configuration}, and  \cite[13.2.4]{LodayVallette} for a discussion of the basis of $\mathrm{LieOperad}(n)$.)
\end{proof}
\iffalse
 \cref{lem: de rham} implies the following

\begin{cor}\label{cor: Lie}
\be \Hom_{\kk}\left(h_{dR}\left( H_{\ddts}\left(\TS_{\AA^n,\basepoint}^{k}((\oms)^{\boxtimes k})\right)\right),\dd\zz\shifted\ox \kk\right)\cong \LieOperad(k) ,
\nn\ee
where we regard the space $\LieOperad(k)$ of $k$-ary operations of the Lie operad as a chain complex concentrated in degree zero.
\end{cor}
\fi
We get a map
\begin{align}
&\Hom_{\kk[z^s_{v\bp},\del_{z^s_v}]^{1\leq s\leq n}_{v\in I \setminus\{\bp\}}}\left(H_{\ddts}\left(\TS_{\AA^n,\basepoint}^{k,\bullet-(n-1)}((\oms)^{\boxtimes k})\right),\dd\zz\ox\kk[\lambda^r_v]^{1\leq r\leq n}_{v\in I\setminus\{\bp\}}\right)\nn\\
&\qquad  \xrightarrow{[h_{dR}]} \Hom_{\kk}\left(h_{dR}\left( H_{\ddts}\left(\TS_{\AA^n,\basepoint}^{k,\bullet-(n-1)}((\oms)^{\boxtimes k})\right)\right),\kk\cdot \dd\zz \right).
\nn\end{align}
induced by taking de Rham cohomologies.
\begin{prop}\label{prop: dR injective}
This map is $[h_{dR}]$ injective.
\end{prop}
\begin{proof}
  This is proved in \cite[\S2.2.7]{BDChiralAlgebras}. We present an elementary proof here. Suppose that
\[\mu\in   \Hom_{\kk[z^s_{v\bp},\del_{z^s_v}]^{1\leq s\leq n}_{v\in I \setminus\{\bp\}}}\left(H^p_{\ddts}\left(\TS_{\AA^n,\basepoint}^{k,\bullet-(n-1)}((\oms)^{\boxtimes k})\right),\dd\zz\ox\kk[\lambda^r_v]^{1\leq r\leq n}_{v\in I\setminus\{\bp\}}\right)
\]
is in the kernel of this map, $[h_{dR}](\mu)=0$. That means, suppose that
\[
 \mu(\beta)\equiv 0 \mod \bigoplus_{r=1}^n \bigoplus_{i\in I\setminus\{\bp\}} \dd\zz \ox \lambda^r_i \kk[\lambda^s_v]^{1\leq s\leq n}_{v\in I\setminus\{\bp\}} \]
for all $\beta\in H^p_{\ddts}\left(\TS_{\AA^n,\basepoint}^{k,\bullet-(n-1)}((\oms)^{\boxtimes k})\right)$.
We want to argue that in that case $\mu(\alpha) = 0$ holds on the nose for all $\alpha\in H^p_{\ddts}\left(\TS_{\AA^n,\basepoint}^{k,\bullet-(n-1)}((\oms)^{\boxtimes k})\right)$. We may write
\be \mu(\alpha) = \dd\zz\ox\sum_{\mathbf v;\mathbf s} c_{\mathbf v;\mathbf s}(\alpha) \lambda^{s_1}_{v_1} \dots \lambda^{s_\ell}_{v_{\ell}} , \qquad c_{\mathbf v;\mathbf s}(\alpha) \in \kk.\nn\ee
for some sum over distinct monomials.
But then for each such monomial we compute
\[ \mu\left(\alpha \cdot z_{v_1\bp}^{s_1}\dots z_{v_\ell\bp}^{s_\ell} \right)
 \equiv c_{\mathbf v;\mathbf s}  \mod \bigoplus_{r=1}^n \bigoplus_{i\in I\setminus\{\bp\}} \dd\zz \ox \lambda^r_i \kk[\lambda^s_v]^{1\leq s\leq n}_{v\in I\setminus\{\bp\}}\]
and we conclude (setting $\beta = \alpha \cdot z_{v_1\bp}^{s_1}\dots z_{v_\ell\bp}^{s_\ell}$) that $c_{\mathbf v;\mathbf s}=0$ on the nose. So $\mu$ must indeed be the zero map, and thus $[h_{dR}]$ is indeed injective.
\end{proof}

Now we are ready to complete the proof of \cref{thm: ChiralIsoLie}, which is the main theorem of the paper.
\begin{proof}[Proof of \cref{thm: ChiralIsoLie}]
  We already proved that the composition of maps
\begin{align}
&\bigoplus_{p\in \ZZ}H^p_{\dd_\P}\left(P^{ch}_{\AA^n}(k)\right)\nn\\
&\underset{\text{\cref{prop: translation invariant Pch}}} \cong \bigoplus_{p\in \ZZ} H^p_{\dd_\P}\left(  \Hom_{\kk[z^s_{v\bp},\del_{z^s_v}]^{1\leq s\leq n}_{v\in I \setminus\{\bp\}}}\left(\TS_{\AA^n,\basepoint}^{k,\bullet-(n-1)}((\oms)^{\boxtimes k}),\dd\zz\ox\kk[\lambda^r_v]^{1\leq r\leq n}_{v\in I\setminus\{\bp\}}\right)\right)\nn\\
&\underset{\text{\cref{prop: translation invariant injectivity}}} \hookrightarrow\bigoplus_{p\in\ZZ}\Hom_{\kk[z^s_{v\bp},\del_{z^s_v}]^{1\leq s\leq n}_{v\in I \setminus\{\bp\}}}\left(H^p_{\ddts}\left(\TS_{\AA^n,\basepoint}^{k,\bullet-(n-1)}((\oms)^{\boxtimes k})\right),\dd\zz\ox\kk[\lambda^r_v]^{1\leq r\leq n}_{v\in I\setminus\{\bp\}}\right)\nn\\
&  \underset{\text{\cref{prop: dR injective}}} \hookrightarrow \bigoplus_{p\in\ZZ}\Hom_{\kk}\left(h_{dR}\left( H^p_{\ddts}\left(\TS_{\AA^n,\basepoint}^{k,\bullet-(n-1)}((\oms)^{\boxtimes k})\right)\right),\kk\cdot \dd\zz \right)\underset{\text{\cref{lem: de rham}}}\cong \LieOperad(k).\nn
\end{align}
is injective. We only need to observe that it is surjective, and this follows from the existence of the residue map $\mu_2\in P^{ch}_{\AA^n,0}(2)$.

So we have found an isomorphism of differential graded vector spaces for each $k$ which is in fact an isomorphism of $S_k$ modules.
By inspection, one sees that the operadic compositions agree, c.f. the composition of residue operators \cref{residue operator} and their action on the basis elements \cref{b3}.
\end{proof}

\section{The Chevalley-Cousin complex}\label{sec: Chevalley-Cousin complex}
From \cite{BDChiralAlgebras}, a chiral algebra on a smooth curve can be thought of as a Lie algebra object in a certain tensor category. The Chevalley-Cousin complex of this chiral algebra is defined to be the reduced Chevalley complex of the corresponding Lie algebra object. It is an important ingredient when defining the chiral homology \cite[Chapter 4]{BDChiralAlgebras}. The derived section of the Chevalley-Cousin complex for the unit chiral algebra can be thought of as singular differential forms  (associated to Feynman diagrams) that appear in 2d Lagrangian chiral QFT. The theory of chiral homology can be used to define such a priori divergent Feynman integrals, see \cite{gui2022elliptictracemapchiral,gui2023tracemapchiralweyl}. In this section, we give an analogous construction of the Chevalley-Cousin complex for the homotopy polysimplicial chiral algebra $\oms.$

Following \cite[\S1.3.1]{BDChiralAlgebras} we let
\be Q(I):= \{ I \xonto \pi T \}\nn\ee
denote the lattice of surjections out of a finite set $I$, and abuse notation by sometimes writing $T\in Q(I)$ for the surjection $\pi:I\onto T$.

We define the graded vector space
\be \ChirCompTot^\bul_{(\AA^n)^I} := \bigoplus_{T \in Q(I)} \Delta_*^{(I/T)} \TS^{T,\bul}((\oms)^{\boxtimes T}[1]). \nn\ee
It comes equipped with a differential $D$ which we are about to describe. The resulting chain complex $(\ChirCompTot^\bul_{(\AA^n)^I},D)$ is the \dfn{Chevalley-Cousin complex} in our setting. Compare \cite[\S3.4.11]{BDChiralAlgebras}.

First observe that as a graded vector space $\ChirCompTot^\bul_{(\AA^n)^I}$ is equivalently  the totalization
%\begin{align*}
$\ChirCompTot_{(\AA^n)^I}^\bul = \Tot(\ChirComp_{(\AA^n)^I}^{\bul,\bul})$
%\qquad\text{i.e.}\qquad     \ChirCompTot_{(\AA^n)^k}^p = \bigoplus_{q=0}^{k-1} \ChirComp_{(\AA^n)^k}^{p-q,q}
%\end{align*}
of the bigraded vector space
\be %\ChirComp_{(\AA^n)^I}^{\bul,\bul}:= \bigoplus_{p\geq 0} \bigoplus_{q=0}^{k-1} \ChirComp_{(\AA^n)^k}^{p,q},\qquad
\ChirComp_{(\AA^n)^I}^{q,p} :=
\bigoplus_{\substack{T\in Q(I) : \\ |T|=-q}} \Delta_*^{(I/T)}\TS^{T,p}((\oms)^{\boxtimes T}).\nn\ee

Recall that for us the canonical sheaf $\omega_{\AA^n}$ sits in cohomological degree zero, and $\TS^{k,\bul}((\omega_{\AA^n})^{\boxtimes k})$ is concentrated in cohomological degrees $\{0,1\dots,\binom k 2 (n-1)\}$. Thus our model $\TS^{k,\bul}((\oms)^{\boxtimes k})$ of the derived global sections of the \emph{shifted} canonical sheaf $(\oms)^{\boxtimes k}:= (\omega_{\AA^n}[n-1])^{\boxtimes k}$ is concentrated in cohomological degrees
$\{ -k(n-1), -k(n-1) + 1, \dots, \left(\binom k 2 - k\right)(n-1)\}$.
The chiral operations $\mu_\ell$ from \cref{sec: higher operations} define maps of bidegree $(-1+\ell,2-\ell)$
\be \ddmu \ell : \ChirComp_{(\AA^n)^k}^{q,p} \to \ChirComp_{(\AA^n)^k}^{q+\ell-1,p+2-\ell} \nn\ee
for every $\ell\geq 2$, as follows. We first define $\ddmu \ell$ on $\TS^T((\oms)^{\boxtimes T})$ to be the sum
\begin{align} \ddmu\ell|_{\TS^T((\oms)^{\boxtimes T})} &\quad: \TS^T((\oms)^{\boxtimes T}) \to \bigoplus_{ \substack{L\subset T \\: |L| = \ell}} \Delta_*^{T/(\{\star\}\sqcup T\setminus L)} \TS^{\{\star\}\sqcup T\setminus L}((\oms)^{\boxtimes (\{\star\}\sqcup T\setminus L)})\nn
\end{align}
over all ways of partially applying the operation $\mathbf{I}^{[L,T\setminus L]}_{[\{\star\},T\setminus L]}\circ\mu_\ell$, where the map $\mathbf{I}^{[L,T\setminus L]}_{[\{\star\},T\setminus L]}$ is defined in \cref{def: Imap} (none, if $|T|<\ell$, of course, so it is the zero map in such cases). See \cref{fig:mudifferential}.
\begin{figure}[htp]
    \centering

\tikzset{every picture/.style={line width=0.75pt}} %set default line width to 0.75pt

\begin{tikzpicture}[x=0.75pt,y=0.75pt,yscale=-1,xscale=1]
%uncomment if require: \path (0,237); %set diagram left start at 0, and has height of 237

%Straight Lines [id:da3038819776157313]
\draw    (100,112) ;
\draw [shift={(100,112)}, rotate = 0] [color={rgb, 255:red, 0; green, 0; blue, 0 }  ][fill={rgb, 255:red, 0; green, 0; blue, 0 }  ][line width=0.75]      (0, 0) circle [x radius= 3.35, y radius= 3.35]   ;
%Straight Lines [id:da6956089784960429]
\draw    (120,132) ;
\draw [shift={(120,132)}, rotate = 0] [color={rgb, 255:red, 0; green, 0; blue, 0 }  ][fill={rgb, 255:red, 0; green, 0; blue, 0 }  ][line width=0.75]      (0, 0) circle [x radius= 3.35, y radius= 3.35]   ;
%Straight Lines [id:da27841843865736915]
\draw    (120,132) ;
\draw [shift={(120,132)}, rotate = 0] [color={rgb, 255:red, 0; green, 0; blue, 0 }  ][fill={rgb, 255:red, 0; green, 0; blue, 0 }  ][line width=0.75]      (0, 0) circle [x radius= 3.35, y radius= 3.35]   ;
%Straight Lines [id:da7220518545238082]
\draw    (120,132) ;
\draw [shift={(120,132)}, rotate = 0] [color={rgb, 255:red, 0; green, 0; blue, 0 }  ][fill={rgb, 255:red, 0; green, 0; blue, 0 }  ][line width=0.75]      (0, 0) circle [x radius= 3.35, y radius= 3.35]   ;
%Straight Lines [id:da13443014429229505]
\draw    (127,97) ;
\draw [shift={(127,97)}, rotate = 0] [color={rgb, 255:red, 0; green, 0; blue, 0 }  ][fill={rgb, 255:red, 0; green, 0; blue, 0 }  ][line width=0.75]      (0, 0) circle [x radius= 3.35, y radius= 3.35]   ;
%Straight Lines [id:da9373301909463239]
\draw    (188,117) -- (261.08,116.88) ;
\draw [shift={(263.08,116.88)}, rotate = 179.91] [color={rgb, 255:red, 0; green, 0; blue, 0 }  ][line width=0.75]    (10.93,-3.29) .. controls (6.95,-1.4) and (3.31,-0.3) .. (0,0) .. controls (3.31,0.3) and (6.95,1.4) .. (10.93,3.29)   ;
%Straight Lines [id:da06306458031903261]
\draw    (306,109) ;
\draw [shift={(306,109)}, rotate = 0] [color={rgb, 255:red, 0; green, 0; blue, 0 }  ][fill={rgb, 255:red, 0; green, 0; blue, 0 }  ][line width=0.75]      (0, 0) circle [x radius= 3.35, y radius= 3.35]   ;
%Straight Lines [id:da33968577373151354]
\draw    (315,138) ;
\draw [shift={(315,138)}, rotate = 0] [color={rgb, 255:red, 0; green, 0; blue, 0 }  ][fill={rgb, 255:red, 0; green, 0; blue, 0 }  ][line width=0.75]      (0, 0) circle [x radius= 3.35, y radius= 3.35]   ;
%Straight Lines [id:da2713005559698085]
\draw    (315,138) ;
\draw [shift={(315,138)}, rotate = 0] [color={rgb, 255:red, 0; green, 0; blue, 0 }  ][fill={rgb, 255:red, 0; green, 0; blue, 0 }  ][line width=0.75]      (0, 0) circle [x radius= 3.35, y radius= 3.35]   ;
%Straight Lines [id:da21399630853225915]
\draw    (315,138) ;
\draw [shift={(315,138)}, rotate = 0] [color={rgb, 255:red, 0; green, 0; blue, 0 }  ][fill={rgb, 255:red, 0; green, 0; blue, 0 }  ][line width=0.75]      (0, 0) circle [x radius= 3.35, y radius= 3.35]   ;
%Straight Lines [id:da08953945587510725]
\draw    (322,103) ;
\draw [shift={(322,103)}, rotate = 0] [color={rgb, 255:red, 0; green, 0; blue, 0 }  ][fill={rgb, 255:red, 0; green, 0; blue, 0 }  ][line width=0.75]      (0, 0) circle [x radius= 3.35, y radius= 3.35]   ;
%Shape: Circle [id:dp04628484187248394]
\draw  [dash pattern={on 0.84pt off 2.51pt}] (293.89,106.49) .. controls (298.76,96.27) and (311,87.98) .. (321.22,87.98) .. controls (331.44,87.98) and (335.78,96.27) .. (330.91,106.49) .. controls (326.03,116.71) and (313.8,125) .. (303.57,125) .. controls (293.35,125) and (289.02,116.71) .. (293.89,106.49) -- cycle ;
%Straight Lines [id:da8362614952971064]
\draw    (422.17,130.91) ;
\draw [shift={(422.17,130.91)}, rotate = 0] [color={rgb, 255:red, 0; green, 0; blue, 0 }  ][fill={rgb, 255:red, 0; green, 0; blue, 0 }  ][line width=0.75]      (0, 0) circle [x radius= 3.35, y radius= 3.35]   ;
%Straight Lines [id:da422674303964955]
\draw    (428.98,87.4) ;
\draw [shift={(428.98,87.4)}, rotate = 0] [color={rgb, 255:red, 0; green, 0; blue, 0 }  ][fill={rgb, 255:red, 0; green, 0; blue, 0 }  ][line width=0.75]      (0, 0) circle [x radius= 3.35, y radius= 3.35]   ;
%Straight Lines [id:da13622561032904335]
\draw    (406.8,123.45) ;
\draw [shift={(406.8,123.45)}, rotate = 0] [color={rgb, 255:red, 0; green, 0; blue, 0 }  ][fill={rgb, 255:red, 0; green, 0; blue, 0 }  ][line width=0.75]      (0, 0) circle [x radius= 3.35, y radius= 3.35]   ;
%Shape: Circle [id:dp6320077102294217]
\draw  [dash pattern={on 0.84pt off 2.51pt}] (428.7,141.42) .. controls (417.94,144.93) and (403.5,141.78) .. (396.46,134.37) .. controls (389.41,126.97) and (392.42,118.11) .. (403.19,114.59) .. controls (413.95,111.08) and (428.39,114.23) .. (435.44,121.64) .. controls (442.48,129.04) and (439.47,137.9) .. (428.7,141.42) -- cycle ;
%Straight Lines [id:da20979393533079538]
\draw    (537.89,103.17) ;
\draw [shift={(537.89,103.17)}, rotate = 0] [color={rgb, 255:red, 0; green, 0; blue, 0 }  ][fill={rgb, 255:red, 0; green, 0; blue, 0 }  ][line width=0.75]      (0, 0) circle [x radius= 3.35, y radius= 3.35]   ;
%Straight Lines [id:da982864199931913]
\draw    (502.35,117.78) ;
\draw [shift={(502.35,117.78)}, rotate = 0] [color={rgb, 255:red, 0; green, 0; blue, 0 }  ][fill={rgb, 255:red, 0; green, 0; blue, 0 }  ][line width=0.75]      (0, 0) circle [x radius= 3.35, y radius= 3.35]   ;
%Straight Lines [id:da679451631311389]
\draw    (538.61,120.24) ;
\draw [shift={(538.61,120.24)}, rotate = 0] [color={rgb, 255:red, 0; green, 0; blue, 0 }  ][fill={rgb, 255:red, 0; green, 0; blue, 0 }  ][line width=0.75]      (0, 0) circle [x radius= 3.35, y radius= 3.35]   ;
%Shape: Circle [id:dp9799447863805744]
\draw  [dash pattern={on 0.84pt off 2.51pt}] (544.05,92.44) .. controls (552.25,100.25) and (556.31,114.46) .. (553.13,124.17) .. controls (549.95,133.89) and (540.72,135.43) .. (532.52,127.61) .. controls (524.33,119.8) and (520.26,105.59) .. (523.45,95.88) .. controls (526.63,86.16) and (535.86,84.62) .. (544.05,92.44) -- cycle ;

% Text Node
\draw (110,164.4) node [anchor=north west][inner sep=0.75pt]    {$\TS^T((\oms)^{\boxtimes T})$};
% Text Node
\draw (190,84.4) node [anchor=north west][inner sep=0.75pt]    {$\ddmu2|_{\TS^T((\oms)^{\boxtimes T})}$};
% Text Node
\draw (363,106.4) node [anchor=north west][inner sep=0.75pt]    {$\oplus $};
% Text Node
\draw (468,106.4) node [anchor=north west][inner sep=0.75pt]    {$\oplus $};

\end{tikzpicture}
    \caption{$\ddmu2|_{\TS^T((\oms)^{\boxtimes T})}$}
    \label{fig:mudifferential}
\end{figure}

Then we extend $\ddmu\ell$ uniquely as a map of $\D$-modules to all the summands appearing in the definition of $\ChirComp_{(\AA^n)^I}$, i.e.
\(\ddmu \ell := \sum\limits_{\substack{T\in Q(I)}} \Delta_*^{I/T} \left(\ddmu\ell|_{\TS^T((\oms)^{\boxtimes T})}\right). \nn
\)
The sum
\begin{align}    D=\ddts+\sum_{\ell=2}^k\ddmu{\ell}:\ChirCompTot_{(\AA^n)^I}^p\to\ChirCompTot_{(\AA^n)^I}^{p+1}
\end{align}
is then a map of cohomological degree $+1$ on the total complex, and the $L_\8$-algebra coherence relations are equivalent to the statement that $D^2=0$, i.e. that
\begin{align}
    \label{eq:differentials}
\ddts\circ\ddmu{\ell}+\ddmu{\ell}\circ\ddts+\sum_{m=2}^{\ell-1}\ddmu{m}\circ \ddmu{\ell-m+1}=0
\end{align}
for all $\ell\geq 2$, in addition to $\ddts^2=0$.

\begin{exmp}\label{exmp: cousin}
In the case  $I=\{1,2,3\}$ of three marked points we have the following. (In our notation here, we are using the fact that surjections $(\pi:I \twoheadrightarrow T)\in Q(I)$ out of $I$ can be seen as equivalence relations on $I$.)
\begin{align} \ChirComp_{(\AA^n)^3}^{-3,\bul} &= \TS^{3,\bul}(\oms)\nn\\
\ChirComp_{(\AA^n)^3}^{-2,\bul} &= \Delta_*^{1\sim 2} \TS^{2,\bul}(\oms) \oplus
\Delta_*^{1\sim 3} \TS^{2,\bul}(\oms) \oplus
\Delta_*^{2\sim 3} \TS^{2,\bul}(\oms)\nn\\
\ChirComp_{(\AA^n)^3}^{-1,\bul} &=  \Delta_*^{1\sim 2 \sim 3} (\oms)^{\bul}
\end{align}
If we further specialize to  $n=1$ dimensions, we recover  the usual Chevalley-Cousin complex for 3 marked points (in this case, $\ddts=0$ and $d_{\mu^l}=0$ for all $l\geq 3$),
\be
\begin{tikzcd}
\ChirComp_{(\AA^1)^3}^{-3,0} \rar{\ddmu 2}& \ChirComp_{(\AA^1)^3}^{-2,0} \rar{\ddmu 2}& \ChirComp_{(\AA^1)^3}^{-1,0}  \\
\end{tikzcd}
\nn\ee
while in $n=2$ and $n=3$ dimensions respectively, $\ChirComp$ is concentrated in the following bidegrees (we have non-trivial $\ddts$ and $d_{\mu^l},l\geq 3$)
\be
\begin{tikzcd}
\ChirComp_{(\AA^2)^3}^{-3,-3} \dar{\ddts}& & \\
\ChirComp_{(\AA^2)^3}^{-3,-2} \dar{\ddts}\rar{\ddmu 2}& \ChirComp_{(\AA^2)^3}^{-2,-2} \dar{\ddts}& \\
\ChirComp_{(\AA^2)^3}^{-3,-1} \dar{\ddts}\rar{\ddmu 2}& \ChirComp_{(\AA^2)^3}^{-2,-1} \rar{\ddmu 2}& \ChirComp_{(\AA^2)^3}^{-1,-1}  \\
\ChirComp_{(\AA^2)^3}^{-3,0} \arrow[urr,bend right=0,"\ddmu 3"']& &
\end{tikzcd}
% \nn\ee
% In the case of $n=3$ dimensions, we have instead
% \be
\qquad
\begin{tikzcd}
\ChirComp_{(\AA^3)^3}^{-3,-6} \dar{\ddts}& & \\
\ChirComp_{(\AA^3)^3}^{-3,-5} \dar{\ddts}& & \\
\ChirComp_{(\AA^3)^3}^{-3,-4} \dar{\ddts}\rar{\ddmu 2}& \ChirComp_{(\AA^3)^3}^{-2,-4} \dar{\ddts}& \\
\ChirComp_{(\AA^3)^3}^{-3,-3} \dar{\ddts}\rar{\ddmu 2}& \ChirComp_{(\AA^3)^3}^{-2,-3} \dar{\ddts}& \\
\ChirComp_{(\AA^3)^3}^{-3,-2} \dar{\ddts}\rar{\ddmu 2}& \ChirComp_{(\AA^3)^3}^{-2,-2} \rar{\ddmu 2}& \ChirComp_{(\AA^3)^3}^{-1,-2}  \\
\ChirComp_{(\AA^3)^3}^{-3,-1} \dar{\ddts}\arrow[urr,bend right=0,"\ddmu 3"']& &\\
\ChirComp_{(\AA^3)^3}^{-3,0} & &
\end{tikzcd}
\nn\ee
\end{exmp}
Beilinson and Drinfeld prove the following theorem.

\begin{thm}[\cite{BDChiralAlgebras}] Let $k\geq 2$.
The Chevalley-Cousin complex $\left(\ChirCompTot_{(\AA^1)^k}^\bul,D=d_{\mu^2}\right)$ is a resolution of the
%$\kk[z_i]_{1\leq i\leq k} = \O_{(\AA^{1})^k}$-module $\Gamma\left((\AA^{1})^k,(\omega_{(\AA^1)})^{\boxtimes k}\right)$ of
global sections of the shifted canonical sheaf. Namely, we have the following complex of $\D$-modules
\be 0 \to\Gamma\left((\AA^{1})^k,(\omega_{(\AA^1)})^{\boxtimes k}\right) \to \ChirCompTot_{(\AA^1)^k}^{-k} \xrightarrow{d_{\mu^2}} \ChirCompTot_{(\AA^1)^k}^{-k+1} \xrightarrow{d_{\mu^2}} \cdots  \xrightarrow{d_{\mu^2}} \ChirCompTot_{(\AA^1)^k}^{-1} \rightarrow 0\nn\ee
which is exact.
\end{thm}
Our goal in this section is to generalize this theorem to arbitrary dimensions $n\geq 1$. We need some preparations.  From \cref{eq:differentials}, we have
\begin{align*}
\ddts\circ\ddmu{3}+\ddmu{3}\circ\ddts+\ddmu{2}\circ \ddmu{2}=0.
\end{align*}
Thus $\ddmu{2}$ induces a differential on $\ddts$-cohomology. We denote the induced differential by $[\ddmu 2]$. The proof of Beilinson and Drinfeld can be generalized to obtain the following lemma and we get their original theorem if we take $n=1$.
\begin{lem}\label{lem:cohom}
Let $k\geq 2$. The chain complex
\be
  \begin{tikzcd}
 0  \rar&   \Coh_{\ddts}^p(\ChirComp_{(\AA^n)^k}^{-k,\bul}) \rar{[\ddmu 2]} & \Coh_{\ddts}^p(\ChirComp_{(\AA^n)^k}^{-k+1,\bul})\rar{[\ddmu 2]} &\cdots\\
  \end{tikzcd}
  \nn\ee
is exact for any $p>-k(n-1)$.
% \begin{enumerate}
% \item If $p \notin (n-1) \ZZ$ then $\Coh_{\ddts}^p(\TS^{q,\bul}(\oms))=0$ and hence
% \be \Coh_{\ddts}^p(\ChirComp_{(\AA^n)^I}^{-q,\bul}) = 0.\nn\ee
% \item If $p \in (n-1) \ZZ$ then the map of $\ddts$-cohomologies induced by $\ddmu 2$ -- cf. \cref{eq:differentials} --
% \be [\ddmu2] : \Coh_{\ddts}^p(\ChirComp_{(\AA^n)^I}^{-q,\bul}) \to \Coh_{\ddts}^p(\ChirComp_{(\AA^n)^I}^{-q+1,\bul}) \nn\ee
% is injective.
% \end{enumerate}
\qed\end{lem}
\begin{proof} We have $\ChirComp_{(\AA^n)^k}^{-k,\bul} = \TS^{k,\bul}((\oms)^{\boxtimes k})$, by definition.
In \cref{sec: cohomology}  we gave an explicit basis of the cohomology $\Coh_{\ddts}^p(\TS^{k,\bul}((\oms)^{\boxtimes k}))$ as a $\kk$-vector space.
The basis vectors are of the form
\be P_{f,\Gamma} = f(z,\del) \left(\prod_{e\in \Gamma} P_{e}\right) (\dd \mathbf z[n-1])^{\ox k} \nn\ee
labelled by certain tree graphs $\Gamma$ on the vertex set $I=\{1,\dots,k\}$, and certain elements $f(z,\del) \in \kk[z^r_i,\del_r^i]^{1\leq r\leq n}_{1\leq i\leq k} = \D((\AA^n)^k)$.

(Such a product involving $|\Gamma|$ many propagators sits in cohomological degree $|\Gamma|(n-1) - k(n-1)$. Whenever $p$ is not of that form for some tree graph $\Gamma$ then the cohomology is trivial,
 $\Coh_{\ddts}^p(\TS^{k,\bul}((\oms))^{\boxtimes k})=0$, and there is nothing more to check.)

Since the push forward functors $\Delta_*^{(I/T)}, T\in Q(I)$, are exact, we have
\[
\Coh_{\ddts}^p(\ChirComp_{(\AA^n)^k}^{q,\bul})=\bigoplus_{\substack{T\in Q(I) : \\ |T|=-q}} \Delta_*^{(I/T)}\left(\Coh_{\ddts}^p(\TS^{T,\bul}((\oms)^{\boxtimes T}))\right).
\]
For $T=(T,\pi:I\twoheadrightarrow T)\in Q(I)$, and $t\in T$, let us write
\be I_t: =\pi^{-1}(t). \nn\ee
If we pick a vertex $1 \in I$, then the above direct sum can be written as
\begin{align}
\Coh_{\ddts}^p(\ChirComp_{(\AA^n)^k}^{q,\bul})&=\bigoplus_{\substack{T\in Q(I) : \\ |T|=-q,\\ I_{\pi(1)}=\{1\}}} \Delta_*^{(I/T)}\left(\Coh_{\ddts}^p(\TS^{T,\bul}((\oms)^{\boxtimes T}))\right)\nn\\&\qquad\oplus \bigoplus_{\substack{T\in Q(I) : \\ |T|=-q,\\ |I_{\pi(1)}|\geq 2}} \Delta_*^{(I/T)}\left(\Coh_{\ddts}^p(\TS^{T,\bul}((\oms)^{\boxtimes T}))\right).
\nn\end{align}
Suppose that we have an element
\begin{align}
\alpha \in \Delta_*^{(I/T)}\left(\Coh_{\ddts}^p(\TS^{T,\bul}((\oms)^{\boxtimes T}))\right)&=\Coh_{\ddts}^p(\TS^{T,\bul}((\oms)^{\boxtimes T}))\ox_{\kk[\lambda_T]}\kk[\lambda_I],\nn\\&\qquad\qquad\quad T\in Q(I) : |T|=-q, |I_{\pi(1)}|\geq 2.\nn
\end{align}
We claim that we can find
\[
\beta \in \Delta_*^{(I/T')}\left(\Coh_{\ddts}^p(\TS^{T',\bul}(\oms))\right),\quad |T'|=-q+1,
\]
such that
\[
[\ddmu 2](\beta)-\alpha\in \bigoplus_{\substack{T\in Q(I) : \\ |T|=-q,\\ I_{\pi(1)}=\{1\}}} \Delta_*^{(I/T)}\left(\Coh_{\ddts}^p(\TS^{T,\bul}((\oms)^{\boxtimes T}))\right).
\]
Since $\Coh_{\ddts}^p(\TS^{T,\bul}((\oms)^{\boxtimes T}))$ generates $\Delta_*^{(I/T)}\left(\Coh_{\ddts}^p(\TS^{T,\bul}((\oms)^{\boxtimes T}))\right)$ as a $\D$-module, we can assume that $\alpha\in \Coh_{\ddts}^p(\TS^{T,\bul}((\oms)^{\boxtimes T}))$. Let $(T',\pi')=(T\sqcup\{\star\},\pi')\in Q(I)$, and we have sequence of surjective maps
\[
\pi=\pi^{(T'/T)}\circ\pi':I\twoheadrightarrow T'\twoheadrightarrow T,\quad \pi^{(T'/T)}:T'\twoheadrightarrow T,
\]
such that $\pi'^{-1}(\star)=1,\pi^{(T'/T)}(\star)=\pi(1)=t, \pi^{(T'/T)}|_{T}=\mathrm{id}_T$. We view $\alpha$ as an element in $\Coh_{\ddts}^p(\TS^{T',\bul}((\oms)^{\boxtimes T'}))$ and define
\[
\beta=\alpha\cdot [P_{\star t}]\in \Coh_{\ddts}^p(\TS^{T,\bul}((\oms)^{\boxtimes T})).
\]
Then it is clear that (see Fig. \ref{TwoSummand})
\[
[\ddmu 2](\beta)-\alpha\in \bigoplus_{\substack{T\in Q(I) : \\ |T|=-q,\\ I_{\pi(1)}=\{1\}}} \left(\Coh_{\ddts}^p(\TS^{T,\bul}((\oms)^{\boxtimes T}))\right).
\]
\begin{figure}[htp]

\tikzset{every picture/.style={line width=0.75pt}} %set default line width to 0.75pt

\begin{tikzpicture}[x=0.75pt,y=0.75pt,yscale=-1,xscale=1]
%uncomment if require: \path (0,331); %set diagram left start at 0, and has height of 331

%Straight Lines [id:da5098366431180996]
\draw    (122.66,67.19) ;
\draw [shift={(122.66,67.19)}, rotate = 0] [color={rgb, 255:red, 0; green, 0; blue, 0 }  ][line width=0.75]      (0, 0) circle [x radius= 3.35, y radius= 3.35]   ;
%Straight Lines [id:da5868580563138825]
\draw    (140.79,84.38) ;
\draw [shift={(140.79,84.38)}, rotate = 0] [color={rgb, 255:red, 0; green, 0; blue, 0 }  ][fill={rgb, 255:red, 0; green, 0; blue, 0 }  ][line width=0.75]      (0, 0) circle [x radius= 3.35, y radius= 3.35]   ;
%Straight Lines [id:da2702829924349792]
\draw    (120.85,85.24) ;
\draw [shift={(120.85,85.24)}, rotate = 0] [color={rgb, 255:red, 0; green, 0; blue, 0 }  ][fill={rgb, 255:red, 0; green, 0; blue, 0 }  ][line width=0.75]      (0, 0) circle [x radius= 3.35, y radius= 3.35]   ;
%Shape: Ellipse [id:dp747582632675774]
\draw  [dash pattern={on 0.84pt off 2.51pt}] (100,77.52) .. controls (100,62.34) and (113,50.02) .. (129.04,50.02) .. controls (145.08,50.02) and (158.08,62.34) .. (158.08,77.52) .. controls (158.08,92.71) and (145.08,105.02) .. (129.04,105.02) .. controls (113,105.02) and (100,92.71) .. (100,77.52) -- cycle ;
%Straight Lines [id:da7332708432653088]
\draw    (231.08,63.02) ;
\draw [shift={(231.08,63.02)}, rotate = 0] [color={rgb, 255:red, 0; green, 0; blue, 0 }  ][fill={rgb, 255:red, 0; green, 0; blue, 0 }  ][line width=0.75]      (0, 0) circle [x radius= 3.35, y radius= 3.35]   ;
%Straight Lines [id:da7255110984028141]
\draw    (218.1,66.07) ;
\draw [shift={(218.1,66.07)}, rotate = 0] [color={rgb, 255:red, 0; green, 0; blue, 0 }  ][fill={rgb, 255:red, 0; green, 0; blue, 0 }  ][line width=0.75]      (0, 0) circle [x radius= 3.35, y radius= 3.35]   ;
%Shape: Ellipse [id:dp107547891102457]
\draw  [dash pattern={on 0.84pt off 2.51pt}] (205,63.02) .. controls (205,54.19) and (213.08,47.02) .. (223.04,47.02) .. controls (233.01,47.02) and (241.08,54.19) .. (241.08,63.02) .. controls (241.08,71.86) and (233.01,79.02) .. (223.04,79.02) .. controls (213.08,79.02) and (205,71.86) .. (205,63.02) -- cycle ;
%Straight Lines [id:da6052557937864549]
\draw    (115.08,147.02) ;
\draw [shift={(115.08,147.02)}, rotate = 0] [color={rgb, 255:red, 0; green, 0; blue, 0 }  ][fill={rgb, 255:red, 0; green, 0; blue, 0 }  ][line width=0.75]      (0, 0) circle [x radius= 3.35, y radius= 3.35]   ;
%Straight Lines [id:da4576684393229913]
\draw    (102.1,150.07) ;
\draw [shift={(102.1,150.07)}, rotate = 0] [color={rgb, 255:red, 0; green, 0; blue, 0 }  ][fill={rgb, 255:red, 0; green, 0; blue, 0 }  ][line width=0.75]      (0, 0) circle [x radius= 3.35, y radius= 3.35]   ;
%Shape: Ellipse [id:dp9823114059554341]
\draw  [dash pattern={on 0.84pt off 2.51pt}] (89,147.02) .. controls (89,138.19) and (97.08,131.02) .. (107.04,131.02) .. controls (117.01,131.02) and (125.08,138.19) .. (125.08,147.02) .. controls (125.08,155.86) and (117.01,163.02) .. (107.04,163.02) .. controls (97.08,163.02) and (89,155.86) .. (89,147.02) -- cycle ;
%Straight Lines [id:da13685527886677962]
\draw    (56.08,66.02) ;
\draw [shift={(56.08,66.02)}, rotate = 0] [color={rgb, 255:red, 0; green, 0; blue, 0 }  ][fill={rgb, 255:red, 0; green, 0; blue, 0 }  ][line width=0.75]      (0, 0) circle [x radius= 3.35, y radius= 3.35]   ;
%Straight Lines [id:da5689854284008986]
\draw    (43.1,69.07) ;
\draw [shift={(43.1,69.07)}, rotate = 0] [color={rgb, 255:red, 0; green, 0; blue, 0 }  ][fill={rgb, 255:red, 0; green, 0; blue, 0 }  ][line width=0.75]      (0, 0) circle [x radius= 3.35, y radius= 3.35]   ;
%Shape: Ellipse [id:dp5202423592840466]
\draw  [dash pattern={on 0.84pt off 2.51pt}] (30,66.02) .. controls (30,57.19) and (38.08,50.02) .. (48.04,50.02) .. controls (58.01,50.02) and (66.08,57.19) .. (66.08,66.02) .. controls (66.08,74.86) and (58.01,82.02) .. (48.04,82.02) .. controls (38.08,82.02) and (30,74.86) .. (30,66.02) -- cycle ;
%Straight Lines [id:da8053089480397759]
\draw    (103.08,142.02) ;
\draw [shift={(103.08,142.02)}, rotate = 0] [color={rgb, 255:red, 0; green, 0; blue, 0 }  ][fill={rgb, 255:red, 0; green, 0; blue, 0 }  ][line width=0.75]      (0, 0) circle [x radius= 3.35, y radius= 3.35]   ;
%Straight Lines [id:da06733721965233164]
\draw    (203.08,141.02) ;
\draw [shift={(203.08,141.02)}, rotate = 0] [color={rgb, 255:red, 0; green, 0; blue, 0 }  ][fill={rgb, 255:red, 0; green, 0; blue, 0 }  ][line width=0.75]      (0, 0) circle [x radius= 3.35, y radius= 3.35]   ;
%Straight Lines [id:da4870587384348073]
\draw    (190.1,144.07) ;
\draw [shift={(190.1,144.07)}, rotate = 0] [color={rgb, 255:red, 0; green, 0; blue, 0 }  ][fill={rgb, 255:red, 0; green, 0; blue, 0 }  ][line width=0.75]      (0, 0) circle [x radius= 3.35, y radius= 3.35]   ;
%Shape: Ellipse [id:dp7051298437031039]
\draw  [dash pattern={on 0.84pt off 2.51pt}] (177,141.02) .. controls (177,132.19) and (185.08,125.02) .. (195.04,125.02) .. controls (205.01,125.02) and (213.08,132.19) .. (213.08,141.02) .. controls (213.08,149.86) and (205.01,157.02) .. (195.04,157.02) .. controls (185.08,157.02) and (177,149.86) .. (177,141.02) -- cycle ;
%Straight Lines [id:da5002637245882986]
\draw    (191.08,136.02) ;
\draw [shift={(191.08,136.02)}, rotate = 0] [color={rgb, 255:red, 0; green, 0; blue, 0 }  ][fill={rgb, 255:red, 0; green, 0; blue, 0 }  ][line width=0.75]      (0, 0) circle [x radius= 3.35, y radius= 3.35]   ;
%Straight Lines [id:da7004591237131352]
\draw    (65.08,71.02) -- (100,77.52) ;
%Straight Lines [id:da33258302600467093]
\draw    (107.04,131.02) -- (117.08,103.02) ;
%Straight Lines [id:da5150166773704179]
\draw    (125.08,147.02) -- (177,141.02) ;
%Straight Lines [id:da8576945592396861]
\draw    (158.08,77.52) -- (205,63.02) ;
%Straight Lines [id:da15309662564542892]
\draw    (434.66,30.19) ;
\draw [shift={(434.66,30.19)}, rotate = 0] [color={rgb, 255:red, 0; green, 0; blue, 0 }  ][line width=0.75]      (0, 0) circle [x radius= 3.35, y radius= 3.35]   ;
%Straight Lines [id:da9139277997528157]
\draw    (450.79,97.38) ;
\draw [shift={(450.79,97.38)}, rotate = 0] [color={rgb, 255:red, 0; green, 0; blue, 0 }  ][fill={rgb, 255:red, 0; green, 0; blue, 0 }  ][line width=0.75]      (0, 0) circle [x radius= 3.35, y radius= 3.35]   ;
%Straight Lines [id:da7168605892260345]
\draw    (430.85,98.24) ;
\draw [shift={(430.85,98.24)}, rotate = 0] [color={rgb, 255:red, 0; green, 0; blue, 0 }  ][fill={rgb, 255:red, 0; green, 0; blue, 0 }  ][line width=0.75]      (0, 0) circle [x radius= 3.35, y radius= 3.35]   ;
%Shape: Ellipse [id:dp6179637612111752]
\draw  [dash pattern={on 0.84pt off 2.51pt}] (410,90.52) .. controls (410,75.34) and (423,63.02) .. (439.04,63.02) .. controls (455.08,63.02) and (468.08,75.34) .. (468.08,90.52) .. controls (468.08,105.71) and (455.08,118.02) .. (439.04,118.02) .. controls (423,118.02) and (410,105.71) .. (410,90.52) -- cycle ;
%Straight Lines [id:da7313975778128612]
\draw    (433.1,84.07) ;
\draw [shift={(433.1,84.07)}, rotate = 0] [color={rgb, 255:red, 0; green, 0; blue, 0 }  ][fill={rgb, 255:red, 0; green, 0; blue, 0 }  ][line width=0.75]      (0, 0) circle [x radius= 3.35, y radius= 3.35]   ;
%Straight Lines [id:da4012696656345951]
\draw    (425.08,160.02) ;
\draw [shift={(425.08,160.02)}, rotate = 0] [color={rgb, 255:red, 0; green, 0; blue, 0 }  ][fill={rgb, 255:red, 0; green, 0; blue, 0 }  ][line width=0.75]      (0, 0) circle [x radius= 3.35, y radius= 3.35]   ;
%Straight Lines [id:da8949606061270403]
\draw    (412.1,163.07) ;
\draw [shift={(412.1,163.07)}, rotate = 0] [color={rgb, 255:red, 0; green, 0; blue, 0 }  ][fill={rgb, 255:red, 0; green, 0; blue, 0 }  ][line width=0.75]      (0, 0) circle [x radius= 3.35, y radius= 3.35]   ;
%Shape: Ellipse [id:dp6685863922130224]
\draw  [dash pattern={on 0.84pt off 2.51pt}] (399,160.02) .. controls (399,151.19) and (407.08,144.02) .. (417.04,144.02) .. controls (427.01,144.02) and (435.08,151.19) .. (435.08,160.02) .. controls (435.08,168.86) and (427.01,176.02) .. (417.04,176.02) .. controls (407.08,176.02) and (399,168.86) .. (399,160.02) -- cycle ;
%Straight Lines [id:da3367558829520205]
\draw    (366.08,79.02) ;
\draw [shift={(366.08,79.02)}, rotate = 0] [color={rgb, 255:red, 0; green, 0; blue, 0 }  ][fill={rgb, 255:red, 0; green, 0; blue, 0 }  ][line width=0.75]      (0, 0) circle [x radius= 3.35, y radius= 3.35]   ;
%Straight Lines [id:da4499313301404717]
\draw    (353.1,82.07) ;
\draw [shift={(353.1,82.07)}, rotate = 0] [color={rgb, 255:red, 0; green, 0; blue, 0 }  ][fill={rgb, 255:red, 0; green, 0; blue, 0 }  ][line width=0.75]      (0, 0) circle [x radius= 3.35, y radius= 3.35]   ;
%Shape: Ellipse [id:dp5276367226764547]
\draw  [dash pattern={on 0.84pt off 2.51pt}] (340,79.02) .. controls (340,70.19) and (348.08,63.02) .. (358.04,63.02) .. controls (368.01,63.02) and (376.08,70.19) .. (376.08,79.02) .. controls (376.08,87.86) and (368.01,95.02) .. (358.04,95.02) .. controls (348.08,95.02) and (340,87.86) .. (340,79.02) -- cycle ;
%Straight Lines [id:da5845051266366754]
\draw    (413.08,155.02) ;
\draw [shift={(413.08,155.02)}, rotate = 0] [color={rgb, 255:red, 0; green, 0; blue, 0 }  ][fill={rgb, 255:red, 0; green, 0; blue, 0 }  ][line width=0.75]      (0, 0) circle [x radius= 3.35, y radius= 3.35]   ;
%Straight Lines [id:da07818663432602291]
\draw    (513.08,154.02) ;
\draw [shift={(513.08,154.02)}, rotate = 0] [color={rgb, 255:red, 0; green, 0; blue, 0 }  ][fill={rgb, 255:red, 0; green, 0; blue, 0 }  ][line width=0.75]      (0, 0) circle [x radius= 3.35, y radius= 3.35]   ;
%Straight Lines [id:da0336212558851019]
\draw    (500.1,157.07) ;
\draw [shift={(500.1,157.07)}, rotate = 0] [color={rgb, 255:red, 0; green, 0; blue, 0 }  ][fill={rgb, 255:red, 0; green, 0; blue, 0 }  ][line width=0.75]      (0, 0) circle [x radius= 3.35, y radius= 3.35]   ;
%Shape: Ellipse [id:dp4172329261503316]
\draw  [dash pattern={on 0.84pt off 2.51pt}] (487,154.02) .. controls (487,145.19) and (495.08,138.02) .. (505.04,138.02) .. controls (515.01,138.02) and (523.08,145.19) .. (523.08,154.02) .. controls (523.08,162.86) and (515.01,170.02) .. (505.04,170.02) .. controls (495.08,170.02) and (487,162.86) .. (487,154.02) -- cycle ;
%Straight Lines [id:da9591209403522774]
\draw    (501.08,149.02) ;
\draw [shift={(501.08,149.02)}, rotate = 0] [color={rgb, 255:red, 0; green, 0; blue, 0 }  ][fill={rgb, 255:red, 0; green, 0; blue, 0 }  ][line width=0.75]      (0, 0) circle [x radius= 3.35, y radius= 3.35]   ;
%Straight Lines [id:da07939845839438875]
\draw    (375.08,84.02) -- (410,90.52) ;
%Straight Lines [id:da6359863593460653]
\draw    (417.04,144.02) -- (427.08,116.02) ;
%Straight Lines [id:da43650684936033657]
\draw    (435.08,160.02) -- (487,154.02) ;
%Straight Lines [id:da5846717973802147]
\draw    (448.1,80.07) ;
\draw [shift={(448.1,80.07)}, rotate = 0] [color={rgb, 255:red, 0; green, 0; blue, 0 }  ][fill={rgb, 255:red, 0; green, 0; blue, 0 }  ][line width=0.75]      (0, 0) circle [x radius= 3.35, y radius= 3.35]   ;
%Straight Lines [id:da21841518930841186]
\draw    (434.66,33.19) -- (434.08,63.02) ;
%Straight Lines [id:da6636308528277945]
\draw    (256.79,237.38) ;
\draw [shift={(256.79,237.38)}, rotate = 0] [color={rgb, 255:red, 0; green, 0; blue, 0 }  ][fill={rgb, 255:red, 0; green, 0; blue, 0 }  ][line width=0.75]      (0, 0) circle [x radius= 3.35, y radius= 3.35]   ;
%Straight Lines [id:da3041967218854611]
\draw    (237.85,231.24) ;
\draw [shift={(237.85,231.24)}, rotate = 0] [color={rgb, 255:red, 0; green, 0; blue, 0 }  ][fill={rgb, 255:red, 0; green, 0; blue, 0 }  ][line width=0.75]      (0, 0) circle [x radius= 3.35, y radius= 3.35]   ;
%Shape: Ellipse [id:dp5630585495828826]
\draw  [dash pattern={on 0.84pt off 2.51pt}] (216,230.52) .. controls (216,215.34) and (229,203.02) .. (245.04,203.02) .. controls (261.08,203.02) and (274.08,215.34) .. (274.08,230.52) .. controls (274.08,245.71) and (261.08,258.02) .. (245.04,258.02) .. controls (229,258.02) and (216,245.71) .. (216,230.52) -- cycle ;
%Straight Lines [id:da7367328858847835]
\draw    (347.08,216.02) ;
\draw [shift={(347.08,216.02)}, rotate = 0] [color={rgb, 255:red, 0; green, 0; blue, 0 }  ][fill={rgb, 255:red, 0; green, 0; blue, 0 }  ][line width=0.75]      (0, 0) circle [x radius= 3.35, y radius= 3.35]   ;
%Straight Lines [id:da15654182651845128]
\draw    (334.1,219.07) ;
\draw [shift={(334.1,219.07)}, rotate = 0] [color={rgb, 255:red, 0; green, 0; blue, 0 }  ][fill={rgb, 255:red, 0; green, 0; blue, 0 }  ][line width=0.75]      (0, 0) circle [x radius= 3.35, y radius= 3.35]   ;
%Shape: Ellipse [id:dp1471809973669158]
\draw  [dash pattern={on 0.84pt off 2.51pt}] (321,216.02) .. controls (321,207.19) and (329.08,200.02) .. (339.04,200.02) .. controls (349.01,200.02) and (357.08,207.19) .. (357.08,216.02) .. controls (357.08,224.86) and (349.01,232.02) .. (339.04,232.02) .. controls (329.08,232.02) and (321,224.86) .. (321,216.02) -- cycle ;
%Straight Lines [id:da8456494611509275]
\draw    (231.08,300.02) ;
\draw [shift={(231.08,300.02)}, rotate = 0] [color={rgb, 255:red, 0; green, 0; blue, 0 }  ][fill={rgb, 255:red, 0; green, 0; blue, 0 }  ][line width=0.75]      (0, 0) circle [x radius= 3.35, y radius= 3.35]   ;
%Straight Lines [id:da212401736414344]
\draw    (218.1,303.07) ;
\draw [shift={(218.1,303.07)}, rotate = 0] [color={rgb, 255:red, 0; green, 0; blue, 0 }  ][fill={rgb, 255:red, 0; green, 0; blue, 0 }  ][line width=0.75]      (0, 0) circle [x radius= 3.35, y radius= 3.35]   ;
%Shape: Ellipse [id:dp7995603717115229]
\draw  [dash pattern={on 0.84pt off 2.51pt}] (205,300.02) .. controls (205,291.19) and (213.08,284.02) .. (223.04,284.02) .. controls (233.01,284.02) and (241.08,291.19) .. (241.08,300.02) .. controls (241.08,308.86) and (233.01,316.02) .. (223.04,316.02) .. controls (213.08,316.02) and (205,308.86) .. (205,300.02) -- cycle ;
%Straight Lines [id:da0453075102581153]
\draw    (172.08,219.02) ;
\draw [shift={(172.08,219.02)}, rotate = 0] [color={rgb, 255:red, 0; green, 0; blue, 0 }  ][fill={rgb, 255:red, 0; green, 0; blue, 0 }  ][line width=0.75]      (0, 0) circle [x radius= 3.35, y radius= 3.35]   ;
%Straight Lines [id:da576221624719621]
\draw    (159.1,222.07) ;
\draw [shift={(159.1,222.07)}, rotate = 0] [color={rgb, 255:red, 0; green, 0; blue, 0 }  ][fill={rgb, 255:red, 0; green, 0; blue, 0 }  ][line width=0.75]      (0, 0) circle [x radius= 3.35, y radius= 3.35]   ;
%Shape: Ellipse [id:dp5861483877954006]
\draw  [dash pattern={on 0.84pt off 2.51pt}] (146,219.02) .. controls (146,210.19) and (154.08,203.02) .. (164.04,203.02) .. controls (174.01,203.02) and (182.08,210.19) .. (182.08,219.02) .. controls (182.08,227.86) and (174.01,235.02) .. (164.04,235.02) .. controls (154.08,235.02) and (146,227.86) .. (146,219.02) -- cycle ;
%Straight Lines [id:da6861379938931051]
\draw    (219.08,295.02) ;
\draw [shift={(219.08,295.02)}, rotate = 0] [color={rgb, 255:red, 0; green, 0; blue, 0 }  ][fill={rgb, 255:red, 0; green, 0; blue, 0 }  ][line width=0.75]      (0, 0) circle [x radius= 3.35, y radius= 3.35]   ;
%Straight Lines [id:da8002395580825545]
\draw    (319.08,294.02) ;
\draw [shift={(319.08,294.02)}, rotate = 0] [color={rgb, 255:red, 0; green, 0; blue, 0 }  ][fill={rgb, 255:red, 0; green, 0; blue, 0 }  ][line width=0.75]      (0, 0) circle [x radius= 3.35, y radius= 3.35]   ;
%Straight Lines [id:da704037885765733]
\draw    (306.1,297.07) ;
\draw [shift={(306.1,297.07)}, rotate = 0] [color={rgb, 255:red, 0; green, 0; blue, 0 }  ][fill={rgb, 255:red, 0; green, 0; blue, 0 }  ][line width=0.75]      (0, 0) circle [x radius= 3.35, y radius= 3.35]   ;
%Shape: Ellipse [id:dp48179626636658734]
\draw  [dash pattern={on 0.84pt off 2.51pt}] (293,294.02) .. controls (293,285.19) and (301.08,278.02) .. (311.04,278.02) .. controls (321.01,278.02) and (329.08,285.19) .. (329.08,294.02) .. controls (329.08,302.86) and (321.01,310.02) .. (311.04,310.02) .. controls (301.08,310.02) and (293,302.86) .. (293,294.02) -- cycle ;
%Straight Lines [id:da640203764911013]
\draw    (307.08,289.02) ;
\draw [shift={(307.08,289.02)}, rotate = 0] [color={rgb, 255:red, 0; green, 0; blue, 0 }  ][fill={rgb, 255:red, 0; green, 0; blue, 0 }  ][line width=0.75]      (0, 0) circle [x radius= 3.35, y radius= 3.35]   ;
%Straight Lines [id:da2569341182050351]
\draw    (181.08,224.02) -- (216,230.52) ;
%Straight Lines [id:da5641209419268309]
\draw    (223.04,284.02) -- (233.08,256.02) ;
%Straight Lines [id:da671635923143937]
\draw    (241.08,300.02) -- (293,294.02) ;
%Straight Lines [id:da6167057335735444]
\draw    (274.08,230.52) -- (321,216.02) ;
%Straight Lines [id:da3893725013237108]
\draw    (290.66,171.19) ;
\draw [shift={(290.66,171.19)}, rotate = 0] [color={rgb, 255:red, 0; green, 0; blue, 0 }  ][line width=0.75]      (0, 0) circle [x radius= 3.35, y radius= 3.35]   ;
%Straight Lines [id:da7940439004175355]
\draw    (288.66,174.19) -- (259.08,207.02) ;

% Text Node
\draw (131.04,53.42) node [anchor=north west][inner sep=0.75pt]    {$I_{t}$};
% Text Node
\draw (116,25.4) node [anchor=north west][inner sep=0.75pt]    {$t$};
% Text Node
\draw (258,89.4) node [anchor=north west][inner sep=0.75pt]    {$=$};
% Text Node
\draw (411,17.4) node [anchor=north west][inner sep=0.75pt]    {$\star $};
% Text Node
\draw (514,87.4) node [anchor=north west][inner sep=0.75pt]    {$+\cdots $};
% Text Node
\draw (81,248.4) node [anchor=north west][inner sep=0.75pt]    {$+d_{\mu ^{2}}$};
% Text Node
\draw (264,153.4) node [anchor=north west][inner sep=0.75pt]    {$\star $};
% Text Node
\draw (215,188.4) node [anchor=north west][inner sep=0.75pt]    {$t$};
% Text Node
\draw (244,174.4) node [anchor=north west][inner sep=0.75pt]    {$P_{\star t}$};
% Text Node
\draw (51,95.4) node [anchor=north west][inner sep=0.75pt]    {$\alpha $};

\end{tikzpicture}
  \centering
  \caption{$\alpha$ is cohomologous to an element in the first summand.}\label{TwoSummand}
\end{figure}
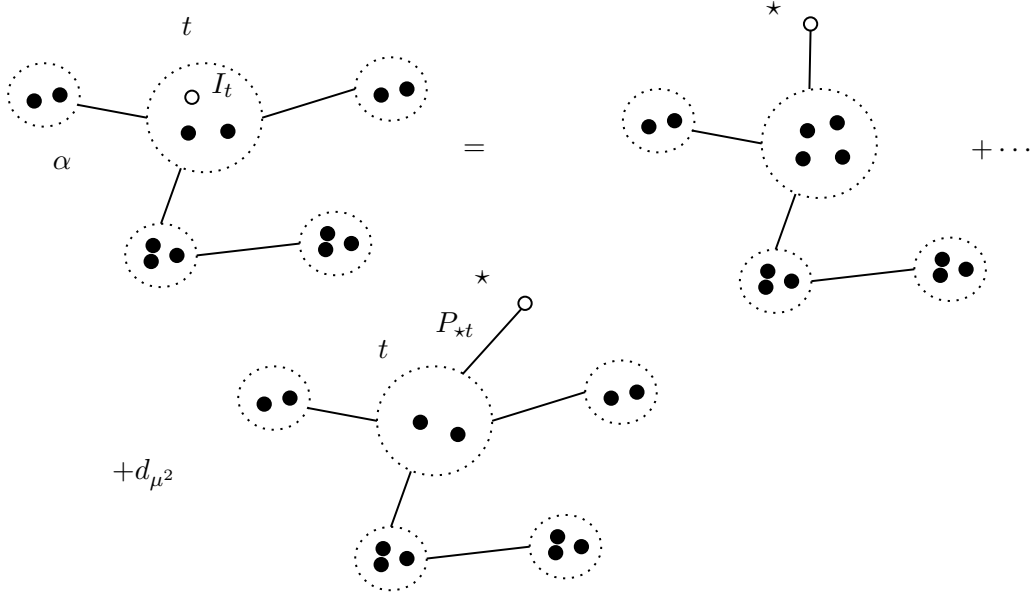
%By inspection, one sees that whenever $\Gamma$ has at least one edge, the basis vector $P_{f,\Gamma}$ can be unambiguously recovered given the data of all the pairwise chiral operations, $\mu^{(I/T)}(P_{f,\Gamma})$ with $|T| = |I|-1$.

Now suppose that we have $\alpha\in\Coh_{\ddts}^p(\ChirComp_{(\AA^n)^k}^{q,\bul})$ such that $[\ddmu 2](\alpha)=0$. Modifying $\alpha$ by a $[\ddmu 2]$-exact term, we can assume that
\[
\alpha\in \bigoplus_{\substack{T\in Q(I) : \\ |T|=-q,\\ I_{\pi(1)}=\{1\}}} \Delta_*^{(I/T)}\left(\Coh_{\ddts}^p(\TS^{T,\bul}((\oms)^{\boxtimes T}))\right).
\]
The condition $[\ddmu 2](\alpha)=0$ implies that $\alpha$ is regular between the vertex $1$ and other vertices (that is, no propagator connects 1 and other $v\in I$). We can repeat this procedure and get either $\alpha=[\ddmu 2](\beta)$ or
\[
\alpha\in \bigoplus_{\substack{T\in Q(I) : \\ |T|=-q,\\ I_{\pi(1)}=\{1\},\dots,I_{\pi(q-1)}=\{q-1\}}} \Delta_*^{(I/T)}\left(\Coh_{\ddts}^p(\TS^{T,\bul}((\oms)^{\boxtimes T}))\right)
\]
which is regular between every vertex in $\{1,\dots,q-1\}$ and all other vertices (that is, no propagator emanates from any of these vertices). If $|\pi^{-1}(\pi(q))|=1$, then $-q=k,p=-k(n-1)$ which is a contradiction since we assume that $p>-k(n-1)$. Thus $|\pi^{-1}(\pi(q))|>1$, and by a similar construction of $\beta$ above we can see that this element $\alpha$ is again $[\ddmu 2]$-exact(see Fig.\ref{MorePointsExact}).
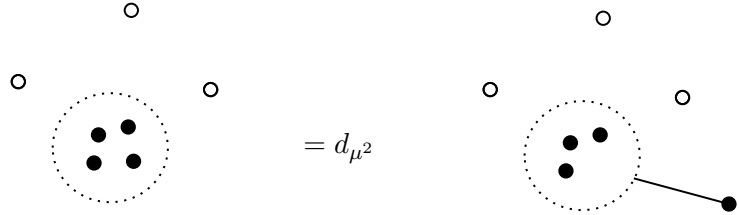
\begin{figure}[htp]
  \centering

\tikzset{every picture/.style={line width=0.75pt}} %set default line width to 0.75pt

\begin{tikzpicture}[x=0.75pt,y=0.75pt,yscale=-1,xscale=1]
%uncomment if require: \path (0,205); %set diagram left start at 0, and has height of 205

%Straight Lines [id:da6041102842099777]
\draw    (209.79,118.38) ;
\draw [shift={(209.79,118.38)}, rotate = 0] [color={rgb, 255:red, 0; green, 0; blue, 0 }  ][fill={rgb, 255:red, 0; green, 0; blue, 0 }  ][line width=0.75]      (0, 0) circle [x radius= 3.35, y radius= 3.35]   ;
%Straight Lines [id:da968689244559449]
\draw    (189.85,119.24) ;
\draw [shift={(189.85,119.24)}, rotate = 0] [color={rgb, 255:red, 0; green, 0; blue, 0 }  ][fill={rgb, 255:red, 0; green, 0; blue, 0 }  ][line width=0.75]      (0, 0) circle [x radius= 3.35, y radius= 3.35]   ;
%Shape: Ellipse [id:dp751423797118191]
\draw  [dash pattern={on 0.84pt off 2.51pt}] (169,111.52) .. controls (169,96.34) and (182,84.02) .. (198.04,84.02) .. controls (214.08,84.02) and (227.08,96.34) .. (227.08,111.52) .. controls (227.08,126.71) and (214.08,139.02) .. (198.04,139.02) .. controls (182,139.02) and (169,126.71) .. (169,111.52) -- cycle ;
%Straight Lines [id:da1298357353620796]
\draw    (192.1,105.07) ;
\draw [shift={(192.1,105.07)}, rotate = 0] [color={rgb, 255:red, 0; green, 0; blue, 0 }  ][fill={rgb, 255:red, 0; green, 0; blue, 0 }  ][line width=0.75]      (0, 0) circle [x radius= 3.35, y radius= 3.35]   ;
%Straight Lines [id:da6272357876917616]
\draw    (207.1,101.07) ;
\draw [shift={(207.1,101.07)}, rotate = 0] [color={rgb, 255:red, 0; green, 0; blue, 0 }  ][fill={rgb, 255:red, 0; green, 0; blue, 0 }  ][line width=0.75]      (0, 0) circle [x radius= 3.35, y radius= 3.35]   ;
%Straight Lines [id:da9735823213446644]
\draw    (208.66,42.19) ;
\draw [shift={(208.66,42.19)}, rotate = 0] [color={rgb, 255:red, 0; green, 0; blue, 0 }  ][line width=0.75]      (0, 0) circle [x radius= 3.35, y radius= 3.35]   ;
%Straight Lines [id:da3151786887850321]
\draw    (151.66,78.19) ;
\draw [shift={(151.66,78.19)}, rotate = 0] [color={rgb, 255:red, 0; green, 0; blue, 0 }  ][line width=0.75]      (0, 0) circle [x radius= 3.35, y radius= 3.35]   ;
%Straight Lines [id:da2306620212027677]
\draw    (151.66,78.19) ;
\draw [shift={(151.66,78.19)}, rotate = 0] [color={rgb, 255:red, 0; green, 0; blue, 0 }  ][line width=0.75]      (0, 0) circle [x radius= 3.35, y radius= 3.35]   ;
%Straight Lines [id:da5412285291547232]
\draw    (248.66,82.19) ;
\draw [shift={(248.66,82.19)}, rotate = 0] [color={rgb, 255:red, 0; green, 0; blue, 0 }  ][line width=0.75]      (0, 0) circle [x radius= 3.35, y radius= 3.35]   ;
%Straight Lines [id:da263967254270449]
\draw    (248.66,82.19) ;
\draw [shift={(248.66,82.19)}, rotate = 0] [color={rgb, 255:red, 0; green, 0; blue, 0 }  ][line width=0.75]      (0, 0) circle [x radius= 3.35, y radius= 3.35]   ;
%Straight Lines [id:da18267292217920428]
\draw    (427.85,123.24) ;
\draw [shift={(427.85,123.24)}, rotate = 0] [color={rgb, 255:red, 0; green, 0; blue, 0 }  ][fill={rgb, 255:red, 0; green, 0; blue, 0 }  ][line width=0.75]      (0, 0) circle [x radius= 3.35, y radius= 3.35]   ;
%Shape: Ellipse [id:dp9041273628963711]
\draw  [dash pattern={on 0.84pt off 2.51pt}] (407,115.52) .. controls (407,100.34) and (420,88.02) .. (436.04,88.02) .. controls (452.08,88.02) and (465.08,100.34) .. (465.08,115.52) .. controls (465.08,130.71) and (452.08,143.02) .. (436.04,143.02) .. controls (420,143.02) and (407,130.71) .. (407,115.52) -- cycle ;
%Straight Lines [id:da9554772661887192]
\draw    (430.1,109.07) ;
\draw [shift={(430.1,109.07)}, rotate = 0] [color={rgb, 255:red, 0; green, 0; blue, 0 }  ][fill={rgb, 255:red, 0; green, 0; blue, 0 }  ][line width=0.75]      (0, 0) circle [x radius= 3.35, y radius= 3.35]   ;
%Straight Lines [id:da8214703507931362]
\draw    (445.1,105.07) ;
\draw [shift={(445.1,105.07)}, rotate = 0] [color={rgb, 255:red, 0; green, 0; blue, 0 }  ][fill={rgb, 255:red, 0; green, 0; blue, 0 }  ][line width=0.75]      (0, 0) circle [x radius= 3.35, y radius= 3.35]   ;
%Straight Lines [id:da11429813298217661]
\draw    (446.66,46.19) ;
\draw [shift={(446.66,46.19)}, rotate = 0] [color={rgb, 255:red, 0; green, 0; blue, 0 }  ][line width=0.75]      (0, 0) circle [x radius= 3.35, y radius= 3.35]   ;
%Straight Lines [id:da19140922931742588]
\draw    (389.66,82.19) ;
\draw [shift={(389.66,82.19)}, rotate = 0] [color={rgb, 255:red, 0; green, 0; blue, 0 }  ][line width=0.75]      (0, 0) circle [x radius= 3.35, y radius= 3.35]   ;
%Straight Lines [id:da5420304120052812]
\draw    (389.66,82.19) ;
\draw [shift={(389.66,82.19)}, rotate = 0] [color={rgb, 255:red, 0; green, 0; blue, 0 }  ][line width=0.75]      (0, 0) circle [x radius= 3.35, y radius= 3.35]   ;
%Straight Lines [id:da9872236584640033]
\draw    (486.66,86.19) ;
\draw [shift={(486.66,86.19)}, rotate = 0] [color={rgb, 255:red, 0; green, 0; blue, 0 }  ][line width=0.75]      (0, 0) circle [x radius= 3.35, y radius= 3.35]   ;
%Straight Lines [id:da5551178419922933]
\draw    (486.66,86.19) ;
\draw [shift={(486.66,86.19)}, rotate = 0] [color={rgb, 255:red, 0; green, 0; blue, 0 }  ][line width=0.75]      (0, 0) circle [x radius= 3.35, y radius= 3.35]   ;
%Straight Lines [id:da8989475926300758]
\draw    (462.08,127) -- (510.08,139.95) ;
\draw [shift={(510.08,139.95)}, rotate = 15.1] [color={rgb, 255:red, 0; green, 0; blue, 0 }  ][fill={rgb, 255:red, 0; green, 0; blue, 0 }  ][line width=0.75]      (0, 0) circle [x radius= 3.35, y radius= 3.35]   ;

% Text Node
\draw (294,102.4) node [anchor=north west][inner sep=0.75pt]    {$=d_{\mu ^{2}}$};

\end{tikzpicture}
  \caption{$|\pi^{-1}(\pi(q))|>1$ implies that $\alpha$ is exact.}\label{MorePointsExact}
\end{figure}
\end{proof}

\begin{thm}\label{thm: Chevalley-Cousin} Let $n\geq 2$ and $k\geq 2$.
The Chevalley-Cousin complex $\left(\ChirCompTot_{(\AA^n)^k}^\bul,D\right)$ is a resolution of the
%$\kk[z^r_i]^{1\leq r\leq n}_{1\leq i\leq k} = \O_{(\AA^{n})^k}$-module $\Gamma\left((\AA^{n})^k,(\omega\shifted_{(\AA^n)})^{\boxtimes k}\right)$
of global sections of the shifted canonical sheaf. Namely, we have the following complex of $\D$-modules
\be 0 \to\Gamma\left((\AA^{n})^k,(\omega\shifted_{(\AA^n)})^{\boxtimes k}\right) \to \ChirCompTot_{(\AA^n)^k}^{-kn} \xrightarrow{D} \ChirCompTot_{(\AA^n)^k}^{-kn+1} \xrightarrow{D} \cdots  \xrightarrow{D} \ChirCompTot_{(\AA^n)^k}^{-n} \xrightarrow{}0\nn\ee
which is exact.
\end{thm}
\begin{proof}
We need to prove that the complex  $(\ChirCompTot_{(\AA^n)^k}^\bul,D)$ has vanishing cohomology in all degrees except $-kn$. Note that
\[
\ChirCompTot_{(\AA^n)^k}^{-kn}=\ChirComp_{(\AA^n)^I}^{-k,-k(n-1)} :=
\TS^{I,-k(n-1)}(\oms).
\]
Furthermore $\mathrm{Ker}(D)=\mathrm{Ker}(\ddts)=\Gamma\left((\AA^{n})^k,(\omega\shifted_{(\AA^n)})^{\boxtimes k}\right)$ from the computation of the cohomology in the previous section.

Now consider an element $\alpha \in \ChirCompTot_{(\AA^n)^k}^{p},p>-kn$. Such an element is, more explictly, a sum
\be \alpha = \alpha^{-k,p+k} + \alpha^{-k+1,p+k-1} + \dots + \alpha^{-1,p+1} \nn\ee
with each $\alpha^{-q,p+q} \in \ChirComp_{(\AA^n)^k}^{-q,p+q}$. Let us demand that $\alpha$ be $D$-closed,
\be (\ddts + \ddmu 2 + \dots + \ddmu k)\alpha = 0, \nn\ee
so that it defines a cohomology class $[\alpha]_D \in \Coh_D^p(\ChirCompTot_{(\AA^n)^k}^{\bul})$.
That demand entails in particular that
\be \ddts \alpha^{-k,p+k} = 0 \quad\text{and}\quad \ddmu 2 \alpha^{-k,p+k} = -\ddts \alpha^{-k+1,p+k-1} .\nn\ee
The first of these implies that $\alpha^{-k,p+k}$ represents class $[\alpha^{-k,p+k}]_{\ddts} \in \Coh_{\ddts}^{p+k}(\ChirComp_{(\AA^n)^k}^{-k,\bul})$ for the cohomology of $\ddts$. Now we shall use \cref{lem:cohom} to argue that this class is the zero class. Indeed, \cref{lem:cohom} implies that the class $[\alpha^{-k,p+k}]_{\ddts}$ is zero whenever its image $[\ddmu2 \alpha^{-k,p+k}]_{\ddts} \in \Coh_{\ddts}^{p+k}(\ChirComp_{(\AA^n)^k}^{-k+1,\bul})$ is zero. By the second equation above this image is $[-\ddts \alpha^{-k+1,p+k-1}]_{\ddts}$ which is indeed the zero class. We conclude that
\be \alpha^{-k,p+k} = \ddts \beta^{-k,p+k-1}\nn\ee
for some $\beta^{-k,p+k-1} \in \ChirComp_{(\AA^n)^k}^{-k,p+k-1}$.
Therefore we have that
\begin{align} [\alpha]_D &= [\alpha^{-k+1,p+k-1} + \dots + \alpha^{-1,p+1} + \ddts \beta^{-k,p+k-1}]_D \nn\\
&= [\alpha^{-k+1,p+k-1} + \dots + \alpha^{-1,p+1} + \ddts \beta^{-k,p+k-1} - D\beta^{-k,p+k-1}]_D \nn\\
&= [\alpha^{-k+1,p+k-1} + \dots + \alpha^{-1,p+1} - (\ddmu 2+\dots+\ddmu k)\beta^{-k,p+k-1} ]_D\nn\\
&= [\tilde\alpha]_D
\end{align}
where
\be \tilde\alpha = \tilde\alpha^{-k+1,p+k-1} + \dots + \tilde\alpha^{-1,p+1} \nn\ee
with $\tilde\alpha^{-k+q,p+k-q} := \alpha^{-k+q,p+k+q} - \ddmu {q+1} \beta^{-k,p+k-1}$.
So we have found a new representative, $\tilde\alpha$, of our original cohomology class $[\alpha]_D$. The condition $D\tilde{\alpha}$ now implies that
\be \ddts \tilde\alpha^{-k+1,p+k-1} = 0 \quad\text{and}\quad \ddmu 2 \tilde\alpha^{-k+1,p+k-1} = -\ddts \tilde{\alpha}^{-k+2,p+k-2} .\nn\ee
We conclude that $[\alpha^{-k+1,p+k-1}]\in \Coh_{\ddts}^{p+k-1}(\ChirComp_{(\AA^n)^k}^{-k,\bul})$ and $[\ddmu 2][\alpha^{-k+1,p+k-1}]=0$. Again by Lemma \ref{lem:cohom}, $[\alpha^{-k+1,p+k-1}]$ is $[\ddmu 2]$-exact which means that there is $\gamma^{-k,p+k-1},\delta^{-k+1,p+k-2}\in \ChirComp_{(\AA^n)^k}^{-k,p+k-1}$ such that
\be \ddts\gamma^{-k,p+k-1} = 0 \quad\text{and}\quad \alpha^{-k+1,p+k-1}-\ddmu 2\gamma^{-k,p+k-1} = \ddts \delta^{-k+1,p+k-2} .\nn\ee
Define
\[
\tilde{\tilde{\alpha}}=\tilde{\alpha}-D(\gamma^{-k,p+k-1}+\delta^{-k+1,p+k-2}).
\]
Then
\be \tilde{\tilde{\alpha}} =\tilde{ \tilde{\alpha}}^{-k+2,p+k-2} + \dots + \tilde{\tilde{\alpha}}^{-1,p+1} .\nn\ee
Repeating the above procedure, we conclude that $\alpha$ is in fact $D$-exact.

\end{proof}

\section{Examples of higher chiral operations}\label{sec: examples}
Our main result, in the form \cref{thm: quasi-isomorphism from LieInfinity}, asserts the existence of chiral operations $\mu_k\in \P^{ch}_{\AA^n,k-2}[k]$, for $k=2,3,4,\dots$ obeying the coherence relations governing \Linfinity-algebras as in \cref{mucoherence}.
% :
% \begin{align*}
%   (\mu_k\circ \ddts)(-) & =\sum_{\substack{ k_1+k_2=k+1\\k_1,k_2>1}} \sum_{\sigma\in \mathrm{Sh}^{-1}_{k_1-1,k_2}}\mathrm{sgn}(\sigma)(-1)^{(k_1-1)k_2} (\mu_{k_1}\circ_1\mu_{k_2})^{\sigma}(-).
% \end{align*}
It is instructive to construct the first higher operation, $\mu_3$, explicitly in low dimensions.

\subsection{The Jacobiator and $\mu_3$ in $n=2$ dimensions}
In $n=2$ dimensions the \polysimplicial model $\TS_{\AA^2}^k$ of $R\Gamma(\Conf_k(\AA^2),\O)$ reduces to a model of polynomial differential forms on hypercubes by identifying $u^1_e=u_e$ and $u^2_e=1-u_e$, cf. \cref{rem: cubes}, obeying certain boundary conditions.
Explicitly, we have
\begin{align}
\TS_{\AA^2}^k = &\bigl\{ \tau \in
 \kk[z_i^r,(z_j^r-z_l^r)^{-1})]^{1\leq r\leq 2}_{1\leq i\leq k; 1\leq j<\ell<k} \ox \kk[u_{j\ell},\dd u_{j\ell}]_{1\leq j<\ell<k}\nn\\
& \quad : \text{$\omega|_{u_{ij}=0}$ is regular in $(z^1_i-z^1_j)$ }\nn\\ &\quad \qquad\text{and $\omega|_{u_{ij}=1}$ is regular in $(z^2_i-z^2_j)$, for $1\leq i<j\leq k$} \bigr\}. \nn
\end{align}
In particular, $\TS_{\AA^2}^3$ consists of polynomial differential forms on the $\binom 3 2 =3$-cube, cf. \cref{fig}.

The first coherence relation is
\be \mu_3 \circ \ddts = \mu_2 \circ_1 \mu_2 - (\mu_2 \circ_1 \mu_2)^{[2 1 3]} + (\mu_2 \circ_1 \mu_2)^{[2 3 1]}=:\Jac_3 \nn\ee
where the three $(1,2)$-unshuffles are $[1 2 3]$, $[2 1 3]$ and $[2 3 1]$.

If we identify $\omega_{\AA^2} \ox_{\kk[\lambda^r]^{1\leq r\leq 2}} \kk[\lambda^r_i]^{1\leq r\leq 2}_{1\leq i\leq k}\cong \omega_{\AA^2} \ox_{\kk} \kk[\lambda^r_i]^{1\leq r\leq 2}_{2\leq i\leq k}$,
by carefully working out the signs and using the explicit description of $\mu_2$ from \cref{sec: higher operations}, the Jacobiator in the cube model is given by
\begin{align} \Jac_3(\tau)
%&:= \mu^2_{3\to 1} \mu^2_{2\to 1}(-) + \mu^2_{2\to 1} \mu^2_{3\to 1}(-) + \mu^2_{(32) \to 1} \mu^2_{3 \to 2}(-)\nn \\
&= \Biggl(+\int_{u_{13}=0}^1 \int_{u_{12}=0}^1 \res_{z^1_3 \to z^1_1} \res_{z^1_2\to z^1_1} \res_{z^2_3 \to z^2_1} \res_{z^2_2\to z^2_1}(-)|_{u_{23} = u_{13}}
\nn\\&\qquad
+ \int_{u_{12}=0}^1 \int_{u_{13}=0}^1 \res_{z^1_2 \to z^1_1} \res_{z^1_3\to z^1_1} \res_{z^2_2 \to z^2_1} \res_{z^2_3\to z^2_1}(-)|_{u_{23} = u_{12}} \nn\\
&\qquad + \int_{u_{12}=0}^1 \int_{u_{23}=0}^1 \res_{z^1_2 \to z^1_1} \res_{z^1_3\to z^1_2} \res_{z^2_2\to z^2_1} \res_{z^2_3\to z^2_2}(-)|_{u_{13} = u_{12}}\Biggr)\nn\\
&\qquad\qquad\qquad \circ e^{\lambda^1_3(z^1_3-z^1_1) + \lambda^1_2(z^1_2-z^1_1) + \lambda^2_3(z^2_3-z^2_1) + \lambda^2_2(z^2_2-z^2_1)} \tau. \nn
%\\
% &= -\int_{u_{12}=0}^1 \int_{u_{13}=0}^1 \res_{x_3 \to x_1} \res_{x_2\to x_1} \res_{y_3 \to y_1} \res_{y_2\to y_1}(-)|_{u_{23} = u_{13}}
% \nn\\&\quad
% + \int_{u_{12}=0}^1 \int_{u_{13}=0}^1 \res_{x_2 \to x_1} \res_{x_3\to x_1} \res_{y_2 \to y_1} \res_{y_3\to y_1}(-)|_{u_{32} = u_{12}} \nn\\
% &\qquad\quad + \int_{u_{12}=0}^1 \int_{u_{23}=0}^1 \res_{x_2 \to x_1} \res_{x_3\to x_2} \res_{y_2\to y_1} \res_{y_3\to y_2}(-)|_{u_{31} = u_{21}} \nn\\
\label{jacobiator}
\end{align}
The integrals in these three terms are over, respectively, the yellow, purple and green rectangles in \cref{fig}.

By a direct calculation (using Stokes theorem and the Jacobi relation satisfied by the residue maps) one finds that the solution to the first coherence relation $\mu_3\circ \ddts = \Jac_3$ is in fact unique and is given by
\begin{align*}
    \mu_3(\tau) &= \Biggl(+\int_{u_{23}=0}^1\int_{u_{12}=0}^{u_{23}}\int_{u_{13}=0}^{u_{12}} \res_{z^1_2\to z^1_1}\res_{z^1_3\to z^1_1}\res_{z^2_2\to z^2_1}\res_{z^2_3\to z^2_2}(-)\\
    &\qquad+\int_{u_{13}=0}^1\int_{u_{12}=0}^{u_{13}}\int_{u_{23}=0}^{u_{12}} \res_{z^1_2\to z^1_1}\res_{z^1_3\to z^1_2}\res_{z^2_2\to z^2_1}\res_{z^2_3\to z^2_1}(-)\\
    &\qquad-\int_{u_{12}=0}^1\int_{u_{13}=0}^{u_{12}}\int_{u_{23}=0}^{u_{13}} \res_{z^1_2\to z^1_1}\res_{z^1_3\to z^1_2}\res_{z^2_3\to z^2_1}\res_{z^2_2\to z^2_1}(-)\\
    &\qquad-\int_{u_{23}=0}^1\int_{u_{13}=0}^{u_{23}}\int_{u_{12}=0}^{u_{13}} \res_{z^1_3\to z^1_1}\res_{z^1_2\to z^1_1}\res_{z^2_2\to z^2_1}\res_{z^2_3\to z^2_2}(-)\\
    &\qquad-\int_{u_{13}=0}^1\int_{u_{23}=0}^{u_{13}}\int_{u_{12}=0}^{u_{23}} \res_{z^1_3\to z^1_1}\res_{z^1_2\to z^1_1}\res_{z^2_2\to z^2_1}\res_{z^2_3\to z^2_1}(-)\\
    &\qquad-\int_{u_{12}=0}^1\int_{u_{23}=0}^{u_{12}}\int_{u_{13}=0}^{u_{23}} \res_{z^1_2\to z^1_1}\res_{z^1_3\to z^1_1}\res_{z^2_3\to z^2_1}\res_{z^2_2\to z^2_1}(-)\Biggl)\\
    &\qquad\qquad\qquad \circ e^{\lambda^1_3(z^1_3-z^1_1) + \lambda^1_2(z^1_2-z^1_1) + \lambda^2_3(z^2_3-z^2_1) + \lambda^2_2(z^2_2-z^2_1)} \tau. \nn
\end{align*}
The integrals in these six terms are over the tetrahedra dissecting the cube which are obtained by cutting the cube along the yellow, purple and green rectangles in \cref{fig}.

\subsection{Application to products of propagators}\label{sec: triangle diagram}
With these explicit descriptions one can also easily verify that $\Jac_3\neq0\neq\mu_3$. In $n=2$ dimensions the propagator from \cref{sec: arnold relations} is $\Prop_{ij}=\frac{\dd u_{ij}}{(z^1_i-z^1_j)(z^2_i-z^2_j)}$, and  $\dd\mathbf{z}=\dd z^1\dd z^2$.
% One checks that
% \begin{align*}
%     \mu_2(P_{12}\dd\zz_1\shifted\dd\zz_2\shifted)=1\dd\zz_1\shifted.
% \end{align*}
Now, one may verify that
\begin{align*}
    &P_{12}P_{13}P_{23}=\dd_\TS\left(\frac{u_{12}}{z^1_1-z^1_2}P_{13}\partial_{z^2_2}P_{23}+P_{12}\frac{u_{13}}{z^1_1-z^1_3}\partial_{z^2_2}P_{23}\right)=:\dd_\TS(V)
\end{align*}
or equivalently
\begin{equation}
    \begin{tikzpicture}[scale=0.7,every node/.style={scale=0.7},x=1cm,y=1cm,baseline={0.7cm-0.5*height("$=$")}]
    % Define the vertices of the triangle
    \coordinate (1) at (0,0);
    \coordinate (2) at (2,0);
    \coordinate (3) at (1,1.732);

    \coordinate (12) at (1,0);
    \coordinate (13) at (0.5,0.867);
    \coordinate (23) at (1.5,0.867);

    % Draw the edges of the triangle
    \draw[thick] (1) -- (2);
    \draw[thick] (1) -- (3);
    \draw[thick] (2) -- (3);

    % Draw round nodes at the triangle with labels
    \node[draw, circle, fill=black, inner sep=2pt, label=below:{$1$}] at (1) {};
    \node[draw, circle, fill=black, inner sep=2pt, label=below:{$2$}] at (2) {};
    \node[draw, circle, fill=black, inner sep=2pt, label=above:{$3$}] at (3) {};

    % Draw round nodes at the edges with labels
    \node[draw, circle, fill=white, inner sep=2pt, label=below:{$P_{12}$}] at (12) {};
    \node[draw, circle, fill=white, inner sep=2pt, label=left:{$P_{13}$}] at (13) {};
    \node[draw, circle, fill=white, inner sep=2pt, label=right:{$P_{23}$}] at (23) {};
\end{tikzpicture}
=\dd_\TS\left(
    \begin{tikzpicture}[scale=0.7,every node/.style={scale=0.7},x=1cm,y=1cm,baseline={0.7cm-0.5*height("$=$")}]
    % Define the vertices of the triangle
    \coordinate (1) at (0,0);
    \coordinate (2) at (2,0);
    \coordinate (3) at (1,1.732);

    \coordinate (12) at (1,0);
    \coordinate (13) at (0.5,0.867);
    \coordinate (23) at (1.5,0.867);

    % Draw the edges of the triangle
    \draw[thick] (1) -- (2);
    \draw[thick] (1) -- (3);
    \draw[thick] (2) -- (3);

    % Draw round nodes at the triangle with labels
    \node[draw, circle, fill=black, inner sep=2pt, label=below:{$1$}] at (1) {};
    \node[draw, circle, fill=black, inner sep=2pt, label=below:{$2$}] at (2) {};
    \node[draw, circle, fill=black, inner sep=2pt, label=above:{$3$}] at (3) {};

    % Draw round nodes at the edges with labels
    \node[draw, circle, fill=white, inner sep=2pt, label=below:{$\frac{u_{12}}{z^1_1-z^1_2}$}] at (12) {};
    \node[draw, circle, fill=white, inner sep=2pt, label=left:{$P_{13}$}] at (13) {};
    \node[draw, circle, fill=white, inner sep=2pt, label=right:{$\partial_{z^2_2}P_{23}$}] at (23) {};
    \end{tikzpicture}
+
    \begin{tikzpicture}[scale=0.7,every node/.style={scale=0.7},x=1cm,y=1cm,baseline={0.7cm-0.5*height("$=$")}]
    % Define the vertices of the triangle
    \coordinate (1) at (0,0);
    \coordinate (2) at (2,0);
    \coordinate (3) at (1,1.732);

    \coordinate (12) at (1,0);
    \coordinate (13) at (0.5,0.867);
    \coordinate (23) at (1.5,0.867);

    % Draw the edges of the triangle
    \draw[thick] (1) -- (2);
    \draw[thick] (1) -- (3);
    \draw[thick] (2) -- (3);

    % Draw round nodes at the triangle with labels
    \node[draw, circle, fill=black, inner sep=2pt, label=below:{$1$}] at (1) {};
    \node[draw, circle, fill=black, inner sep=2pt, label=below:{$2$}] at (2) {};
    \node[draw, circle, fill=black, inner sep=2pt, label=above:{$3$}] at (3) {};

    % Draw round nodes at the edges with labels
    \node[draw, circle, fill=white, inner sep=2pt, label=below:{$P_{12}$}] at (12) {};
    \node[draw, circle, fill=white, inner sep=2pt, label=left:{$\frac{u_{13}}{z^1_1-z^1_3}$}] at (13) {};
    \node[draw, circle, fill=white, inner sep=2pt, label=right:{$\partial_{z^2_2}P_{23}$}] at (23) {};
\end{tikzpicture}
\right).\nn
\end{equation}
Observe that, crucially, $V$ fulfils the boundary conditions of the model (in contrast to for example $\frac{u_{12}}{(z^1_1-z^1_2)(z^2_1-z^2_2)}$).

Direct calculation then shows
\begin{align*}
    &\Jac_3(V\dd\zz_1\shifted\dd\zz_2\shifted\dd\zz_3\shifted)=\mu_3(P_{12}P_{13}P_{23}\dd\zz_1\shifted\dd\zz_2\shifted\dd\zz_3\shifted)\\
    &\qquad\qquad=\frac{1}{2}(\lambda^1_2\lambda^2_3-\lambda^2_2\lambda^1_3)\dd\zz_1\shifted =: \frac{1}{2} \ul\lambda_2\wedge\ul\lambda_3\dd\zz_1\shifted.
\end{align*}
This agrees with a calculation
%using a different model (the Dolbeault model)
by Budzik, Gaiotto, Kulp, Wu and Yu; see \cite[equation (4.20)]{BGKWY2023}.

\appendix
\addtocontents{toc}{\protect\setcounter{tocdepth}{1}}
\section{Conventions for $\D$-modules}\label{sec: D module background}
In this section we set the notations related to $\D$-modules in this paper. For basic backgrounds on $\D$-modules, we refer to \cite{HTT_Dmodules2008}.
%%Here we follow the presentation and notation in \cite{van2021chiral}.

Let $X$ be a smooth algebraic variety over $\kk$. We denoted by $\omega_X$ the canonical sheaf of $X$. The tangent sheaf $\Theta_X$ acts on $\omega_X$ by minus the Lie derivative, and this extends to a right action of the sheaf of differential operators $\D_X$. More explicitly, for local sections $\tau\in \omega_X, \theta,\theta_1,\dots,\theta_n\in \Theta_X$ ($n=\dim X$)
\[
\left(\mathcal{L}_\theta\tau\right)(\theta_1,\dots,\theta_n):=\theta\left(\tau(\theta_1,\dots,\theta_n)\right)-\sum^n_{i=1}\tau(\theta_1,\dots,[\theta,\theta_i],\dots,\theta_n).
\]
The right $\D_X$-action is then given by $\tau\theta:=-\mathcal{L}_{\theta}\tau.$
Given a left $\D_X$-module $M$, one can form a right $\D_X$-module $M\ox_{\O_X}\omega_X$ via the Leibniz rule
\[
(m\otimes \tau)\theta := m\otimes \tau\theta-\theta m\otimes \tau.
\]
Conversely, for a right $\D_X$-module $M^r$ we can equip $M^r\ox_{\O_X}\omega_X^{-1}$ with a left $\D_X$-module structure using the canonical isomorphism $\D_X^{op}\cong \omega_X\ox_{\O_X}\D_X\ox_{\O_X}\omega_X^{-1}$.

Suppose that we have a morphism $f:X\rightarrow Y$. We denote by $f^{\centerdot}$ and $f_{\centerdot}$ the sheaf-theoretic pullback and pushforward (in \cite{HTT_Dmodules2008}, they use $f^{-1}.f_*$). Define the transfer bimodule $\D_{f}$ by
\[\D_{f} :=\O_X\ox_{f^{\centerdot}\O_Y}f^{\centerdot}\D_Y.\]
It is naturally a right $f^{\centerdot}\D_Y$-module. The left $\D_X$-module structure is given in terms of local coordinates $\{y_i,\del_i\}_{i=1,\dots,\dim Y}$ on $Y$ by
\[
\theta(s\otimes P)=\theta(s)\otimes P+ \sum_{i=1}^{\dim Y}\left(s\theta(y_i\circ f)\right)\otimes\del_iP.
\]

The pullback of a left $\D_Y$-module $M$ is
\[ f^*(M) := \D_{f}\ox_{f^{\centerdot}\D_Y}f^{\centerdot}(M). \]
For a right $\D_Y$-module $M^r$, the pullback is defined to be
\[
f^*(M^r) := \omega_X\ox_{\O_X}\left(\D_{f}\ox_{f^{\centerdot}\D_Y}f^{\centerdot}(\omega^{-1}_Y\ox_{\O_Y}M^r)\right).
\]
To define the pushforward of $\D$-modules, we need to use derived functors. We are mainly interested in the pushforwards of right $\D$-modules. Denote by $D^b(\D^{op}_X)$ the bounded derived category of right $\D_X$-modules. Define a functor $\int_f:D^b(\D^{op}_X)\rightarrow D^b(\D^{op}_Y)$
\[
\int_f N^{\bul}:=Rf_{\centerdot}\left(N^{\bul}\otimes^{L}_{\D_X}\D_{f}\right),\quad N^{\bul}\in D^b(\D^{op}_X),
\]
where we use a flat resolution of $N^{\bul}$ to build $N^{\bul}\otimes^{L}_{\D_X}\D_{f}$ and an injective resolution of $N^{\bul}\otimes^{L}_{\D_X}\D_{f}$ to define $Rf_{\centerdot}$.

If $j:U\hookrightarrow Y$ is an open embedding, then we have $\D_j=j^{\centerdot}\D_Y=\D_U$ and $j^*$ is exact
\[
j^*M=j^{\centerdot}M=M|_{U},
\]
and for the pushforward we have
\[
\smallint_jN^{\bul}=Rj_{\centerdot}\left(N^{\bul}\otimes^{L}_{\D_U}\D_{j}\right)=Rj_{\centerdot}N^{\bul},\quad N^{\bul}\in D^b(\D^{op}_U).
\]
We may have non-trivial higher cohomology groups.
(This is the case for example when $j:\AA^{2}-\{0\}\hookrightarrow\AA^2$.)

If $i: X\hookrightarrow Y$ is a closed embedding of smooth algebraic varieties, then $\D_i$ is a locally free $\D_X$-module and
\[
i_*(N):=H^0\left(\int_iN\right)=i_{\centerdot}\left(N\ox_{\D_X}\D_{f}\right).
\] is exact.

In this paper, we mainly focus on the open inclusion
\[
j^{(k)}:\Conf_k(\AA^n)\hookrightarrow (\AA^n)^k
\]
and the diagonal embedding
\[
\Delta^{(k)}:\AA^n\hookrightarrow (\AA^n)^k.
\]
One of our main results is a construction of a complex of $\DAAnk$-modules $\TS^{k,\bul}_{\AA^n}(M)$ which computes $\Gamma\left((\AA^{n})^k,\smallint_j(j^*M)\right)$, where $M$ is a right $\D$-module.% which is homotopically $\O_{(\AA^n)^k}$-flat (K-flat).

For the diagonal embedding $\Delta^{(k)}:\AA^n\hookrightarrow (\AA^n)^k$, we can write everything explicitly as in \cref{sec: shifted canonical bundle}.

Since we mainly consider $\D$-modules on $(\AA^{n})^k$ which is $\D$-affine (more precisely, is $\D^{op}$-affine since we are considering right $\D$-modules), we will not distinguish the module and its global section.

\section{Construction of the polysimplicial model}%Proof of \cref{thm: polysimplex model NN}}
In this section we prove \cref{thm: polysimplex model NN}. We first recall the definition of the Thom-Sullivan functor $\Th$.
\subsection{Semisimplicial objects}\label{sec: simplicial sets}
Let $\triangle$ denote the category whose objects are the finite totally-ordered sets
\([n] := \{0<1<\dots<n\}, \, n\in \ZZ_{\geq 0}, \nn\)
and whose morphisms are the strictly %weakly
order-preserving maps $\theta: [n] \to [N]$.
% Such maps are generated by \dfn{coface maps},
% \begin{equation} d_j : [n] \to [n+1];\quad i\mapsto \begin{cases} i & i<j \\ i+1 & i\geq j \end{cases} \quad\text{for}\quad j=0,1,\dots,n+1.\nn\end{equation}
% and \dfn{codegeneracy maps},
% \begin{equation} s_j : [n] \to [n-1];\quad i\mapsto \begin{cases} i & i<j \\ i-1 & i\geq j \end{cases} \quad\text{for}\quad
% j=1,\dots,n.\nn\end{equation}
% (These maps obey certain relations; see e.g. \cite[\S8]{Weibel}.)
% We write
% \begin{equation}
% \begin{tikzcd}
%  \dots \quad \left[2\right]
% \rar[<-,shift left=8pt]\rar[<-,shift left=4pt]\rar[<-]\rar[shift right = 4pt]\rar[shift right = 8pt]&
% \left[1\right] \rar[<-,shift left=4pt]\rar[<-] \rar[shift right=4pt]& \left[0\right]
% \end{tikzcd}.
% \nn\end{equation}
A \dfn{semisimplicial object} $Z$ in a category $\C$ is a functor $Z: \triangle^\op\to \C$.  A \dfn{semicosimplicial object} $A$ is a functor $A: \triangle \to \C$.

(\dfn{(Co)simplicial objects} are obtained by taking altering ``strictly order-preserving'' to ``weakly order-preserving'' in the definition above.% One then has (co)degeneracy as well as (co)face maps.
)

\subsection{Polynomial differential forms on the algebro-geometric simplex}\label{sec: omega}
There is a semisimplicial dg commutative algebra
$\Omega:\triangle^\op \to \dgCAlg \nn$
defined as follows.
For each $n\geq 0$, $\Omega([n])$ is the dg commutative algebra
\begin{equation} \Omega([n]) :=  \kk[t_0,\dots,t_n; \dd t_0, \dots \dd t_n]\big/\langle \sum_{i=0}^n t_i -1, \sum_{i=0}^n \dd t_i \rangle \nn\end{equation}
with $t_i$ in degree $0$ and $\dd t_i$ in degree $1$, for each $i$, and equipped with the usual de Rham differential. One should think of $\Omega([n])$ as the complex of polynomial differential forms on the algebro-geometric $n$-simplex.
For any map $\phi: [n] \to [N]$ of $\triangle$, the map
\(\Omega(\phi) : \Omega([N]) \to \Omega([n]) \nn\)
is the map of dg commutative algebras defined by $t_i \mapsto \sum_{j\in \phi^{-1}(i)} t_j$. (The set $\phi^{-1}(i)$ here is either empty or has one element, since for us $\phi$ is \emph{strictly} order-preserving.)

\subsection{The functor $\Th$}\label{sec: Thom-Sullivan functor}
Suppose we are given a functor $A : \triangle \to \CAlg(\dgVect_\kk)$; that is, suppose we are given a semicosimplicial object in dg commutative algebras. One can construct a dg commutative algebra, given by the graded vector space
\begin{align} \Th^\bul(A)  &:= \Tot\biggl\{ \mathbf a = (a_X)_{X\in \Ob(\triangle)} \in \prod_{X \in \Ob(\triangle) } A^\bul(X) \ox \Omega^\bul(X) :\nn\\&\qquad\qquad\qquad
\left(A(\phi) \ox \id \right) a_X
= \left(\id\ox\Omega(\phi)\right) a_Y \quad \text{in}\quad A(Y) \ox \Omega(X)   \nn\\&\qquad\qquad\qquad
\qquad\text{for all maps $\phi:X \to Y$ of $\triangle$} \biggr\}\label{def: Thom-Sullivan functor}
\end{align}
equipped with the differential
\begin{equation} \dd = \dd_A\ox \id + \id \ox \dd_{\text{de Rham}}\nn\end{equation}
and the graded commutative product given by $(a\ox \omega)(b\ox \tau) := (-1)^{\gr \omega \gr b} ab \ox \omega \wedge \tau$.
This defines the action on objects of a functor, the Thom-Sullivan or Thom-Whitney functor, from semicosimplicial dg commutative algebras to dg commutative algebras,
\begin{equation} \Th:[\triangle, \CAlg(\dgVect_\kk)] \to \dgCAlg. \nn\end{equation}
This functor computes (a model of) the homotopy limit of the functor $A$,
\begin{equation} \Th A \simeq \holim A .\nn\end{equation}
See \cite{HS} and e.g. \cite{Manetti}, \cite[Appendix A]{Kapranov}, \cite[Appendix A]{FHK}.

\subsection{Cech cohomology}\label{sec: cech cohomology}
Let $\I := \Hom_\Set\left(\{ (i,j)\}\ijk , \{1,\dots,n\}\right)$ denote the set of all ordered partitions of the set $\{(i,j)\}\ijk$ into $n$ disjoint subsets: if $X\in \I$ then
\begin{equation} \{(i,j)\}\ijk = X^1\sqcup \dots \sqcup X^n, \quad\text{where}\quad X^r := X^{-1}(r)\quad\text{for each}\quad r.\nn\end{equation}
For each $X\in \I$, let us set
\begin{equation} U_X := (\AA^{n})^k\setminus \bigcup_{(i,j) \in X_1} \{z^1_i=z^1_j\} \setminus \dots \setminus \bigcup_{(i,j)\in  X_n} \{z^n_i=z^n_j\}. \nn\end{equation}
Then $\mc U := (U_X)_{X\in \I}$ is an open cover of
\begin{equation} \Conf_k(\AA^n)  := (\AA^{n})^k \setminus \bigcup\ijk \{z^1_i=z^1_j,\dots,z^n_i=z^n_j\}\nn. \end{equation}
The sets $U_X$ of this cover and all their intersections are affine schemes: for every tuple $(X_1,\dots,X_K) \in \I^K$ we have
$U_{X_1} \cap \dots \cap U_{X_K} =  \Spec \Gamma(U_{X_1} \cap \dots \cap U_{X_K},\O)$
where
\begin{align} &\Gamma(U_{X_1} \cap \dots \cap U_{X_K},\O) \label{OUXXX} %\\&\quad
= \bigotimes_{r=1}^n \kk[z^r_1,\dots,z^r_k]\left[ \prod_{(i,j) \in \bigcup_{p=1}^K X^r_p} (z^r_i-z^r_j)^{-1}\right]  .
\end{align}
So $\mc U$ is a Leray cover: it is sufficiently refined that the Cech complex of $\O$ with respect to $\mc U$ computes the sheaf cohomology of the structure sheaf $\O$.
Recall that the usual Cech complex is the cochain complex which computes the homotopy limit in dg vector spaces of the diagram
\begin{equation} \Gamma(\Cech(\mc U),\O) = \left(
\begin{tikzcd}
\dots \quad %\prod_{\substack{(X,Y,Z) \in \I^3\\ X<Y<Z}} \O(U_X \cap U_Y \cap U_Z)
%\rar[<-,shift left=8pt]\rar[<-,shift left=4pt]\rar[<-]\rar[->,shift right = 4pt]\rar[->,shift right = 8pt]&
\rar[<-,shift left=4pt]\rar[<-,shift left=0pt]\rar[<-,shift right = 4pt]&
\prod_{\substack{(X,Y)\in \I^2\\X<Y}} \Gamma(U_X \cap U_Y,\O)
%\rar[<-shift left=4pt]\rar[<-] \rar[->,shift right=4pt]&
\rar[<-,shift left=2pt]\rar[<-,shift right=2pt]&
\prod_{X\in \I} \Gamma(U_X,\O)
\end{tikzcd}\right).
\nn
\end{equation}
Here we pick and fix any total ordering $<$ on the index set $\I$. But this is also a diagram in dg commutative algebras and we want to compute its homotopy limit there. The Thom-Sullivan functor allows us to do this, because this diagram is given by a semicosimplicial object, i.e. the indexing category is $\triangle$ and $\Gamma(\Cech(\mc U),\O)$ the functor
\begin{equation} \triangle \xrightarrow{\Gamma(\Cech(\mc U),\O)}  \CAlg(\dgVect_\kk) \label{gammacech}\end{equation}
whose action on objects is
\begin{equation} \Gamma(\Cech(\mc U),\O)([p]) := \prod_{\substack{(X_0,\dots,X_p)\in \I^{p+1}\\ X_0<\dots<X_p}} \Gamma\left(\bigcap_{i=0}^p U_{X_i},\O\right). \nn\end{equation}
Thus, we have that
\begin{equation} R\Gamma(\Conf_k(\AA^n),\O)\simeq \holim \Gamma(\Cech(\mc U),\O)  \simeq \Th\left( \Gamma(\Cech(\mc U),\O) \right).\label{Th computes RGamma}\end{equation}
The dg commutative algebra $\Th\left( \Gamma(\Cech(\mc U),\O)\right)$ has an intuitively natural geometric interpretation as a dg algebra of forms on the $N$-simplex, where
\[N:= |\I|-1,\]
as we shall now describe.

First, observe that the functor $\Gamma(\Cech(\mc U),\O)$ in \cref{gammacech} sends $[p] \mapsto \mathbf 0$ for all $p > N$ and it sends $[N] \mapsto \Gamma\left(\bigcap_{X\in \I} U_X,\O\right)= \kk[z^r_i,(z^r_j-z^r_\ell)^{-1}]^{1\leq r\leq n}_{1\leq i\leq k; 1\leq j<\ell<k}$. Indeed, we picked a totally ordered cover $\mc U$ consisting of $|\I|=N+1<\8$ many open sets. So there \emph{are} no strictly ordered tuples of $p>|\I|$ of opens from our cover.
(Here it matters that we are working with a totally ordered cover and therefore with \emph{semi}(co)simplicial objects.)

Now, we may regard the power set $2^\I$ as a poset and hence as a category. Our choice of total order $<$ on $\I$ defines an identification of $2^\I$ with the overcategory $\triangle/\topsimplex$.\footnote{Recall that the overcategory $\triangle/\topsimplex$ is the category in which an object is an object $[n]$ of $\triangle$ together with a morphism $[n]\to \topsimplex$ of $\triangle$, and a morphism from $[n] \to \topsimplex$ to $[m] \to \topsimplex$ is a morphism $[n]\xrightarrow\phi [m]$ of $\triangle$ such that the diagram
\[\begin{tikzcd}[ampersand replacement = \&]\left[n\right] \arrow{rr}{\phi}\drar \& \& \left[m\right] \dlar \\ \& \topsimplex \&  \end{tikzcd}\]
commutes. The identification with the poset $2^\I$ is by picking a bijection $[N] \isom \I$, which amounts to picking a total order on $\I$.} The latter is, intuitively speaking, the category of facets of the $N$-simplex.
We have the functor
\begin{equation} 2^\I \cong \triangle/\topsimplex \xrightarrow{F} \CAlg(\dgVect_\kk)  \nn\end{equation}
whose action on objects is
\begin{equation} F(\mc J) := \Gamma\left(\bigcap_{X \in \mc J} U_{X},\O\right),\qquad \mc J \subset \mc I.\nn\end{equation}
The forgetful functor $2^\I \to \triangle$ sends $\mc J\mapsto [|\mc J|-1]$ for each $\mc J\subset \I$.
(The functor $\Gamma(\Cech(\mc U),\O)$ is actually the left Kan extension of $F$ along this forgetful functor.)

On inspection of the definitions \cref{def: Thom-Sullivan functor} and \cref{gammacech}, one sees that $\Th^\bul(\Gamma(\Cech(\mc U), \O))$ is equivalently given by
\begin{align} &\Th^\bul(\Gamma(\Cech(\mc U), \O))\label{ontheNsimplex} \\  &\cong \Tot\biggl\{ \mathbf a = (a_\mc J)_{\mc J\in \Ob(\triangle/[N])} \in \prod_{\mc J \in \Ob(\triangle/[N]) } F^\bul(\mc J) \ox \Omega^\bul([|\mc J|-1]) :\nn\\&\qquad\qquad\qquad
\left(F(\phi) \ox \id \right) a_\mc J
= \left(\id\ox\Omega(\phi)\right) a_\mc K \quad \text{in}\quad F(\mc K) \ox \Omega([|\mc J|-1])   \nn\\&
\qquad\qquad\qquad\qquad\qquad\qquad\text{whenever there is a map $\phi:\mc J \to \mc K$ of $\triangle/[N]$}\biggr\}.
\nn\end{align}
In words, this just says that an element of $\Th^\bul(\Gamma(\Cech(\mc U), \O))$ is equivalent to the data of a $F(\mc J)$-valued polynomial differential form on the facet labelled by $\mc J$ of the $N$-simplex, for each such facet, subject to the compatibility conditions that whenever $\mc J\subset \mc K$ the pullback of the form on the facet labelled by $\mc K$ to the facet labelled by $\mc J$ must agree with the image, under the application of the sheaf restriction map $\Gamma(\bigcap_{X\in \mc J} U_{X},\O)  \to \Gamma(\bigcap_{X\in \mc K} U_{X},\O)$, of the form on the facet labelled by $\mc J$.

Moreover, the algebra maps \(\Gamma(\bigcap_{X\in \mc J} U_X,\O) \into \Gamma\left(\bigcap_{X\in \I} U_X,\O\right)\) are injective, for all $\mc J\subset \mc I$. That means once the form on the top facet, labelled $\I$, is fixed, then all the forms on lower dimensional facets are fixed too.
Thus $\Th^\bul\left( \Gamma(\Cech(\mc U),\O)\right)$ is equivalently the dg commutative algebra of $\Gamma\left(\bigcap_{X\in \I} U_X,\O\right)$-valued polynomial differential forms on the algebro-geometric $N:=(|\I|-1)$-simplex, obeying certain boundary conditions:
\begin{align} \Th^\bul\left( \Gamma(\Cech(\mc U),\O) \right) &\cong \biggl\{ \tau \in \Gamma\left(\bigcap_{X\in \I} U_X,\O\right) \ox \kk[t_X; \dd t_X]_{X\in \I}\big/\langle \sum_{X\in \I} t_X -1, \sum_{X\in \I} \dd t_X \rangle  \label{boundaryconditions}\\
&\quad :\text{ $\tau|_{\bigcap_{X\in \mc I \setminus \mc J}(t_{X} = 0)}$ has coefficients in $\Gamma\left(\bigcap_{X\in \mc J} U_{X},\O\right)$ for all $\mc J \subset \mc I$} \biggr\}.\nn
\end{align}
(Cf. \cref{ontheNsimplex}, where the tuple $\mathbf a$ consists of such a form $\tau$ together with all its pullbacks to facets of the $N$-simplex.)

This much is true for any Leray cover $\mc U = \{U_X\}_{X\in \I}$ of finite cardinality $N+1=|\I|$. But in our case one can do rather better because $\Gamma(U_{X_1} \cap \dots \cap U_{X_K},\O)$ depends on the tuple of set partitions $X_1,\dots,X_K$ only through
\begin{equation} \bigcup_{p=1}^K X^1_p, \quad\bigcup_{p=1}^K X^2_p,\quad \dots, \quad\bigcup_{p=1}^K X^n_p, \nn\end{equation} as one sees in \cref{OUXXX}. It follows that the boundary conditions on $\Th\left( \Gamma(\Cech(\mc U),\O) \right)$ are nonempty only for facets of the $(|\I|-1)$-simplex on which at least one of these unions fails to be the full set $\{(i,j)\}\ijk$. The facets of highest dimension on which this happens are those for which, for precisely one direction $r$, $1\leq r\leq n$, and for precisely one pair of the marked points, with indices say $(i,j)$, $1\leq i<j\leq k$, the union $\bigcup_{p=1}^K X^r_p$ fails to contain $(i,j)$. (Such a facet corresponds to an intersection of opens in which it may happen that $z^r_i=z^r_j$, for this specific pair of points $(i,j)$ and this specific direction $r$ only.) Moreover, one sees that all boundary conditions on facets of lower dimensions follow from taking intersections of these. Reasoning in this way, one checks that
\begin{subequations}\label{ThSimplex}
\begin{equation} \Th\left( \Gamma(\Cech(\mc U),\O) \right)\cong \Th^k_{\AA^n} \end{equation}
where
\begin{align} \Th^k_{\AA^n}&:= \biggl\{ \tau \in \kk[z^r_i,(z^r_j-z^r_\ell)^{-1}]^{1\leq r\leq n}_{1\leq i\leq k; 1\leq j<\ell<k}  \nn\\&\qquad\qquad\ox \kk[t_X; \dd t_X]_{X\in \I}\big/\langle \sum_{X\in \I} t_X -1, \sum_{X\in \I} \dd t_X \rangle  \\
&\qquad\quad : \text{$\tau|_{\bigcap_{X\in \I: X^r \ni (i,j)}(t_{X} = 0)}$ is regular in $(z^r_i-z^r_j)$}\nn\\ & \quad\qquad\qquad\text{for all $1\leq i<j\leq k$ and all $1\leq r\leq n$}  \biggr\}. \nn
\end{align}
\end{subequations}
\begin{rem}
    The construction $\Th\left( \Gamma(\Cech(\mc U),\mathcal{E}) \right)$ works for any quasi-coherent sheaf $\mathcal{E}$ on $\Conf_k(\AA^n)$. However, we will not have the above simplification if $\mathcal{E}$ is not locally free.
\end{rem}
\subsection{Proof of \cref{thm: polysimplex model NN}}\label{sec: proof of polysimplex model NN}
In view of \cref{ThSimplex}, \cref{thm: polysimplex model NN} now follows from \cref{Th computes RGamma} together with the following lemma.

\begin{prop}\label{prop: retract of simplex to polysimplex}
There is a deformation retract of dg vector spaces, in which both $\id\ox\pi^*$ and $\id\ox\iota^*$ are maps of dg commutative algebras,
\begin{equation}\begin{tikzcd}
\TS_{\AA^n}^k  \rar[shift left]{\id\ox\pi^*}  & \lar[shift left]{\id\ox\iota^*}
\ar[loop right,""] \Th_{\AA^n}^k \end{tikzcd}.\nn\end{equation}
%where $\TS_{\AA^n}^k$ is the dg commutative algebra defined in \cref{sec: higher configuration space and movable punctures in two complex dimensions}.
Recall the definition of $\TS_{\AA^n}^k$ from \cref{sec: higher configuration space and movable punctures in two complex dimensions},
\begin{align} \TS_{\AA^n}^{k,\bul}  := &\bigl\{ \tau \in
 \kk[z^r_i,(z^r_j-z^r_\ell)^{-1}]^{1\leq r\leq n}_{1\leq i\leq k; 1\leq j<\ell<k} \ox \Omega(\triangle_{n-1}^{\times \binom k 2}) \nn\\
& \quad :          \text{$\tau|_{u^r_{ij}=0}$ is regular in $(z^r_i-z^r_j)$,  for $1\leq r\leq n$ and $1\leq i<j\leq k$} \bigr\}. \nn
\end{align}

\end{prop}
\begin{proof} Recall that $\I = \{1,\dots,n\}^S$ is the set of ordered partitions into $n$ parts of a set $S$. Let $m=|S|$, so $|\I|= n^m$.

(In our case $S$ is actually the set $\{(i,j): 1\leq i<j\leq k\}$ and $m= \binom k 2$.)

Let $\pi^*$ and $\iota^*$ be the maps of dg commutative algebras
\begin{equation}\begin{tikzcd}
\kk[u^r_\alpha, \dd u^r_\alpha]^{1\leq r\leq n}_{\alpha \in S} \rar[shift left]{\pi^*}  &\lar[shift left]{\iota^*} \kk[t_X, \dd t_X]_{X\in \I}
\end{tikzcd}\nn
\end{equation}
defined by setting
\begin{equation} \pi^*(u^r_\alpha) := \sum_{\substack{X \in \I  \\ X^r \ni \alpha }} t_X\qquad\text{and}\qquad \iota^*(t_X) := \prod_{r=1}^n \prod_{\alpha \in X^r} u^r_\alpha .\nn\end{equation}
Let us also write \[ t_X^*:=\pi^*(\iota^*(t_X)) \quad\text{and}\quad t_X^\lambda:=\lambda t_X + (1-\lambda) t_X^*,\]
and define a $\kk$-linear map $h:\kk[t_X, \dd t_X]_{X\in \I} \to \kk[t_X, \dd t_X]_{X\in \I}$ by setting
\begin{align}
    &h(t_{X_1} \dots t_{X_p} \dd t_{X'_{1}} \dots \dd t_{X'_{q}})\nn\\ &:= \int_{\lambda=0}^1 \tl_{X_1} \dots \tl_{X_p}\sum_{s=1}^q(-1)^{s-1}(t_{X'_{s}}-t_{X'_{s}}^*)\dd\tl_{X'_{1}} \dots \widehat{\dd\tl_{X'_{s}}} \dots \dd\tl_{X'_{q}} \dd\lambda,\nn
\end{align}
for any collection of $p\geq 0$ elements $X_1,\dots,X_p \in \mc I$ and any collection of $q\geq 0$ pairwise distinct elements $X'_1,\dots,X'_q \in \mc I$.

Let $\mathscr I$ and $\mathscr K$ denote  the dg ideals which define $\Omega(\triangle_{n^m-1})$ and  $\Omega(\triangle_{n-1}^{\times m})$ respectively:
\begin{align} \mathscr I& :=\left< \sum_{X\in \I} t_X-1 ,\sum_{X\in\I} \dd t_X\right>\subset \kk[t_X,\dd t_X]_{X\in \I}\nn\\
\mathscr K&:=\left< \sum_{r=1}^n u^r_\alpha-1,\sum_{r=1}^n \dd u^r_\alpha\right>_{\alpha\in S}\subset \kk[u^r_\alpha,\dd u^r_\alpha]^{1\leq r\leq n}_{\alpha \in S}\nn
\end{align}
Let us introduce also the dg ideals, for each $1\leq r\leq n$ and $\alpha \in S$,
\begin{align}\mathscr I^r_\alpha &:= \bigl< t_X, \dd t_X \bigr>_{X\in \I: X^r\ni \alpha} \subset \kk[t_X,\dd t_X]_{X\in \I}\nn\\
\mathscr K^r_\alpha &:= \left< u^r_\alpha, \dd u^r_\alpha\right> \subset \kk[u^s_\beta,\dd u^s_\beta]^{1\leq s\leq n}_{\beta\in S}.\nn
\end{align}
To prove \cref{prop: retract of simplex to polysimplex} it is enough to establish the claims of \cref{lem: retract} below.
Indeed, parts (i) and (iii) imply that the maps $\pi^*$, $\iota^*$ and $h$ induce well-defined maps
\begin{equation}\begin{tikzcd} \Omega^\bul(\triangle_{n-1}^{\times m})
\rar[shift left]{\pi^*}  & \lar[shift left]{\iota^*}
 \ar[loop right,"h"] \Omega^\bul(\triangle_{n^m-1})   \end{tikzcd}\label{hrt}\end{equation}
(whose names we keep the same, by a slight abuse).
Parts (ii) and (iv) ensure that this diagram is a deformation retract.

It remains to show that the maps respect the defining boundary conditions of $\TS_{\AA^n}^k$ and  $\Th_{\AA^n}^k$. This is the content of parts (v) and (vi). Indeed, the boundary conditions of $\Th_{\AA^n}^k$ can be expressed as the statement that $\tau \in \kk[z^r_i,(z^r_j-z^r_\ell)^{-1}]^{1\leq r\leq n}_{1\leq i\leq k; 1\leq j<\ell<k} \ox \kk[t_X; \dd t_X]_{X\in \I}\big/\mathscr I$ belongs to $\Th_{\AA^n}^k$ if and only if, for each $1\leq r\leq n$ and each $(i,j) \in S$, the $\D$-module residue $\res_{z^r_{i}-z^r_{j}} \tau$ has coefficients in (the image in $\Omega(\Delta_{n^m-1})$ of) the ideal $\mathscr I^r_{(ij)}$; and similarly for $\TS_{\AA^n}^k$.
\end{proof}
\begin{lem}\label{lem: retract} The maps $\pi^*$, $\iota^*$ and $h$ are such that:
\begin{enumerate}[(i)]
\item $\pi^*(\mathscr K) \subset \mathscr I$ and $\iota^*(\mathscr I) \subset \mathscr K$
\item $\iota^*(\pi^*(u^r_\alpha)) \equiv u^r_\alpha \mod \mathscr K$ for each $1\leq r\leq n$ and each $\alpha\in S$
\item $h(\mathscr I) \subset \mathscr I$
\item $\pi^*\circ\iota^* - \id = [h,\dd]$ as an equality of $\kk$-linear maps $\kk[t_X,\dd t_X]_{X\in \I}\to\kk[t_X,\dd t_X]_{X\in \I}$
\item $\iota^*(\mathscr I^r_\alpha) \subset \mathscr K^r_\alpha$ and $
\pi^*(\mathscr K^r_\alpha) \subset \mathscr I^r_\alpha$ for each $1\leq r\leq n$ and $\alpha\in S$
\item  $h(\mathscr I^r_\alpha) \subset \mathscr I^r_\alpha$ for each $1\leq r\leq n$ and $\alpha\in S$.
\end{enumerate}
\end{lem}
\begin{proof}
We observe that
\(\iota^*(\sum\limits_{X \in \I } t_X) = \sum\limits_{X \in \I} \prod\limits_{r=1}^n \prod\limits_{\alpha \in X^r}u^r_\alpha = \prod\limits_{\alpha \in S} \sum\limits_{r=1}^n u^r_\alpha
    \label{IdSimplex}
\)
and, for each $\alpha \in S$,
\(\pi^*(\sum\limits_{r=1}^nu^r_\alpha) = \sum\limits_{r=1}^n \sum\limits_{\substack{X \in \I  \\ X^r \ni \alpha }} t_X = \sum\limits_{X \in \I} t_X.\nn\) This establishes part (i).

Now, in fact
\(\iota^*(\sum\limits_{X \in \{1,\dots,n\}^{S'} } t_X) = \sum\limits_{X \in \{1,\dots,n\}^{S'}} \prod\limits_{r=1}^n \prod\limits_{\alpha \in X^r}u^r_\alpha = \prod\limits_{\alpha \in S'} \sum\limits_{r=1}^n u^r_\alpha
\) holds
for any subset $S'\subset S$ of the index set $S$. In particular this holds for the subset $S\setminus \{\alpha\}$ for any $\alpha \in S$. It follows that for every $\alpha \in S$, we have
$$
\iota^*(\pi^*(u^r_\alpha)) =\sum_{\substack{X \in \I  \\ X^r \ni \alpha }}\prod_{p=1}^n \prod_{\beta \in X^p}u^p_\beta%\nn\\&
=u^r_\alpha \sum_{\substack{X \in \{1,\dots,n\}^{S\setminus\{\alpha\}} }}\prod_{p=1}^n \prod_{\beta \in X^p}u^p_\beta
=u^r_\alpha \prod_{\beta \in S\setminus\{\alpha\}} \sum_{r=1}^n u_\beta^r
.
$$ This establishes part (ii).

We turn to the statements (iii) and (iv) about the homotopy, whose definition is similar in spirit to homotopies used in \cite{AlfonsiYoung} (morally the definition is ``$h= \int_{\iota(\pi(*))}^*$'').

Consider part (iii). We must show that $h(\mathscr I) \subset \mathscr I$. As we have already checked,   $\sum\limits_{X \in \I} t^*_X-1$ and $\sum\limits_{X \in \I} \dd t^*_X$ belong to $\mathscr I$. Therefore $\sum\limits_{X \in \I} t^\lambda_X-1$ and $\sum\limits_{X \in \I} \dd t^\lambda_X$
%both belong to the ideal generated by $\sum_{X\in \I} t_X-1$ and $\sum_{X\in \I} \dd t_X$ in $\kk[t_X, \dd t_X,]_{X\in \I}[\lambda,\dd\lambda]$. That is, they
are polynomials in $\lambda$ and $\dd\lambda$ whose coefficients lie in $\mathscr I$. In this way, we observe that in the expressions
\begin{align}
    &h\biggl(\bigl(\sum_{X\in \I} t_X-1\bigr)t_{X_1} \dots t_{X_p} \dd t_{X'_{1}} \dots \dd t_{X'_{q}}\biggr)\nn\\ &= \int_{\lambda=0}^1 \bigl(\sum_{X\in \I} t^\lambda_X-1\bigr)\tl_{X_1} \dots \tl_{X_p}\sum_{s=1}^q(-1)^{s-1}(t_{X'_{s}}-t_{X'_{s}}^*)\dd\tl_{X'_{1}} \dots \widehat{\dd\tl_{X'_{s}}} \dots \dd\tl_{X'_{q}} \dd\lambda\nn,
\end{align}
and
\begin{align}
    &h\biggl(t_{X_1} \dots t_{X_p} \dd \bigl(\sum_{X\in \I} t_X\bigr) \dd t_{X'_{1}} \dots \dd t_{X'_{q}}\biggr)\nn\\
&=  \int_{\lambda=0}^1 \tl_{X_1} \dots \tl_{X_p}
\bigl(\sum_{X\in \I} (t_{X}-t_{X}^*)\bigr) \dd\tl_{X'_{1}} \dots \widehat{\dd\tl_{X'_{s}}} \dots \dd\tl_{X'_{q}} \dd\lambda\nn\\&
\quad
+\int_{\lambda=0}^1 \tl_{X_1} \dots\tl_{X_p}\bigl(\sum_{X\in \I} \dd t_X^\lambda \bigr)\sum_{s=1}^q(-1)^{s}(t_{X'_{s}}-t_{X'_{s}}^*)\dd\tl_{X'_{1}} \dots \widehat{\dd\tl_{X'_{s}}} \dots \dd\tl_{X'_{q}} \dd\lambda\nn
\end{align}
the integrands are polynomials in $\lambda$ with coefficients in $\mathscr I$.
%(Note that $\sum_{X\in \I} (t_{X}-t_{X}^*)\bigr) = \sum_{X\in \I} ((t_{X}-1)-(t_{X}^*-1))\bigr)$.)
The integral $\int_{\lambda=0}^1P(\lambda) \dd \lambda$ of any such polynomial $P(\lambda) \in \mathscr I[\lambda]$ belongs to $\mathscr I$. This completes the proof of part (iii).

We turn to part (iv). %Let us consider $h\circ\dd+\dd\circ h$, as a $\kk$-linear map $\kk[t_X, \dd t_X]_{X\in \I} \to \kk[t_X, \dd t_X]_{X\in \I}$.
We have \[\dd(t_{X_1} \dots t_{X_p} \dd t_{X'_{1}} \dots \dd t_{X'_{q}}) = \sum_{r=1}^p  t_{X_1} \dots \widehat{t_{X_r}} \dots t_{X_s} \dd t_{X_r} \dd t_{X'_{1}} \dots \dd t_{X'_{q}}\]
and hence
\begin{align}
    &h(\dd(t_{X_1} \dots t_{X_p} \dd t_{X'_{1}} \dots \dd t_{X'_{q}}))\nn\\
& = \int_0^1 \sum_{r=1}^p \tl_{X_1} \dots \wh{\tl_{X_r}} \dots \tl_{X_p} \nn \\
   &\quad\left((t_{X_r}-t_{X_r}^*) \dd\tl_{X'_{1}} \dots  \dd\tl_{X'_{q}} - \sum_{s=1}^q (-1)^{s-1}(t_{X'_{s}}-t_{X'_{s}}^*) \dd\tl_{X_r} \dd\tl_{X'_{1}} \dots \widehat{\dd\tl_{X'_{s}}} \dots \dd\tl_{X'_{q}} \right)\dd\lambda \nn
\end{align}
and similarly
\begin{align}
    &\dd (h(t_{X_1} \dots t_{X_p} \dd t_{X'_{1}} \dots \dd t_{X'_{q}}))\nn\\
&= \dd \int_{\lambda=0}^1 \tl_{X_1} \dots \tl_{X_p}\sum_{s=1}^q(-1)^{s-1}(t_{X'_{s}}-t_{X'_{s}}^*)\dd\tl_{X'_{1}} \dots \widehat{\dd\tl_{X'_{s}}} \dots \dd\tl_{X'_{q}} \dd\lambda,\nn\\
&= \int_0^1 \sum_{r=1}^p  \tl_{X_1} \dots \wh{\tl_{X_r}} \dots \tl_{X_p} \sum_{s=1}^q(-1)^{s-1} (t_{X'_{s}}-t_{X'_{s}}^*) \dd \tl_{X_r} \dd \tl_{X'_{1}} \dots \widehat{\dd \tl_{X'_{s}}} \dots \dd \tl_{X'_{q}} \dd \lambda \nn \\
    &\qquad +\int_0^1 \tl_{X_1} \dots \tl_{X_p} \sum_{s=1}^q(-1)^{s-1} (\dd t_{X'_{s}}-\dd t_{X'_{s}}^*) \dd \tl_{X'_{1}} \dots \widehat{\dd \tl_{X'_{s}}} \dots \dd \tl_{X'_{q}} \dd \lambda. \nn
\end{align}
In this way, one finds that
\begin{align}
    &h(\dd (t_{X_1} \dots t_{X_p} \dd t_{X'_{1}} \dots \dd t_{X'_{q}})) + \dd h(t_{X_1} \dots t_{X_p} \dd t_{X'_{1}} \dots \dd t_{X'_{q}}) \nn \\
    &= \int_0^1 \sum_{r=1}^p  \tl_{X_1} \dots \wh{\tl_{X_r}} \dots \tl_{X_p}(t_{X_r}-t_{X_r}^*) \dd \tl_{X'_{1}} \dots  \dd \tl_{X'_{q}} \dd \lambda \nn \\
    &\qquad +\int_0^1 \tl_{X_1} \dots \tl_{X_p} \sum_{s=1}^q \dd \tl_{X'_{1}} \dots (\dd t_{X'_{s}}-\dd t_{X'_{s}}^*) \dots \dd \tl_{X'_{q}} \dd \lambda. \nn \\
    &= \int_0^1 \frac{d}{d\lambda}\left( \tl_{X_1} \dots \tl_{X_p}
   \dd \tl_{X'_{1}} \dots  \dd \tl_{X'_{q}}  \right)\dd \lambda \nn \\
   &\qquad = t_{X_1} \dots t_{X_p} \dd t_{X'_{1}} \dots \dd t_{X'_{q}} - t^*_{X_1} \dots t^*_{X_p} \dd t^*_{X'_{1}} \dots \dd t^*_{X'_{q}} \nn.
\end{align}
This establishes that
\begin{equation} h\circ\dd + \dd\circ h = \id  - \pi^*\circ \iota^* \nn\end{equation}
as an equality of $\kk$-linear endomorphisms of $\kk[t_X,\dd t_X]_{X\in \I}$, which is part (iv).

Part (v) is clear on inspection.

The argument for part (vi) is parallel to that for part (iii), as follows. Consider any generator $t_Y$ of $\mathscr I^r_\alpha$, i.e. pick any $Y\in \I$ be such that $Y^r\ni \alpha$. It follows from part (v) that $t_Y^* \in \mathscr I^r_\alpha$ also. Therefore $t_Y-t_Y^*\in \mathscr I^r_\alpha$ and hence $t^\lambda_Y \in \mathscr I^r_\alpha[\lambda,\dd \lambda]$. Then one argues as for part (iii) that, for any $p\in \kk[t_X,\dd t_X]_{X\in \I}$, $h(t_Y p)$ and $h(\dd t_Y p)$ are both integrals of polynomials in $\lambda$ with coefficients in $\mathscr I^r_\alpha$ and therefore belong to $\mathscr I^r_\alpha$. Thus $h(\mathscr I^r_\alpha)\subset \mathscr I^r_\alpha$, which is part (vi).

The proof of \cref{lem: retract} is complete.
\end{proof}

This completes the proof of \cref{prop: retract of simplex to polysimplex} and hence of \cref{thm: polysimplex model NN}.

\subsection{Proof of \cref{thm: F model}}\label{sec: proof of F model}
For any complex $\G$ of quasi-coherent sheaves on $(\AA^{n})^k$, one defines the ordered Cech sheaf (complex) $\mathfrak{C}_{ord}^{p}(\mathcal{U},j^*\G)$ on $\Conf_k(\AA^n)$ as follows
\[
\mathfrak{C}_{ord}^{\bullet}(\mathcal{U},j^*\G):=\prod_{\substack{(X_0,\dots,X_p)\in \I^{p+1}\\ X_0<\dots<X_p}} (j_{X_0\cdots X_p})_*\left(j^*\G\right)|_{\bigcap_{i=0}^p U_{X_i}}=\prod_{\substack{(X_0,\dots,X_p)\in \I^{p+1}\\ X_0<\dots<X_p}} (j_{X_0\cdots X_p})_*\G|_{\bigcap_{i=0}^p U_{X_i}},
\]
here $\quad j_{X_0\cdots X_p}:\bigcap_{i=0}^p U_{X_i}\hookrightarrow \Conf_k(\AA^n).
$ Taking global sections, we get the usual ordered Cech complex
\[
\Gamma\left((\AA^{n})^k,j_*\mathfrak{C}_{ord}^{\bullet}(\mathcal{U},j^*\G)\right)=\Gamma\left(\Conf_k(\AA^n) ,\mathfrak{C}_{ord}^{\bullet}(\mathcal{U},j^*\G)\right)=\Cech_{ord}^{\bullet}(\mc U,j^*\G),
\]
\[
\Cech_{ord}^{p}(\mc U,j^*\G)=\prod_{\substack{(X_0,\dots,X_p)\in \I^{p+1}\\ X_0<\dots<X_p}}\G({\bigcap_{i=0}^p U_{X_i}}).
\]
Since $(\AA^{n})^k$ is affine, we have $j_*\mathfrak{C}_{ord}(\mathcal{U},j^*\G)=\underline{\Cech_{ord}^{\bullet}(\mc U,j^*\G)}$. The integration along simplices gives rise to a quasi-isomorphism (see \cite[Theorem 7.4.5]{Manetti} or \cite[(4.4.1)]{hinich2006homotopy})
\[
\int_{\triangle}:\Th\left( \Gamma(\Cech(\mc U),j^*\G) \right)\xrightarrow{\sim}\Cech_{ord}^{\bullet}(\mc U,j^*\G) .
\]
Since the cover $\mathcal{U}$ is finite, the double complex $\Cech_{ord}^{\bullet}(\mc U,j^*\G)$ (thus $\Th\left( \Gamma(\Cech(\mc U),j^*\G) \right)$) computes $R\Gamma\left(\Conf_k(\AA^n),\G\right)$ without any boundedness condition on $\G$ (see  \cite[\href{https://stacks.math.columbia.edu/tag/08BZ}{Tag 08BZ}]{stacks-project}). And if $\G\xrightarrow{\sim}\M$ is a quasi-isomorphism of complexes of quasi-coherent sheaves, we have the induced quasi-isomorphism (see \cite[Lemma 7.4.3]{Manetti})
\[
\Th\left( \Gamma(\Cech(\mc U),j^*\G)\right)\xrightarrow{\sim}\Th\left( \Gamma(\Cech(\mc U),j^*\M)\right).
\]
Now suppose that $\F$ is a bounded complex of $\D_{(\AA^{n})^k}$-modules with quasi-coherent cohomology. By a theorem by Bernstein \cite[1.5.7]{HTT_Dmodules2008} we can find a bounded complex $\G$ of quasi-coherent $\D_{(\AA^{n})^k}$-modules quasi-isomorphic to $\F$. Then we can find a bounded above complex of locally free $\D_{(\AA^{n})^k}$-modules $\M$ which resolves $\G$, cf. \cite[1.4.20]{HTT_Dmodules2008}. The claim that $\Tot\left(\underline{\TS_{\AA^n}^{k}} \ox_{\O_{(\AA^{n})^k}} \F\right)\sim \int_jj^*\F$ follows from the following diagram.
\[
\begin{tikzcd}
\Tot\left(\underline{\TS_{\AA^n}^{k}} \ox_{\O_{(\AA^{n})^k}} \F\right) \rar["\sim"] & \Tot\left(\underline{\TS_{\AA^n}^{k}} \ox_{\O_{(\AA^{n})^k}} \G\right)\dar["\sim"]  & \lar["\sim"] \Tot(\underline{\TS_{\AA^n}^k} \ox_{\O_{(\AA^{n})^k}} \M) \dar["\sim"] \\
 %\rar
& \Tot(\underline{\Th_{\AA^n}^k} \ox_{\O_{(\AA^{n})^k}} \G) \dar["i"] & \lar["\sim"]\Tot(\underline{\Th_{\AA^n}^k} \ox_{\O_{(\AA^{n})^k}} \M)\dar["\cong"] \\
\int_jj^{*}\F\sim \int_jj^{*}\G \rar["\sim"]& \underline{\Th\left( \Gamma(\Cech(\mc U),j^*\G) \right) } & \lar["\sim"] \underline{\Th\left( \Gamma(\Cech(\mc U),j^*\M) \right) }
\end{tikzcd}
\]
In this diagram,
\begin{enumerate}[-]
 \item Every arrow respects $\D_{(\AA^{n})^k}$-module structures.
 \item The arrows in the first row are quasi-isomorphisms. This is because $\TS_{\AA^n}^k$ is a bounded complex of flat $\O_{(\AA^{n})^k}$-modules. (We prove flatness in \cref{lem: FlatE}, cf. \cref{thm: TSIE gens and rels}). Therefore it is K-flat. (See \cite[\href{https://stacks.math.columbia.edu/tag/06Y7}{\S06Y7}]{stacks-project}.)
 \item The first two vertical arrows are quasi-isomorphisms since $\underline{\TS_{\AA^n}^{k}} $ and $\Th_{\AA^n}^{k}$ are chain homotopy equivalent by \cref{prop: retract of simplex to polysimplex}. This implies that the horizontal arrow in the second row is a quasi-isomorphism.
 \item Since $\M$ is locally free, we have the identification $\underline{\Th\left( \Gamma(\Cech(\mc U),j^*\M) \right) }\cong \Tot(\underline{\Th_{\AA^n}^k} \ox_{\O_{(\AA^{n})^k}} \M) $ by the same argument as before.
 \item The map $i$ comes from the map $\Th\left( \Gamma(\Cech(\mc U),\O) \right)\otimes_{\Gamma((\AA^n)^k,O_{(\AA^n)^k})}\Gamma((\AA^n)^k,\G)\rightarrow \Th\left( \Gamma(\Cech(\mc U),j^*\G) \right)$ and we use the fact that $\G$ is quasi-coherent.
    \item The first two columns are bounded complexes of $\D_{(\AA^{n})^k}$-modules.
    \item $\int_jj^{*}\G$ is represented by $j_*\mathfrak{C}_{ord}(\mathcal{U},j^*\G)=\underline{\Cech_{ord}^{\bullet}(\mc U,j^*\G)}$. This can be seen using injective resolutions and the fact that $\mathcal{U}$ is affine and finite.
\end{enumerate}

\begin{rem}
One could also use the projection formula \cite[Corollary 1.7.5]{HTT_Dmodules2008} directly by the flatness of $\TS_{\AA^n}^k$.
\end{rem}
\begin{rem}
The statement seems to be true for unbounded complexes. However, we did not find proof of the unbounded Bernstein theorem in the literature (it is mentioned in \cite[pp.207]{neeman1996grothendieck}). For bounded below complexes of quasi-coherent $D_{(\AA^{n})^k}$-modules, the statement is true by using the truncation of the complexes.
\end{rem}
\subsection{Proof of \cref{thm: TSIE gens and rels}}{\label{sec: proof of thm TSIE gens and rels}}
In this subsection we prove \cref{thm: TSIE gens and rels}, i.e. that $\TS_{\AA^n}^{I,E} \cong \mathbf B_{\AA^n}^{I,E}$. We first recall the notations:
\begin{enumerate}[-]
%\item  $\B_{\AA^1}^{I,E} := \kk[z_i,\frac 1{z_e}]_{i\in I; e\in E}$,
%\item $\B_{\AA^n}^{I,E} := \left( \B_{\AA^1}^{I,E} \right)^{\ox n} = \kk[z^r_i,\frac 1{z^r_e}]^{1\leq r\leq n}_{i\in I; e \in E}$
\item $\Omega_{n}^{I,E} := \kk[u^r_e; \dd u^r_e]^{1\leq r\leq n}_{e \in E} \big/\langle \sum\limits_{r=1}^{n} u^r_e -1, \sum\limits_{r=1}^{n} \dd u^r_e \rangle_{e\in E}$
\item $\mathbf \Omega_{\AA^n}^{I,E} := \kk[z^r_i,\frac 1{z^r_e}]^{1\leq r\leq n}_{i\in I; e \in E} \ox_\kk \Omega_{n}^{I,E}$
\end{enumerate}
Given also an undirected edge $e\in E$ and a direction $s\in \{1,\dots,n\}$, let
\begin{enumerate}[-]
\item $\Omega^{I,E}_{n,(e,\del_s)} := \Omega_{n}^{I,E}\big/ \langle u_e^s = 0, \dd u_e^s=0\rangle$
\item $\mathbf \Omega^{I,E}_{\AA^n,(e,\del_s)} %= \B^{I,E}_{\AA^n} \ox \Omega^{I,E}_{\AA^n,(e,\del_s)}
:= \mathbf \Omega^{I,E}_{n}\big/  \langle u_e^s = 0, \dd u_e^s=0\rangle$
\item $\mathbf \Omega^{I,E}_{\AA^n,(e,\del_s)^{\mathbf{reg}}} := \kk[z^r_i,\frac 1{z^t_{f}}]^{(r,i)\in \{1,\dots,n\}\times I}_{(t,f)\in (\{1,\dots,n\}\times E) \setminus \{(s,e)\}} \ox \Omega^{I,E}_{n,(e,\del_s)}$
\item $\mathbf \Upsilon^{I,E}_{\AA^n,u^s_e=0} := \mathbf \Omega^{I,E}_{\AA^n,(e,\del_s)}\big/  \mathbf \Omega^{I,E}_{\AA^n,(e,\del_s)^{\mathbf{reg}}}$
\end{enumerate}
We have the canonical quotient maps
$\Res_e^{I,E}|_{u^s_e=0} : \mathbf \Omega^{I,E}_{\AA^n} \to \mathbf\Upsilon^{I,E}_{\AA^n,u^s_e=0}.$

The general strategy is to start with simple cases and add vertices and edges.
We begin with the case of two vertices and one edge.
\begin{prop} Let $e = (v',v'')$.
$\TS_{\AA^n}^{\{v',v''\},\{e\}}$ is isomorphic to
$
\mathbf{B}^{\{v',v''\},\{e\}}_{\AA^n}%=\frac{\kk[z_{v'}^s,z^s_{v''},\frac{u^s_e}{(z^s_e)^{\mathbf{m}_{e,s}}},\frac{\dd u^s_e}{(z^s_e)^{\mathbf{m}_{e,s}}}]^{1\leq s\leq n}_{\mathbf{m}_{e,s}\geq 0}}{\langle \sum^n\limits_{s=1}u^s_e=1, \sum^n\limits_{s=1}\dd u^s_e=0,u^s_{i_1i_2}=-u^s_{i_2i_1}\rangle}
.
$
% Similarly, $\TS^{\{v',v''\},\{e\}}_{\AA^n,\basepoint}$ is isomorphic to  $\mathbf{B}^{\{v',v''\},\{e\}}_{\AA^n,\basepoint}%=\frac{\kk[z^s_{v'v''},\frac{u^s_e}{(z^s_e)^{\mathbf{m}_{e,s}}},\frac{\dd u^s_e}{(z^s_e)^{\mathbf{m}_{e,s}}}]^{1\leq s\leq n}_{\mathbf{m}_{e,s}\geq 0}}{\langle \sum^n\limits_{s=1}u^s_e=1, \sum^n\limits_{s=1}\dd u^s_e=0,u^s_{i_1i_2}=-u^s_{i_2i_1}\rangle}
% .
% $
\end{prop}
\begin{proof}
It is clear that
\[ \mathbf{B}^{\{v',v''\},\{e\}}_{\AA^n}\hookrightarrow \TS_{\AA^n}^{\{v',v''\},\{e\}} \]
 is a subalgebra. It remains to establish the containment in the other direction.
As a $\kk$-vector space, $\TS_{\AA^n}^{\{v',v''\},\{e\}}$ is spanned by elements of the form
\[
\sum_{l_s \in \ZZ_{\geq 0}} \sum_{r_s\in \ZZ} \left(\prod^n_{s=1}(z^s_{v''})^{l_s} (z^s_{e})^{-r_s}\right) \alpha_{(l_1,\dots,l_n),(r_1,\dots,r_n)} \in\TS_{\AA^n}^{\{v',v''\},\{e\}}
\]
with each $\alpha_{(l_1,\dots,l_n),(r_1,\dots,r_n)} \in \Omega_{\AA^n}^{\{v',v''\},\{e\}}$.
This element must belong to $\mathbf{\Omega}^{\{v',v''\},\{e\}}_{\AA^n,(e,\del_s)^{\mathbf{reg}}}$ for each $s$, i.e. it must obey the defining boundary conditions, namely that
\[\alpha_{(l_1,\dots,l_n),(r_1,\dots,r_n)}|_{u^s_e=0}=0 \quad \text{if}\ r_s>0.\]
Then the proposition follows from the following claim: For any subset $S\subset \{1,\dots,n\}$, if a polynomial differential form $\alpha\in \Omega^{\bul}(\triangle_{n-1})$ on the $(n-1)$-simplex satisfies the boundary condition
\[\alpha|_{u^{s}_e=0} \quad\text{for all } s\in S,\]
then it can be written as $\kk$-linear combination of monomials of the form
\begin{equation} \prod_{s\in S} (\varepsilon u_e^{s})^{p_s}\quad\text{with}\quad\varepsilon u^{s}_e \in \{u^{s}_e, \dd u^{s}_e\} \quad\text{and}\quad p_s>0\quad\text{for each $s\in S$}.\label{form}\end{equation}
To see this, note that for any $t\in \{1,\dots,n\}$, we may use the relations
$\sum\limits_{s=1}^{n}u^{s}_e=1, \sum\limits_{s=1}^{n}\dd u^{s}_e=0$
to write $\alpha$ uniquely as a polynomial $\alpha_t$ in the variables $\{u^s_e,\dd u^s_e\}_{s\neq t}$. If $S\neq \{1,\dots,n\}$ we may pick any $t\notin S$, write $\alpha = \alpha_t$, and the claim follows immediately from the boundary condition.
When $S=\{1,\dots,n\}$, then we argue in two steps. First, for every $t\in S$, $\alpha$ satisfies in particular the boundary conditions above for $S\setminus\{t\}\subset S$. Therefore, as we just saw, $\alpha_t$ is of the form \cref{form} for $S\setminus\{t\}$. Then we need only note that
\[
  \alpha =\sum_{s=1}^{n}u^s_e\alpha=\sum_{s=1}^{n}u^s_e\alpha_s
\]
is of the required form. This establishes the claim and hence the proposition.
\end{proof}

Now we consider the case of many vertices but still only one edge.

\begin{cor}\label{cor: oneedge} Let $e\in I\two$.
 $ \TS_{\AA^n}^{I,\{e\}}$ is isomorphic to
  \[
\mathbf{B}^{I,\{e\}}_{\AA^n}=\O_{(\AA^n)^I}\ox_{\O_{(\AA^n)^{\{v',v''\}}}}\mathbf{B}^{\{v',v''\},\{e\}}_{\AA^n}%=\frac{\kk[z^s_{i},\frac{u^s_e}{(z^s_e)^{\mathbf{m}_{e,s}}},\frac{\dd u^s_e}{(z^s_e)^{\mathbf{m}_{e,s}}}]^{1\leq s\leq n}_{i\in I,\mathbf{m}_{e,s}\geq 0}}{\langle \sum^n\limits_{s=1}u^s_e=1, \sum^n\limits_{s=1}\dd u^s_e=0,u^s_{v'v''}=-u^s_{v''v'}\rangle}
.
\]
% Similarly,
%  $ \TS_{\AA^n,\basepoint}^{I,\{e\}}$ is isomorphic to
%   \[
% \mathbf{B}^{I,\{e\}}_{\AA^n,\basepoint}=\O^I_{\AA^n,\basepoint}\ox_{\O^{\{v',v''\}}_{\AA^n,\basepoint}}\mathbf{B}^{\{v',v''\},\{e\}}_{\AA^n,\basepoint}
% %=\frac{\kk[z^s_{ij},\frac{u^s_e}{(z^s_e)^{\mathbf{m}_{e,s}}},\frac{\dd u^s_e}{(z^s_e)^{\mathbf{m}_{e,s}}}]^{1\leq s\leq n}_{(i,j)\in I^{[2]},\mathbf{m}_{e,s}\geq 0}}{\langle \sum^n\limits_{s=1}u^s_e=1, \sum^n\limits_{s=1}\dd u^s_e=0,u^s_{v'v''}=-u^s_{v''v'}\rangle}
% .
% \]
\end{cor}
\begin{proof}
  Since $\O_{(\AA^n)^{I}}=\kk[z^s_i]^{1\leq s\leq n}_{i\in I}$ is flat over $\O_{(\AA^n)^{\{v',v''\}}}=\kk[z^s_{v'},z^s_{v''}]^{1\leq s\leq n}$, we can tensor it with the exact sequence
  \[
  0\rightarrow \mathbf{B}^{\{v',v''\},\{e\}}_{\AA^n} \rightarrow  \mathbf{\Omega}_{\AA^n}^{\{v',v''\},\{e\}}\xrightarrow{\oplus^n_{s=1}\Res^{\{v,v'\},\{e\}}_e|_{u^s_e=0}} \mathbf{\Upsilon}^{\{v',v''\},\{e\}}_{\AA^n,u^1_e=0}  \oplus\cdots \oplus \mathbf{\Upsilon}^{\{v',v''\},\{e\}}_{\AA^n,u^n_e=0} \rightarrow 0
  \]
to get
\[
  0\rightarrow \mathbf{B}^{I,\{e\}}_{\AA^n} \rightarrow  \mathbf{\Omega}^{I,\{e\}}_{\AA^n}\xrightarrow{\oplus^n_{s=1}\Res^{I,\{e\}}_e|_{u^s_e=0}} \mathbf{\Upsilon}^{I,\{e\}}_{\AA^n,u^1_e=0}  \oplus\cdots \oplus \mathbf{\Upsilon}^{I,\{e\}}_{\AA^n,u^n_e=0} \rightarrow 0
  \]
\end{proof}

We also need the following simple lemma.

\begin{lem}\label{RegularBoth} Let $e,e'\in E\subset I\two$ and $s,s'\in \{1,\dots,n\}$ such that $(e,s) \neq (e',s')$.
We have
\begin{align} &\kk[z^r_i]^{1\leq r\leq n}_{i\in I}\left[  \frac 1{z^r_f}\right]_{\substack{(r,f) \in (\{1\dots,n\}\times E)\setminus\{(s,e)\}}} \cap
\kk[z^r_i]^{1\leq r\leq n}_{i\in I}\left[  \frac 1{z^r_f}\right]_{\substack{(r,f) \in (\{1\dots,n\}\times E)\setminus\{(s',e')\}}}\nn\\
& \qquad\qquad=
\kk[z^r_i]^{1\leq r\leq n}_{i\in I}\left[  \frac 1{z^r_f}\right]_{\substack{(r,f) \in (\{1\dots,n\}\times E)\setminus\{(s,e),(s',e')\} }},
\nn\end{align}
where the intersection is taken in $\kk[z^r_i]^{1\leq r\leq n}_{i\in I}\left[  \frac 1{z^r_f}\right]_{\substack{(r,f) \in (\{1\dots,n\}\times E)}} $.
\end{lem}
\begin{proof}
We write an element in the intersection as
\[
F(z)=\sum_{l'}\frac{F'_{l'}(z)}{(z^{s'}_{e'})^{l'}}=\sum_{l}\frac{F_l(z)}{(z^s_e)^l},
\]
where $F'_{l'}(z)$ and $F_l(z)$  are regular in both $z^s_e$ and $z^{s'}_{e'}$. We can find $N\gg0$ and set $G(z)=\left(\prod\limits_{(\tilde{e},\tilde{s})\neq (e,s),(e',s')}z^{\tilde{s}}_{\tilde{e}}\right)^N$ such that $G(z)F_l(z),G(z)F'_{l'}(z)\in \kk[z^r_i]^{1\leq r\leq n}_{i\in I}$. Then
\[
G(z)F(z)\in\kk[z_i^r,{\frac{1}{z^s_{e}}}]^{1\leq r\leq n}_{i\in I}\cap  \kk[z_i^r,{\frac{1}{z^{s'}_{e'}}}]^{1\leq r\leq n}_{i\in I}=\kk[z^r_i]^{1\leq r\leq n}_{i\in I}.
\]
We conclude that
\[
F(z)=\frac{G(z)F(z)}{G(z)}\in \kk[z^r_i]^{1\leq r\leq n}_{i\in I}\left[  \frac 1{z^r_f}\right]_{\substack{(r,f) \in (\{1\dots,n\}\times E)\setminus\{(s,e),(s',e')\} }}
\]
as required.
\end{proof}

Now we can complete the proof of \cref{thm: TSIE gens and rels}.
We prove the statement inductively by adding edges, with \cref{cor: oneedge} as the base case.
\begin{comment}
\begin{thm}\label{FlatnessTensor}
  $\TS_{\AA^n}^{I,E}$ is isomorphic to $\bigotimes_{\O^I_{\AA^n}}^{e\in E}\TS_{\AA^n}^{I,\{e\}}\simeq\bigotimes_{\O^I_{\AA^n}}^{e\in E}\mathbf{B}_{\AA^n}^{I,\{e\}}=\mathbf{B}_{\AA^n}^{I,E}$, in other words
\[
\TS_{\AA^n}^{I,E}\simeq\mathbf{B}_{\AA^n}^{I,E}=\frac{\kk[z^s_{i},\frac{u^s_e}{(z^s_e)^{\mathbf{m}_{e,s}}},\frac{\dd u^s_e}{(z^s_e)^{\mathbf{m}_{e,s}}}]^{1\leq s\leq n}_{i\in I,e\in E,\mathbf{m}_{e,s}\geq 0}}{\langle \sum^n\limits_{s=1}u^s_e=1, \sum^n\limits_{s=1}\dd u^s_e=0,u^s_{i_1i_2}=-u^s_{i_2i_1}\rangle}.
\]
% Similarly,
% \[
% \TS_{\AA^n,\basepoint}^{I,E}\simeq\mathbf{B}_{\AA^n,\basepoint}^{I,E}=\frac{\kk[z^s_{ij},\frac{u^s_e}{(z^s_e)^{\mathbf{m}_{e,s}}},\frac{\dd u^s_e}{(z^s_e)^{\mathbf{m}_{e,s}}}]^{1\leq s\leq n}_{(i,j)\in I^{[2]},e\in E,\mathbf{m}_{e,s}\geq 0}}{\langle \sum^n\limits_{s=1}u^s_e=1, \sum^n\limits_{s=1}\dd u^s_e=0,u^s_{i_1i_2}=-u^s_{i_2i_1}\rangle}.
% \]
\end{thm}
\begin{proof}
\end{comment}

So suppose, inductively, that we already have $\TS^{I,E}_{\AA^n}\cong\mathbf{B}_{\AA^n}^{I,E}$.
It is manifest from the definition that
$\mathbf{B}^{I,E\sqcup \{e\}}_{\AA^n} \cong \mathbf{B}^{I,E}_{\AA^n} \ox_{\O_{(\AA^n)^I}} \mathbf{B}^{I,\{e\}}_{\AA^n}$. Therefore, by our inductive assumption and \cref{cor: oneedge} we have \[\mathbf{B}^{I,E\sqcup \{e\}}_{\AA^n} \cong \TS^{I,E}_{\AA^n} \ox_{\O_{(\AA^n)^I}} \TS^{I,\{e\}}_{\AA^n}.\]
Thus, our goal is to show that $\TS^{I,E\sqcup \{e\}}_{\AA^n}$ is isomorphic to $\TS_{\AA^n}^{I,E}\ox_{\O_{(\AA^n)^I}}\TS_{\AA^n}^{I,\{e\}}$. From \cref{cor: oneedge} we have the short exact sequence
  \[
  0\rightarrow \mathbf{B}^{I,\{e\}}_{\AA^n} \rightarrow  \mathbf{\Omega}_{\AA^n}^{I,\{e\}}\xrightarrow{\mathop{\oplus}\limits^n_{s=1}\Res^{I,\{e\}}_e|_{u^s_e=0}} \mathbf{\Upsilon}^{I,\{e\}}_{\AA^n,u^1_e=0}  \oplus\cdots \oplus \mathbf{\Upsilon}^{I,\{e\}}_{\AA^n,u^n_e=0} \rightarrow 0.
  \]
The rings $\TS_{\AA^n}^{I,E}=\mathbf{B}_{\AA^n}^{I,E}, \mathbf{\Omega}_{\AA^n}^{I,\{e\}},\mathbf{\Omega}_{\AA^n}^{I,E}$ and $\TS_{\AA^n}^{I,\{e\}}$ are flat over $\O_{(\AA^n)^I}=\kk[z^s_i]^{1\leq s\leq n}_{i\in I}$. (The fact that $\mathbf{B}_{\AA^n}^{I,E}$ is flat over $\O_{(\AA^n)^I}$ is proved in \cref{lem: FlatE} below.) Therefore we have the commutative diagram
  \[
  % https://tikzcd.yichuanshen.de/#N4Igdg9gJgpgziAXAbVABwnAlgFyxMJZAJgBoBGAXVJADcBDAGwFcYkQAdDgIywHMIaFnAD6wGKTgBfLgBEYjHPTHMAeqPEByKQF4ADDI4Q8AW3giASlwDyZvsrZTS6TLnyEUZYtTpNW7GztlLgBxehMTei5jLDNRKw5bGHsRR2cQDGw8AiIyPR8GFjZEED0QJxcs91zSAGYCv2LODkicAAtuADNgABUAZSkxUPDIwxi4y0Dkh3L0zLccz1IAFgai9jKKjNdsj2RaurX-Eq5Wju7+weBhiPox03ME3gEhZg1pOQUlFXUxGF0DLNKgs9stDjRCsdSkDtlVFvtSN4IY0AokgkMOGFbvdYo8uM9BMIxB8OPJFMpgGoNP99FIYfNdkQwUjfOsSps5jtqihyIijk0pilrpiRndog94qd6O0ur0Bql6Vz4bz6si2c0zrLLhisaNxbjJS1pec5YM0sDGTyKPzUQTXtTJIYyd9Kb8tACcRMEpqLvLzbCQURefk1VCORbucheatQ01wwHLcg9ODWWHFXC9smWZC4+nAygDjHU7mpD4YFA+PAiKBOgAnCAmJC8kA4CBIAzpOsN9s0VtIYhbLuNxAHFttxDkQf14dgsdIWpT7sj3vjgCsi+Hq5XSAAbBukFu54gAOw0RhYMBNOAQc9QECx9gXgBWMAAxnhaP6h0hT0eABz7hO24noBACcwGTp205NsmR4LlBS7kM2faIB2NbQROsEoaBYEQZB6GIaOKHkGhIDfhOs7EaR5HkIe2GAeQO7AeuCHDkhEEDpQUhAA
\begin{tikzcd}
            & 0                                                                              & 0                                                               &                                                                                               &   \\
            & {\bigoplus\limits_{e'\in E}^{1\leq s\leq n}\mathbf{\Upsilon}^{I,E}_{u^s_{e'}=0}\otimes\TS^{I,\{e\}}} \arrow[r] \arrow[u] & {\bigoplus\limits_{e'\in E}^{1\leq s\leq n}\mathbf{\Upsilon}^{I,E}_{u^s_{e'}=0}\otimes\mathbf{\Omega}^{I,\{e\}}} \arrow[u] &                                                                                               &   \\
0 \arrow[r] & \mathbf{\Omega}^{I,E}\otimes\TS^{I,\{e\}} \arrow[r] \arrow[u]                      & \mathbf{\Omega}^{I,E}\otimes\mathbf{\Omega}^{I,\{e\}} \arrow[u,"{\mathop{\oplus}^{1\leq s\leq n}\limits_{e'\in E}\Res^{I,E}_{e'}|_{u^s_{e'}=0}}"] \arrow[r,"
\small
{}^{{}^{{}^{\mathop{\oplus}^n\limits_{s=1}\Res^{I,\{e\}}_e|_{u^s_e=0}}}}
"]              & \mathbf{\Omega}^{I,E}\otimes\bigoplus^n\limits_{s=1}\mathbf{\Upsilon}^{I,\{e\}}_{u^s_{e}=0} \arrow[r]                             & 0 \\
0 \arrow[r] & \TS^{I,E}\otimes\TS^{I,\{e\}} \arrow[r] \arrow[u]                 & \TS^{I,E}\otimes\mathbf{\Omega}^{I,\{e\}} \arrow[u] \arrow[r]       & \TS^{I,E}\otimes\bigoplus^n\limits_{s=1}\mathbf{\Upsilon}^{I,\{e\}}_{u^s_{e}=0} \arrow[r] \arrow[u, "\iota"] & 0 \\
            & 0 \arrow[u]                                                                    & 0 \arrow[u]                                                     & 0 \arrow[u]                                                                                   &
\end{tikzcd}
  \]
  Here all tensor products $\otimes$ are over $\O_{(\AA^n)^I}=\kk[z^s_i]^{1\leq s\leq n}_{i\in I}$ and we have omitted the subscript ${\AA^n}$.

  From the diagram we know that the map \[\TS_{\AA^n}^{I,E}\ox_{{\O_{(\AA^n)^I}}} \TS^{I,\{e\}}_{\AA^n}\rightarrow \mathbf{\Omega}^{I,E}_{\AA^n}\ox_{{\O_{(\AA^n)^I}}}\mathbf{\Omega}^{I,\{e\}}_{\AA^n}=\mathbf{\Omega}_{\AA^n}^{I,E\sqcup \{e\}}\]
  is injective and its image is contained in $\TS_{\AA^n}^{I,E\sqcup \{e\}}\subset   \mathbf{\Omega}_{\AA^n}^{I,E\sqcup \{e\}}$. We claim that this map is also onto $\TS_{\AA^n}^{I,E\sqcup \{e\}}$: that is, $
  \TS^{I,E\sqcup \{e\}}_{\AA^n}=\TS^{I,E}_{\AA^n}\ox_{{\O_{(\AA^n)^I}}} \TS_{\AA^n}^{I,\{e\}}.$

  We need to show that if $\alpha\otimes \beta\in \mathbf{\Omega}_{\AA^n}^{I,E}\ox_{{\O_{(\AA^n)^I}}}\mathbf{\Omega}_{\AA^n}^{I,\{e\}}$ such that
  \[
  \Res^{I,E}_{e'}|_{u^s_{e'}=0}(\alpha)\otimes \beta=\alpha\otimes \Res^{I,\{e\}}_e|_{u^s_e=0}(\beta)=0,\quad e'\in E,1\leq s\leq n,
  \]
  then $\alpha\otimes\beta\in \TS_{\AA^n}^{I,E}\ox_{{\O_{(\AA^n)^I}}} \TS_{\AA^n}^{I,\{e\}}$. We first conclude that $\alpha\otimes\beta\in \TS^{I,E}_{\AA^n}\ox_{{\O_{(\AA^n)^I}}}\mathbf{\Omega}^{I,\{e\}}_{\AA^n}$.
Then it is enough to show that
\[
\iota:\TS_{\AA^n}^{I,E}\ox_{{\O_{(\AA^n)^I}}}\bigoplus^n\limits_{s=1}\mathbf{\Upsilon}^{I,\{e\}}_{\AA^n,u^s_{e}=0}  \rightarrow \mathbf{\Omega}^{I,E}_{\AA^n}\ox_{{\O_{(\AA^n)^I}}}\bigoplus^n\limits_{s=1}\mathbf{\Upsilon}^{I,\{e\}}_{\AA^n,u^s_{e}=0}
\]
is injective. Suppose that $\eta\otimes \sum\limits_{s=1}^n \gamma_s\in\TS^{I,E}_{\AA^n}\ox_{\O_{(\AA^n)^I}}\bigoplus\limits_{s=1}^n\mathbf{\Omega}^{I,\{e\}}_{{\AA^n},(e,\del_s)}$ is such that
\[
\iota\left([\eta\otimes \sum_{s=1}^n \gamma_s]\right)=0,\quad [\eta\otimes \sum_{s=1}^n \gamma_s]\in \TS^{I,E}_{\AA^n}\ox_{{\O_{(\AA^n)^I}}}\bigoplus^n\limits_{s=1}\mathbf{\Upsilon}^{I,\{e\}}_{\AA^n,u^s_{e}=0}.
\]
%Use the flatness of $ \Omega_{\Gamma}$ over $R$,
Given this, we can reexpress
\[
\eta\otimes \sum_{s=1}^n \gamma_s=\eta'\otimes \sum_{s=1}^n \gamma'_s\in \mathbf{\Omega}^{I,E}_{\AA^n}\ox_{{\O_{(\AA^n)^I}}}\bigoplus^n\limits_{s=1}\mathbf{\Omega}^{I,\{e\}}_{\AA^n,(e,\del_s)^\mathbf{reg}}.
\]
Then $\eta|_{u^{s'}_{e'}=0}\otimes \gamma_s=\eta'|_{u^{s'}_{e'}=0}\otimes  \gamma'_s$ is in the intersection
\[
 \left(\mathbf{\Omega}^{I,E}_{\AA^n,(e',\del_{s'})^\mathbf{reg}}\ox_{\O_{(\AA^n)^I}}\mathbf{\Omega}^{I,\{e\}}_{\AA^n,(e,\del_s)}\right)\cap \left(\mathbf{\Omega}^{I,E}_{\AA^n,(e',\del_{s'})}\ox_{\O_{(\AA^n)^I}}\mathbf{\Omega}^{I,\{e\}}_{\AA^n,(e,\del_s)^\mathbf{reg}}\right).
\]
Notice that we have
\[
\mathbf{\Omega}^{I,E}_{\AA^n,(e',\del_{s'})}\ox_{\O_{(\AA^n)^I}}\mathbf{\Omega}^{I,\{e\}}_{\AA^n,(e,\del_{s})}={\Omega}^{I,E}_{n,(e',\del_{s'})}\ox_{\kk}{\Omega}^{I,\{e\}}_{n,(e,\del_s)}\ox_{\kk}\kk[z^r_i,\frac 1{z^t_{f}}]^{(r,i)\in \{1,\dots,n\}\times I}_{(t,f)\in (\{1,\dots,n\}\times E\sqcup\{e\})}.
\]
From Lemma \ref{RegularBoth}, we get
\[
\eta|_{u^{s'}_{e'}=0}\otimes\gamma_s=\eta'|_{u^{s'}_{e'}=0}\otimes  \gamma'_s\in \left(\mathbf{\Omega}^{I,E}_{\AA^n,(e',\del_{s'})^\mathbf{reg}}\ox_{\O_{(\AA^n)^I}}\mathbf{\Omega}^{I,\{e\}}_{\AA^n,(e,\del_s)^\mathbf{reg}}\right).
\]
From the short exact sequence
\[
\TS^{I,E}_{\AA^n}\ox_{\O_{(\AA^n)^I}}\mathbf{\Omega}^{I,\{e\}}_{\AA^n,(e,\del_s)^\mathbf{reg}} \rightarrow\mathbf{\Omega}_{\AA^n}^{I,E}\ox_{\O_{(\AA^n)^I}}\mathbf{\Omega}^{I,\{e\}}_{\AA^n,(e,\del_s)^\mathbf{reg}}\xrightarrow{\mathop{\oplus}^{1\leq s\leq n}\limits_{e'\in E}\Res^{I,E}_{e'}|_{u^s_{e'}=0}} \bigoplus\limits_{e'\in E}^{1\leq s\leq n}\mathbf{\Upsilon}^{I,E}_{\AA^n,u^s_{e'}=0}\ox_{\O_{(\AA^n)^I}}\mathbf{\Omega}^{I,\{e\}}_{\AA^n,(e,\del_s)^\mathbf{reg}},
\]
we get
\[
\eta\otimes \sum_{s=1}^n \gamma_s=\eta'\otimes \sum_{s=1}^n \gamma'_s\in  \TS^{I,E}_{\AA^n}\ox_{\O_{(\AA^n)^I}}\bigoplus\limits_{s=1}^n\mathbf{\Omega}^{I,\{e\}}_{\AA^n,(e,\del_s)^\mathbf{reg}} .
\]
Thus
\[
[\eta\otimes \sum_{s=1}^n \gamma_s]=0\in \TS^{I,E}_{\AA^n}\ox_{{\O_{(\AA^n)^I}}}\bigoplus^n\limits_{s=1}\mathbf{\Upsilon}^{I,\{e\}}_{\AA^n,u^s_{e}=0}.
\]
Thus the map $\iota$ is injective, as required. This completes the inductive step, and therefore the proof of \cref{thm: TSIE gens and rels}.

\subsection{A lemma on flatness}
In the proof above, and also in \cref{sec: cohomology}, we required the following result.
\begin{lem}\label{lem: FlatE}
 The ring $\mathbf{B}^{I,E}_{\AA^n}$ is flat over $\O_{(\AA^n)^I}.$
\end{lem}
\begin{proof}
It is manifest from the definition, \cref{sec: gens and rels for TS}, of $\mathbf{B}^{I,E}_{\AA^n}$ that  \begin{equation} \mathbf{B}_{\AA^n}^{I,E}=\bigotimes_{e\in E}\mathbf{B}_{\AA^n}^{I,\{e\}},\nn\end{equation}
the tensor product being taken in $\O_{(\AA^n)^I}$-modules.
Therefore, it is enough to show that $\mathbf{B}_{\AA^n}^{I,\{e\}}$ is flat over $\O_{(\AA^n)^I}$, for if so then the tensor product above is a tensor product of flat $\O^I_{\AA^n}$-modules and thus a flat $\O_{(\AA^n)^I}$-module.

Our goal is thus to prove the following:
The ring $\mathbf{B}^{I,\{e\}}_{\AA^n}$ is flat over $\O_{(\AA^n)^I}$.

We have the sequence of inclusions
\[
\O_{(\AA^n)^I}\into  \mathbf{B}^{I,\{e\},0}_{\AA^n}=
\directlim_{m\geq0}\kk\left[ z^r_i, \frac{u^r_{e}}{(z^r_{e})^m}\right]^{1\leq r\leq n}_{i\in I}\big/\langle \sum_{r=1}^{n} u^r_{e} -1\rangle
\into \mathbf{B}^{I,\{e\}}_{\AA^n}.\]
Given any fixed positive integer $m$, consider% the affine scheme
\[ Y_m := \Spec\left( \kk\left[ z^r_i, \frac{u^r_{e}}{(z^r_{e})^m}\right]^{1\leq r\leq n}_{i\in I}\big/\langle \sum_{r=1}^{n} u^r_{e} -1\rangle \right).\]
We have %(recall here $z_e = z_i-z_j$ where $e=(i,j)$)
\[ Y_m =\{(z^r_i,y^r)^{1\leq r\leq n}_{i\in I} : \sum^n_{s=1}(z^s_e)^{m}y^s=1\}\subset(\AA^{n})^k\times \AA^{n}.\]
Notice that the affine variety $Y_{m}$ is non-singular and every fiber of the morphism $Y_{m}\rightarrow (\AA^{n})^k$ is of the same dimension $n-1$. Hence $Y_{m}$ is flat over $(\AA^{n})^k$ by the miracle flatness (\cite[\href{https://stacks.math.columbia.edu/tag/00R4}{Tag 00R4}]{stacks-project}).

Thus the ring $\mathbf{B}^{I,\{e\},0}_{\AA^n}$ is the direct limit of flat rings, and is therefore itself flat.

Now we prove the second inclusion, $\mathbf{B}^{I,\{e\},0}_{\AA^n} \into \mathbf{B}^{I,\{e\}}_{\AA^n}$, is also flat.
The degree 1 component $\mathbf{B}^{I,\{e\},1}_{\AA^n}$ is a $\mathbf{B}^{I,\{e\},0}_{\AA^n}$-module and
\[ \mathbf{B}^{I,\{e\}}_{\AA^n}=\bigwedge{}_{\mathbf{B}^{I,\{e\},0}_{\AA^n}}\mathbf{B}^{I,\{e\},1}_{\AA^n}.\]
We only need to prove that $\mathbf{B}^{I,\{e\},1}_{\AA^n}$ is flat over $\mathbf{B}^{I,\{e\},0}_{\AA^n}$. We shall use the equational criterion for flatness. Namely, suppose that we have a relation in $\mathbf{B}^{I,\{e\},1}_{\AA^n}$,
\[\sum_{s=1}^n F_s\cdot \frac{\dd u^s_e}{(z^s_e)^{\ell_s}}=0,\quad F_s\in\mathbf{B}^{I,\{e\},0}_{\AA^n}.\]
We must show that it is possible to rewrite $\frac{\dd u^s_e}{(z^s_e)^{\ell_s}} = \sum\limits_{j} a_{sj} b_j$, for some elements $b_j \in \mathbf{B}^{I,\{e\},1}_{\AA^n}$ and coefficients $a_{ij}\in \mathbf{B}^{I,\{e\},0}_{\AA^n}$, in such a way that for each $j$ we have the relations $\sum\limits_{s=1}^n F_s a_{sj} = 0$ in $\mathbf{B}^{I,\{e\},0}_{\AA^n}$.

To that end, first observe that we have the relations
\begin{equation} F_s \cdot (z^r_e)^{\ell_r} - F_r \cdot (z^s_e)^{\ell_s} =0 \quad\text{in}\quad \mathbf{B}^{I,\{e\},0}_{\AA^n}. \label{star}
\end{equation}
To see this, we note that $\mathbf{B}^{I,\{e\},1}_{\AA^n}$ embeds into $\kk[z_i^s, \frac 1 {z^s_e},u^s_e]^{1\leq s\leq n}_{i\in I}\{\dd u^s_e\}^{1\leq s\leq n}\big/\langle \sum_{r=1}^{n} \dd u^r_{e}\rangle$. There, we may eliminate $\dd u_e^n$ and then compare coefficients of the remaining $\dd u^r_e$ to find that $\frac{F_r}{(z^r_e)^{\ell_r}}=\frac{F_n}{(z^n_e)^{\ell_n}}$ for each $r$, and hence $\frac{F_s}{(z^s_e)^{\ell_s}}=\frac{F_t}{(z^t_e)^{\ell_t}}$ for all $1\leq s<t\leq n$. The relations \cref{star} in $\mathbf{B}^{I,\{e\},0}_{\AA^n}$ follow.

Now we rewrite
\begin{gather}
\frac{\dd u^s_e}{(z^s_e)^{\ell_s}}=\frac{\dd u^s_e}{(z^s_e)^{\ell_s}} \sum_{t=1}^n u^t_e=\sum_{\substack{t=1\\t\neq s}}^n(z^t_e)^{\ell_t}\cdot\left(\frac{u^t_e \dd u^s_e}{(z^s_e)^{\ell_s} (z^t_e)^{\ell_t}}\right)+\frac{u^s_e \dd u^s_e}{(z^s_e)^{\ell_s}}\nn\\
%&=\sum_{\substack{t=1\\t\neq s}}^n(z^t_e)^{\ell_t} \left(\frac{u^t_e \dd u^s_e}{(z^s_e)^{\ell_s} (z^t_e)^{\ell_t}}\right)-\sum_{\substack{t=1\\t\neq s}}^n(z^t_e)^{\ell_t} \left(\frac{u^s_e du^t_e}{(z^s_e)^{\ell_s} (z^t_e)^{\ell_t}}\right)\nn\\
=\sum_{\substack{t=1\\t\neq s}}^n(z^t_e)^{\ell_t} \cdot \left(\frac{u^t_e \dd u^s_e -u^s_e du^t_e}{(z^s_e)^{\ell_s} (z^t_e)^{\ell_t}}\right),\quad \text{using}\ \sum^n_{t=1}du^t_e=0,\nn
\end{gather}
and hence find that the relation in  $\mathbf{B}^{I,\{e\},1}_{\AA^n}$  comes from the relations \cref{star} in $\mathbf{B}^{I,\{e\},0}_{\AA^n}$:
\begin{gather}
\sum_{s=1}^n F_s\cdot\frac{\dd u^s_e}{(z^s_e)^{\ell_s}}=\sum_{\substack{s,t=1\\s\neq t}}^n
F_s \cdot (z^t_e)^{\ell_t} \cdot\left(\frac{u^t_e \dd u^s_e -u^s_e du^t_e}{(z^s_e)^{\ell_s} (z^t_e)^{\ell_t}}\right)\nn\\
=\sum_{\substack{s,t=1\\s<t}}^n (F_s\cdot (z^t_e)^{\ell_t}-F_t\cdot (z^s_e)^{\ell_s})\cdot \left(\frac{u^t_e \dd u^s_e -u^s_e du^t_e}{(z^s_e)^{\ell_s} (z^t_e)^{\ell_t}}\right).\nn
\end{gather}
This establishes that   $\mathbf{B}^{I,\{e\},1}_{\AA^n}$ is flat over   $\mathbf{B}^{I,\{e\},0}_{\AA^n}$, completing the proof.
\end{proof}

\printbibliography
%\bibliography{bibliography}
%\bibliographystyle{amsalpha}

\end{document}

% Local Variables:
% TeX-command-extra-options: "-shell-escape"
% End:

%********************************